\documentclass[11pt, a4paper, english]{smfart}

\usepackage[utf8]{inputenc}

\usepackage{amsmath}
\usepackage{amssymb}
\usepackage{amsfonts}
\usepackage{amsthm}
\usepackage{mathrsfs}
\usepackage{mathtools}

\usepackage{tikz-cd}
\tikzcdset{scale cd/.style={every label/.append style={scale=#1},
    cells={nodes={scale=#1}}}}
\usepackage{tikz}
\usetikzlibrary{patterns,shapes,arrows}
\usepackage{array}
\usepackage{setspace}
\usepackage{url}
\usepackage{enumerate}
\usepackage{comment}
\usepackage{graphicx}
\usepackage{hyperref}
\hypersetup{colorlinks=true,linkcolor=black,citecolor=black,urlcolor=black}
\usepackage{caption}
\usepackage{subcaption}
\usepackage{stmaryrd}

\swapnumbers

\newtheoremstyle{numpar}
{\topsep}			
{\topsep}			
{}			
{}			
{\bfseries}			
{ }			
{5pt plus 1pt minus 1pt}			
{}			

\theoremstyle{numpar}
\newtheorem{para}[subsubsection]{\unskip}

\newtheoremstyle{mythm}
  {\topsep}   
  {\topsep}   
  {\itshape}  
  {0pt}       
  {\bfseries} 
  {.}         
  {5pt plus 1pt minus 1pt} 
  {}          

\theoremstyle{mythm}
\newtheorem{assumption}[subsubsection]{Assumption}
\newtheorem{theorem}[subsubsection]{Theorem}
\newtheorem{lemma}[subsubsection]{Lemma}
\newtheorem{prop}[subsubsection]{Proposition}
\newtheorem{conjecture}[subsubsection]{Conjecture}
\newtheorem{corollary}[subsubsection]{Corollary}

\newtheoremstyle{myrmk}
  {\topsep}   
  {\topsep}   
  {}  
  {0pt}       
  {\bfseries} 
  {.}         
  {5pt plus 1pt minus 1pt} 
  {}          
  
\theoremstyle{myrmk}
\newtheorem{definition}[subsubsection]{Definition}
\newtheorem{example}[subsubsection]{Example}

\theoremstyle{myrmk}
\newtheorem{remark}[subsubsection]{Remark}

\numberwithin{equation}{subsection}

\newcommand{\ints}{\mathbf{Z}}
\newcommand{\rationals}{\mathbf{Q}}

\newcommand{\cO}{{\mathcal{O}}}
\newcommand{\rec}{\operatorname{rec}}

\newcommand{\Frob}{\operatorname{Frob}}
\newcommand{\Hom}{\operatorname{Hom}}
\newcommand{\Ind}{\operatorname{Ind}}

\newcommand{\End}{\operatorname{End}}
\newcommand{\id}{\operatorname{id}}

\newcommand{\Gal}{\operatorname{Gal}}

\newcommand{\Spec}{\operatorname{Spec}}

\newcommand{\Spf}{\operatorname{Spf}}

\newcommand{\GL}{\operatorname{GL}}

\newcommand{\ad}{\operatorname{ad}}

\newcommand{\im}{\operatorname{im}}

\newcommand{\Res}{\operatorname{Res}}
\newcommand{\gr}{\operatorname{gr}}
\newcommand{\Mod}{\operatorname{Mod}}
\newcommand{\coker}{\operatorname{coker}}

\allowdisplaybreaks

\title[Local-global compatibility]{Local-global compatibility of automorphic Galois representations over CM fields at \texorpdfstring{$p$}{p}}

\author{L. A'Campo \and B. Hevesi \and J. A. Thorne \and D. Whitmore }

\begin{document}

\begin{abstract}
    Let $F$ be a CM number field; then, to any cuspidal, regular algebraic automorphic representation of $\GL_n(\mathbf{A}_F)$ is associated a compatible system of $p$-adic Galois representations of the absolute Galois group of $F$. We prove that these representations are potentially semi-stable, in the sense of $p$-adic Hodge theory, and satisfy compatibility with the local Langlands correspondence, up to semi-simplification. 
\end{abstract}

\maketitle

\tableofcontents

\section{Introduction}

\subsection{Context} Let $F$ be a number field, and let $n \geq 1$. According to the philosophy of the Langlands programme, automorphic representations of $\GL_n(\mathbf{A}_F)$ should be related to Galois representations, i.e.\ $n$-dimensional representations of the absolute Galois group $G_F = \Gal(\overline{F} / F)$. A conjectural answer to the question of which automorphic representations should participate in this correspondence was given by Clozel \cite{Clo90}, who singled out those representations that are `algebraic' (a condition depending only on the infinite component $\pi_\infty$ of an automorphic representation $\pi$). Clozel made the following conjecture:
\begin{conjecture}\label{introconj_existence} Let $\pi$ be a cuspidal, algebraic automorphic representation of $\GL_n(\mathbf{A}_F)$. Then:
\begin{enumerate}
    \item The finite part $\pi^\infty$ (an irreducible admissible representation of $\GL_n(\mathbf{A}_F^\infty)$) is defined over a number field $K_\pi \leq  \mathbf{C}$.
    \item For any prime $p$ and embedding $K_\pi \to \overline{\mathbf{Q}}_p$, there is a continuous, semi-simple representation $r_{\pi, p} \colon G_F \to \GL_n(\overline{\mathbf{Q}}_p)$ with the following property: for any finite place $v \nmid p$ such that $\pi_v$ is unramified, the local representation $r_{\pi, p}|_{G_{F_v}}$ is unramified, and $r_{\pi, p}(\Frob_v)^{\textnormal{ss}} \in \GL_n(\overline{\mathbf{Q}}_p)$ is conjugate to the (Tate--normalised) Satake parameter $t_{\pi_v}^T \in \GL_n(\overline{\mathbf{Q}}_p)$.
\end{enumerate}
\end{conjecture}
(Here $G_{F_v} \leq G_F$ denotes the decomposition group, defined up to conjugation, and $\Frob_v \in G_{F_v}$ is a choice of geometric Frobenius lift; see \S \ref{subsec_notation} below for any further undefined notation.) The Chebotarev density theorem implies that $r_{\pi, p}$ is determined uniquely up to isomorphism by the matrices $t_{\pi_v}^T$. Henceforth we usually avoid mentioning $K_\pi$ and speak in terms of a choice of isomorphism $\iota \colon \overline{\mathbf{Q}}_p \to \mathbf{C}$; we obtain a representation $r_{\pi, \iota}$ that is characterized by the relation $r_{\pi, \iota}(\Frob_v)^{\textnormal{ss}} \sim \iota^{-1} t_{\pi_v}^T$ at unramified places. 

If $\pi$ is a regular algebraic automorphic representation, then it is cohomological, and Clozel proved in \cite{Clo90} that it is defined over a number field. If further $F$ is a CM field, then  Harris--Lan--Taylor--Thorne \cite{Har16} and Scholze \cite{Sch15} proved that $r_{\pi, \iota}$ can be constructed by approximating it (or more precisely, by approximating $r_{\pi, \iota} \oplus r_{\pi, \iota}^{c, \vee} \otimes \epsilon^{1-2n}$) by Galois representations associated to conjugate self-dual automorphic representations, whose properties were already well-understood. 

A more refined version of Conjecture \ref{introconj_existence} would ask for a description of $r_{\pi, \iota}|_{G_{F_v}}$ at all places, not just the unramified ones. This is the requirement of local-global compatibility, phrased in terms of the local Langlands correspondence. Recall that if $v$ is a finite place of $F$, then the (Tate-normalised) local Langlands correspondence $\rec_{F_v}^T$ \cite{HT01} gives a bijection between following two sets:
\begin{itemize}
    \item The set of isomorphism classes of $n$-dimensional Frobenius-semi-simple Weil--Deligne representations $(r, N)$ of $W_{F_v}$ over $\mathbf{C}$.
    \item The set of isomorphism classes of irreducible admissible representations of $\GL_n(F_v)$ over $\mathbf{C}$.
\end{itemize}
(The Tate-normalised correspondence has the virtue that it is invariant under field automorphisms, so $\mathbf{C}$ can be replaced by any abstractly isomorphic field here, such as $\overline{\mathbf{Q}}_p$.) It extends the construction of Satake parameters, in the sense that if $\pi_v$ is unramified then $\rec^T_{F_v}(\pi_v)$ is unramified, and $\rec^T_{F_v}(\pi_v)(\Frob_v) \sim t^T_{\pi_v}$. We state a refinement of Conjecture \ref{introconj_existence} in the case that $\pi$ is regular algebraic (not just algebraic), in which case there is also a convenient recipe for the expected Hodge--Tate weights:
\begin{conjecture}\label{introconj_lgc}
    Let $\pi$ be a cuspidal, regular algebraic automorphic representation of $\GL_n(\mathbf{A}_F)$ of weight $\lambda$, and let $\iota \colon \overline{\mathbf{Q}}_p \to \mathbf{C}$ be an isomorphism. Then there exists a continuous, semi-simple representation $r_{\pi, \iota} \colon G_F \to \GL_n(\overline{\mathbf{Q}}_p)$ satisfying the following conditions:
    \begin{enumerate}
        \item For any finite, prime-to-$p$ place $v$ of $F$, there is an isomorphism 
        \[ \mathrm{WD}(r_{\pi, \iota}|_{G_{F_v}})^{F-ss} \cong \iota^{-1}\rec_{F_v}^T(\pi_v) \]
        of Weil--Deligne representations of $W_{F_v}$.
        \item For any place $v | p$ of $F$, the representation $r_{\pi, \iota}|_{G_{F_v}}$ is de Rham of Hodge--Tate weights 
        \[ \mathrm{HT}_\tau(r_{\pi, \iota}) = \{ \lambda_{\iota \tau, 1} + (n-1), \lambda_{\iota \tau, 2} + (n-2), \dots, \lambda_{\iota \tau, n} \}, \]
        and there is an isomorphism 
        \[ \mathrm{WD}(r_{\pi, \iota}|_{G_{F_v}})^{F-ss} \cong \iota^{-1} \rec_{F_v}^T(\pi_v) \]
        of Weil--Deligne representations of $W_{F_v}$.
    \end{enumerate}
\end{conjecture}
Here we recall that the weight $\lambda = (\lambda_\tau)_\tau \in (\mathbf{Z}^n)^{\Hom(F, \mathbf{C})}$ is defined by the requirement that $\pi_\infty$ have the same infinitesimal character as $V_\lambda^\vee$, where $V_\lambda$ is the algebraic representation of $\GL_n(F \otimes_\mathbf{Q} \mathbf{C}) \cong \prod_\tau \GL_n(\mathbf{C})$ of highest weight $\lambda$. The definition of the Weil--Deligne representation of a continuous representation $r_v \colon G_{F_v} \to \GL_n(\overline{\mathbf{Q}}_p)$ depends, in the case $v \nmid p$, on Grothendieck's $\ell$-adic monodromy theorem, see e.g.\ \cite{Tat79}. In contrast, when $v | p$, it is defined only under the condition that $r_v$ is de Rham (equivalently, potentially semi-stable), using a recipe of Fontaine \cite{Fon94}. It is this conjecture that we study in this paper. In fact, we will be able to prove results only towards a slightly weaker, `semi-simplified' version:
\begin{conjecture}\label{introconj_lgc_ss}
    Let $\pi$ be a cuspidal, regular algebraic automorphic representation of $\GL_n(\mathbf{A}_F)$ of weight $\lambda$, and let $\iota \colon \overline{\mathbf{Q}}_p \to \mathbf{C}$ be an isomorphism. Then there exists a continuous, semi-simple representation $r_{\pi, \iota} \colon G_F \to \GL_n(\overline{\mathbf{Q}}_p)$ satisfying the following conditions:
    \begin{enumerate}
        \item For any finite, prime-to-$p$ place $v$ of $F$, there is an isomorphism 
        \[ \mathrm{WD}(r_{\pi, \iota}|_{G_{F_v}})^{\textnormal{ss}} \cong \iota^{-1}\rec_{F_v}^T(\pi_v)^{\textnormal{ss}} \]
        of Weil--Deligne representations of $W_{F_v}$.
        \item For any place $v | p$ of $F$, the representation $r_{\pi, \iota}|_{G_{F_v}}$ is de Rham of Hodge--Tate weights 
        \[ \mathrm{HT}_\tau(r_{\pi, \iota}) = \{ \lambda_{\iota \tau, 1} + (n-1), \lambda_{\iota \tau, 2} + (n-2), \dots, \lambda_{\iota \tau, n} \}, \]
        and there is an isomorphism 
        \[ \mathrm{WD}(r_{\pi, \iota}|_{G_{F_v}})^{\textnormal{ss}} \cong \iota^{-1} \rec_{F_v}^T(\pi_v)^{\textnormal{ss}} \]
        of Weil--Deligne representations of $W_{F_v}$.
    \end{enumerate}
\end{conjecture}
(This is the same as Conjecture \ref{introconj_lgc}, up to `forgetting the $N$'.) Before coming to our results, we recall what is known about Conjectures \ref{introconj_lgc} and \ref{introconj_lgc_ss}. As mentioned above, a candidate representation $r_{\pi, \iota}$ satisfying local-global compatibility at all but finitely many unramified places has already been constructed, so the question is really whether this representation satisfies local-global compatibility at every place. The following results are known:
\begin{itemize}
    \item If $\pi$ is conjugate self-dual, then Conjecture \ref{introconj_lgc} is known to hold in complete generality (see e.g.\ \cite{Car14} for the capstone in a long series of works by many authors on this problem). 
    \item If $v \nmid p$, then Varma proved \cite{Var24} that $r_{\pi, \iota}$ satisfies semi-simplified local-global compatibility at $v$, using a refinement of the `$p$-adic approximation by overconvergent conjugate self-dual cusp forms' approach given in \cite{Har16}. This approach is harder to apply when $v | p$, because $p$-adic limits of de Rham Galois representations of $G_{F_v}$ need not be de Rham. 
    \item If $v | p$, then many cases of local-global compatibility at $v$ are known, provided that $r_{\pi, \iota}$ satisfies auxiliary global hypotheses, such as that the residual representation $\overline{r}_{\pi, \iota}$ be irreducible, and that there exist a `decomposed generic' prime $\ell \neq p$, such that at places $w | \ell$ of $F$, $\overline{r}_{\pi, \iota}|_{G_{F_w}}$ is unramified, with Frobenius eigenvalues that are `as generic as possible'. As a representative example, we cite e.g. \cite[Theorem 1.1]{Hev24}, which proves Conjecture \ref{introconj_lgc_ss} under conditions of this type, using an elaboration of the degree-shifting argument introduced in \cite{10author}. 
\end{itemize}
    Let $U(n, n)$ denote the quasi-split unitary group in $2n$ variables associated to the quadratic extension $F / F^+$ (of $F$ over its maximal totally real subfield). The arithmetic locally symmetric spaces associated to $U(n, n)$ (or rather, their Borel--Serre compactifications) admit boundary components corresponding to parabolic subgroups, in particular to the Siegel parabolic, with Levi $\GL_{n / F}$. The global conditions mentioned in the previous paragraph greatly simplify the analysis of the cohomology of the $U(n, n)$ symmetric space with integral (say, $\mathbf{Z}_p$ or $\mathbf{Z} / p^N \mathbf{Z}$) coefficients. 
    
    For example, the `non-Eisenstein' (or irreducibility of residual representation) condition implies that no boundary components contribute to the cohomology groups, after localisation at a suitable maximal ideal of the Hecke algebra, except those corresponding to the Siegel parabolic. The `decomposed generic' condition implies that one can invoke the vanishing results of \cite{caraiani_scholze}, which are the key input into the degree-shifting argument of \cite{10author}: they imply that, for suitable $\mathbf{Z}_p$-free local systems $\mathcal{V}_\lambda$, middle degree cohomology classes with coefficients in $\mathcal{V}_\lambda / (p^N)$ can always be lifted to middle degree classes with coefficients in $\mathcal{V}_\lambda$. This allows the construction of congruences with cuspidal, regular algebraic, conjugate self-dual automorphic forms for $\GL_{2n}(\mathbf{A}_F)$, with controlled Hodge--Tate weights at $p$-adic embeddings $\tau$ lying above $v$.
    \subsection{Results of this paper} We complete the proof of Conjecture \ref{introconj_lgc_ss} for CM number fields $F$, by establishing semi-simplified local-global compatibility at the $p$-adic places:
    \begin{theorem}\label{introthm_lgc}
        Let $F$ be a CM number field, let $n \geq 1$, and let $\iota \colon \overline{\mathbf{Q}}_p \to \mathbf{C}$ be an isomorphism. Let $\pi$ be a cuspidal, regular algebraic automorphic representation of $\GL_n(\mathbf{A}_F)$ of weight $\lambda$. Then for any place $v | p$ of $F$, the representation $r_{\pi, \iota}|_{G_{F_v}}$ is de Rham of Hodge--Tate weights
        \[ \mathrm{HT}_\tau(r_{\pi, \iota}) = \{ \lambda_{\iota \tau, 1} + (n-1), \lambda_{\iota \tau, 2} + (n-2), \dots, \lambda_{\iota \tau, n} \}, \]
        and there is an isomorphism $\mathrm{WD}(r_{\pi, \iota}|_{G_{F_v}})^{\textnormal{ss}} \cong \iota^{-1} \rec^T_{F_v}(\pi_v)^{\textnormal{ss}}$. 
    \end{theorem}
    In fact, Theorem \ref{introthm_lgc} has the following apparently stronger corollary:
    \begin{corollary}\label{introcor_lgc_nilp}
        In the situation of Theorem \ref{introthm_lgc}, we have
        \[ \mathrm{WD}(r_{\pi, \iota}|_{G_{F_v}})^{\textnormal{F-ss}} \prec \iota^{-1} \rec^T_{F_v}(\pi_v). \]
    \end{corollary}
    If $(r, N)$ and $(r', N')$ are Frobenius-semisimple Weil--Deligne representations, the relation $(r, N) \prec (r', N')$, defined in \cite[\S 8]{Var24}, means that $N'$ is `closer to regular nilpotent than $N$'. We recall the definition of this relation, and explain why Theorem \ref{introthm_lgc} implies Corollary \ref{introcor_lgc_nilp}, in \S \ref{sec_nilp_prec} below.
    
    We now sketch the proof of Theorem \ref{introthm_lgc}. The basic idea is to carry out the degree-shifting argument of \cite{10author} without any additional global hypotheses. We first need to find a replacement for the torsion-vanishing results of \cite{caraiani_scholze}, which state that, after localisation at a decomposed generic maximal ideal of the Hecke algebra, the cohomology of the locally symmetric space of $U(n, n)$ vanishes below the middle degree $d$. (There is a dual statement for compactly supported cohomology, which follows by Poincar\'e duality.) 

    \para We prove a quantitative version of this statement. More precisely, if $w$ is an unramified place, split over the maximal totally real subfield of $F$, then we construct an unramified Hecke operator $T_w$ with the following properties:
    \begin{itemize}
        \item The eigenvalue of $T_w$ in any generic unramified representation of $\GL_{2n}(F_w)$ over $\mathbf{C}$ is non-zero.
        \item There is an explicit integer $N = N(n, F, w) \geq 0$ such that $T_w^{N}$ annihilates the degree $i$ cohomology of the locally symmetric space for all $i < d$. 
    \end{itemize}
    See Theorem \ref{VanishingUpToBoundedTorsion} for a precise statement. This implies the main theorem of \cite{caraiani_scholze} in the case that $p$ is a banal characteristic for $\GL_{2n}(F_w)$, since $T_w$ then becomes a unit after localisation at a generic maximal ideal. It is also well-adapted to our target theorem, because any cuspidal automorphic representation of $\GL_n$ has the property that all of its local components are generic \cite{Sha74}. The proof of Theorem \ref{VanishingUpToBoundedTorsion} follows the strategy introduced in \cite{caraiani_scholze}. More precisely, our argument follows the variant developed in \cite{Kos21}: rather than analysing the cohomology of Igusa varieties, we study the cohomology of local Shimura varieties appearing in Mantovan's product formula, using the machinery of \cite{FS24}.

    \para We next need to account for the removal of the non-Eisenstein condition. This creates two difficulties. First, on the Galois side: if the residual representation $\overline{r}_{\pi, \iota}$ \emph{is} irreducible, then it has a well-defined deformation ring, and we can write down quotients corresponding to conditions from $p$-adic Hodge theory (such as `crystalline with given Hodge--Tate weights') using the results of \cite{Kis08}. This shows in particular that, if $r_{\pi, \iota}$ can be presented as a $p$-adic limit of Galois representations satisfying such conditions, then it will also satisfy these conditions.

    If $\overline{r}_{\pi, \iota}$ is reducible, then we need to deal with pseudocharacters and their deformations in the sense of \cite{Che14}. These admit a pseudodeformation space, but the question of imposing conditions from $p$-adic Hodge theory, studied e.g.\ in \cite{WE18}, is more subtle. For example, in \emph{op. cit.} pseudodeformation spaces with $p$-adic Hodge--theoretic conditions are constructed when $\overline{r}_{\pi, \iota}$ is residually multiplicity-free; but we want a result with no condition at all on the residual representation. The key new observation we make here is that the finiteness of cohomology of reductive groups over Noetherian rings can be combined with the ideas developed in \cite{WE18}, and the theory of `pseudodeformations with conditions' developed in \cite{WWE19}, to construct the needed pseudodeformation spaces. This construction is the subject of \S \ref{sec_potentially_semistable_pseudodeformation_rings}. 

    \para The second difficulty that arises without the non-Eisenstein condition is that, in considering the cohomology of the (Borel--Serre compactification of the) locally symmetric space for $U(n, n)$, we need to consider contributions from parabolic subgroups other than the Siegel parabolic. Indeed, it is the cohomology of the whole Borel--Serre boundary that admits comparison with the cohomology of the $U(n, n)$ locally symmetric space (via the excision exact sequence); in the non-Eisenstein case, the cohomology of the Borel--Serre boundary and the cohomology of the Siegel component agree, after localisation, but otherwise the best one can hope for is that the interior cohomology of the Siegel component is a subquotient of the cohomology of the Borel--Serre boundary. 

    As interior cohomology, being defined as the image of the compactly supported cohomology in the usual cohomology, does not behave well in exact sequences, this necessitates a complicated induction argument, where we prove Theorem \ref{introthm_lgc} (and in fact, a version for the cohomology with torsion coefficients of the $\GL_{n / F}$-locally symmetric space) by induction on $n$. Since we need to use coefficient systems that pick out automorphic representations with the right Hodge--Tate weights and inertial type at the place $v$, this further requires us to understand what these local systems look like when restricted to boundary components, via an excursion into type theory for the group $\GL_n(F_v)$. We recall the needed local results in \S \ref{sec_local_representation_theory}, and carry out this new degree-shifting argument, including the necessary induction on $n$, in \S \ref{sec_LGC_for_Betti}. 

\subsection{Remarks}    The induction argument in \S \ref{sec_LGC_for_Betti} would not be needed if we knew that the representations $r_{\pi, \iota}$ were irreducible (as is conjectured to be the case). On the other hand, the knowledge that the representations $r_{\pi, \iota}$ satisfy the expected conditions from $p$-adic Hodge theory is a useful input into many proofs of irreducibility (see e.g.\ \cite{Sha26}).

    Although the main theorem of this paper concerns the $\overline{\mathbf{Q}}_p$-representations associated to automorphic representations, it is also of interest to establish local-global compatibility for torsion classes, with the aim of proving e.g.\ $R = \mathbf{T}$ theorems under more degenerate conditions than the ones considered in \cite{10author}. Our arguments do give meaningful results for torsion classes. However, instead of proving the expected result `up to a nilpotent ideal', we are able to prove a result `up to bounded $T_w$-torsion', where $T_w$ is the operator described above, which is (at least in the banal case) nonzero modulo any decomposed generic maximal ideal of the Hecke algebra. We apply these results to prove the vanishing of the adjoint Bloch--Kato Selmer group for many automorphic Galois representations in the companion paper \cite{Aca26}, generalising the theorems proved in \cite{New23} (which just treated the conjugate self-dual part of the adjoint Bloch--Kato Selmer group for conjugate self-dual automorphic Galois representations). 

    \subsection*{Acknowledgments} We thank Frank Calegari, Toby Gee, and Peter Scholze for helpful discussions and correspondence. B.H. thanks Patrick Allen and Vincent Pilloni for some conversations from which this work has benefited. 

    B.H.'s work was supported by the Engineering and Physical Sciences Research Council grant number EP/Y030648/1 and by the European Union’s Horizon 2020 research and innovation programme under the Marie Skłodowska-Curie grant agreement No 101034255.

    J.T.'s work was funded by the European Union (ERC CoG-101169866). Views and opinions expressed are however those of the author(s) only and do not necessarily reflect those of the European Union or the European Research Council. Neither the European Union nor the granting authority can be held responsible for them. 

    \subsection{Notation}\label{subsec_notation}

    If $F$ is a field of characteristic zero, we generally fix an algebraic closure $\overline{F} / F$ and write $G_F$ for the absolute Galois group of $F$ with respect to this choice. If $F$ is a number field, then we will also fix embeddings $\overline{F} \to \overline{F}_v$ extending the map $F\to F_v$ for each place $v$ of $F$; this choice determines a homomorphism $G_{F_v} \to G_F$. When $v$ is a finite place, we will write $\cO_{F_v} \subset F_v$ for the valuation ring, $\varpi_v \in \cO_{F_v}$ for a fixed choice of uniformizer, $\Frob_v \in G_{F_v}$ for a fixed choice of Frobenius lift, $k(v) = \cO_{F_v} / (\varpi_v)$ for the residue field, and $q_v = \# k(v)$ for the cardinality of the residue field. When $v$ is a real place, we write $c_v \in G_{F_v}$ for complex conjugation. 

    If $p$ is a prime, then we call a coefficient field a finite extension $E / \mathbf{Q}_p$ contained inside our fixed algebraic closure $\overline{\mathbf{Q}}_p$, and write $\cO$ for the valuation ring of $E$, $\varpi \in \cO$ for a fixed choice of uniformizer, and $k = \cO / (\varpi)$ for the residue field. (In the body of the paper, we generally use $p$-adic coefficients, except in \S \ref{sec_torsion_in_the_cohomology}, where we use $\ell$-adic coefficients in order to be consistent with the relevant literature.) 

    Let $K$ be a non-archimedean characteristic $0$ local field, and let $\Omega$ 
be an algebraically 
closed field of characteristic $0$. We write $W_K \subset G_K$ for the Weil group of $K$ and $I_K \subset W_K$ for the inertia subgroup.  We write $\rec^T_{K}$ for the Tate normalisation of the local Langlands correspondence for $\GL_n(K)$: it is the bijection $\rec_K^T$ between isomorphism classes of irreducible, 
admissible $\Omega[\GL_n(K)]$-modules and isomorphism classes of Frobenius-semisimple Weil--Deligne 
representations over $\Omega$ of rank $n$ defined in \cite[\S 2.1]{Clo14}. When $\Omega = \mathbf{C}$, it is related to the usual unitary normalisation by the formula $\rec_{K}^T(\pi) = \rec_{K}(\pi |\cdot|^{(1-n)/2})$.

If $K$ is a $p$-adic local field, $\tau \colon K \to \overline{\mathbf{Q}}_p$ is a continuous embedding, and $\rho \colon G_K \to \GL_n(\overline{\mathbf{Q}}_p)$ is a continuous representation, then we write $\mathrm{HT}_\tau(\rho)$ for the multiset of integers $i$ such that $\operatorname{gr}^i(\rho \otimes_{\tau, K} B_{dR})^{G_K} \neq 0$, with multiplicity equal to the dimension of this $\overline{\mathbf{Q}}_p$-vector space. If $\rho$ is de Rham, then this set has $n$ elements (counted with multiplicity). In particular, if $\rho$ is the $p$-adic cyclotomic character, we have $\textnormal{HT}_\tau(\rho) = \{-1\}$ for all $\tau$.

\section{Torsion in the cohomology of Shimura varieties}\label{sec_torsion_in_the_cohomology}

An essential input in the proof of the local–global compatibility theorems of \cite{10author} is the main result of \cite{caraiani_scholze}, which establishes the vanishing of the mod $\ell$ compactly supported cohomology of $\mathrm{Res}_{F^+/\mathbf{Q}}GU(n,n)$-Shimura varieties above the middle degree after localisation at decomposed generic maximal ideals of the global Hecke algebra. Equivalently, this can be formulated as vanishing above the middle degree after inverting a Hecke operator $\Delta$ at an auxiliary prime $p \neq \ell$, which measures the genericity of unramified principal series representations. The goal of this section is to prove a quantitative version of this vanishing result.

In \cite{Kos21}, Koshikawa gives a new proof of the vanishing result of \cite{caraiani_scholze} using the framework of Fargues--Scholze \cite{FS24}. He first proves that, after localisation at generic maximal ideals of the local Hecke algebra $\mathcal{H}_p$ for $p \neq \ell$, or equivalently after inverting $\Delta \in \mathcal{H}_p$, the mod $\ell$ compactly supported cohomology groups of local Shimura varieties associated with non-ordinary Kottwitz elements vanish. He then deduces the global vanishing result for unitary Shimura varieties by combining a version of Mantovan’s product formula for \textit{cohomology without compact support} with the semi-perversity result of \cite{caraiani_scholze}, §4.6.

Following Koshikawa’s strategy, we combine the description of the cohomology of local Shimura varieties provided by the Fargues–Scholze machinery with the local–local compatibility result of \cite{HKW22}. This allows us to show that the compactly supported cohomology of non-ordinary local Shimura varieties with coefficients in any $\mathbf{Z}_\ell[p^{1/2}]$-algebra $\Lambda$ is annihilated by $M\Delta$, for some integer $M \geq 1$ depending only on $n$ and $F$.

Finally, we deduce that, above the middle degree, the compactly supported cohomology of the unitary Shimura variety with coefficients in $\Lambda$ is annihilated by a power of $M\Delta$, with exponent depending only on $n$ and $F$. For this, we use Mantovan’s product formula for \textit{compactly supported cohomology} established in \cite{HL25}, reducing the problem to the vanishing above the middle degree of the contribution of the ordinary stratum. This is then obtained from the corresponding result for partially compactly supported cohomology of the ordinary Igusa variety, together with cases of the Harris–Viehmann conjecture proved in \cite{GI22}.

\subsection{Setup} Fix an imaginary CM field $F$ with totally real subfield $F^+\subset F$ and an integer $n\geq 2$. Denote by $c$ the non-trivial element of $\textnormal{Gal}(F/F^+)$. For the $F$-vector space $V:=F^{2n}$, we consider the alternating form
\begin{align*}
    (., .) \colon V\times V & \to \mathbf{Q} \\
    (x,y) & \mapsto \textnormal{tr}_{F/\mathbf{Q}}\langle x,y\rangle,
\end{align*}
associated with the skew-hermitian form
\begin{equation*}
    \langle x,y\rangle :={}^txJ_ny^c.
\end{equation*}
Here $J_n$ is the $2n\times 2n$ matrix
$$J_n:=\begin{pmatrix}
0 & \Psi_n \\
-\Psi_n & 0
\end{pmatrix}$$
with the matrix $\Psi_n$ in the formula given by the $n\times n$ matrix with $1$'s on the anti-diagonal and $0$'s elsewhere. We consider the associated quasi-split unitary similitude $\mathbf{Q}$-group $\mathbf{G}$ with functor of points
\begin{equation*}
    \mathbf{G}(R)=\{(g,d)\in \textnormal{GL}_{F\otimes_{\mathbf{Q}}R}(V\otimes_{\mathbf{Q}}R)\times \mathbf{G}_m(R)\mid (gx,gy)=d(x,y) \textnormal{ } \forall x,y\in V\}.
\end{equation*}

\begin{para}
We define the morphism
\begin{align*}
h=\prod_{\tau:F^+\hookrightarrow\mathbf{R}}h_{\tau}\colon\mathbf{S} &\to \mathbf{G}_{\mathbf{R}} 
\end{align*}
from the Deligne torus $\mathbf{S}=\textnormal{Res}_{\mathbf{C}/\mathbf{R}}\mathbf{G}_m$ by setting $h(i)=J_n$. We consider the Shimura datum $(\mathbf{G},X)$ associated with the PEL datum $(F,c,V,(.,.),h)$ of type $A$. The corresponding $\mathbf{G}(\mathbf{C})$-conjugacy class of the Hodge cocharacter $\mu=\mu_h\colon\mathbf{G}_m\to \mathbf{G}_{\mathbf{C}}$, under the identification $\mathbf{G}_{\mathbf{C}}\cong (\prod_{\tau:F^+\hookrightarrow \mathbf{C}}\textnormal{GL}_{2n})\times \mathbf{G}_m$, is identified with the tuple $((1,1,...,1,0,...,0)_{\tau},1)\in X_{\ast}((\prod_{\tau:F^+\hookrightarrow \mathbf{C}}\textnormal{GL}_{2n})\times \mathbf{G}_m)\cong (\mathbf{Z}^{2n})^{\Hom(F^+,\mathbf{C})}\times\mathbf{Z}$ with $1$ appearing exactly $n$ times for each $\tau$-factor.

For neat compact open subgroups $K\subset \mathbf{G}(\mathbf{A}^{\infty})$, we obtain the corresponding Shimura variety $\textnormal{Sh}_K:=\textnormal{Sh}_K(\mathbf{G},X)$, a quasi-projective smooth $\mathbf{Q}$-scheme.

\end{para}

\begin{para}
We fix a rational prime $p$, a field isomorphism $\iota \colon\overline{\mathbf{Q}}_p\xrightarrow{\sim} \mathbf{C}$ and assume that $p$ splits completely in $F$. 
We set $\textnormal{G}:=\mathbf{G}_{\mathbf{Q}_p}$ and note that there is an isomorphism
\begin{equation}\label{tau_p}
    \tau_p\colon\textnormal{G}\xrightarrow{\sim} (\prod_{v\mid p}\textnormal{GL}_{2n})\times \mathbf{G}_m=\textnormal{GL}_{2n}^{[F^+:\mathbf{Q}]}\times \mathbf{G}_m.
\end{equation}
We further set $\textnormal{T}=(\prod_{v\mid p}\textnormal{T}_{2n})\times \mathbf{G}_m\subset \textnormal{B}=(\prod_{v\mid p}\textnormal{B}_{2n})\times \mathbf{G}_m$ to be our choice of torus and Borel subgroup.

We fix another rational prime $\ell\neq p$.
We denote by $\mathcal{O}$ a complete DVR over $\mathbf{Z}_{\ell}[p^{1/2}]$ and by $\varpi\in \mathcal{O}$ a choice of uniformiser.

The goal of this section is to prove boundedness of $\ell$-torsion in some suitably generic part of the cohomology groups
\begin{equation*}
    H^{\ast}(\textnormal{Sh}_K(\mathbf{C}),\mathcal{O})
\end{equation*}
below the middle degree $d=\textnormal{dim}_{\mathbf{C}}(\textnormal{Sh}_K(\mathbf{C}))$.
\end{para}

\subsection{The statement}\label{subsection_Vanishing statement}
To state the main result, we introduce some notation. Consider the hyperspecial level subgroup $K_p^{\textnormal{hs}}:=\textnormal{G}(\mathbf{Z}_p)$ and the corresponding Hecke algebra
\begin{equation*}
    \mathcal{H}:=\mathcal{O}[K_p^{\textnormal{hs}}\backslash\textnormal{G}(\mathbf{Q}_p)/K_p^{\textnormal{hs}}].
\end{equation*}
The isomorphism $\tau_p$ induces an algebra isomorphism $\mathcal{H}\xrightarrow{\sim}(\otimes_{v\mid p}\mathcal{H}_v)\otimes \mathcal{H}_0$
where, for $v|p$, we have
\begin{equation*}
    \mathcal{H}_v=\mathcal{O}[\textnormal{GL}_{2n}(\mathbf{Z}_p)\backslash\textnormal{GL}_{2n}(\mathbf{Q}_p)/\textnormal{GL}_{2n}(\mathbf{Z}_p)]
\end{equation*}
and
\begin{equation*}
    \mathcal{H}_0=\mathcal{O}[\mathbf{Z}_{p}^{\times}\backslash\mathbf{Q}_p^{\times}/\mathbf{Z}_{p}^{\times}].
\end{equation*}

\begin{para} \label{Deltapara}
For $v| p$ we have the Satake isomorphism
\begin{equation*}
    \mathcal{H}_v \cong \mathcal{O}[X_{v,1}^{\pm 1}, \dots, X_{v,2n}^{\pm 1}]^{S_{2n}}
\end{equation*}
and we define the element
\begin{equation*}
    \Delta_v := \prod_{1\leq i\neq j\leq 2n}(X_{v,i} X_{v,j}^{-1} - p)\in \mathcal{O}[X_{v,1}^{\pm1},...,X_{v,2n}^{\pm 1}]^{S_{2n}}.
\end{equation*}
By abuse of notation we also consider it as an element of $\mathcal{H}_v$. Let $\Delta:=\prod_{v|p}\Delta_v\in \mathcal{H}$.
Let $\Lambda$ be a local Artinian $\mathcal{O}$-algebra. 
The main result of this section is as follows.
\end{para}

\begin{theorem}[= Theorem \ref{VanishingUpToBoundedTorsion}]
    Let 
    \[
        M=(1-p)^{[F^+:\mathbf{Q}](2n-1)} \prod_{m = 1}^{2n}(p^m - 1)^{2n[F^+:\mathbf{Q}]+1}.
    \]
    For every good subgroup $K^p\leq \mathbf{G}(\mathbf{A}^{\{\infty,p\}})$, and integer $0\leq q\leq d-1$, the Hecke operator $M\Delta\in \mathcal{H}$ acts nilpotently on the cohomology groups
    \begin{equation*}
    H^q(\textnormal{Sh}_{K^pK_p^{\textnormal{hs}}}(\mathbf{C}),\Lambda)
    \end{equation*}
    and
    \begin{equation*}
        H^{2d-q}_c(\textnormal{Sh}_{K^pK_p^{\textnormal{hs}}}(\mathbf{C}),\Lambda).
    \end{equation*}
    More precisely, each such group is annihilated by\footnote{Here $B(\textnormal{G},\mu^{-1})$ denotes the $\mu^{-1}$-admissible elements of the Kottwitz set of $\textnormal{G}$.} 
    \[ M^{[F^+ : \mathbf{Q}] n + | B(\textnormal{G},\mu^{-1})|} \Delta^{| B(\textnormal{G},\mu^{-1})|}. \]
\end{theorem}

\begin{remark}
    Note that $M$ only depends on $p$, $n$, and $[F:\rationals]$. This is important in the degree-shifting argument, where we will apply the theorem with many different choices of $K^p$. Moreover, a prime $\ell \neq p$ divides $M$ if and only if it divides the order of $\GL_{2n}(\mathbf{F}_p)$. If $\ell$ does not divide $M$, then the theorem implies the Caraiani--Scholze vanishing theorem \cite{caraiani_scholze}, since $\Delta$ becomes a unit modulo decomposed generic maximal ideals.
\end{remark}

\subsection{Recollection on the work of Fargues--Scholze}
Let $H$ be a connected reductive group over $\mathbf{Q}_p$. For a fixed algebraic closure $\overline{\mathbf{F}}_p/\mathbf{F}_p$, set $\check{\mathbf{Z}}_p=W(\overline{\mathbf{F}}_p)$ and $\check{\mathbf{Q}}_p=\check{\mathbf{Z}}_p[1/p]$. This admits an automorphism $\sigma$ induced by the absolute Frobenius $x\mapsto x^p$ on $\overline{\mathbf{F}}_p$. We denote by $B(H)$ the Kottwitz set of $\sigma$-conjugacy classes in $H(\check{\mathbf{Q}}_p)$ parametrising $H$-isocrystals over $\overline{\mathbf{F}}_p$. Kottwitz described $B(H)$ via two invariants, the Newton point and the Kottwitz point (cf. \cite{Kot97}). The first one is a map
\begin{equation*}
    \nu:B(H)\to (X_{\ast}(H)_{\mathbf{Q}}^+)^{\Gamma}
\end{equation*}
where $X_{\ast}(H)^+:=\Hom(\mathbf{G}_{m,\overline{\mathbf{Q}}_p},H_{\overline{\mathbf{Q}}_p})/(H(\overline{\mathbf{Q}}_p)-\textnormal{conjugacy})=X_{\ast}(T)^+$ is the set of $B$-dominant cocharacters for any choice of Borel and torus for $H_{\overline{\mathbf{Q}}_p}$ and $\Gamma:=\textnormal{Gal}(\overline{\mathbf{Q}}_p/\mathbf{Q}_p)$. The latter is a map
\begin{equation*}
    \kappa:B(H)\to \pi_1(H)_{\Gamma}
\end{equation*}
where $\pi_1(H):=\pi_1(H_{\overline{\mathbf{Q}}_p})$ is Borovoi's fundamental group for $H_{\overline{\mathbf{Q}}_p}$. Kottwitz proves that these two maps turn out to combine to an injective map
\begin{equation*}
    (\nu,\kappa):B(H)\hookrightarrow (X_{\ast}(H)_{\mathbf{Q}}^+)^{\Gamma}\times \pi_1(H)_{\Gamma}
\end{equation*}
describing $B(H)$ completely. Moreover, they are compatible in the sense that they have the same image in $\pi_1(H)^{\Gamma}\otimes_{\mathbf{Z}}\mathbf{Q}$ after post-composition with the natural quotient map and the averaging map along Galois conjugates, respectively.

In \cite[\S2.3]{RR96}, they describe a partial order $\prec$ on $B(H)$ using the Newton and Kottwitz maps. We equip $B(H)$ with the topology induced by \textit{opposite} of this partial order. In particular, for the $H(\check{\mathbf{Q}}_p)$-conjugacy class of a minuscule cocharacter $\mu$, we can speak of the finite subset $B(H,\mu)\subset B(H)$ of $\mu$\textit{-admissible elements} that forms an open subset with respect to the induced topology. For later use, we note that we will write $\mu^{-1}$ for the dominant inverse of $\mu$, another minuscule cocharacter.

For every $b\in B(H)$, we obtain a reductive group $H_b$ over $\mathbf{Q}_p$ with functor of points
\begin{equation*}
    H_b(R)=\{g\in H(R\otimes_{\mathbf{Q}_p}\check{\mathbf{Q}}_p)\mid gb=b\sigma(g)\}.
\end{equation*}
If $H$ is quasisplit, $H_b$ is an inner form of a Levi subgroup of $H$. In general, it is an inner form of the Levi of the quasisplit inner form of $H$.
\begin{example}\label{KottwitzSetExample}
    We now spell out the case $H=\textnormal{GL}_m$, our group of interest. In this case, the Kottwitz set parametrises isocrystals of height $m$. Namely, to the element $b\in B(H)$ corresponds the isocrystal
    \begin{equation*}
        (\check{\mathbf{Q}}_p^m,b\sigma).
    \end{equation*}
    The Dieudonn\'e--Manin classification shows that the category of isocrystals is semisimple with simple objects
    \begin{equation*}
        V_{r/s}=(\check{\mathbf{Q}}_p^{r},\begin{pmatrix}
            0 & 1 & \dots  & 0 \\
    \vdots &  & \ddots  & \vdots \\
    0 &  &  & 1 \\
    p^s & 0 & \dots  & 0
        \end{pmatrix})
    \end{equation*}
    labelled by slopes $\lambda=s/r\in \mathbf{Q}$; rational numbers in their primitive form. In particular, every isocrystal $V$ admits a unique "slope decomposition" $V\cong \oplus_{\lambda}V_{\lambda}^{\oplus m_{\lambda}}$ and corresponds to a unique tuple \[
    (\lambda_1,\dots,\lambda_1,\dots,\lambda_k,\dots,\lambda_k)\in(\mathbf{Q}^m)^+=\{(\mu_1,\dots,\mu_m)\in \mathbf{Q}^{m}\mid \mu_1\geq \dots \geq \mu_m\}
    \]
    with $\lambda=s/r$ occurring exactly $r\cdot m_{\lambda}$ times. Finally, the endomorphism ring $D_{\lambda}:=\textnormal{End}(V_{\lambda})$ of the isocrystal of slope $\lambda$ recovers the central division algebra of invariant $[\lambda]\in \textnormal{Br}(\mathbf{Q}_p)\cong \mathbf{Q}/\mathbf{Z}$.
    
    The Newton map
    \begin{align*} 
        \nu:B(H) &\to (X_{\ast}(H)^+_{\mathbf{Q}})^{\Gamma}=(\mathbf{Q}^{m})^+
     \\ 
        b &\mapsto (\nu_{b,1},...,\nu_{b,m}),
     \end{align*}
    recovers the slope decomposition of the isocrystal attached to $b$.
    Moreover, the Kottwitz map
    \begin{align*} 
        \kappa \colon B(H) &\to \pi_1(H)_{\Gamma}\cong \mathbf{Z} \\ 
        b &\mapsto \sum_{i=1}^m\nu_{b,i},
     \end{align*}
    gives the endpoint of the Newton polygon. In particular, the injectivity of $(\nu,\kappa)$ follows from the Dieudonn\'e--Manin classification of isocrystals.

    For $b\in B(H)$ with associated isocrystal $V\cong \oplus_{\lambda} V_{\lambda}^{\oplus m_{\lambda}}$ we see from the definitions that
    \begin{equation*}
        H_b\cong \prod_{\lambda, m_{\lambda}\neq 0}\textnormal{GL}_{m_{\lambda}}(D_{\lambda}),
    \end{equation*}
    an inner form of $\prod_{\lambda=s/r}\textnormal{GL}_{r\cdot m_{\lambda}}$. In particular, it is quasisplit if and only if $\nu_{b,i}$ is an integer for every $1\leq i\leq m$.
    
    Finally, the partial order is given by the usual partial order
    \begin{align*}
        \nu\prec \mu & \Leftrightarrow \\
        \sum_{i=1}^k\nu_i\leq \sum_{i=1}^k\mu_i \textit{ }\forall 1\leq  k\leq m & \textnormal{ and }\sum_{i=1}^m\nu_i=\sum_{i=1}^m\mu_i
    \end{align*}
    of Newton polygons. In particular, for the $H(\overline{\mathbf{Q}}_p)$-conjugacy class corresponding to the minuscule cocharacter
    $\mu=(1,...,1,0,...,0)\in X_{\ast}(H)^+$ with $1$ appearing exactly $q$ times, 
    \begin{equation*}
        B(H,\mu^{-1})=\{b\mid -1\leq \nu_{b,i}\leq 0,\textnormal{ }\sum_{i=1}^{m}\nu_{b,i}=-q\}.
    \end{equation*} Therefore, we see that there is a unique element $b\in B(H,\mu^{-1})$ such that $\nu_{b,i}$ are all integers. Concretely, it is represented by
    \begin{equation*}
        b_{\mu}:=\textnormal{diag}(1,...,1,p^{-1},...,p^{-1})
    \end{equation*} 
    with $p^{-1}$ appearing in the diagonal exactly $q$ times, and is referred to as the $\mu$-ordinary element. In particular, we can conclude that for $b\in B(H,\mu^{-1})$, $H_b$ is non-quasisplit if $b\neq b_{\mu}$.
\end{example}

\begin{para} For a field extension $k/\mathbf{F}_p$, let $\textnormal{Perf}_k$ denote the category of characteristic $p$ perfectoid spaces over $\textnormal{Spd}(k)$ equipped with the $v$-topology of \cite{Sch17}. For a non-archimedean field extension $K/\mathbf{Q}_p$, we will write $\textnormal{Perf}_K$ for the slice category $\textnormal{Perf}_{/\textnormal{Spd(K)}}$. We will denote by $\textnormal{Bun}_H$ the small v-stack of $H$-bundles on ``the'' Fargues--Fontaine curve sending $S\in \textnormal{Perf}_{\overline{\mathbf{F}}_p}$ to the groupoid of $H$-torsors on the relative (adic) Fargues--Fontaine curve $X_S$ (cf. \cite[Definition III.0.1]{FS24}). 

A theorem of Fargues asserts the existence of a bijection of sets
\begin{equation*}
    |\textnormal{Bun}_H|\xrightarrow{\sim} B(H)
\end{equation*}
and a theorem of Viehmann shows that it in fact upgrades to a homeomorphism of topological spaces (where the underlying topological space $|\textnormal{Bun}_H|$ is defined as in \cite[Proposition 12.7]{Sch17}). In particular, by \cite[Proposition 12.9]{Sch17}, $B(H,\mu)\subset B(H)$ defines an open substack
\begin{equation*}
    j_{\mu}\colon\textnormal{Bun}_{H,\mu}\hookrightarrow \textnormal{Bun}_H.
\end{equation*} 
Moreover, we obtain locally closed substacks 
\begin{equation*}
j_b\colon\textnormal{Bun}_H^{b}:=\textnormal{Bun}_H\times_{|B(H)|}\{b\}\hookrightarrow \textnormal{Bun}_H
\end{equation*} for every Kottwitz element $b\in B(H)$ (\cite[Theorem III.02.(v)]{FS24}).

We now discuss how the \'etale sheaf theory of $\textnormal{Bun}_H$ governs the representation theory of $H$ and its inner forms $H_b$, $b\in B(H)$. For a small $v$-stack $X$, \cite[Definition 14.13]{Sch17} introduces the triangulated category $D(X,\Lambda):=D_{\textit{\'et}}(X,\Lambda)$ that can be thought of as the derived category of \'etale sheaves of $\Lambda$-modules on $X$. As is shown in \cite[Proposition V.2.2]{FS24}, for every $b\in B(H)$, there is an equivalence of categories
\begin{equation*}
    D(\textnormal{Bun}_H^{b},\Lambda)\cong D(H_b(\mathbf{Q}_p),\Lambda).
\end{equation*}
In particular, by slight abuse of notation, the formalism of \cite{Sch17} (cf. \textit{loc. cit.} \S17, \S22, \S23) yields adjoint functors
\begin{align*}
    j_b^{\ast} \colon D(\textnormal{Bun}_H,\Lambda) & \rightleftarrows D(H_b(\mathbf{Q}_p),\Lambda):Rj_{b\ast}, \\
    j_{b!} \colon D(\textnormal{Bun}_H,\Lambda) & \rightleftarrows D(H_b(\mathbf{Q}_p),\Lambda):Rj_{b}^!.
\end{align*}
Moreover, when $b\in B(H)$ is a \textit{basic} element (i.e. minimal for the partial order), then $j_{b!}$ is a left adjoint to $j_b^{\ast}$ (cf. \cite{Sch17}, Definition/Proposition 19.1).

Denote by $\widehat{H}$ the Langlands dual group of $H$ and write $Q$ for a discrete quotient of $W_{\mathbf{Q}_p}$ through which the natural action on $\widehat{H}$ factors. For any finite set $I$ and algebraic representation $V\in\textnormal{Rep}_{\Lambda}\left((\widehat{H}\rtimes Q)^I\right)$, an associated Hecke operator
\begin{equation*}
    T_V:D(\textnormal{Bun}_H,\Lambda)\to D(\textnormal{Bun}_H\times [\ast/\underline{W_{\mathbf{Q}_p}}^I],\Lambda)
\end{equation*}
is constructed in \cite[Chapter IX]{FS24}. For an algebraically closed non-archimedean field $C/\mathbf{Q}_p$, define the map
\begin{equation*}
    \pi^I_C\colon\textnormal{Bun}_H\times \textnormal{Spd}(C)\to \textnormal{Bun}_H\times[\ast/\underline{W_{\mathbf{Q}_p}}^I]
\end{equation*}
induced by $\textnormal{Spd}(C)\to \ast \to [\ast /W_{\mathbf{Q}_p}^I]$. By pulling back along $\pi_C^I$
(i.e. forgetting the $W_{\mathbf{Q}_p}^I$-action), we obtain an endofunctor
\begin{equation*}
    T_{V,C}:=(\pi^I_C)^{\ast}T_V:D(\textnormal{Bun}_H,\Lambda)\to D(\textnormal{Bun}_H\times \textnormal{Spd}(C),\Lambda)\cong D(\textnormal{Bun}_H,\Lambda)
\end{equation*}
where the identification at the end is \cite[Corollary V.2.3]{FS24}.
By \cite[Theorem IX.0.1]{FS24}, $T_{V,C}$ preserves both compact and ULA objects. For an $H(\overline{\mathbf{Q}}_p)$-conjugacy class of minuscule cocharacters $\mu$ with associated highest weight representation $V_{\mu}$, we will write $T_{\mu}:=T_{V_{\mu}}$. To explicate $T_{\mu,C}$, consider the Hecke correspondence
\begin{equation*}
    \begin{tikzcd}
	&& {\textnormal{Hck}_{H,\leq\mu}} \\
	{\textnormal{Bun}_{H}} &&&& {\textnormal{Bun}_{H}\times \textnormal{Spd}(C).}
	\arrow["{h_{\mu}^{\leftarrow}}"', from=1-3, to=2-1]
	\arrow["{h_{\mu}^{\rightarrow}}", from=1-3, to=2-5]
\end{tikzcd}
\end{equation*}
Here $\textnormal{Hck}_{H,\leq\mu}$ is the $v$-stack sending $S\in\textnormal{Perf}_C$ to the groupoid of modifications $\mathcal{E}_1\dashrightarrow \mathcal{E}_2$ of $H$-bundles on $X_S$ at the divisor defined by $S\to\textnormal{Spd}(C)$ with meromorphy bounded by $\mu$. Moreover, $h_{\mu}^{\leftarrow}$ and $h_{\mu}^{\rightarrow}$ remember $\mathcal{E}_1$ and $(\mathcal{E}_2, S\to \textnormal{Spd}(C))$, respectively. Then, for $A\in D(\textnormal{Bun}_H,\Lambda)$ we have the formula
\begin{equation*}
    T_{\mu,C}(A)=Rh_{\mu,\ast}^{\rightarrow}(h_{\mu}^{\ast}(A))[d_{\mu}]
\end{equation*}
where $d_{\mu}=\langle2\rho_{H},\mu\rangle$.
\end{para}

\begin{para}
Finally, we discuss the Fargues--Scholze semisimple correspondence. To do so, we write 
\begin{equation*}
    \mathcal{Z}^{\textnormal{spec}}(H,\mathcal{O}):=\mathcal{O}(Z^1(W_{\mathbf{Q}_p},\widehat{H})_{\mathcal{O}})^{\widehat{H}}
\end{equation*} 
for the global sections of the Artin stack $[Z^1(W_{\mathbf{Q}_p},\widehat{H})_{\mathcal{O}}/\widehat{H}]$ over $\textnormal{Spec}(\mathcal{O})$ (cf. \cite[Chapter VIII]{FS24}) and will refer to it as the \textit{spectral Bernstein centre}. Note that it coincides with the ring of functions of the coarse moduli given by the GIT quotient. In particular, one sees that $\overline{\mathbf{Q}}_{\ell}$-points of the spectral Bernstein centre are in bijection with \textnormal{semisimple} $L$-parameters of $H$ (cf. \cite{FS24}, Proposition VIII.3.2).

On the other  side of the correspondence, we have the \textit{geometric Bernstein centre}
\begin{equation*}
    \mathcal{Z}^{\textnormal{geom}}(H,\mathcal{O}):=\pi_0\textnormal{End}(\textnormal{id}_{\mathcal{D}_{\textnormal{lis}}(\textnormal{Bun}_H,\mathcal{O})}).
\end{equation*}

Under the assumption that the order of $\pi_0Z(H)$ is invertible in $\mathcal{O}$, \cite[Corollary IX.0.3]{FS24}  asserts the existence of a map of $\mathcal{O}$-algebras
\begin{equation*}
\Psi^{\textnormal{geom}}\colon\mathcal{Z}^{\textnormal{spec}}(H,\mathcal{O})\to \mathcal{Z}^{\textnormal{geom}}(H,\mathcal{O})
\end{equation*}
with image landing in the subalgebra $\mathcal{Z}^{\textnormal{geom}}_{\textnormal{Hecke}}(H,\mathcal{O})\subset \mathcal{Z}^{\textnormal{geom}}(H,\mathcal{O})$ of elements commuting with the Hecke operators $T_{V,C}$.

For $b\in B(H)$, shriek pushforward along $j_b$ induces an $\mathcal{O}$-algebra map
\begin{equation*}
    \Psi_H^b\colon\mathcal{Z}^{\textnormal{spec}}(H,\mathcal{O})\xrightarrow{\Psi^{\textnormal{geom}}}\mathcal{Z}^{\textnormal{geom}}(H,\mathcal{O})\xrightarrow{j_{b!}}\mathfrak{Z}(H_b)_{\mathcal{O}}
\end{equation*}
towards the Bernstein centre for $H_b(\mathbf{Q}_p)$
\begin{equation*}
    \mathfrak{Z}(H_b)_{\mathcal{O}}:=\textnormal{End}(\textnormal{id}_{\textnormal{Mod}_{\textnormal{sm}}(H_b(\mathbf{Q}_p),\mathcal{O})})=\varprojlim_{K\subset H_b(\mathbf{Q}_p)}\mathcal{O}[K\backslash H_b(\mathbf{Q}_p)/K].
\end{equation*}
For $b=1$, we will sometimes abbreviate $\Psi_H:=\Psi^1_{H}$.
As a consequence, for $b\in B(H)$ and $\pi$ an irreducible smooth $\overline{\mathbf{Q}}_{\ell}$-representation $H_b(\mathbf{Q}_{p})$, post-composition by $\Psi_H^b$ of the induced $\overline{\mathbf{Q}}_{\ell}$-point of the Bernstein centre $\mathfrak{Z}(H_b)_{\mathcal{O}}$ yields a semisimple $L$-parameter
\begin{equation*}
\varphi_{H,\pi}^{\textnormal{FS}}:W_{\mathbf{Q}_{p}}\to {}^LH(\overline{\mathbf{Q}}_{\ell}).
\end{equation*}
The $\pi\mapsto \varphi_{H,\pi}^{\textnormal{FS}}$ correspondence satisfies several desirable properties (cf. \cite[Theorem IX.0.5]{FS24}). 
\end{para}

\subsubsection{Local-local compatibility}
An essential property of the Fargues--Scholze correspondence for our argument is that it recovers the semisimplification of the classical local Langlands correspondence constructed in \cite{HT01} and \cite{Hen00}.

For $H=\textnormal{GL}_m$ and $b\in B(H)$, $H_b$ is an inner form of some standard Levi subgroup $M_b\subset H$. In particular, using the Jacquet--Langlands correspondence \cite{DKV84}, the Langlands classification and the local Langlands correspondence for general linear groups \cite{HT01}, \cite{Zel80}, to every smooth irreducible $\overline{\mathbf{Q}}_{\ell}$-representation $\sigma$ of $H_b(\mathbf{Q}_p)$, one can attach an $L$-parameter
$\textnormal{rec}^{H_b}_{\mathbf{Q}_p}(\sigma):W_{\mathbf{Q}_p}\to M_b(\overline{\mathbf{Q}}_{\ell})$
for $M_b$ (see for instance \cite{ABPS16}). This association is always injective and is a bijection exactly when $b$ is basic, that is, $H_b=M_b$.

We set $\textnormal{JL}(\sigma):=(\textnormal{rec}_{\mathbf{Q}_p}^{M_b})^{-1}(\textnormal{rec}_{\mathbf{Q}_p}^{H_b}(\sigma))\in \textnormal{Mod}_{\textnormal{sm}}(M_b(\mathbf{Q}_p),\overline{\mathbf{Q}}_{\ell})$. We note that if $\sigma$ has supercuspidal support $[H_b^{\textnormal{sc}}(\mathbf{Q}_p),\sigma^{\textnormal{sc}}]$, and $M_b^{\textnormal{sc}}\subset M_b$ is the Levi subgroup associated with $H_b^{\textnormal{sc}}$, then $\textnormal{JL}(\sigma)$ is a subquotient of $\textnormal{n-Ind}_{M_b^{\textnormal{sc}}(\mathbf{Q}_{p})}^{M_b(\mathbf{Q}_p)}\textnormal{JL}(\sigma^{\textnormal{sc}})$.

\begin{theorem}[{``local-local compatibility"}] \label{LocalLocal}
    Assume that $H=\textnormal{GL}_m$, $b\in B(H)$ and that $\sigma$ is a smooth irreducible $\overline{\mathbf{Q}}_{\ell}$-representation of $H_b(\mathbf{Q}_p)$. Let $\pi$ be any Jordan--H\"older constituent of $\textnormal{n-Ind}_{Q_b(\mathbf{Q}_p)}^{H(\mathbf{Q}_p)}\textnormal{JL}(\sigma)$ where $Q_b\leq H$ denotes the standard parabolic subgroup such that $H_b$ is an inner form of its Levi quotient $M_b$.

    Then we have an isomorphism
    \begin{equation*}
\varphi_{H,\sigma}^{\textnormal{FS}}\cong\textnormal{rec}_{\mathbf{Q}_p}(\pi)^{\textnormal{ss}}
    \end{equation*}
    of semisimple $L$-parameters $W_{\mathbf{Q}_p}\to \textnormal{GL}_m(\overline{\mathbf{Q}}_{\ell})$.
     Here $\textnormal{rec}_{\mathbf{Q}_p}(-)$ is the classical local Langlands correspondence for general linear groups constructed in \cite{HT01}.

    In particular, if $\varphi_{H,\sigma}^{\textnormal{FS}}$ is a generic unramified $L$-parameter, then $H_b$ has to be the quasi-split inner form (i.e. $H_b=M_b$).
\end{theorem}
\begin{proof}
By \cite[Theorem IX.0.5 (vi) and Theorem IX.7.2]{FS24}, to see the first claim, we can assume that $H_b$ is an inner form of $H=\textnormal{GL}_m$. In that case, it is exactly \cite[Theorem 6.6.1.]{HKW22}.

    For the last part, note that $\varphi_{H,\sigma}^{\textnormal{FS}}\cong\textnormal{rec}_{\mathbf{Q}_p}(\pi)^{\textnormal{ss}}$ being unramified and generic means that the Jordan--H\"older constituent $\pi$ of $\textnormal{n-Ind}_{Q_b(\mathbf{Q}_p)}^{H(\mathbf{Q}_p)}\textnormal{JL}(\sigma)$ is a generic unramified principal series. In particular, if $(\sigma^{\textnormal{sc}},H_b^{\textnormal{sc}}(\mathbf{Q}_p))$ denotes the supercuspidal support of $\sigma$, and $Q_b^{\textnormal{sc}}\leq Q_b$ is the standard parabolic subgroup with Levi quotient $M_b^{\textnormal{sc}}\leq M_b$ corresponding to $H_{b}^{\textnormal{sc}}\leq H_b$, then $\textnormal{n-Ind}_{Q_b^{\textnormal{sc}}(\mathbf{Q}_p)}^{H(\mathbf{Q}_p)}\textnormal{JL}(\sigma^{\textnormal{sc}})$ admits a generic unramified principal series as a Jordan--H\"older constituent. As $\textnormal{JL}(\sigma^{\textnormal{sc}})$ is a discrete representation of $M_b^{\textnormal{sc}}(\mathbf{Q}_p)$, this can only happen if $M_b^{\textnormal{sc}}$ is a torus, forcing $H_b$ to be quasi-split.
\end{proof}

\subsection{Mantovan's product formula}
Recall that $\Lambda$ denoted a local Artinian algebra over a complete DVR $\mathcal{O}$ over $\mathbf{Z}_{\ell}[p^{1/2}]$ and set $C=\widehat{\overline{\mathbf{Q}}}_p$. We revisit our global setup of the PEL type A Shimura datum associated with the quasisplit unitary similitude group $\mathbf{G}$ attached to an integer $n\geq 2$ and imaginary CM field $F$ with corresponding tower of Shimura varieties $\{\textnormal{Sh}_K\}_{K\subset \mathbf{G}(\mathbf{A}^{\infty})}$. Following the strategy of \cite{Kos21}, an essential tool of our proof of Theorem~\ref{VanishingUpToBoundedTorsion} will be Mantovan's product formula. We will use the version proved in \cite{HL25} and follow their notation.

For any sufficiently small compact open subgroup $K\subset \mathbf{G}(\mathbf{A}^{\infty})$, we set
\begin{equation*}
    \mathcal{S}_K:=(\textnormal{Sh}_{K,\mathbf{Q}_p})^{\textnormal{ad},\Diamond}\subset \mathcal{S}_K^{\ast}:=(\textnormal{Sh}_{K,\mathbf{Q}_p}^{\ast})^{\textnormal{ad},\Diamond},
\end{equation*}
the diamonds over $\textnormal{Spd}(\mathbf{Q}_p)$ associated with the Shimura variety $\textnormal{Sh}_{K,\mathbf{Q}_p}$ and its minimal compactification $\textnormal{Sh}_{K,\mathbf{Q}_p}^{\ast}$. We further consider their infinite level counterparts
\begin{equation*}
    \mathcal{S}_{K^p}:=\varprojlim_{K_p\subset \textnormal{G}(\mathbf{Q}_p)}\mathcal{S}_{K^pK_p}\subset \mathcal{S}_{K^p}^{\ast}:=\varprojlim_{K_p\subset \textnormal{G}(\mathbf{Q}_p)}\mathcal{S}_{K^pK_p}^{\ast}.
\end{equation*}
As shown in \cite{Sch15}, for $C=\widehat{\overline{\mathbf{Q}}}_p$, $\mathcal{S}_{K^p,C}^{}$ and $\mathcal{S}_{K^p,C}^{\ast}$ are in fact represented by perfectoid spaces and are referred to as perfectoid Shimura varieties. In \cite{CS17}, and \cite{caraiani_scholze} $\textnormal{G}(\mathbf{Q}_p)$- and Hecke-equivariant Hodge--Tate period morphisms

\begin{equation*}
    \begin{tikzcd}
	{[\mathcal{S}_{K^p,C}/\underline{\textnormal{G}(\mathbf{Q}_p)}]} && {[\mathcal{F}l_{\textnormal{G},\mu^{-1}}/\underline{\textnormal{G}(\mathbf{Q}_p)}]} \\
	{[\mathcal{S}_{K^p,C}^{\ast}/\underline{\textnormal{G}(\mathbf{Q}_p)}]} & {}
	\arrow["{\pi_{\textnormal{HT}}}", from=1-1, to=1-3]
	\arrow[hook, from=1-1, to=2-1]
	\arrow["{\pi_{\textnormal{HT}}^{\ast}}"', from=2-1, to=1-3]
\end{tikzcd}
\end{equation*}
were constructed towards the flag variety\footnote{We note that the appearance of $\mu^{-1}$ instead of $\mu$ in the notation originates in the difference of the normalisations of the Schubert cells $\textnormal{Gr}_{\textnormal{G},\mu}$ in \cite{CS17} and \cite{SW20}. We are using the normalisations introduced in \cite{SW20}.} $\mathcal{F}l_{\textnormal{G},\mu^{-1}}:=( \textnormal{G}/\textnormal{P}_{\mu})^{\textnormal{ad},\Diamond}_C$. The Bialynicki--Birula map identifies this flag variety with the de Rham affine Grassmannian $\textnormal{Gr}_{\textnormal{G},\mu^{-1}}$ of \cite{SW20}, the functor sending $S\in \textnormal{Perf}_C$ to the set of modifications $\mathcal{E}_1\dashrightarrow \mathcal{E}_0$ of the trivial $\textnormal{G}$-bundle $\mathcal{E}_0$ over the relative Fargues--Fontaine curve $X_S$ at the untilt $S\to \textnormal{Spd}(C)$ with meromorphy given by $\mu$. In particular, under the identification $\textnormal{Bun}_{\textnormal{G},C}^1\cong[\textnormal{Spd}(C)/\underline{\textnormal{G}(\mathbf{Q}_p)}]$, the structure map
\begin{equation*}
    h^{\rightarrow}:[\mathcal{F}l_{\textnormal{G},\mu^{-1}}/\underline{\textnormal{G}(\mathbf{Q}_p)}]\to[\textnormal{Spd}(C)/\underline{\textnormal{G}(\mathbf{Q}_p)}]\cong \textnormal{Bun}_{\textnormal{G},C}^1
\end{equation*}
is identified with the map sending $\mathcal{E}_1\dashrightarrow \mathcal{E}_2$ to $\mathcal{E}_2$. For the structure map \begin{equation*}
    f_{K^p}:[\mathcal{S}_{K^p}/\underline{\textnormal{G}(\mathbf{Q}_p)}]\to[\textnormal{Spd}(C)/\underline{\textnormal{G}(\mathbf{Q}_p)}]\cong\textnormal{Bun}_{\textnormal{G},C}^1
\end{equation*}
we consider the compactly supported cohomology complex
\begin{equation*}
R\Gamma_c(\mathcal{S}_{K^p},\Lambda):=Rf_{K^p,!}\Lambda=Rh_{\ast}^{\rightarrow}R\pi_{\textnormal{HT,!}}\Lambda\in D_{\textnormal{sm}}(\textnormal{G}(\mathbf{Q}_p),\Lambda).
\end{equation*}
We also note that, on top of the natural $\textnormal{G}(\mathbf{Q}_p)$-action, one defines a prime-to-$p$ Hecke action on these complexes via correspondences.

For $b\in B(\textnormal{G},\mu^{-1})$, we pick a completely slope divisible $p$-divisible group $\mathbf{X}_b$ and write $\textnormal{Ig}_{K^p}^b$ for the perfect Igusa variety associated with it by \cite{CS17}, a perfect $\overline{\mathbf{F}}_p$-scheme admitting a natural action of $\underline{\textnormal{G}_b(\mathbf{Q}_p)}$.
We will denote by $\textnormal{Ig}_{K^p}^{b,\ast}$ its partial minimal compactification and by $g_b\colon\textnormal{Ig}^b_{K^p}\hookrightarrow \textnormal{Ig}_{K^p}^{b,\ast}$ the open embedding with dense image, both constructed in \cite{caraiani_scholze}. It is a perfect normal scheme that turns out to be \textit{affine} (cf. \textit{loc. cit.}, Theorem 2.8.1) and consequently, by a version of Hartogs's principle, we get that $\textnormal{Ig}^{b,\ast}_{K^p}=\textnormal{Spec}(\mathcal{O}(\textnormal{Ig}_{K^p}^b))$. In particular, we obtain an extension of the $\underline{\textnormal{G}_b(\mathbf{Q}_p)}$-action compatible with the map $g_b$. We consider the partially compactly supported cohomology
\begin{equation*}
    V_b:=R\Gamma_{c-\partial}(\textnormal{Ig}^b_{K^p},\Lambda):=R\Gamma(\textnormal{Ig}^{b,\ast}_{K^p},g_{b!}\Lambda)
\end{equation*}
that is naturally an object of $D_{\textnormal{sm}}(\textnormal{G}_b(\mathbf{Q}_p),\Lambda)$ and admits an action of Hecke operators away from $p$. One has the following crucial result about partially compactly supported cohomology of Igusa varieties.
\begin{theorem}\label{CohOfIV's}
    The complex $V_b$ vanishes outside the range $[0,d_b]$ for $d_b:=\langle2\rho_{\textnormal{G}},\nu_b\rangle=\dim( \textnormal{Ig}_{K^p}^b)$ and is an admissible object of $D_{\textnormal{sm}}(\textnormal{G}_b(\mathbf{Q}_p),\Lambda)$ in the sense that $V_b^{K}$ is a perfect complex of $\Lambda$-modules for every pro-$p$ compact open subgroup $K\subset \textnormal{G}_b(\mathbf{Q}_p)$.
\end{theorem}
\begin{proof}
    The vanishing is \cite[Proposition 2.8.2]{caraiani_scholze} (see \cite[Proposition 3.7]{HL25} as well) and follows from Artin's vanishing theorem using that $\textnormal{Ig}_{K^p}^{b,\ast}$ is affine.

    For admissibility, see the proof of \cite[Proposition 8.21]{zhang_thesis}.
\end{proof}
We now spell out the relation of $V_b$ with $R\pi_{\textnormal{HT!}}\Lambda$ and therefore with the compactly supported cohomology of the Shimura variety. Consider the map
\begin{align*}
    h^{\leftarrow} \colon [\mathcal{F}l_{\textnormal{G},\mu^{-1}}/\underline{\textnormal{G}(\mathbf{Q}_p)}] & \to \textnormal{Bun}_{\textnormal{G},\mu^{-1}} \\
    (\mathcal{E}_1\dashrightarrow \mathcal{E}_2) &\mapsto \mathcal{E}_1
\end{align*}
and write $\mathcal{F}l_{\textnormal{G},\mu^{-1}}^b$ for the fiber of $j_{b,\mu^{-1}}\colon\textnormal{Bun}_{\textnormal{G}}^b\hookrightarrow\textnormal{Bun}_{\textnormal{G},\mu^{-1}}$ along this map. As observed in \cite{CS17}, \cite{caraiani_scholze}, the fibers of $R(\pi_{\textnormal{HT}}^{\ast})_{\ast}\Lambda$ over $\mathcal{F}l_{\textnormal{G},\mu^{-1}}^b$ compute $R\Gamma(\textnormal{Ig}^{b,\ast}_{K^p},\Lambda)$. In \cite{HL25}, it was proved that the fibers of $R\pi_{\textnormal{HT}!}\Lambda$ compute $V_b$ and a further interpolation of this result is proved as well. Namely, consider the Cartesian square
\begin{equation*}
    \begin{tikzcd}
	{[\mathcal{F}l_{\textnormal{G},\mu^{-1}}^b/\underline{\textnormal{G}(\mathbf{Q}_p)}]} && {\textnormal{Bun}_{\textnormal{G}}^b} \\
	{[\mathcal{F}l_{\textnormal{G},\mu^{-1}}/\underline{\textnormal{G}(\mathbf{Q}_p)}]} && {\textnormal{Bun}_{\textnormal{G},\mu^{-1}}.} \\
	& {}
	\arrow["{h_b^{\leftarrow}}", from=1-1, to=1-3]
	\arrow["{i_b}"', hook', from=1-1, to=2-1]
	\arrow["{j_{b,\mu^{-1}}}", hook', from=1-3, to=2-3]
	\arrow["{h^{\leftarrow}}"', from=2-1, to=2-3]
\end{tikzcd}
\end{equation*}
Then \cite[\S3]{HL25} shows that there is a canonical isomorphism
\begin{equation*}
    i_b^{\ast}(R\pi_{\textnormal{HT}!}\Lambda)\cong h_b^{\leftarrow\ast}(V_b)
\end{equation*}
where we view $V_b$ as a sheaf on $\textnormal{Bun}_{\textnormal{G}}^b$ via $D(\textnormal{Bun}_{\textnormal{G}}^b,\Lambda)\cong D_{\textnormal{sm}}(\textnormal{G}_b(\mathbf{Q}_p),\Lambda)$. By proper base change (\cite[Proposition 22.19]{Sch17}), after applying $Rh^{\rightarrow}_{\ast}i_{b!}$, we obtain an isomorphism
\begin{equation*}
    R\Gamma(\mathcal{F}l_{\textnormal{G},{\mu^{-1}}},i_{b!}i_b^{\ast}(R\pi_{\textnormal{HT}!}\Lambda))\cong Rh^{\rightarrow}_{\ast}h^{\leftarrow\ast}j_{b,\mu^{-1}!}V_b
\end{equation*}
in $D_{\textnormal{sm}}(\textnormal{G}(\mathbf{Q}_p),\Lambda)$. Finally, we can relate these to the Hecke operators $T_{\mu,C}$
using the following diagram

\begin{equation*}
    \begin{tikzcd}
	&& {[\mathcal{F}l_{\textnormal{G},\mu^{-1}}/\underline{\textnormal{G}(\mathbf{Q}_p)}]} \\
	{\textnormal{Bun}_{\textnormal{G},\mu^{-1}}} && {\textnormal{Hck}_{\textnormal{G},\leq\mu}} & {} & {\textnormal{Bun}^1_{\textnormal{G}}\times \textnormal{Spd}(C)} \\
	{\textnormal{Bun}_{\textnormal{G}}} &&&& {\textnormal{Bun}_{\textnormal{G}}\times\textnormal{Spd}(C)}
	\arrow["{h^{\leftarrow}}"', from=1-3, to=2-1]
	\arrow["{\widetilde{j}_1}"', hook, from=1-3, to=2-3]
	\arrow["{h^{\rightarrow}}", from=1-3, to=2-5]
	\arrow["{j_{\mu^{-1}}}", hook, from=2-1, to=3-1]
	\arrow["{h_{\mu}^{\leftarrow}}", from=2-3, to=3-1]
	\arrow["{h_{\mu}^{\rightarrow}}"', from=2-3, to=3-5]
	\arrow["{j_1}", hook, from=2-5, to=3-5]
\end{tikzcd}
\end{equation*}
where $\widetilde{j}_1$ is the open embedding of the locus where $\mathcal{E}_2$ is geometrically fiberwise trivial. In particular, the right square is Cartesian. We then compute
\begin{align*}
    Rh_{\ast}^{\rightarrow}h^{\leftarrow\ast}j_{b,\mu^{-1},!}(V_b) & \cong
    Rh_{\ast}^{\rightarrow}\widetilde{j}_1^{\ast}h^{\leftarrow\ast}_{\mu}j_{b!}(V_b) \\
    & \cong j_1^{\ast}Rh_{\mu\ast}^{\rightarrow}h_{\mu}^{\leftarrow\ast}j_{b!}(V_b),
\end{align*}
where the latter holds by proper base-change.
Using the formula $T_{\mu}(-)=Rh_{\mu\ast}^{\rightarrow}(h^{\leftarrow\ast}_{\mu}(-))[d_{b_{\mu}}]$ and noting that $d_{b_{\mu}}=d=\dim \textnormal{Sh}_{K}$, we arrive at Mantovan's formula \cite[Corollary 3.18]{HL25}.\footnote{We note that we are ignoring the action of the Weil group $W_{\mathbf{Q}_p}$.}

\begin{theorem}[Mantovan's formula]\label{thm_mantovan}
    There is a prime-to-$p$ Hecke equivariant filtration on the complex $R\Gamma_c(\mathcal{S}_{K^p,C},\Lambda)$ in $D_{\textnormal{sm}}(\textnormal{G}(\mathbf{Q}_{p}),\Lambda)$ labelled by $B(\textnormal{G},\mu^{-1})$ with graded pieces given by admissible objects
    \begin{equation*}
    R\Gamma_{K^p,b}:=j_1^{\ast}T_{\mu,C}\left(j_{b!}(V_b[-d])\right)\in D_{\textnormal{sm}}(\textnormal{G}(\mathbf{Q}_p),\Lambda).
    \end{equation*}
\end{theorem}
\begin{proof}
    Most of the statement is \cite[Corollary 3.18]{HL25} when $\Lambda=\overline{\mathbf{F}}_{\ell}$. The proof works with general $\ell$-torsion coefficients $\Lambda$ as well. 
    
    To see that the filtration is Hecke equivariant, we note that the filtrations are compatible when varying $K^p$ by construction. 
    
    Finally, to see that $R\Gamma_{K^p,b}$ is admissible, we recall that along the identifications $D(\textnormal{Bun}_{\textnormal{G}}^{b},\Lambda)\cong D_{\textnormal{sm}}(\textnormal{G}_{b}(\mathbf{Q}_p),\Lambda)$, admissibility corresponds to being ULA and the functors $j_{b!}$, $T_{\mu,C}$ and $j_1^{\ast}$ all preserve the ULA property. In particular, admissibility follows from the admissibility of $V_b$.
\end{proof}
\begin{remark}
     Even though we will not need it explicitly, it certainly deserves mentioning that the graded pieces $R\Gamma_{K^p,b}$ are described in \cite{HL25} in more classical terms. Namely, they are identified with
    \begin{equation*}
R\Gamma_c(G,b,\mu,\Lambda)\otimes^{\mathbf{L}}_{\mathcal{H}(\textnormal{G}_b)}V_b[2d_b]
    \end{equation*}
    where $R\Gamma_c(G,b,\mu,\Lambda)$ is a certain (normalised) compactly supported cohomology of the infinite level local Shimura variety $\mathcal{M}_{(G,b,\mu),\infty,C}$ associated with the local Shimura datum $(\textnormal{G},b,\mu)$ parametrising modifications $\mathcal{E}_b\dashrightarrow \mathcal{E}_0$ at $S\to \textnormal{Spd}(C)$ of meromorphy $\mu$ and $\mathcal{H}(\textnormal{G}_b):=C_c(\textnormal{G}_b(\mathbf{Q}_p),\Lambda)$ is the smooth Hecke algebra for $\textnormal{G}_b(\mathbf{Q}_p)$.
\end{remark}
\subsubsection{Cohomology of local Shimura varieties}
By the definition of the Bernstein centre, we have a map
\begin{equation*}
    t_1^b\colon\mathfrak{Z}(\textnormal{G})_{\mathcal{O}}\to\textnormal{End}_{D_{\textnormal{sm}}(\textnormal{G}(\mathbf{Q}_p),\Lambda)}(R\Gamma_{K^p,b}).
\end{equation*}
On the other hand, by acting with $\mathfrak{Z}(\textnormal{G}_b)_{\mathcal{O}}$ on $V_b$, we obtain a map
\begin{equation*}
    t_b^b\colon\mathfrak{Z}(\textnormal{G}_b)_{\mathcal{O}}\to \textnormal{End}_{D_{\textnormal{sm}}(\textnormal{G}_b(\mathbf{Q}_p),\Lambda)}(R\Gamma_{K^p,b}).
\end{equation*}
As a consequence of the Fargues--Scholze formalism, we have the following compatibility result, a shadow of the Kottwitz conjecture for the local Shimura variety $\mathcal{M}_{(\textnormal{G},b,\mu),\infty,C}$.
\begin{prop}\label{KottwitzConjShadow}
    The following is a commutative diagram
    
    \begin{equation*}
        \begin{tikzcd}
	&& {\mathfrak{Z}(\textnormal{G})_{\mathcal{O}}} \\
	{\mathcal{Z}^{\textnormal{spec}}(\textnormal{G},\mathcal{O})} &&&& {Z\left(\textnormal{End}_{D_{\textnormal{sm}}(\textnormal{G}(\mathbf{Q}_p),\Lambda)}(R\Gamma_{K^p,b})\right).} \\
	&& {\mathfrak{Z}(\textnormal{G}_b)_{\mathcal{O}}}
	\arrow["{t_1^b}", from=1-3, to=2-5]
	\arrow["{\Psi_{\textnormal{G}}^1}", from=2-1, to=1-3]
	\arrow["{\Psi_{\textnormal{G}}^b}"', from=2-1, to=3-3]
	\arrow["{t_b^b}"', from=3-3, to=2-5]
\end{tikzcd}
    \end{equation*}
\end{prop}
\begin{proof}
    Recall that by definition we have $R\Gamma_{K^p,b}=j_1^{\ast}T_{\mu,C}j_{b!}(V_b[-d])$.
    Given $z\in\mathcal{Z}^{\textnormal{spec}}(\textnormal{G},\mathcal{O})$, it suffices to show the equality \[
        \Psi_{\textnormal{G}}^1(z) = j_1^{\ast}T_{\mu,C}j_{b!} \Psi_{\textnormal{G}}^b(z)
    \]
    of endomorphisms of $j_1^{\ast}T_{\mu,C}j_{b!} A$ for all objects $A$ in $D_{\textnormal{sm}}(\textnormal{G}_{b}(\mathbf{Q}_p),\Lambda)$.

    Since $\Psi_\textnormal{G}^{\textnormal{geom}}(z)=j_{b!}\Psi_{\textnormal{G}}^b(z)$ by definition and excursion operators commute with Hecke operators, it suffices to prove the equality
    \[
        \Psi_{\textnormal{G}}^1(z) = j_1^* \Psi_{\textnormal{G}}^{\textnormal{geom}}(z)
    \]
    as elements of the endomorphism ring of $j_1^{\ast}T_{\mu,C}j_{b!} A$.

    By Yoneda, it suffices to check that for every $B\in D_{\textnormal{sm}}(\textnormal{G}(\mathbf{Q}_p),\Lambda)$ the two actions agree on
    \begin{equation}\label{adjunctionequation}
        \Hom_{D_{\textnormal{sm}}(\textnormal{G}(\mathbf{Q}_p),\Lambda)}(B,j_1^{\ast}T_{\mu,C}j_{b!} A)=\Hom_{D(\textnormal{Bun}_{\textnormal{G}},\Lambda)}(j_{1!}B,T_{\mu,C}j_{b!} A)
    \end{equation}
    where we used that $j_1^{\ast}$ admits $j_{1!}$ as left adjoint as $j_1$ is an open embedding (cf. \cite{Sch17}, Definition/Proposition 19.1). However, in the endomorphism ring of the RHS of \ref{adjunctionequation} we have
    \begin{equation*}
        \Hom(\id,\Psi_{\textnormal{G}}^{\textnormal{geom}}(z))=\Hom(\Psi_{\textnormal{G}}^{\textnormal{geom}}(z),\id)=\Hom(j_{1!}\Psi_{\textnormal{G}}^1(z),\id)
    \end{equation*}
    using that $\Psi_{\textnormal{G}}^{\textnormal{geom}}(z)$ lies in the centre of the category and the definition of $\Psi^1_{\textnormal{G}}$. In particular, by functoriality of adjunction, and centrality of $\Psi^1_{\textnormal{G}}(z)$, we obtain the equality
    \begin{equation*}
        \Hom(\id,j_1^{\ast}\Psi_{\textnormal{G}}^{\textnormal{geom}}(z))=\Hom(\id, \Psi^1_{\textnormal{G}}(z))
    \end{equation*}
    in the endomorphism ring of the LHS of \ref{adjunctionequation}
    as desired.
\end{proof}

\begin{para} Finally, we will need a consequence of an instance of the Harris--Viehmann conjecture for $\mathcal{M}_{(\textnormal{G},b_{\mu},\mu),C,\infty}$. Let  $\textnormal{P}_{\mu}$ denote the (non-standard)\footnote{This is the parabolic subgroup containing the opposite of the product of the Borel subgroups of lower triangular matrices.} dynamic parabolic, and  $\textnormal{M}_{\mu}$ its Levi quotient, the centraliser of $\mu$ in $\textnormal{G}$. Note that for the $\mu$-ordinary element $b_{\mu}$ we have $\textnormal{G}_{b_{\mu}}=\textnormal{M}_{\mu}$.
\end{para}

\begin{prop}\label{RapoportViehmann}
    For the $\mu$-ordinary element $b_{\mu}\in B(\textnormal{G},\mu^{-1})$, the functor
    \begin{equation*}
    j_1^{\ast}T_{\mu,C}j_{b_{\mu}!}(-):D_{\textnormal{sm}}(\textnormal{M}_{\mu}(\mathbf{Q}_p),\Lambda)\to D_{\textnormal{sm}}(\textnormal{G}(\mathbf{Q}_p),\Lambda)
   \end{equation*}
can be identified with the exact functor $\textnormal{Ind}_{\textnormal{P}_{\mu}(\mathbf{Q}_p)}^{\textnormal{G}(\mathbf{Q}_p)}(-)[d]$. In particular, the complex $R\Gamma_{K^p, b_{\mu}}$ is concentrated in $[0,d]$.
\end{prop}
\begin{proof}
    This follows from \cite{GI22} as shown in \cite[Proposition 10.2.5]{DvHKZ24}.
\end{proof}

\subsection{The proof} \label{sec_VanishingProof}
Let $X^{\textnormal{tame}} \subset Z^1(W_{\rationals_p}, \widehat{\textnormal{G}})$ be the connected component of tame parameters. The functor of points of $X^{\textnormal{tame}}$ can be described as
\[
    X^{\textnormal{tame}}(R) = \{ ((\phi_0, \tau_0), (\phi_v, \tau_v)) \in \mathbf{G}_m(R)^2 \times (\GL_{2n}(R)^2)^{[F^+ : \rationals]} : \phi_v \tau_v \phi_v^{-1} = \tau_v^p \}.
\]
Let $\ad^0 \colon \GL_{2n} \to \GL(\mathfrak{sl}_{2n})$ be the adjoint representation.
Let $\Delta^{\textnormal{spec}}_\phi \in \mathcal{O}(X^{\textnormal{tame}})^{\widehat{\textnormal{G}}} \subset \mathcal{Z}^{\textnormal{spec}}(\textnormal{G}, \mathcal{O})$ be the conjugation invariant function
defined by
\[
    \Delta^{\textnormal{spec}}_\phi = \prod_{v \mid p} \det(\ad^0(\phi_v) - p \cdot \operatorname{id}).
\]
After specialising at a $\overline{\mathbf{Q}}_{\ell}$-point $x$ of $X^{\textnormal{tame}}$, we have the formula 
\[
    \Delta^{\textnormal{spec}}_\phi = (1 - p)^{[F^+:\rationals] (2n-1)} \prod_{v \mid p} \prod_{i \neq j} (\alpha_{v,i} \alpha_{v,j}^{-1} - p), 
\]
where $\{\alpha_{v,i} \}$ are the eigenvalues of $x(\phi_v)$.
Moreover, we define
\[
    \Delta^{\textnormal{spec}}_{\tau} = f(\tau_0) \prod_{v \mid p} \det(f(\tau_v)) \in \mathcal{O}(X^{\textnormal{tame}})^{\widehat{\textnormal{G}}},
\]
where 
\[
f(x) = \frac{x^{ N } - 1}{x - 1} \qquad \textnormal{and} \qquad N = \prod_{m = 1}^{2n} (p^{m} - 1).
\]
Finally, we set $\Delta^{\textnormal{spec}} = \Delta^{\textnormal{spec}}_{\phi} \Delta^{\textnormal{spec}}_{\tau}$.

\begin{lemma} \label{deltaspecpoints_lemma}
Let $y \colon \Spec \overline{\rationals}_\ell \to X^\textnormal{tame}\sslash \widehat{\textnormal{G}}$ be a point corresponding to a semisimple representation $\varphi_y \colon W_{\rationals_p} \to \widehat{\textnormal{G}}(\overline{\rationals}_\ell)$.
Then $\Delta^{\textnormal{spec}}(y) \neq 0$ if and only if $\varphi_y$ is a generic unramified $L$-parameter.
\end{lemma}

\begin{proof}
    The relation $\tau_v^p = \phi_v \tau_v \phi_v^{-1}$ implies that the eigenvalues of $\tau_v$ are roots of unity of order $p^m - 1$ for some $1 \leq m \leq 2n$.
    Hence, we see that $\Delta_{\tau}^{\textnormal{spec}}(y) \neq 0$ if and only if $\varphi_y$ is unramified. For such parameters it follows from the definition that $\Delta^{\textnormal{spec}}_\phi(y) \neq 0$ if and only if $\varphi_y$ is generic.
    Therefore, $\Delta^{\textnormal{spec}}(y) = \Delta^{\textnormal{spec}}_{\phi}(y) \Delta^{\textnormal{spec}}_{\tau}(y) \neq 0$ if and only if $\varphi_y$ is a generic unramified $L$-parameter.
\end{proof}

\begin{para}
For $b\in B(\textnormal{G})$, we define
\begin{equation*}
    \Delta^b:=\Psi_{\textnormal{G}}^b(\Delta^{\textnormal{spec}}) \in \mathfrak{Z}(\textnormal{G}_b)_{\mathcal{O}}.
\end{equation*}
\end{para}

\begin{lemma} \label{Qlbarpoint_lemma}
    Let $b \in B(\textnormal{G})$ and $x, y \in \mathfrak{Z}(\textnormal{G}_b)_{\mathcal{O}}$. Then we have that $x = y$ if and only if 
    $h_\pi(x) = h_\pi(y)$ for all smooth irreducible $\overline{\rationals}_\ell$-representations $\pi$ of $\textnormal{G}_b$, where $h_\pi \colon \mathfrak{Z}(\textnormal{G}_b)_{\mathcal{O}} \to \End(\pi) \cong \overline{\rationals}_\ell$ is the corresponding ring homomorphism.
\end{lemma}

\begin{proof}
    By \cite[Cor 3.3]{DHKM24} it suffices to prove that the image of $x - y$
    in $\mathfrak{Z}(\textnormal{G}_b)_{\mathcal{O}}[1/\ell] \subset \mathfrak{Z}(\textnormal{G}_b)_{\overline{\rationals}_\ell}$ vanishes.
    Since
    \[
    \mathfrak{Z}(\textnormal{G}_b)_{\overline{\rationals}_\ell} = \lim_{\substack{\leftarrow \\ K}} Z(\mathcal{H}(\textnormal{G}_b, K)_{\overline{\rationals}_\ell}),
    \]
    it suffices to show that the image $z_K$ of $x - y$ in $Z(\mathcal{H}(\textnormal{G}_b, K)_{\overline{\rationals}_\ell})$ vanishes for a system of shrinking open compact subgroups $K$. By \cite[2.13 and 3.9]{bernstein_deligne84}, there is such a system such that each $Z(\mathcal{H}(\textnormal{G}_b, K)_{\overline{\rationals}_\ell})$ is a reduced finitely generated $\overline{\rationals}_\ell$-algebra. By the Nullstellensatz, it suffices to show that $z_K \in \mathfrak{m}$ for every maximal ideal of $Z(\mathcal{H}(\textnormal{G}_b, K)_{\overline{\rationals}_\ell})$. Such a maximal ideal corresponds to a supercuspidal support $(M, \sigma)$ by \cite[2.13]{bernstein_deligne84} and the image of $z_K$ in $Z(\mathcal{H}(\textnormal{G}_b, K)_{\overline{\rationals}_\ell})/\mathfrak{m} \cong \overline{\rationals}_\ell$ is equal to $h_\pi(z)$, where $\pi$ is any smooth irreducible representation with supercuspidal support $(M, \sigma)$. 
\end{proof}

\begin{lemma} \label{keylemma}
    If $b \in B(\textnormal{G}, \mu^{-1})$ and $b \neq b_{\mu}$ , then $\Delta^b = 0$.
\end{lemma}

\begin{proof}
    By Lemma \ref{Qlbarpoint_lemma} it suffices to prove that $\Delta^b$ acts as $0$ on any smooth irreducible $\overline{\mathbf{Q}}_{\ell}$-representation $\sigma$ of $\textnormal{G}_b(\mathbf{Q}_p)$.
    As the group $\textnormal{G}_b$ is not quasi-split by Example \ref{KottwitzSetExample}, this follows from the second part of Theorem \ref{LocalLocal} and Lemma \ref{deltaspecpoints_lemma}.
\end{proof}

\begin{lemma} \label{delta_compare_lemma}
    The image of $\Delta^1 = \Psi_{\textnormal{G}}(\Delta^{\textnormal{spec}})$ under the natural map $\mathfrak{Z}(\textnormal{G})_\cO \to \mathcal{H}$ coincides with 
    $M \Delta$, where $\Delta$ is the element defined in paragraph \ref{Deltapara} and
    \begin{equation*}
        M=(1-p)^{[F^+:\mathbf{Q}](2n-1)}N^{2n[F^+:\mathbf{Q}]+1}.
    \end{equation*}
\end{lemma}

\begin{proof}
    By Lemma \ref{Qlbarpoint_lemma}, it suffices to prove that 
    \[
    h_\pi(\Psi_{\textnormal{G}}(\Delta^{\textnormal{spec}})) = h_\pi(M\Delta)
    \]
    for all unramified smooth irreducible $\overline{\mathbf{Q}}_{\ell}$-representations $\pi$ of $\textnormal{G}(\mathbf{Q}_p)$. 
    Consider unramified characters $\chi_{v,1}, \dots, \chi_{v,2n} \colon \rationals_p^\times \to \overline{\rationals}_\ell^{\times}$ such that $\pi$ is an unramified twist of a subquotient of the normalised parabolic induction $\textnormal{n-Ind}_{\textnormal{B}(\mathbf{Q}_p)}^{\textnormal{G}( \mathbf{Q}_p)}(\bigotimes_{v|p} \chi_{v,1} \otimes \dots \otimes\chi_{v,2n})$. Then $\varphi^{\textnormal{FS}}_{\textnormal{G},\pi}$ is a twist of $(\oplus_{i=1}^{2n} \varphi^{\textnormal{FS}}_{\mathbf{G}_m,\chi_{v,i}})_{v | p}$ and
    \[
        h_\pi(\Psi_{\textnormal{G}}(\Delta^{\textnormal{spec}}_\phi)) = (1-p)^{[F^+:\mathbf{Q}](2n-1)}\prod_{v \mid p} \prod_{i \neq j} (\chi_{v,i}(p)\chi_{v,j}(p)^{-1} - p )
    \]
    and
    \[
        h_\pi(\Psi_{\textnormal{G}}(\Delta^{\textnormal{spec}}_\tau)) = f(1) \prod_{v \mid p} f(1)^{2n} = N^{1 + 2n [F^+ : \rationals]}.
    \]
    Hence, the claim follows from the compatibility between the Satake isomorphism and normalised induction.
\end{proof}

\begin{corollary}\label{nonordinaryvanishing}
    For every element $b\in B(\textnormal{G},\mu^{-1})$ different from the $\mu$-ordinary element $b_{\mu}$, $\Delta^1$ is sent to the zero endomorphism under
    \begin{equation*}
        t_1^b\colon\mathfrak{Z}(\textnormal{G})_{\mathcal{O}}\to \textnormal{End}_{D_{\textnormal{sm}}(\textnormal{G}(\mathbf{Q}_p),\Lambda)}(R\Gamma_{K^p,b}).
    \end{equation*}
\end{corollary}
\begin{proof}
    By Proposition~\ref{KottwitzConjShadow}, it suffices to show that $t_b^b(\Psi_{\textnormal{G}}^b(\Delta^{\textnormal{spec}}))$ vanishes. However, by Lemma~\ref{keylemma}, $\Psi_{\textnormal{G}}^b(\Delta^{\textnormal{spec}})$ is zero already, finishing the proof.
\end{proof}

We can finally prove our vanishing theorem.

\begin{theorem}\label{VanishingUpToBoundedTorsion}
    For every good subgroup $K^p\leq \mathbf{G}(\mathbf{A}^{\{\infty,p\}})$, and integer $0\leq q\leq d-1$, the Hecke operator $M\Delta\in \mathcal{H}$ acts nilpotently on the cohomology groups
    \begin{equation*}
    H^q(\textnormal{Sh}_{K^pK_p^{\textnormal{hs}}}(\mathbf{C}),\Lambda)
    \end{equation*}
    and
    \begin{equation*}
        H^{2d-q}_c(\textnormal{Sh}_{K^pK_p^{\textnormal{hs}}}(\mathbf{C}),\Lambda).
    \end{equation*}
    More precisely, each such group is annihilated by the operator \[
    M^{[F^+ : \mathbf{Q}] n + |B(\textnormal{G}, \mu^{-1})|} \Delta^{|B(\textnormal{G}, \mu^{-1})|}.
    \]
\end{theorem}
\begin{proof}
    Without loss of generality, we can take $\Lambda=\mathcal{O}/\varpi^m$ for some $m\geq 1$, an injective module over itself. Let $\iota\colon\mathcal{H}\to \mathcal{H}$ denote the involution $f\mapsto (\iota(f):g\mapsto f(g^{-1}))$ and recall that the Poincar\'e duality isomorphism\footnote{Here we denote by $(\cdot)^{\vee}$ the Pontryagin dual.}
    \begin{equation*}
        H^q(\mathcal{S}_{K^pK_p^{\textnormal{hs}},C},\Lambda)\cong H^{2d-q}_c(\mathcal{S}_{K^pK_p^{\textnormal{hs}},C},\Lambda)^{\vee}
    \end{equation*}
    is $\mathcal{H}$-equivariant when we act via $\iota$ (and transpose) on the RHS. We claim that $\iota(\Delta)=\Delta$. To see this, note that for every $\overline{\mathbf{Q}}_{\ell}$-point $x_{\pi}\colon\mathcal{H}\to \overline{\mathbf{Q}}_{\ell}$ corresponding to an irreducible induction $\pi=\textnormal{n-Ind}_{\textnormal{B}(\mathbf{Q}_p)}^{\textnormal{G}(\mathbf{Q}_p)}\chi$, we have $x_{\pi}(\iota(\Delta))=x_{\pi^{\vee}}(\Delta)$. Furthermore, the existence of an isomorphism
    \begin{equation*}
        \pi^{\vee}\cong \textnormal{n-Ind}_{\textnormal{B}(\mathbf{Q}_p)}^{\textnormal{G}(\mathbf{Q}_p)}(\chi^{\vee})
    \end{equation*}
    implies that $x_{\pi^{\vee}}(\Delta)=x_{\pi}(\Delta)$. The claim now follows noting that $\mathcal{H}$ is a reduced $\mathcal{O}$-flat finitely generated $\mathcal{O}$-algebra, and that the prime ideals corresponding to irreducible unitary inductions are dense. Consequently, it suffices to prove that $M \Delta$ acts nilpotently on $H^q_c(\mathcal{S}_{K^pK_p^{\textnormal{hs}},C},\Lambda)$ for $q\geq d+1$.

    For $K_p:=\ker(K_p^{\textnormal{hs}}\to \textnormal{G}(\mathbf{F}_p))$, we have a $\mathfrak{Z}(\textnormal{G})_{\mathcal{O}}$-equivariant spectral sequence
    \begin{equation*}
E_2^{r,s}=H^r(K_p^{\textnormal{hs}}/K_p,H^0(K_p,H^s_c(\mathcal{S}_{K^p,C},\Lambda)))\Rightarrow H^{r+s}_c(\mathcal{S}_{K^pK_p^{\textnormal{hs}}},\Lambda).
    \end{equation*}
    Suppose that $q > d$. Then $H^{q}_c(\mathcal{S}_{K^pK_p^{\textnormal{hs}}},\Lambda)$ has an abutment filtration whose graded pieces are subquotients of the groups
    \[ E_2^{r, q-r} = H^{r}(K_p^{\textnormal{hs}}/K_p,H^0(K_p,H^{q-r}_c(\mathcal{S}_{K^p,C},\Lambda))), \]
    for $r = 0, \dots, q$. The group $E_2^{0, q}$ is annihilated by $(M \Delta)^{|B(\textnormal{G}, \mu^{-1})|}$, because Corollary \ref{nonordinaryvanishing},  Proposition \ref{RapoportViehmann}, and  Theorem \ref{thm_mantovan} together show that $H^{q}_c(\mathcal{S}_{K^p,C},\Lambda)$ is annihilated by $(\Delta^1)^{|B(\textnormal{G}, \mu^{-1})|}$,  and the induced action of $\Delta^1$ on $E_2^{0, q}$ agrees with the action of $M \Delta$ by Lemma \ref{delta_compare_lemma}. If $r > 0$, then the group $E_2^{r, q-r}$ is annihilated by the prime-to-$p$ part $(p-1)N^{[F^+:\mathbf{Q}]}=(p-1)\prod_{m=1}^{2n}(p^m-1)^{[F^+:\mathbf{Q}]}$ of $|\textnormal{G}(\mathbf{F}_p)|$. 
    
    Since the number of terms in the abutment filtration is bounded by $q \leq 2d = \dim_{\mathbf{R}} \textnormal{Sh}_{K^p K_p^{\textnormal{hs}}}(\mathbf{C}) = 2 n^2[ F^+ : \mathbf{Q}]$, we conclude that $H^q_c(\mathcal{S}_{K^pK_p^{\textnormal{hs}},C},\Lambda)$ is annihilated by 
    \[ \left( (p-1) N^{[F^+ : \mathbf{Q}]} \right)^{2 n^2 [ F^+ : \mathbf{Q} ] -1} (M \Delta)^{| B(\textnormal{G}, \mu^{-1}) |}. \]
    The integer $\left( (p-1) N^{[F^+ : \mathbf{Q}]} \right)^{2 n^2 [ F^+ : \mathbf{Q} ] -1}$ divides $M^{[F^+ : \mathbf{Q}]n}$, so we conclude that $H^q_c(\mathcal{S}_{K^pK_p^{\textnormal{hs}},C},\Lambda)$ is annihilated by $M^{[F^+ : \mathbf{Q}] n} (M \Delta)^{|B(\textnormal{G}, \mu^{-1})|}$, and in particular that the action of $M \Delta$ on this group is nilpotent. 
\end{proof}

\section{Potentially semistable pseudodeformation rings}\label{sec_potentially_semistable_pseudodeformation_rings}

The goal of this section is to define spaces of pseudocharacters, i.e. determinants in the sense of \cite{Che14}, satisfying conditions from $p$-adic Hodge theory (e.g.\ potentially semistable with fixed Hodge and inertial type). Wang-Erickson \cite{WE18} showed how to do this under the condition that the residual pseudocharacter is multiplicity-free, and also under the (unproved) condition that the algebraic stack of representations with given pseudocharacter satisfies formal GAGA. Roughly speaking, the techniques of \cite{Kis08} allow one to define a closed locus in the associated formal stack of representations satisfying the given conditions; using formal GAGA, one can algebraise this locus and take its image in the spectrum of the universal pseudodeformation ring to get the desired space of pseudocharacters.

We do not wish to put any hypothesis on the residual representation, such as being multiplicity-free. We get around the need to establish formal GAGA by combining the approach of \cite{WE18} with the ideas of Wake--Wang-Erickson \cite{WWE19}. In the latter paper, the authors showed how to define a quotient of the universal pseudodeformation ring corresponding to conditions such as `semistable after passage to $L / K$, with Hodge--Tate weights in an interval $[a, b]$'. Returning to the setting of \cite{WE18}, this replaces the problem of algebraising an arbitrary closed formal substack with the much easier one of algebraising a union of connected components of its generic fibre. We complete the argument here by proving `half' of formal GAGA in our setting, namely the statement that passage to completion is a fully faithful functor for certain quotient stacks. This necessitates a detour through certain finiteness results, that we turn to first. 

\subsection{Finiteness of cohomology}
Let $S$ be a Noetherian ring and $G$ a reductive group over $S$ (i.e.\ a smooth affine group scheme whose geometric fibres are reductive groups, in particular, connected). A finitely generated $S$-algebra $A$ with $G$-action corresponds to an affine scheme of finite type over $S$ with an action of $G$. We will need the following theorem. 

\begin{theorem} \label{thm:group_coh_fg}
    \begin{enumerate}
        \item $A^G$ is a finitely generated $S$-algebra.
        \item If $M$ is a finite $A$-module with semilinear $G$-action, then for any $n \geq 0$, $H^n(G, M)$ is a finite $A^G$-module. 
    \end{enumerate}
\end{theorem}
    Part (1) is proved as \cite[Theorem 3]{vanderkallen_franjou}, using general properties of the Grosshans filtration, and part (2) is proved as \cite[Theorem 49]{vanderkallen_franjou}, under the additional assumption that $S = \mathbf{Z}$, and as \cite[Theorem 10.5]{vdK15}, without any assumption on $S$. (See also \cite[Theorem 6.3.3]{alper_adequate} for the case $n = 0$.) The deduction of \cite[Theorem 10.5]{vdK15} uses advanced tools (including results of \cite{Sus97}), but in fact part (2) of the theorem can be deduced from the case $n = 0$ using a relatively simple dimension-shifting argument. Since this seems worth recording, we sketch the details here. 

    By descent, we can assume that $G$ is split. Choose Borel subgroups $B, B^- \leq G$ which are opposite with respect to a maximal torus $T = B \cap B^- \leq G$. Let $U \leq B, U^- \leq B^-$ denote the respective unipotent radicals. We take the set $\Phi(G, T)^+$ of positive roots to be those belonging to $B$. Recall \cite[Lemma 6]{grosshans} that there exists a homomorphism $h \colon X^\bullet(T) \to \ints$ such that 
\begin{enumerate}
    \item for all $\lambda \in X^\bullet(T)^+$, $h(\lambda) \geq 0$,
    \item if $\lambda > \mu$, then $h(\lambda) > h(\mu)$,
    \item for all $\lambda \in X^\bullet(G)$, we have $h(\lambda) = 0$.
\end{enumerate}
    For example, if $G = \GL_2$ and $T \leq B \leq G$ are the standard diagonal torus and upper-triangular Borel subgroup, then we can take $h(\lambda_1, \lambda_2) = \lambda_1 - \lambda_2$.
\begin{definition}
    For any $S$-module $M$ with $G$-action, we define the Grosshans filtration
    \[
        M_0 \leq M_1 \leq \dots \leq M,
    \]
    where $M_n \subset M$ is the largest $G$-stable $S$-submodule such that every weight $\lambda \in X^\bullet(T)$ appearing in $M_n$ satisfies $h(\lambda) \leq n$. The associated graded module is denoted by $\gr M$. It is again an $S$-module with $G$-action.
\end{definition}
    We record the following properties of this filtration (where $M, N$ are $S$-modules with $G$-action, and $A$ is an $S$-algebra with $G$-action):
    \begin{enumerate}
        \item The weights appearing in $M \otimes_S N$ are sums of weights appearing in $M$ and $N$. Therefore, the image of $M_i \otimes_S N_j$ in $M \otimes_S N$ is contained in $(M \otimes_S N)_{i + j}$. In particular, if $A$ is an $S$-algebra with $G$-action, then $A_i \cdot A_j \leq A_{i + j}$. 
        \item For a $T$-module $P$, let $P_{(i)} = \oplus_{h(\lambda) = i} P_\lambda$. Then $M^{U}$ only has dominant weights, and for any $i \geq 0$, the map $(M^{U})_{(i)} \to (\gr_i M)^{U}$ is an isomorphism. Consequently, there is an isomorphism $M^U \cong (\gr M)^U$ of $S$-modules with $T$-action. Moreover, the induced map $\gr M \to \Ind_{B^-}^G M^U$ is injective \cite[Lemma 26]{vanderkallen_franjou}.
        \item The ring homomorphism $\gr A \to \Ind_{B^-}^G A^U$ is an integral ring extension, by \cite[Theorem 29]{vanderkallen_franjou}. 
        \item If $A$ is finitely generated over $S$, then both $\gr A$ and $\Ind_{B^{-}}^G A^U$ are finitely generated $S$-algebras, by \cite[Theorem 30]{vanderkallen_franjou}, and consequently the map $\gr A \to \Ind_{B^-}^G A^U$ is a finite algebra homomorphism. 
    \end{enumerate}
\begin{lemma} \label{lem:ind_finite_gr_mod}
    Let $A$ be a finitely generated $S$-algebra with $G$-action and $M$ a finite $A$-module with semilinear $G$-action. Then $\Ind_{B^-}^G M^U$ is a finite $\gr A$-module. 
\end{lemma}

\begin{proof}
    Consider the finitely generated $S$-algebra $A' = A \oplus M$ with multiplication $(a,m) \cdot (a',m') = (a a', am' + a'm)$. By (4) above, $\gr A' = \gr A \oplus \gr M$ is a finitely generated $S$-algebra, hence a finite $\gr A$-algebra. Moreover, $\Ind_{B^{-}}^G (A')^U$ is a finite $\gr A'$-module, hence also a finite $\gr A$-module. The result follows. 
\end{proof}

\begin{proof}[Proof of Theorem \ref{thm:group_coh_fg}]
We can assume that $G$ is a split reductive group over $S$; that $A$ is a finitely generated $S$-algebra with $G$-action; and that $M$ is a finite $A$-module. We must show that for each $n \geq 0$, $H^n(G, M)$ is a finite $A^G$-module. We prove this by induction on $n$. 
When $n = 0$, this follows from \cite[Theorem 6.3.3]{alper_adequate}.
Let $n > 0$ and assume the property holds for $n - 1$.
Consider the short exact sequence
\[
    0 \to \gr M \to \Ind_{B^{-}}^G M^U \to M' \to 0,
\]
where $\gr M \to \Ind_{B^{-}}^G M^U$ is the injective map from \cite[Lemma 26]{vanderkallen_franjou}. It follows from Lemma \ref{lem:ind_finite_gr_mod} that $M'$ is a finite $\gr A$-module. By \cite[Prop. 19]{vanderkallen_franjou}, all weights appearing in $M^U$ are dominant, hence the module $\Ind_{B^{-}}^G M^U$ is $G$-acyclic by \cite[B.4]{jantzen2003} and the universal coefficient theorem.
Hence, there is a surjection $H^{n-1}(G, M') \to H^n(G, \gr M)$. By the inductive hypothesis, $H^n(G, \gr M)$ is a finite module over $(\gr A)^G = A^G$.
Since,
\[
    H^{n}(G, \gr M) = \bigoplus_{p \leq 0} H^{n}(G, \gr_{-p} M)
\]
is finite over $A^G$, there are only finitely many $p \leq 0$ such that 
$
H^n(G, \gr_{-p} M) \neq 0.
$
The Grosshans filtration induces a spectral sequence
\[
    E_1^{p,q} = H^{p + q}(G, \gr_{-p}(M)) \implies H^{p + q}(G, M).
\]
There are finitely many pairs $(p,q)$ such that $p + q = n$ and $E_1^{p,q} \neq 0$. Since each $E_1^{p,q}$ is finite over $A^G$, we deduce that $H^{n}(G, M)$ is finite, as desired.
\end{proof}

\subsection{Formal quotient stacks}

As in the previous section let $S$ be a Noetherian ring, $G/S$ a reductive group and $A$ a finitely generated $S$-algebra with $G$-action. Let $I \leq A^G$ be an ideal such that  $A^G$ is $I$-adically complete. 

We define $\Mod^{\text{f.g.}}_G(A)$ to be the category of finitely generated $A$-modules with semilinear $G$-action. Moreover, we let $\widehat{\Mod}^{\text{f.g.}}_G(A)$ be the category of inverse systems 
\[
    M_1 \leftarrow M_2 \leftarrow M_3 \leftarrow \dots
\]
where $M_n$ is a finitely generated $A/I^n A$-module with semilinear $G$-action and the transition maps induce isomorphisms $M_{n}/I^{n-1}M_n \cong M_{n-1}$. This is an abelian category. The cokernel of a morphism $f_n \colon M_n \to N_n$ in $\widehat{\Mod}^{\text{f.g.}}_G(A)$ is the inverse system $(\coker f_n)_{n \geq 1}$. To describe the kernel of $(f_n)$, let 
\[
    K_n = \bigcap_{m \geq n} \im ( \ker f_m \to \ker f_n ).
\]
By the Artin--Rees lemma, the maps $K_{m+1} / I^n \to K_m / I^n$ ($n$ fixed, $m \geq n$) are eventually constant. We define
\[
    K_n' := \varprojlim_m K_m/I^n,
\]
which satisfies $K_n'/I^{n-1} K_n' \cong K_{n-1}'$, where the $G$-action is well-defined because the maps $K_{m+1}/I^n \to K_m/I^n$ are isomorphisms for large enough $m$. The kernel of $(f_n)$ is $(K_n')$.

\para
The categories $\Mod^{\text{f.g.}}_G(A)$ and $\widehat{\Mod}^{\text{f.g.}}_G(A)$ both have natural tensor products and internal homs. The only non-trivial one to define is the internal hom in $\widehat{\Mod}^{\text{f.g.}}_G(A)$. Let $(M_n)$ and $(N_n)$ be objects of $\widehat{\Mod}^{\text{f.g.}}_G(A)$. 
The internal hom $\underline{\Hom}((M_n)_{n \geq 1}, (N_n)_{n \geq 1})$ is defined as the inverse system $(H_n')_{n \geq 1}$, where 
\[
    H_n = \bigcap_{m \geq n} \im( \underline{\Hom}(M_m, N_m) \to \underline{\Hom}(M_n, N_n) )
\]
and
\[
    H_n' := \lim_{\substack{\leftarrow \\ m}} H_m/I^n H_m.
\]
The natural completion functor 
\begin{align*}
    \Mod^{\text{f.g.}}_G(A) &\to \widehat{\Mod}^{\text{f.g.}}_G(A) \\
    M &\mapsto \widehat{M} := (M/IM \leftarrow M/I^2M \leftarrow M/I^3 M \leftarrow \dots )
\end{align*}
is exact and preserves tensor products and internal homs.

\begin{definition}
    We say that the stack $[\Spec A / G]$ satisfies formal GAGA if 
    the completion functor $\Mod^{\text{f.g.}}_G(A) \to \widehat{\Mod}^{\text{f.g.}}_G(A)$
    is an equivalence of categories. 
\end{definition}

\begin{remark}
    When $G/S$ is linearly reductive, it has been shown in \cite{geraschenko2015} that the stack $[\Spec A/G]$ satisfies formal GAGA. 
\end{remark}

\begin{theorem}  \label{thm:general_gaga_thm}
    The completion functor $\Mod^{\text{f.g.}}_G(A) \to \widehat{\Mod}^{\text{f.g.}}_G(A)$ is fully faithful.
\end{theorem}
\begin{remark}
    After completing the present article, we learned that Theorem \ref{thm:general_gaga_thm} also follows from the work of Alper--Hall--Lim \cite{AHL23}. Indeed, in \cite[Example 4.5]{AHL23} they explain that the triple
    \begin{equation*}
        \big([\textnormal{Spec}A/G]\to \textnormal{Spec}A^G, [\textnormal{Spec}(A/I)/G], \textnormal{Spec}(A^G/I)\big)
    \end{equation*} is \textit{strongly cohomologically proper} in their terminology. Therefore, by \cite[Theorem 4.6 (2)]{AHL23}, $([\textnormal{Spec}A/G],[\textnormal{Spec}(A/I)/G])$ satisfies \textit{formal functions} in the sense of \cite[Definition 3.2]{AHL23}; that is, the natural map
    \begin{equation}\label{eqn:formalfunctions}
        H^i([\textnormal{Spec}A/G],\mathcal{F})\to H^i(\widehat{[\textnormal{Spec}A/G]},\widehat{\mathcal{F}})
    \end{equation}
    is an isomorphism for every $i\geq 0$ and $\mathcal{F}\in \textnormal{Coh}([\textnormal{Spec}A/G])$. In particular, as they observe in \cite[Corollary 3.8]{AHL23}, the completion functor $\textnormal{Coh}([\textnormal{Spec}A/G])\to \textnormal{Coh}(\widehat{[\textnormal{Spec}A/G]})$ is fully faithful, or equivalently, by \cite[Theorem 2.3]{ConradFormalGAGA},  $\Mod^{\text{f.g.}}_G(A) \to \widehat{\Mod}^{\text{f.g.}}_G(A)$ is fully faithful. 

    The main point of the argument is to establish the isomorphism \eqref{eqn:formalfunctions}. We give a proof of this in the case of $i=0$, which is sufficient to deduce full faithfulness. Although their proof, like ours, ultimately relies on \cite[Theorem 10.5]{vdK15}, the two arguments are different, and we therefore include ours.
\end{remark}

\begin{proof}
  We claim that it suffices to prove the following property for all objects $M$ of $\Mod^{\text{f.g.}}_G(A)$: 
    \begin{enumerate}
        \item[$(F_M)$] The natural map $H^0(G, M) \to \lim_n H^0(G, M/I^n M)$ is an isomorphism.
    \end{enumerate}
    Indeed, taking $M = \underline{\Hom}(M', M'')$, we see that the natural functor $\Mod^{\text{f.g.}}_G(A) \to \widehat{\Mod}^{\text{f.g.}}_G(A)$ is fully faithful.
    
    We now show $(F_M)$ for a general $M \in \Mod^{\text{f.g.}}_G(A)$. By Theorem \ref{thm:group_coh_fg}, each cohomology group $H^n(G, M)$ is a finite module over the Noetherian ring $A^G$. In particular, $H^n(G, M)$ is $I$-adically complete. 

    We will use the notion of derived completeness for complexes of $A^G$-modules defined in \cite[\href{https://stacks.math.columbia.edu/tag/091S}{Definition 091S}]{stacks-project}. By \cite[\href{https://stacks.math.columbia.edu/tag/091R}{Lemma 091R}]{stacks-project} and \cite[\href{https://stacks.math.columbia.edu/tag/091P}{Lemma 091P}]{stacks-project}, the complex $R \Gamma(G, M)$ is a derived complete complex of $A^G$-modules. 

    Let $x_1, \dots, x_r \in A^G$ be generators of $I$, and let $K_n^\bullet$ be the Koszul complex of the ring $A^G$ with respect to $x_1^n, \dots, x_r^n$. (Thus $K_n^\bullet$ is a complex of finite free $A^G$-modules, concentrated in degrees $[-r, 0]$.) By \cite[\href{https://stacks.math.columbia.edu/tag/091Z}{Lemma 091Z}]{stacks-project}, there is an isomorphism
    \[ R \Gamma(G, M) \cong R \lim (R \Gamma(G, M) \otimes^{\mathbf{L}}_{A^G} K_n^\bullet) \]
    in $D(A^G)$. Thus there is a short exact sequence
    \[
    0 \to  R^1 \lim H^{-1}(R \Gamma(G, M) \otimes^{\mathbf{L}}_{A^G} K_n^\bullet) \to H^0(G, M) \to \varprojlim_n H^0(R \Gamma(G, M) \otimes^{\mathbf{L}}_{A^G} K_n^\bullet) \to 0. 
    \]
    To prove the theorem, we need to show that the first entry in this sequence is zero, and that the third may be identified with $\varprojlim_n H^0(G, M / I^n M)$. We take these in turn. 

    Let $\underline{x}^n = (x_1^n, \dots, x_r^n)$, and let $K_q(\underline{x}^n, N) = H^{-q}(N \otimes_{A^G} K_n^\bullet)$ denote the degree $q$ Koszul homology of an $A^G$-module $N$. For any $n \geq 1$, there exists $m(n) \geq n$ such that the map $K_q(\underline{x}^m, M) \to K_q(\underline{x}^n, M)$ is zero for all $q > 0$. (The proof is the same as the proof of \cite[\href{https://stacks.math.columbia.edu/tag/0921}{Lemma 0921}]{stacks-project}, using that $M$ is a finitely generated module over the Noetherian ring  $A$.) There is a convergent spectral sequence 
    \[ E_2^{p, q} = H^p(G, K_{-q}(\underline{x}^n, M)) \Rightarrow H^{p+q}(G, M \otimes^{\mathbf{L}}_{A^G} K_n^\bullet), \]
    which comes equipped with an edge map
    \[ e_n \colon H^0(G, M \otimes^{\mathbf{L}}_{A^G} K_n^\bullet) \to H^0(G, M / (x_1^n, \dots, x_r^n) M) \]
    (as $K_0(\underline{x}^n, M) = M / (x_1^n, \dots, x_r^n) M$), therefore determining a homomorphism
    \[ \varprojlim_n e_n \colon \varprojlim_n H^0(G, M \otimes^{\mathbf{L}}_{A^G} K_n^\bullet) \to \varprojlim_n H^0(G, M / (x_1^n, \dots, x_r^n) M). \]
    We claim that this map is an isomorphism. If $(y_n)_{n \geq 1}$ lies in the kernel, choose $n = n_0 < n_1 < n_2 < \dots < n_r$ such that the map $K_q(\underline{x}^{n_i}, M) \to K_q(\underline{x}^{n_{i-1}}, M)$ is zero for $i > 0$, $q > 0$. Then $y_{n_r} \in \ker e_{n_r}$, so $y_{n_r}$ lies in $F^1 H^0(G, M \otimes^{\mathbf{L}}_{A^G} K_{n_r}^\bullet)$. The first graded piece of the abutment filtration is a subquotient of $H^1(G, K_1(\underline{x}^{n_r}, M))$, so we see that $y_{n_{r-1}}$ lies in $F^2 H^0(G, M \otimes^{\mathbf{L}}_{A^G} K_{n_{r-1}}^\bullet)$. Continuing in this way, we eventually find that $y_n = 0$. Since $n$ was arbitrary, $\ker \varprojlim_n e_n = 0$. To see that $\varprojlim_n e_n$ is surjective, take an element 
    \[ (z_n)_{n \geq 1} \in \varprojlim_n H^0(G, K_0(\underline{x}^n, M)). \]
    A similar argument shows that $z_n \in \operatorname{im}(e_n)$ for all $n \geq 1$ (because the differentials starting at $H^0(G,  K_0(\underline{x}^n, M))$ go to groups valued in $K_q(\underline{x}^n, M)$ for $q > 0$). So write $z_n = e_n(y_n)$. Another argument of the same type 
    shows that the image of $y_{m}$ in $H^0(G, M \otimes^{\mathbf{L}}_{A^G} K_n^\bullet)$ stabilises for $m \gg n$, showing that $(z_n)_{n \geq 1}$ is in the image of $\varprojlim_n e_n$. Since the sequences of ideals $I^n$ and $(x_1^n, \dots, x_r^n)$ of $A^G$ define isomorphic pro-systems, it follows that the third term of our short exact sequence is as it should be. 

    It remains to show that $R^1 \lim H^{-1}(R \Gamma(G, M) \otimes^{\mathbf{L}}_{A^G} K_n^\bullet) = 0$. We recall that this $R^1 \lim$ may be computed as the cokernel of the map
    \[ \prod_n H^{-1}(R \Gamma(G, M\otimes^{\mathbf{L}}_{A^G} K_n^\bullet)) \to \prod_n H^{-1}(R \Gamma(G, M\otimes^{\mathbf{L}}_{A^G} K_n^\bullet)) \]
    given by $(y_n) \mapsto (y_n - f_{n+1}(y_{n+1}))$ (where we write $f_n$ for the transition maps). The graded pieces of the abutment filtration of $H^{-1}(R \Gamma(G, M \otimes^{\mathbf{L}}_{A^G} K_n^\bullet))$ are subquotients of groups $H^\ast(G, K_q(\underline{x}^n, M))$ with $q > 0$. In particular, for any $n$, there exists $m > n$ such that $f_{n+1} \circ f_{n+2} \circ \dots \circ f_{n+m}$ is the zero map. Refining a prosystem does not change $R^1 \lim$, by \cite[\href{https://stacks.math.columbia.edu/tag/0H9K}{Lemma 0H9K}]{stacks-project}. In our case, this shows that $R^1 \lim H^{-1}(R \Gamma(G, M) \otimes^{\mathbf{L}}_{A^G} K_n^\bullet)$ may be identified with the cokernel of the identity map: in other words, it is zero, as required. This completes the proof. 
\end{proof}

\subsection{Application to Galois deformations}\label{subsection_ExistenceofPSSdefnrings}
In this section we apply theorem \ref{thm:general_gaga_thm} to the deformation theory of pseudorepresentations of absolute Galois groups. In particular, we prove the results of \cite{WE18} without the multiplicity-free hypothesis. As a consequence we are able to define potentially semistable Galois pseudodeformations with fixed $p$-adic Hodge type and Galois type.

\begin{para}
    Let $E$ be a coefficient field, and let $G$ be a profinite group. Fix a continuous $n$-dimensional determinant $\overline{D} \colon G \to k$ such that $H^1(G, \ad \rho^{\textnormal{ss}}_{\overline{D} \otimes \overline{k}})$ is finite dimensional over $\overline{k}$. By \cite[Prop 3.3]{Che14} there exists a profinite local $\mathcal{O}$-algebra $R_{\overline{D}}$ with residue field $k$ and a universal determinant $D \colon G \to R_{\overline{D}}$ such that for every Artinian local $\mathcal{O}$-algebra with residue field $k$, the map 
    \begin{align*}
        \Hom(R_{\overline{D}}, A) & \to \{ D_A \colon G \to A \mid D_A \equiv \overline{D} \pmod{\mathfrak{m}_A} \} \\
        f & \mapsto f \circ D
    \end{align*}
    is a bijection. It follows from \cite[Prop 3.7]{WE18} that $R_{\overline{D}}$ is Noetherian.
\end{para}

\begin{para}
    Let $D^u \colon G \to R_{\overline{D}}$ denote the universal pseudocharacter, and let $\mathcal{E} = R_{\overline{D}} \llbracket G \rrbracket / (\mathrm{CH}(D^u))$ denote the universal Cayley--Hamilton algebra described in \cite[\S 3.2]{WE18}; it is a finite $R_{\overline{D}}$-algebra. If $A$ is an Artinian local $\cO$-algebra with residue field $k$, then there is a natural bijection between the set of homomorphisms $\rho \colon G \to \GL_n(A)$ of pseudocharacter deforming $\overline{D}$, and the set of pairs of  compatible homomorphisms $R_{\overline{D}} \to A$, $\mathcal{E} \to M_n(A)$ (cf. \cite[Theorem 3.7]{WE18}).

    We write $\operatorname{Rep}_{\overline{D}}^\square = \operatorname{Rep}_{\mathcal{E}, D^u}^\square$ for the affine scheme over $\Spec R_{\overline{D}}$ representing the functor of compatible pairs as above, and $\mathcal{R}ep^\square_{\overline{D}}$ for its $\mathfrak{m}_{R_{\overline{D}}}$-adic completion. Similarly, we write $\operatorname{Rep}_{\overline{D}} = [ \operatorname{Rep}_{\overline{D}}^\square / \GL_n ]$ for the quotient stack, where $\GL_n$ acts by conjugation, and  $\mathcal{R}ep_{\overline{D}}$ for its $\mathfrak{m}_{R_{\overline{D}}}$-adic completion. The following result follows from \cite[Theorem 3.8(3)]{WE18}.
\end{para}
    
    \begin{theorem}\label{thm_global_sections_finite}
        The induced morphism $R_{\overline{D}} \to \cO( \operatorname{Rep}_{\overline{D}}) = \cO(\operatorname{Rep}^\square_{\overline{D}})^{\GL_n}$ is a finite algebra homomorphism.
    \end{theorem}
    \begin{corollary}\label{cor_global_sections_of_formal_stack}
        The completion functor from coherent sheaves on $\operatorname{Rep}_{\overline{D}}$ to coherent sheaves on $\mathcal{R}ep_{\overline{D}}$ is fully faithful. In particular, the induced map
        \[ \mathcal{O}(\operatorname{Rep}_{\overline{D}}) \to \mathcal{O}(\mathcal{R}ep_{\overline{D}}) \]
        is an isomorphism.
    \end{corollary}
    \begin{proof}
        By Theorem \ref{thm_global_sections_finite}, $\cO( \operatorname{Rep}_{\mathcal{E}, D^u})$ is a finite $R_{\overline{D}}$-algebra; in particular, $\mathfrak{m}_{R_{\overline{D}}}$-adically complete. We can thus apply Theorem \ref{thm:general_gaga_thm}. 
    \end{proof}

\begin{para} \label{para:CayleyHamilton}
    Now we specialize to the situation where $G = G_K$ is the absolute Galois group of a finite field extension $K/\rationals_p$.
    Let $L/K$ be a finite Galois extension. Fix integers $a \leq b$. The following are stable conditions in the sense of \cite[\S 2.3]{WWE19}:
    \begin{itemize}
        \item $\mathcal{C}_{L/K,[a,b]}^{\textnormal{cr}}$: Subquotient of a $\mathbf{Z}_p$-lattice in a de Rham $\mathbf{Q}_p[G_K]$-module, crystalline after restriction to $L$, with Hodge--Tate weights in $[a, b]$.
        \item $\mathcal{C}_{L/K,[a,b]}^{\textnormal{st}}$: Subquotient of a $\mathbf{Z}_p$-lattice in a de Rham $\mathbf{Q}_p[G_K]$-module, semistable after restriction to $L$, with Hodge--Tate weights in $[a, b]$.
    \end{itemize}
    Following \cite[Definition 2.5.4]{WWE19}, we say that $n$-dimensional pseudorepresentation $D \colon G \to A$ (for $A$ an Artinian local $\cO$-algebra with residue field $k$) satisfies one of these conditions if there exists a Cayley--Hamilton representation $(\mathcal{E}, D^u) \to (\mathcal{E}', D)$ such that $\mathcal{E}'$, when considered as left $\mathbf{Z}_p[G]$-module via the composite $\mathbf{Z}_p[G] \to \mathcal{E} \to \mathcal{E}'$, satisfies the given condition. According to \cite[Theorem 2.5.5]{WWE19}, the subfunctor of pseudorepresentations satisfying a stable condition $\mathcal{C}$ is represented by a quotient of $R_{\overline{D}}$. Moreover, there is a Cayley--Hamilton quotient $(\mathcal{E}, D^u) \to (\mathcal{E}^\mathcal{C}, D^{u, \mathcal{C}})$ with the following property: for any Artinian local $\cO$-algebra $A$ with residue field $k$, and any Cayley--Hamilton representation $(\mathcal{E}, D^u) \to (E, D)$ over $A$, we have that $(E, D)$ satisfies $\mathcal{C}$ if and only if the representation factors through $(\mathcal{E}^{\mathcal{C}}, D^{u, \mathcal{C}})$.

    We can therefore define closed, $\GL_n$-stable subschemes
     \[
    \operatorname{Rep}_{\overline{D}}^{\Box, \textnormal{pcr}, L, [a,b]} \subset \operatorname{Rep}_{\overline{D}}^{\Box, \textnormal{pst}, L, [a,b]} \subset \operatorname{Rep}^\Box_{\overline{D}}
    \]
    by taking e.g.\ $\operatorname{Rep}_{\overline{D}}^{\Box, \textnormal{pcr}, L, [a,b]}$ to represent the subfunctor of representations of $\mathcal{E}$ that factor through $\mathcal{E}^{\textnormal{pcr}, L, [a, b]}$ (in the obvious notation). By construction, the structure morphism to $\Spec R_{\overline{D}}$ factors through $\Spec R_{\overline{D}}^{\textnormal{pcr}, L, [a, b]}$.
    Quotienting by $\GL_n$, we obtain closed algebraic substacks
    \[
    \operatorname{Rep}_{\overline{D}}^{\textnormal{pcr}, L, [a,b]} \subset \operatorname{Rep}_{\overline{D}}^{\textnormal{pst}, L, [a,b]} \subset \operatorname{Rep}_{\overline{D}},
    \]
    sitting above the closed subschemes
    \[ \Spec R_{\overline{D}}^{\textnormal{pcr}, L, [a, b]} \subset \Spec R_{\overline{D}}^{\textnormal{pst}, L, [a, b]} \subset \Spec R_{\overline{D}}.
    \]
    We can also define their $\mathfrak{m}_{R_{\overline{D}}}$-adic completions as before. Our goal now is to use these objects to prove the following theorem:

\end{para}
    
\begin{theorem}\label{thm_existence_of_pseudodeformation_ring_with_conditions}
        Let $\mathbf{v}$, $\tau \colon I_{L / K} \to \GL_n(E)$ be a pair consisting of a Hodge and inertial type, with $\mathbf{v}$ bounded by $[a, b]$. Let $? \in \{ \textnormal{pcr}, \textnormal{pst} \}$. Then there exists a quotient $R_{\overline{D}}^{?, \mathbf{v}, \tau}$ of $R_{\overline{D}}^{?, L, [ a, b]}$ with the following properties:
        \begin{enumerate}
            \item $R_{\overline{D}}^{?, \mathbf{v}, \tau}$ is reduced and $p$-torsion-free.
            \item $\Spec R_{\overline{D}}^{?, \mathbf{v}, \tau}[1/p]$ is a union of connected components of $\Spec R_{\overline{D}}^{?, L, [ a, b]}[1/p]$.
            \item Let $x \colon R_{\overline{D}} \to \overline{\mathbf{Q}}_p$ be an $\cO$-algebra homomorphism, and let $\rho_x \colon G_K \to \GL_n(\overline{\mathbf{Q}}_p)$ be the corresponding semi-simple representation (determined up to isomorphism by $x$). Then $x$ factors through $R_{\overline{D}}^{?, \mathbf{v}, \tau}[1/p]$ if and only if $\rho_x$ is potentially crystalline (resp. potentially semi-stable), of Hodge type $\mathbf{v}$, and of inertial type $\tau$. 
        \end{enumerate}
    \end{theorem}
    We remark that $R_{\overline{D}}^{?, \mathbf{v}, \tau}$, if it exists, is uniquely determined by properties (1) and (3). 

    \begin{prop}
        Let $? \in \{ \textnormal{pcr}, \textnormal{pst} \}$. Then:
        \begin{enumerate}
            \item The map
        \[ R_{\overline{D}}^{?, L, [a, b]} \to \cO( \operatorname{Rep}_{\overline{D}}^{?, L, [a, b]}) \]
        is a finite algebra homomorphism, which becomes an isomorphism after inverting $p$.
        \item The map
        \[ \cO( \operatorname{Rep}_{\overline{D}}^{?, L, [a, b]}) \to \cO( \mathcal{R}ep_{\overline{D}}^{?, L, [a, b]})\]
        is an isomorphism.
        \end{enumerate}
    \end{prop}
    \begin{proof}
        The first part is an analogue of Theorem \ref{thm_global_sections_finite}; it may be proved in the same way as \cite[Theorem 3.7]{WE18} (note that the statement of that result is for the universal pseudocharacter, but the proof reduces immediately to studying representations of the universal Cayley--Hamilton algebra; applying this argument to $\mathcal{E}^{?, L, [ a,b]}$ gives what we need). The second part follows from the first in the same way as in the proof of Corollary \ref{cor_global_sections_of_formal_stack}. 
    \end{proof}
    \begin{corollary}
        To prove Theorem \ref{thm_existence_of_pseudodeformation_ring_with_conditions}, it suffices to construct an idempotent $e \in \cO(\mathcal{R}ep_{\overline{D}}^{?,[ a, b]})[1/p]$ with the following property: for any finite extension $E' / E$ with ring of integers $\cO'$, and any morphism $x \colon \Spf \cO' \to \mathcal{R}ep_{\overline{D}}^{?,[ a, b]}$ corresponding to a representation $\rho_x \colon G_K \to \GL_n(\cO')$, the following are equivalent:
        \begin{enumerate}
            \item $x^\ast(e) = 1$ in $E'$.
            \item $\rho_x[1/p]$ is potentially crystalline (resp. potentially semi-stable), of Hodge type $\mathbf{v}$, and of inertial type $\tau$. 
        \end{enumerate}
    \end{corollary}
    \begin{proof}
        Given $e$, let $e' \in \cO(\operatorname{Rep}_{\overline{D}}^{?,[ a, b]})[1/p]$ be the algebraisation of $e$, let $Y^{\mathbf{v}, \tau}$ be the closure of the closed substack of $(\operatorname{Rep}_{\overline{D}}^{?,[ a, b]})[1/p]$ defined by the equation $e' = 1$, and let $R_{\overline{D}}^{?, \mathbf{v}, \tau}$ be the affine ring of the scheme-theoretic image of $Y^{\mathbf{v}, \tau}$  in $\Spec R_{\overline{D}}$. Then  $R_{\overline{D}}^{?, \mathbf{v}, \tau}$ satisfies properties (1) and (2) of Theorem \ref{thm_existence_of_pseudodeformation_ring_with_conditions} more or less by definition. By \cite[Theorem 3.8]{WE18}, the $\overline{\mathbf{Q}}_p$-points of $R_{\overline{D}}^{?, \mathbf{v}, \tau}$ correspond bijectively to the $\overline{\mathbf{Q}}_p$-points of $Y^{\mathbf{v}, \tau}$ supporting semisimple representations, so property (3) follows from the defining property of $e$, on noting that any $E'$-point of $\operatorname{Rep}_{\overline{D}}^{?,[ a, b]}$ comes from an $\cO'$-point of $\mathcal{R}ep_{\overline{D}}^{?,[ a, b]}$.
    \end{proof}
    \begin{proof}[Proof of Theorem \ref{thm_existence_of_pseudodeformation_ring_with_conditions}]
    Let 
    \[ A = \cO(\mathcal{R}ep_{\overline{D}}^{?,[ a, b]})[1/p], \,\, A^\square = \cO(\mathcal{R}ep_{\overline{D}}^{\square, ?,[ a, b]})[1/p], \]
    and 
    \[ A^{\square, 1} = \cO(\mathcal{R}ep_{\overline{D}}^{\square, ?,[ a, b]} \times_{\operatorname{Spf} \cO} \widehat{\GL}_n)[1/p], \]
    where $\widehat{\GL_n} / \Spf \cO$ denotes the $\varpi$-adic completion of $\GL_n$. These rings are reduced (see \cite[Theorem B]{WE18}). 

    Let $\rho^u \colon G_K \to \GL_n(A^\square)$ denote the universal representation. After passage to twist, we can assume that $a \geq 0$. We now want to apply the results of \cite[\S 6]{WE18}. We first claim that, with the notation of \emph{loc. cit.}, we have $A^\square = (A^{\square})^{L-\operatorname{st}, b}$. After applying the main theorem of \cite{Liu07}, we see that this follows from the defining property of $A^\square$, namely that if we set $A^{\square, \circ} = \cO(\mathcal{R}ep_{\overline{D}}^{\square, ?,[ a, b]})$, and write $\rho^{u, \circ}$ for the corresponding universal representation, then $\rho^{u, \circ} \otimes_{R_{\overline{D}}} R_{\overline{D}} / \mathfrak{m}_{R_{\overline{D}}}^i|_{G_L}$ is, for every $i \geq 0$, a subquotient of a lattice in a $\mathbf{Q}_p[G_L]$-module, semi-stable of Hodge--Tate weights in $[a, b]$. By \cite[Theorem 6.6]{WE18}, there exists a unique idempotent $e \in A^{\square}$ such that for any finite extension $E' / E$ and $\cO$-algebra homomorphism $x \colon A^\square \to E'$, with specialisation $\rho_x$ of the universal representation, we have that $\rho_x$ is potentially crystalline (resp. potentially semi-stable) of Hodge type $\mathbf{v}$ and inertial type $\tau$ if and only if $x(e) = 1$. 

    To complete the proof, we need to show that $e \in A$. By definition, we have $A = \ker (p_1 - a)$, where $p_1 \colon A^\square \to A^{\square, 1}$ induced by projection on the first factor, and $a \colon A^\square \to A^{\square, 1}$ is the coaction map. If $g \in \GL_n(\cO)$, then there is a retraction $r_g \colon A^{\square, 1} \to A^\square$. Let $f = p_1(e) - a(e)$. The uniqueness property of $e$ (which follows from \cite[Lemma 4.17]{WE18}) shows that for any $g \in \GL_n(\cO)$, we have $r_g(f) = 0$. To conclude the proof, it suffices to note that the only element $z \in A^{\square, 1}$ such that $r_g(z) = 0$ for all $g \in \GL_n(\cO)$ is $z = 0$.  
    \end{proof}

\para\label{para_deformation_ring_notations}

In \S\ref{sec_LGC_for_Betti}, it will be convenient to label our deformation rings by invariants appearing on the automorphic side. Let $\mathbf{Q}_p\leq E\leq \overline{\mathbf{Q}}_p$ be a finite field extension with $\# \Hom(K,E)=[K:\mathbf{Q}_p]$ and residue field $k$. Let $\overline{D}:G_K\to k$ be a continuous $n$-dimensional determinant as before.

Given a highest weight vector $\lambda\in (\mathbf{Z}_+^n)^{\Hom(K,E)}$ for $\Res_{K/\mathbf{Q}_p}\textnormal{GL}_n$, we can define an associated $p$-adic Hodge type in the sense of \cite{Kis08} as follows. Let $D$ be an $n$-dimensional $E$-vector space and write
\begin{equation*}
    D_K:=D\otimes_{\mathbf{Q}_p}K=\oplus_{\iota:K\hookrightarrow E}D_{\iota}.
\end{equation*}
For an embedding $\iota:K\hookrightarrow E $, choose a decreasing filtration $\textnormal{Fil}^{\bullet}D_{\iota}$ on $D_{\iota}$ by $E$-subvector spaces such that
\begin{equation*}
    \dim_E\textnormal{gr}^jD_{\iota}=\begin{cases}
        1 & \text{if }j\in\{\lambda_{\iota,i}+n-i\}_{1\leq i\leq n} \\
    0 & \text{otherwise.}
    \end{cases}
\end{equation*}
This yields a filtration $\textnormal{Fil}^{\bullet}D_K$ on $D_K$ by sub-$K\otimes_{\mathbf{Q}_p}E$-modules. We set $\mathbf{v}_{\lambda}=(D,\textnormal{Fil}^{\bullet}D_K)$.

Given a Bernstein block $\Omega$ of $\textnormal{Rep}_{\textnormal{sm}}(\textnormal{GL}_n(K),\overline{\mathbf{Q}}_p)$, there is an associated inertial type $\tau_{\Omega}:I_K\to \textnormal{GL}_n(\overline{\mathbf{Q}}_p)$ that, up to isomorphism, can be defined as $\textnormal{rec}_K(\pi)|_{I_K}$ for any irreducible smooth representation $\pi$ lying in $\Omega$. At the cost of possibly enlarging $E$ and $L$, we can assume that it is of the form $\tau_{\Omega}:I_{L/K}\to\textnormal{GL}_n(E)$.
\begin{definition}\label{def_defnring_notation}
    Given a highest weight vector $\lambda$ of $\textnormal{Res}_{K/\mathbf{Q}_p}\textnormal{GL}_n$ and a Bernstein block $\Omega$ of $\textnormal{Rep}_{\textnormal{sm}}(\textnormal{GL}_n(K),\overline{\mathbf{Q}}_p)$ with associated $p$-adic Hodge type $\mathbf{v}_{\lambda}$ and inertial type $\tau_{\Omega}:I_{L/K}\to \textnormal{GL}_n(E)$, we set
    \begin{equation*}
        R^{\lambda,\Omega}_{\overline{D}}:=R^{\textnormal{pst},\mathbf{v}_{\lambda},\tau_{\Omega}}_{\overline{D}}.
    \end{equation*}
\end{definition}

\para Given $\lambda$ and $\Omega$ as above, we will also need that the semisimple local Langlands correspondence interpolates over $R^{\lambda,\Omega}_{\overline{D}}$.

 To discuss this, let $Z^1(W_K, L, \GL_n)_E$ denote the scheme, of finite type over $E$, of homomorphisms $W_K \to \GL_n$, trivial on $I_L$. Applying Fontaine's construction to the $(\varphi,N,L/K)$-module over $\cO(\mathcal{R}ep_{\overline{D}}^{\square, \textnormal{pst},[ a, b]})[1/p]$ constructed in \cite[\S 6]{WE18} and descending sections to $\cO(\mathcal{R}ep_{\overline{D}}^{ ?,[ a, b]})[1/p]$, as in the proof of Theorem \ref{thm_existence_of_pseudodeformation_ring_with_conditions},  shows that there is a homomorphism $\zeta \colon \cO(Z^1(W_K, L, \GL_n)_E)^{\GL_n} \to R_{\overline{D}}^{\lambda,\Omega}[1/p]$ such that for any $x \colon R_{\overline{D}}^{\lambda,\Omega} \to \overline{\mathbf{Q}}_p$ corresponding to a semisimple representation $\rho_x \colon G_K \to \GL_n(\overline{\mathbf{Q}}_p)$, the composite $x\circ\zeta$ classifies the representation $\mathrm{WD}(\rho_x)^{\textnormal{ss}}$.
 
  Let $\mathfrak{Z}_\Omega$ be the centre of $\Omega$. It is proven in \cite[Prop 3.11]{chenevier20} that there exists a pseudorepresentation $T \colon W_K \to \mathfrak{Z}_\Omega$ such that for every smooth irreducible representation $\pi$ in $\Omega$, therefore determining a specialisation $T(\pi) \colon W_K \to \overline{\mathbf{Q}}_p$, we have $T(\pi)$ is the pseudocharacter associated to $\rec_K^T(\pi)$. 
\begin{lemma}\label{Corollary_InperpolationOfSSLL}
    There is a ring homomorphism
    \[
        \eta \colon \mathfrak{Z}_\Omega \to R_{\overline{D}}^{\lambda,\Omega} \otimes_{\mathcal{O}} \overline{\rationals}_p
    \]
    such that for every $z \in \mathfrak{Z}_\Omega$ and every homomorphism $R_{\overline{D}}^{\lambda,\Omega} \to \overline{\mathbf{Q}}_p$, corresponding to a semisimple representation $\rho_x \colon G_K \to \GL_n(\overline{\mathbf{Q}}_p)$, we have 
    \[
        x(\eta(z)) = z( (\operatorname{rec}^T_K)^{-1}(\textnormal{WD}(\rho_x))^{\textnormal{F-ss}}).
    \]
\end{lemma}

\begin{proof}
    We argue in a similar way as \cite[Theorem 3.3]{Pyv21}.
    Let \[
    U \colon W_K \to \mathcal{O}(Z^1(W_K, \GL_n)_E)^{\GL_n}
    \]
    be the tautological pseudorepresentation. Let $h \colon R_{\overline{D}}^{\lambda,\Omega} \otimes_{\mathcal{O}} \overline{\rationals}_p \to \prod_{x} \overline{\rationals}_p$ be the natural map. Since $R_{\overline{D}}^{\lambda,\Omega}[1/p]$ is reduced and Jacobson by \cite[Lemma 4.17]{WE18}, it follows that $h$ is injective. Further, we define a map $h' \colon \mathfrak{Z}_\Omega \to \prod_{x} \overline{\rationals}_p$ by sending $z$ to its action on $(\operatorname{rec}^T_K)^{-1}(\mathrm{WD}(\rho_{x})^\textnormal{ss})$ for each $x$.

    It follows from the properties on $\overline{\rationals}_p$-points that $h \circ \zeta \circ U = h' \circ T$. It was proven in \cite[lemma 4.5]{6author} that the traces of $T$ generate the $\overline{\rationals}_p$-algebra $\mathfrak{Z}_\Omega$. Therefore, the image of $h'$ lies in the image of $h$ and we obtain the desired ring homomorphism. 
\end{proof}

\begin{para}\label{para_global_deformation_rings}
    We conclude \S\ref{sec_potentially_semistable_pseudodeformation_rings} by introducing the global deformation rings from \cite{WWE19} that we need in this article. Let $F$ be a number field and $\Sigma_p$ be a subset of the set $S_p(F)$ of  $p$-adic places of $F$. Fix a finite field extension $L_v/F_v$ for each $v\in \Sigma_p$ and a pair of integers $a\leq b$. Fix another finite set $S$ of finite places of $F$, containing $S_p(F)$, and a continuous $n$-dimensional determinant $\overline{D}$ of $G_{F, S}$ over $k$.
\end{para}

    \begin{remark} \label{rem:local_global_pseudoring}
    We warn the reader that, for our applications to proving de Rhamness of  automorphic Galois representations, working with, say, $R_{\overline{D}}\widehat{\otimes}_{R_{\overline{D}\mid_{G_{F_v}}}}R_{\overline{D}\mid_{G_{F_v}}}^{\textnormal{pst},\mathbf{v},\tau}$ does not suffice, even in the case when the global residual representation is irreducible. Namely, the $\overline{\mathbf{Q}}_p$-points
    \begin{equation*}
x \colon R_{\overline{D}}\widehat{\otimes}_{R_{\overline{D}\mid_{G_{F_v}}}}R_{\overline{D}\mid_{G_{F_v}}}^{\textnormal{pst},\mathbf{v},\tau}\to \overline{\mathbf{Q}}_p
    \end{equation*} 
    correspond to semisimple Galois representations $\rho_x:G_{F,S}\to \textnormal{GL}_n(\overline{\mathbf{Q}}_p)$ such that $(\rho_x|_{G_{F_v}})^{\textnormal{ss}}$ is de Rham with $p$-adic Hodge type $\mathbf{v}$ and inertial type $\tau$. As the category of de Rham representations is not closed under extensions, this does not show de Rhamness of $\rho_x|_{G_{F_v}}$.

    For example,\footnote{We are grateful to Frank Calegari for pointing out that our original example was not correct, and proposing a replacement.} consider the case $n = 2$, $F = \rationals$ and $p = 11$. Let $\rho_\Delta : G_\rationals \to \GL_2(\rationals_p)$ be the representation associated to Ramanujan's modular form $\Delta$. Then the residual representation $\overline{\rho}_\Delta$ has image  $\GL_2(\mathbf{F}_p)$ and is absolutely irreducible. Since $\Delta$ is ordinary at ${11}$, there is an isomorphism
    \[ \rho_\Delta|_{G_{\rationals_p}} \sim \left( \begin{array}{cc} \mathrm{ur}_\alpha & * \\ 0 & \epsilon^{-11} \mathrm{ur}_{\alpha^{-1}} \end{array} \right), \]
    where $\alpha \in \mathbf{Z}_p$ is the unit root of the Hecke polynomial 
    \[ X^2 - \tau(11) X + 11^{11} \equiv X(X-1) \text{ mod } 11.  \]
    There is a congruence $\Delta(z) \equiv f(z) := \eta(z)^2 \eta(11z)^2 \pmod{11}$ and the Galois representation $\rho_f$ is isomorphic to the dual of the Tate module of the elliptic curve $X_1(11)$. 
    Therefore, the restriction $\overline{\rho}_\Delta\rvert_{G_{\rationals_p}} \cong \overline{\rho}_f \rvert_{G_{\rationals_p}}$ is non-split because $X_1(11)$ has split multiplicative reduction at $11$ and its minimal discriminant is equal to $-11$. 

    Let $\rho : G_{\rationals} \to \GL_2(\mathbf{Z}_p)$ be the Galois representation attached to the weight $-8$ specialisation of the cuspidal $11$-adic Hida family through $\Delta$. Then there is an isomorphism
    \[ \rho|_{G_{\rationals_p}} \sim \left( \begin{array}{cc} \mathrm{ur}_\beta & * \\ 0 & \epsilon^{9} \mathrm{ur}_{\beta^{-1}} \end{array} \right) \]
    for some $\beta \in \mathbf{Z}_p$, and so $\rho$ determines a $\overline{\mathbf{Q}}_p$-point of the ring $R_{\overline{D}}\widehat{\otimes}_{R_{\overline{D}\mid_{G_{\rationals_p}}}}R_{\overline{D}\mid_{G_{\rationals_p}}}^{\textnormal{pst},\mathbf{v},\tau}$ (with $\overline{D} = D(\overline{\rho}_\Delta)$, $\tau = 1$, and $\mathbf{v} = \mathbf{v}_{(-1, -9)}$). It follows from \cite[Proposition 6]{ghate2005} that $\rho|_{G_{\rationals_p}}$ is non-split, and therefore not de Rham, since $H^1_g(\rationals_p, \mathrm{ur}_\beta^2 \epsilon^{-9}) = 0$. 
\end{remark}

   \begin{para}
As the Remark shows, it is not possible to prove our main local-global compatibility theorem by working only with local pseudodeformation rings with conditions. However, the structure of the degree-shifting argument means that they must play a role. (Roughly speaking, to study $n$-dimensional automorphic Galois representations we need to introduce $2n$-dimensional automorphic Galois representations arising from the unitary group; these do not have a canonical invariant $n$-dimensional invariant subspace, but they do after restriction to decomposition groups at the $p$-adic places of interest -- cf. the proof of Lemma \ref{lemma_interiorLGC}.) 

Ultimately, we are saved by the observation that if we have a $p$-adic, semisimple global representation $\rho : G_{F, S} \to \GL_n(\overline{\mathbf{Q}}_p)$ that is potentially semistable, and $(\rho|_{G_{F_v}})^{ss}$ has both the correct Hodge--Tate weights and the correct inertial type, then $\rho|_{G_{F_v}}$ also has the correct Hodge--Tate weights and correct inertial type. This observation is formalised in Lemma \ref{lem_inertial_type_of_semisimplified_representation} and informs the axiomatisation of the induction hypothesis underlying the degree-shifting argument that we give in Definition \ref{def_LGC}. 
\end{para}

   \begin{para}
     Given an Artinian local $\mathcal{O}$-algebra $(A,\mathfrak{m})$ with residue field $k$ and a (continuous) deformation $D:G_{F,S}\to A$ of $\overline{D}$, a Cayley--Hamilton representation $(\mathcal{E},\rho,D')$ lifting $D$ consists of a Cayley--Hamilton pseudorepresentation $D'\colon\mathcal{E}\to A$ in the sense of \cite[Definition 2.1.6]{WWE19} and a (continuous) group homomorphism $\rho:G_{F,S}\to \mathcal{E}^{\times}$ such that $D=D'\circ \rho$. In particular, $\mathcal{E}$ is a finitely generated $A$-module and so it is naturally an object of the category $\textnormal{Mod}_{\textnormal{fin}}(\mathbf{Z}_p[G_{F,S}])$ of continuous $G_{F,S}$-representations on $\mathbf{Z}_p$-modules of \textit{finite cardinality}.
\end{para}
     
    \begin{theorem}[Wake--Wang-Erickson]\label{Thm_WWE_global}
        For $? \in \{\textnormal{pcr},\textnormal{pst}\}$, there is a quotient $R_{\overline{D}}^{\Sigma_p,?,L_v,[a,b]}$ of $R_{\overline{D}}$ such that a point $x:R_{\overline{D}}\to A$ factors through $R_{\overline{D}}^{\Sigma_p,?,L_v,[a,b]}$ if and only if there is a Cayley--Hamilton representation $(\mathcal{E},\rho,D)$ lifting $D_x$ such that, for every $v\in \Sigma_p$, $\mathcal{E}$ satisfies $\mathcal{C}_{L_v/F_v,[a,b]}^{\textnormal{cr}}$ (resp. $\mathcal{C}_{L_v/F_v,[a,b]}^{\textnormal{st}}$) from \S\ref{para:CayleyHamilton}.

        Moreover, a point $x:R_{\overline{D}}\to \overline{\mathbf{Q}}_p$ factors through $R_{\overline{D}}^{\Sigma_p,?,L_v,[a,b]}$ if and only if, for the corresponding semisimple representation $\rho_x:G_{F,S}\to \textnormal{GL}_n(\overline{\mathbf{Q}}_p)$, $\rho_x|_{G_{L_v}}$ is crystalline (resp. semistable) with Hodge--Tate weights lying in $[a,b]$ for every $v\in \Sigma_p$.
    \end{theorem}
    \begin{proof}
        The first part follows from \cite[Theorem 2.5.5]{WWE19}, \S5.2 and \S5.3. The second part follows from \cite{WWE19}, Remark 5.2.2.
    \end{proof}

\begin{para}
        Let $R_{\overline{D}}^{\textnormal{loc}} = \widehat{\bigotimes}_{v \in \Sigma_p} R_{\overline{D}\rvert_{G_{F_v}}}$.
        For $? \in \{\textnormal{pcr}, \textnormal{pst}\}$ and a tuple of highest weight vectors and Bernstein blocks $(\lambda_v, \Omega_v)_{v \in \Sigma_p}$, let
        $R_{\overline{D}}^{\Sigma_p,?,(\lambda_v, \Omega_v)}$ be the maximal $p$-torsion free, reduced quotient of 
        $R_{\overline{D}}^{\Sigma_p, ?, L_v, [a,b]} \otimes_{R_{\overline{D}}^{\textnormal{loc}} } \widehat{\bigotimes}_{v \in \Sigma_p} R_{\overline{D} \rvert_{G_{F_v}}}^{?, \lambda_v, \Omega_v}$. 
    \end{para}
    \begin{lemma}\label{lemma_semisimplification_of_deRham_reps}
    Let $K/\mathbf{Q}_p$ be a finite field extension,
    and $\rho:G_K\to \textnormal{GL}_n(\overline{\mathbf{Q}}_p)$ be a continuous de Rham Galois representation. The following hold true.
    \begin{enumerate}
        \item For every $\tau:K\hookrightarrow \overline{\mathbf{Q}}_p$, $\textnormal{HT}_{\tau}(\rho)=\textnormal{HT}_{\tau}(\rho^{\textnormal{ss}})$.
        \item There is an isomorphism
        \begin{equation*}
            \textnormal{WD}(\rho)^{\textnormal{ss}}\cong \textnormal{WD}(\rho^{\textnormal{ss}})^{\textnormal{ss}}.
        \end{equation*}
    \end{enumerate}
\end{lemma}
\begin{proof}
    The first claim follows from $D_{\textnormal{HT}}$ being exact on the category of Hodge--Tate representations (see for instance \cite[Proposition 2.4.4]{conrad_brinon_notes}). The second claim follows from $D_{\textnormal{st}}$ being exact on the category of semistable (in other words $B_{\textnormal{st}}$-admissible) representations (see for instance \cite[Theorem 5.2.1]{conrad_brinon_notes} with $B=B_{\textnormal{st}}$). 
\end{proof}

    \begin{lemma}\label{lem_inertial_type_of_semisimplified_representation}
        An $\mathcal{O}$-algebra homomorphism 
        $x \colon R_{\overline{D}}^{\Sigma_p,?,L_v,[a,b]} \to \overline{\rationals}_p$ factors through $R_{\overline{D}}^{\Sigma_p,?,(\lambda_v, \Omega_v)}$ if and only if the corresponding semisimple representation $\rho_x \colon G_{F,S} \to \GL_n(\overline{\rationals}_p)$ (which is potentially semistable by Theorem \ref{Thm_WWE_global}), satisfies that for each $v \in \Sigma_p$, $\rho_x \rvert_{G_{F_v}}$ has $p$-adic Hodge type $\mathbf{v}_{\lambda_v}$ and inertial type $\tau_{\Omega_v}$.
    \end{lemma}
 
    \begin{proof}
        By definition $x$ factors through $R_{\overline{D}}^{\Sigma_p,?,(\lambda_v, \Omega_v)}$ if and only if for all $v \in \Sigma_p$, the \emph{semisimplification} of $\rho_x \rvert_{G_{F_v}}$ has $p$-adic Hodge type $\mathbf{v}_{\lambda_v}$ and inertial type $\tau_{\Omega_v}$.
        It follows from Lemma \ref{lemma_semisimplification_of_deRham_reps} that we can drop the semisimplification here and the claim follows.
    \end{proof}

    \begin{corollary}
        For each $v \in \Sigma_p$, there is a map $\eta : \mathfrak{Z}_{\Omega_v} \to R_{\overline{D}}^{\Sigma_p,?,(\lambda_v, \Omega_v)} \otimes_{\mathcal{O}} \overline{\rationals}_p$, such that for every $z \in \mathfrak{Z}_{\Omega_v}$ and every homomorphism $x \colon R_{\overline{D}}^{\Sigma_p,?,(\lambda_v, \Omega_v)} \to \overline{\mathbf{Q}}_p$, corresponding to a semisimple representation $\rho_x \colon G_{F,S} \to \GL_n(\overline{\mathbf{Q}}_p)$, we have 
    \[
        x(\eta(z)) = z( (\operatorname{rec}^T_{F_v})^{-1}(\textnormal{WD}(\rho_x \rvert_{G_{F_v}}))^{\textnormal{F-ss}}).
    \]
    \end{corollary}

    \begin{corollary} \label{cor:local_global_deformation_ring}
        For $v \in \Sigma_p$, let $\pi_v$ be smooth irreducible representations of $\GL_n(F_v)$ lying in $\Omega_v$. Suppose that $x \colon R_{\overline{D}}^{\Sigma_p,?,(\lambda_v, \Omega_v)} \to \overline{\rationals}_p$ is a point such that for each $v \in \Sigma_p$, we have a commutative diagram
        \[
            \begin{tikzcd}
            \mathfrak{Z}_{\Omega_v} \arrow{r} \arrow{dr}[swap]{[\pi_v]} & R_{\overline{D}}^{\Sigma_p,?,(\lambda_v, \Omega_v)} \otimes_{\mathcal{O}} \overline{\rationals}_p \arrow{d}{x} \\
            & \overline{\rationals}_p,
            \end{tikzcd}
        \]
        where $z \in \mathfrak{Z}_{\Omega_v}$ acts on $\pi_v$ as multiplication by $[\pi_v](z)$. Then the corresponding representation $\rho_x \colon G_{F,S} \to \GL_n(\overline{\rationals}_p)$ satisfies for each $v \in \Sigma_p$:
        \begin{enumerate}[(1)]
            \item The representation $\rho_x \rvert_{G_{L_v}}$ is potentially crystalline (resp. semistable) with $p$-adic Hodge type $\mathbf{v}_{\lambda_v}$ and inertial type $\tau_{\Omega_v}$.
            \item We have $\rec_{F_v}^T(\pi_v)^{\textnormal{ss}} \cong \textnormal{WD}(\rho_x\rvert_{G_{F_v}})^{\textnormal{ss}}$.
        \end{enumerate}
    \end{corollary}

\section{Some local representation theory}\label{sec_local_representation_theory}
This section prepares the proof of Proposition \ref{Proposition_EisFunc}, which describes the Hecke eigensystems appearing in the Betti cohomology of Borel–Serre strata of $\mathrm{Res}_{F/\mathbf{Q}}\mathrm{GL}_n$-locally symmetric spaces in terms of the cohomology of their Levi factors.

Understanding the local systems arising in this description reduces to the study of certain continuous group cohomology groups of $\mathbf{Z}_p$-lattices in locally algebraic types. After inverting $p$, the smooth part of such a type is given by a Bushnell–Kutzko type, while the algebraic part is given by a highest weight representation. The corresponding descriptions are provided by Proposition \ref{Proposition_BushnellKutzkoTypesIso} and Theorem \ref{Theorem_Kostant}. In Lemma \ref{Lemma_changeofweightI} and Lemma \ref{Lemma_changeofweightII}, we show that, up to ignoring $p^N$-torsion (with exponent depending only on the choices made in \S\ref{Subsection_integralstructures}), the integral local systems that arise agree with those predicted by the characteristic $0$ formulas of \S\ref{Subsubsection_smooth_char0} and \S\ref{Subsection_locallyalgreps}.

Describing the Hecke action at $p$ is more subtle. To do this, in Lemma \ref{Lemma_explicitdescriptionofgenerators} we single out a collection of Hecke operators arising from Bushnell--Kutzko’s generalisation of Bernstein’s presentation for which such a description is possible, and in Lemma \ref{Lemma_centercontainedinBernsteinsalgebra} we show that these operators suffice to control the action of the Bernstein centre. In \S\ref{Subsection_integralstructures}, we rescale these operators to make them compatible with the chosen lattices in our locally algebraic types. Finally, in \S\ref{subsection_smooth_modp}, we carry out the necessary computations in the derived category of smooth $p$-power torsion representations needed to compare the action of this collection of Hecke operators with the characteristic $0$ formula appearing in Remark \ref{Remark_MinusculeSatakeTransform}.

\subsection{Smooth representations with characteristic \texorpdfstring{$0$}{0}  coefficients}\label{Subsubsection_smooth_char0}
Let $L/\mathbf{Q}_p$ be a finite field extension, $\textnormal{G}:=\textnormal{GL}_n(L)$. For every integer $m\geq 1$ and supercuspidal Bernstein block $\Omega$ of $\textnormal{Rep}_{\textnormal{sm}}(\textnormal{GL}_m(L),\overline{\mathbf{Q}}_p)$, we fix a representative $(\textnormal{GL}_m(L),\pi_{\Omega})$ once and for all. Set $\textnormal{SC}_0$ to be the set of these supercuspidal representations for each supercuspidal block $\Omega$ of $\textnormal{Rep}_{\textnormal{sm}}(\textnormal{GL}_m(L),\overline{\mathbf{Q}}_p)$ for some $m\geq 1$. We also fix a strict ordering $<$ on $\textnormal{SC}_0$. 

Let $\Omega$ be a Bernstein block of $\textnormal{Rep}_{\textnormal{sm}}(\textnormal{G},\overline{\mathbf{Q}}_p)$. 
Recall that a pair $(J,\kappa)$ of a compact open subgroup $J\leq \textnormal{G}$ and an irreducible $\overline{\mathbf{Q}}_p$-representation $\kappa$ is called an $\Omega$-type if for $\pi$ a smooth $\overline{\mathbf{Q}}_p$-representation of $\textnormal{G}$, $\pi$ lies in $\Omega$ if and only if it is generated by its $\kappa$-typical vectors as a $\overline{\mathbf{Q}}_p[\textnormal{G}]$-module.

The work of Bushnell--Kutzko \cite{BK93}, \cite{BK98}, \cite{BK99} allows us to define an $\Omega$-type $(J,\kappa)$ with $J\leq K:=\textnormal{GL}_n(\mathcal{O}_L)$ as follows. By \cite[Theorem 8.4.1]{BK93}, for every $\pi_0\in \textnormal{SC}_0$, we can and do fix a maximal simple type $(J_0,\kappa_0)$ with $J_0\leq \textnormal{GL}_m(\mathcal{O}_L)$ appearing in $\pi_0$. In particular, by \cite[Corollary 6.2.3]{BK93}, if we write $(\textnormal{M},\pi)=(\textnormal{GL}_{n_1}(L)\times...\times \textnormal{GL}_{n_h}(L),\pi_1\otimes...\otimes \pi_h)$ for the unique representative of $\Omega$ with $\pi_i\in \textnormal{SC}_0$ and $\pi_1\leq...\leq\pi_h$, we obtain a well-defined $[\textnormal{M},\pi]_{\textnormal{M}}$-type $(J_{\textnormal{M}},\kappa_{\textnormal{M}})$.

Let $\textnormal{M}\leq \widetilde{\textnormal{M}}\leq \textnormal{G}$ be the unique smallest Levi subgroup containing the $N_\textnormal{G}(\textnormal{M})$-stabiliser of $[\textnormal{M},\pi]_{\textnormal{M}}$. By \cite[Theorem 7.2.17]{BK93} (see also \cite[Proposition 1.4]{BK99}), we can and do fix an $\widetilde{\textnormal{M}}$-cover $(J_{\widetilde{\textnormal{M}}},\kappa_{\widetilde{\textnormal{M}}})$ of $(J_{\textnormal{M}},\kappa_{\textnormal{M}})$ (in the sense of \cite{BK99}, p.55).

Finally, the main result of \cite{BK99} is the construction of a $\textnormal{G}$-cover $(J,\kappa)$ of $(J_{\widetilde{\textnormal{M}}},\kappa_{\widetilde{\textnormal{M}}})$ that we call the \textit{Bushnell--Kutzko type} for the block $\Omega$. 

Given the Bushnell--Kutzko type $(J,\kappa)$ for $\Omega$ we set 
\begin{equation*}
    \sigma_{\Omega}:=\textnormal{c-Ind}_J^K\kappa\in \textnormal{Rep}_{\textnormal{sm}}(K,\overline{\mathbf{Q}}_p).
\end{equation*} 
For a parabolic subgroup $\textnormal{P}=\textnormal{M}\textnormal{N}\subset \textnormal{G}$, we write $K_{\textnormal{P}}=K\cap \textnormal{P}$, $K_{\textnormal{M}}:=K\cap \textnormal{M}$ and $K_{\textnormal{N}}:=K\cap\textnormal{N}$. Moreover, we set $\sigma_{\textnormal{P},\Omega}:=\Gamma(K_{\textnormal{N}},\sigma_{\Omega})\in \textnormal{Rep}_{\textnormal{sm}}(K_{\textnormal{M}},\overline{\mathbf{Q}}_p)$.

\begin{lemma}\label{BKlemma}
    Given a smooth irreducible $\overline{\mathbf{Q}}_p$-representation $\pi$ of $\textnormal{M}$, the following are equivalent.
    \begin{enumerate}
            \item The set $\Hom_{K_{\textnormal{M}}}(\sigma_{\textnormal{P},\Omega},\pi)$ is non-zero.
        \item The representation $\textnormal{Ind}_{\textnormal{P}}^{\textnormal{G}}\pi$ lies in the block $\Omega$.
    \end{enumerate}
\end{lemma}
\begin{proof}
    We have a series of isomorphisms
    \begin{align*}
        \Hom_{K_{\textnormal{M}}}(\sigma_{\textnormal{P},\Omega},\pi) &\cong \Hom_{\textnormal{M}}(\textnormal{c-Ind}_{K_{\textnormal{M}}}^{\textnormal{M}}\sigma_{\textnormal{P},\Omega},\pi) \\
        & \cong
        \Hom_{\textnormal{M}}((\textnormal{c-Ind}_{K}^{\textnormal{G}}\sigma_{\Omega})_{\textnormal{N}},\pi) \\
        & \cong \Hom_{\textnormal{G}}(\textnormal{c-Ind}_K^{\textnormal{G}}\sigma_{\Omega},\textnormal{Ind}_{\textnormal{P}}^{\textnormal{G}}\pi) \\
        &\cong \Hom_K(\sigma_{\Omega},\textnormal{Ind}_{\textnormal{P}}^{\textnormal{G}}\pi) \\
        &\cong \Hom_J(\kappa,\textnormal{Ind}_{\textnormal{P}}^{\textnormal{G}}\pi).
    \end{align*}
    Indeed, the first and last two identifications are via Frobenius reciprocity, the second one is \cite[Lemma 10.3]{BK98}, and the third one is via Bernstein's first adjoint theorem.

    The lemma now follows from the facts that $\kappa$ is a type for the block $\Omega$ and all Jordan--H\"older factors of $\textnormal{Ind}_{\textnormal{P}}^{\textnormal{G}}\pi$ lie in the same Bernstein block.
\end{proof}

We set $\textnormal{Rep}_{\textnormal{sm}}(\textnormal{M},\overline{\mathbf{Q}}_p)[\Omega]$ to be the direct product factor of $\textnormal{Rep}_{\textnormal{sm}}(\textnormal{M},\overline{\mathbf{Q}}_p)$ given by the product of Bernstein blocks $\Omega_M$ that induce $\Omega$ along parabolic induction.

Set $\mathcal{H}(\sigma_{\Omega}):=\textnormal{End}_{\textnormal{G}}(\textnormal{c-Ind}_K^{\textnormal{G}}\sigma_{\Omega})$ and $\mathcal{H}(\sigma_{\textnormal{P},\Omega}):=\textnormal{End}_{\textnormal{M}}(\textnormal{c-Ind}_{K_{\textnormal{M}}}^{\textnormal{M}}\sigma_{\textnormal{P},\Omega})$. 
The $\overline{\mathbf{Q}}_p$-algebra $\mathcal{H}(\sigma_{\Omega})$ can be canonically identified with the convolution algebra of $\sigma_{\Omega}$-bi-invariant compactly supported functions $f\colon\textnormal{G}\to \textnormal{End}(\sigma_{\Omega})$ by acting on the former via convolution. We will freely move between these two ``perspectives" of the Hecke algebra. Using the latter description, we can write down an anti-isomorphism
\begin{align}\label{equation-antiiso}
\begin{split}
    \mathcal{H}(\sigma_{\Omega}) &\xrightarrow{\sim}\mathcal{H}(\sigma_{\Omega}^{\vee}) \\
    F &\mapsto (F^{\vee}:g\mapsto F(g^{-1})^t).
\end{split}
\end{align}

Given $\pi\in \textnormal{Rep}_{\textnormal{sm}}(\textnormal{G},\overline{\mathbf{Q}}_p)$, $\Hom_{\textnormal{G}}(\textnormal{c-Ind}_{K}^{\textnormal{G}}\sigma_{\Omega},\pi)$ becomes a \textit{right} $\mathcal{H}(\sigma_{\Omega})$-module by acting on the first factor. In particular, under the anti-isomorphism \eqref{equation-antiiso}, it becomes a \textit{left} $\mathcal{H}(\sigma_{\Omega}^{\vee})$-module. Under the Frobenius reciprocity isomorphism
\begin{equation*}
    \Hom_{\textnormal{G}}(\textnormal{c-Ind}_K^{\textnormal{G}}\sigma_{\Omega},\pi)\cong \Hom_{K}(\sigma_{\Omega},\pi)\cong \Gamma(K,\pi\otimes\sigma_{\Omega}^{\vee})
\end{equation*}
the described left action on the LHS coincides with the usual Hecke action of $\mathcal{H}(\sigma_{\Omega}^{\vee})$ on the RHS where we treat the elements of the Hecke algebra as bi-invariant functions.

The same remarks apply equally well to $\mathcal{H}(\sigma_{\textnormal{P},\Omega})$.
\begin{corollary}\label{corollary_equiv_with_cat_of_Heckemodules}
    There is an equivalence of categories
    \begin{align*}
        \textnormal{Rep}_{\textnormal{sm}}(\textnormal{M},\overline{\mathbf{Q}}_p)[\Omega] & \xrightarrow{\sim}\textnormal{Mod}(\mathcal{H}(\sigma_{\textnormal{P},\Omega}^{\vee})) \\
    \pi & \mapsto \Hom_{\textnormal{M}}(\textnormal{c-Ind}_{K_{\textnormal{M}}}^{\textnormal{M}}\sigma_{\textnormal{P},\Omega},\pi).
    \end{align*}
    In particular, the centre $\mathfrak{Z}_{\textnormal{M},\Omega}$ of $\textnormal{Rep}_{\textnormal{sm}}(\textnormal{M},\overline{\mathbf{Q}}_p)[\Omega]$ is identified with $Z(\mathcal{H}(\sigma_{\textnormal{P},\Omega}^{\vee}))$.
\end{corollary}
\begin{proof}
    According to \cite[Theorem 1.1]{Roc02} (see also \cite[Proposition 2.3]{Hel16}), it suffices to check that $\textnormal{c-Ind}_{K_{\textnormal{M}}}^{\textnormal{M}}\sigma_{\textnormal{P}}$ is a faithfully projective object of the category $\textnormal{Rep}_{\textnormal{sm}}(\textnormal{M},\overline{\mathbf{Q}}_p)[\Omega]$ in the sense of \cite[Definition 2.2.]{Hel16} It lies in $\textnormal{Rep}_{\textnormal{sm}}(\textnormal{M},\overline{\mathbf{Q}}_p)[\Omega]$ and is faithful by Lemma~\ref{BKlemma}. Being a compact induction of a finite dimensional $\overline{\mathbf{Q}}_p$-representation of $K_{\textnormal{M}}$, it is projective and small.
\end{proof}

Recall that, using the Bernstein--Deligne description of the Bernstein centres $\mathfrak{Z}_{\textnormal{G},\Omega}$ and $\mathfrak{Z}_{\textnormal{M},\Omega}$, one has a unique algebra map $\textnormal{BD}:=\textnormal{BD}_{\textnormal{M},\Omega}\colon\mathfrak{Z}_{\textnormal{G},\Omega}\to \mathfrak{Z}_{\textnormal{M},\Omega}$ such that, for any $\pi \in \textnormal{Rep}_{\textnormal{sm}}(\textnormal{M},\overline{\mathbf{Q}}_p)$, the following diagram commutes
\begin{equation*}
    \begin{tikzcd}
	{\mathfrak{Z}_{\textnormal{G},\Omega}} & {\textnormal{End}_{\textnormal{G}}(\textnormal{Ind}_{\textnormal{P}}^{\textnormal{G}}\pi)} \\
	{\mathfrak{Z}_{\textnormal{M},\Omega}} & {\textnormal{End}_{\textnormal{M}}(\pi)}.
	\arrow[from=1-1, to=1-2]
	\arrow["{\textnormal{BD}}"', from=1-1, to=2-1]
	\arrow[from=2-1, to=2-2]
	\arrow["{\textnormal{Ind}_{\textnormal{P}}^{\textnormal{G}}}"', from=2-2, to=1-2]
\end{tikzcd}
\end{equation*}

On the other hand, there is the usual Satake transform
\begin{equation*}
\mathcal{S}:=\mathcal{S}_{\textnormal{P},\Omega}\colon\mathcal{H}(\sigma_{\Omega}^{\vee})\to \mathcal{H}(\sigma_{\textnormal{P},\Omega}^{\vee}).
\end{equation*}
To define this, recall that $\mathcal{H}(\sigma_{\Omega}^{\vee})$ and $\mathcal{H}(\sigma_{\textnormal{P},\Omega}^{\vee})$ are identified with the convolution algebras of $\sigma_{\Omega}^{\vee}$-, and $\sigma_{\textnormal{P},\Omega}^{\vee}$-bi-invariant compactly supported functions, respectively.
Then the Satake transform is defined by restricting a function $f\colon\textnormal{G}\to \textnormal{End}(\sigma_{\Omega}^{\vee})$ to $\textnormal{P}$ and then integrating $f\mid_{\textnormal{P}}$ along $\textnormal{N}$. In formulas we have
\begin{align*}
    \mathcal{S} \colon \mathcal{H}(\sigma_{\Omega}^{\vee}) &\to \mathcal{H}(\sigma_{\textnormal{P},\Omega}^{\vee}) \\
    f & \mapsto (m\mapsto \sum_{n\in \textnormal{N}/K_{\textnormal{N}}}f(mn)\mid_{\sigma_{\textnormal{P},\Omega}^{\vee}}).
\end{align*}

\begin{lemma}
    The following diagram is commutative
    \begin{equation*}
        \begin{tikzcd}
	{\mathfrak{Z}_{\textnormal{G},\Omega}} & {Z(\mathcal{H}(\sigma_{\Omega}^{\vee}))} & {\mathcal{H}(\sigma_{\Omega}^{\vee})} \\
	{\mathfrak{Z}_{\textnormal{M},\Omega}}  & {Z(\mathcal{H}(\sigma_{\textnormal{P},\Omega}^{\vee}))} & {\mathcal{H}(\sigma_{\textnormal{P},\Omega}^{\vee})}.
	\arrow["\sim", from=1-1, to=1-2]
	\arrow["{\textnormal{BD}}"', from=1-1, to=2-1]
	\arrow[hook, from=1-2, to=1-3]
	\arrow["{\mathcal{S}}", from=1-3, to=2-3]
	\arrow["\sim", from=2-1, to=2-2]
	\arrow[hook, from=2-2, to=2-3]
\end{tikzcd}
    \end{equation*}
\end{lemma}
\begin{proof}
    Let $z\in \mathfrak{Z}_{\textnormal{G},\Omega}$ and first treat it as an element of $\mathcal{H}(\sigma_{\Omega}^{\vee})$. For $\pi \in \textnormal{Rep}_{\textnormal{sm}}(\textnormal{M},\overline{\mathbf{Q}}_p)$, consider the Frobenius reciprocity isomorphism
    \begin{equation}\label{equation_FrobRec}\Hom_{\textnormal{G}}(\textnormal{c-Ind}_K^{\textnormal{G}}\sigma_{\Omega},\textnormal{Ind}_{\textnormal{P}}^{\textnormal{G}}\pi)\cong \Gamma(K,\textnormal{Ind}_{\textnormal{P}}^{\textnormal{G}}\pi\otimes\sigma_{\Omega}^{\vee}).
    \end{equation}
    On the left, $z$ acts by acting on the source $\textnormal{c-Ind}_{K}^{\textnormal{G}}\sigma_{\Omega}$. On the other hand, when we treat $z$ as a $\sigma_{\Omega}^{\vee}$-bi-invariant function $\textnormal{G}\to \textnormal{End}(\sigma^{\vee}_{\Omega})$, it acts on the right via the usual Hecke action. These two actions of $z$ are intertwined by the isomorphism \eqref{equation_FrobRec}.

    We further have a series of isomorphisms
    \begin{equation}\label{SatakeId}
\Gamma(K,\textnormal{Ind}_{\textnormal{P}}^{\textnormal{G}}\pi\otimes\sigma_{\Omega}^{\vee})\cong \Gamma(K_{\textnormal{P}}, \textnormal{Inf}_{\textnormal{M}}^{\textnormal{P}}\pi\otimes\sigma_{\Omega}^{\vee}) \cong
        \Gamma(K_{\textnormal{M}},\pi\otimes \sigma_{\textnormal{P},\Omega}^{\vee})
    \end{equation}
    where the first identification follows from the Iwasawa decomposition $\textnormal{G}=\textnormal{P}K$. Using the definition of the Satake transform $\mathcal{S}$, one sees that the isomorphism above intertwines the Hecke action of $z$ on the LHS with the Hecke action of $\mathcal{S}(z)$ on the RHS.

    Finally, Frobenius reciprocity intertwines the Hecke action of $\mathcal{S}(z)$ on $\Gamma(K_{\textnormal{M}},\pi\otimes \sigma_{\textnormal{P},\Omega}^{\vee})$ with the action of $\mathcal{S}(z)$ via the first factor on $\Hom_{\textnormal{M}}(\textnormal{c-Ind}_{K_{\textnormal{M}}}^{\textnormal{M}}\sigma_{\textnormal{P},\Omega},\pi)$.

    On the other hand, by treating $z$ as an element of the Bernstein centre, the endomorphism of $\Hom_{\textnormal{G}}(\textnormal{c-Ind}_{K}^{\textnormal{G}}\sigma_{\Omega},\textnormal{Ind}_{\textnormal{P}}^{\textnormal{G}}\pi)$ induced by acting on the first and, respectively, second factor coincide. By the characterising property of $\textnormal{BD}$, under the Frobenius reciprocity isomorphism
    \begin{equation*}
        \Hom_{\textnormal{G}}(\textnormal{c-Ind}_K^{\textnormal{G}}\sigma_{\Omega},\textnormal{Ind}_{\textnormal{P}}^{\textnormal{G}}\pi)\cong \Gamma(K, \textnormal{Ind}_{\textnormal{P}}^{\textnormal{G}}\pi\otimes \sigma_{\Omega}^{\vee})
    \end{equation*}
    this endomorphism also coincides with the endomorphism  on the target induced by acting on $\pi$ via $\textnormal{BD}(z)$.

    Under the isomorphism \ref{SatakeId}, the endomorphism on the LHS induced by acting on $\pi$ via $\textnormal{BD}(z)$ matches the endomorphism induced by acting on $\pi$ via $\textnormal{BD}(z)$ on the RHS. Finally, under Frobenius reciprocity, using that $\textnormal{BD}(z)$ lies in the centre of the category, this endomorphism matches the endomorphism of $\Hom_{\textnormal{M}}(\textnormal{c-Ind}_{K_{\textnormal{M}}}^{\textnormal{M}}\sigma_{\textnormal{P},\Omega},\pi)$ induced by acting via $\textnormal{BD}(z)$ on the first factor.

 In summary, we proved that, for any $\pi\in \textnormal{Rep}_{\textnormal{sm}}(\textnormal{M},\overline{\mathbf{Q}}_p)$, $\mathcal{S}(z)$ and $\textnormal{BD}(z)$ act via the same endomorphism on $\pi$. Finally, we note that $\mathcal{H}(\sigma_{\textnormal{P},\Omega}^{\vee})$ acts faithfully on $\textnormal{c-Ind}_{K_{\textnormal{M}}}^{\textnormal{M}}\sigma_{\textnormal{P},\Omega}$. In particular, setting $\pi=\textnormal{c-Ind}_{K_{\textnormal{M}}}^{\textnormal{M}}\sigma_{\textnormal{P},\Omega}$, we obtain the desired equality $\mathcal{S}(z)=\textnormal{BD}(z)$.
\end{proof}

As we observed, $\sigma_{\textnormal{P},\Omega}$ is a $K_{\textnormal{M}}$-typical representation for $\textnormal{Rep}_{\textnormal{sm}}(\textnormal{M},\overline{\mathbf{Q}}_p)[\Omega]$. To gain a better understanding of this representation we would like to relate it to Bushnell--Kutzko types for the blocks $\Omega_{\textnormal{M}}$ appearing in $\textnormal{Rep}_{\textnormal{sm}}(\textnormal{M},\overline{\mathbf{Q}}_p)[\Omega]$. We are able to give such a description after applying $\textnormal{c-Ind}_{K_{\textnormal{M}}}^{\textnormal{M}}$ which turns out to be sufficient for our applications.
\begin{prop}\label{Proposition_BushnellKutzkoTypesIso}
    There is an isomorphism
    \begin{equation*}
        \textnormal{c-Ind}_{K_{\textnormal{M}}}^{\textnormal{M}}\sigma_{\textnormal{P},\Omega}\xrightarrow{\sim}\textnormal{c-Ind}_{K_{\textnormal{M}}}^{\textnormal{M}}\sigma_{\textnormal{M},\Omega}
    \end{equation*}
    for
    \begin{equation*}
        \sigma_{\textnormal{M},\Omega}:=\oplus_{\Omega_{\textnormal{M}}\in \mathfrak{B}(\textnormal{M},\Omega)}\sigma_{\Omega_{\textnormal{M}}}
    \end{equation*}
    where the sum runs over some multiset $\mathfrak{B}(\textnormal{M},\Omega)$ of Bernstein blocks $\Omega_{\textnormal{M}}$ of $\textnormal{Rep}_{\textnormal{sm}}(\textnormal{M},\overline{\mathbf{Q}}_p)$ inducing $\Omega$ under parabolic induction.
\end{prop}
\begin{proof}
    By \cite[Lemma 10.3]{BK98}, we have an isomorphism
    \begin{equation*}
        \textnormal{c-Ind}_{K_{\textnormal{M}}}^{\textnormal{M}}\sigma_{\textnormal{P},\Omega}\xrightarrow{\sim}(\textnormal{c-Ind}_K^{\textnormal{G}}\sigma_{\Omega})_{\textnormal{N}}.
    \end{equation*}
    
    Consider the pair $(\widetilde{\textnormal{M}},\widetilde{\pi})=(\textnormal{GL}_{n_1}(L)\times...\times \textnormal{GL}_{n_h}(L),\widetilde{\pi}_1\otimes...\otimes \widetilde{\pi}_h)$ for the standard parabolic subgroup $\widetilde{\textnormal{P}}=\widetilde{\textnormal{M}}\widetilde{\textnormal{N}}$ of $\textnormal{G}$ and the supercuspidal $\overline{\mathbf{Q}}_p$-representation $\widetilde{\pi}$ of $\widetilde{\textnormal{M}}$ such that $\Omega=[\widetilde{\textnormal{M}},\widetilde{\pi}]$, $\widetilde{\pi}_i\in \textnormal{SC}_0$ and
    $\widetilde{\pi}_1\leq...\leq\widetilde{\pi}_h$ with respect to the ordering chosen on $\textnormal{SC}_0$. Set $\widetilde{\Omega}:=[\widetilde{\textnormal{M}},\widetilde{\pi}]_{\widetilde{\textnormal{M}}}$. By \cite{Dat1999}, section 1.5, we have an isomorphism
    \begin{equation*}
        (\textnormal{c-Ind}_K^{\textnormal{G}}\sigma_{\Omega})_{\textnormal{N}}\xrightarrow{\sim}\left(\textnormal{n-Ind}_{\widetilde{\textnormal{P}}}^{\textnormal{G}}(\textnormal{c-Ind}_{K_{\widetilde{\textnormal{M}}}}^{\widetilde{\textnormal{M}}}\sigma_{\widetilde{\Omega}})\right)_{\textnormal{N}}.
        \end{equation*}

        By the geometric lemma \cite[Lemma 2.12]{Bernstein1977}, the RHS admits a filtration with graded pieces $I_w$ labelled by $^{\textnormal{P}}W^{\widetilde{\textnormal{P}}}$. Moreover, since $\widetilde{\Omega}$ is a cuspidal Bernstein block, any non-trivial Jacquet module of $\textnormal{c-Ind}_{K_{\widetilde{\textnormal{M}}}}^{\widetilde{\textnormal{M}}}\sigma_{\widetilde{\Omega}}$ vanishes as can be seen by combining Bernstein's first adjoint theorem, Frobenius reciprocity for $\textnormal{c-Ind}_{K_{\widetilde{\textnormal{M}}}}^{\widetilde{\textnormal{M}}}$ and the fact that a cuspidal block only consists of supercuspidal representations. 
        In particular, the only graded pieces that survive are those labelled by Kostant representatives $w\in {}^{\textnormal{P}}W^{\widetilde{\textnormal{P}}}$ satisfying $\widetilde{\textnormal{M}}^w:=w\widetilde{\textnormal{M}}w^{-1}\subset \textnormal{P}$. If we write $\widetilde{\textnormal{P}}^w\subset \textnormal{M}$ for the corresponding standard parabolic subgroup, we have
        \begin{align*}
            I_w&=\delta_{\textnormal{P}}^{1/2}\textnormal{n-Ind}_{\widetilde{P}^w}^{\textnormal{M}}w^{\ast}(\textnormal{c-Ind}_{K_{\widetilde{\textnormal{M}}}}^{\widetilde{\textnormal{M}}}\sigma_{\widetilde{\Omega}}) 
        \\
            &\cong \delta_{\textnormal{P}}^{1/2}\textnormal{n-Ind}_{\widetilde{\textnormal{P}}^w}^{\textnormal{M}}(\textnormal{c-Ind}_{K_{\widetilde{\textnormal{M}}^w}}^{\widetilde{\textnormal{M}}^w}\sigma_{w^{\ast}\widetilde{\Omega}}).
        \end{align*}

        Write $\Omega_{\textnormal{M}}^w$ for the Bernstein block of  $\textnormal{Rep}_{\textnormal{sm}}(\textnormal{M},\overline{\mathbf{Q}}_p)$ induced by $w^{\ast}\Omega_{\widetilde{\textnormal{M}}}$ via $\textnormal{Ind}_{\widetilde{\textnormal{P}}^w}^{\textnormal{M}}$ and $(J_w,\kappa_w)$ for the Bushnell--Kutzko type for (the cuspidal block) $w^{\ast}\Omega_{\widetilde{\textnormal{M}}}$. Note that there is a Weyl group element $\widetilde{w}\in W_{\textnormal{M}}^{\widetilde{\textnormal{M}}^w}$ such that the Bushnell--Kutzko type $(J^w,\kappa^w)$ for $\Omega_{\textnormal{M}}^w$ introduced at the beginning of the section is an $\textnormal{M}$-cover of $\widetilde{w}^{\ast}\kappa_{w}$ and, consequently, $(\widetilde{w}^{-1}J^w\widetilde{w},\widetilde{w}^{-1,\ast}\kappa^w)$ is an $\textnormal{M}$-cover of $\kappa_w$. In particular, by another application of \cite{Dat1999}, section 1.5, $I_w$ is isomorphic to
        \begin{equation*}
            \delta_{\textnormal{P}}^{1/2}  \textnormal{c-Ind}_{\widetilde{w}^{-1}J^w\widetilde{w}}^{\textnormal{M}}(\widetilde{w}^{-1,\ast}\kappa^w)\cong\delta_{\textnormal{P}}^{1/2}\textnormal{c-Ind}_{K_{\textnormal{M}}}^{\textnormal{M}}\sigma_{\Omega_{\textnormal{M}}^w}.
        \end{equation*}
         In addition, this is isomorphic to $\textnormal{c-Ind}_{K_{\textnormal{M}}}^{\textnormal{M}}\sigma_{\Omega_{\textnormal{M}}^w}$.

        We proved that $\textnormal{c-Ind}_{K_{\textnormal{M}}}^{\textnormal{M}}\sigma_{\textnormal{P},\Omega}$ admits a filtration with graded pieces $\textnormal{c-Ind}_{K_{\textnormal{M}}}^{\textnormal{M}}\sigma_{\Omega_{\textnormal{M}}^w}$ labelled by $\{w\in {}^{\textnormal{P}}W^{\widetilde{\textnormal{P}}}\mid w\widetilde{\textnormal{M}}w^{-1}\subset \textnormal{P} \}$. As all of these representation are projective objects in $\textnormal{Rep}_{\textnormal{sm}}(\textnormal{M},\overline{\mathbf{Q}}_p)$, the filtration is necessarily split, finishing the proof.
\end{proof}

We will also need an explicit description of the Hecke algebras $\mathcal{H}(\kappa^{\vee})$. Recall the subgroups $\textnormal{M}\leq \widetilde{\textnormal{M}}\leq \textnormal{G}$ and the types $(J_{\textnormal{M}},\kappa_{\textnormal{M}})$, $(J_{\widetilde{M}},\kappa_{\widetilde{\textnormal{M}}})$ appearing in the first paragraphs of the section.

Given a parabolic subgroup $\textnormal{P}$ with Levi factor $\textnormal{M}$, by \cite[Propositions 7.6.2, 7.6.5]{BK93}, there is a unique \textit{support preserving} injective algebra homomorphism
\begin{equation*}
    t_{\textnormal{P}\cap \widetilde{\textnormal{M}}}\colon\mathcal{H}(\kappa_{\textnormal{M}}^{\vee})\hookrightarrow \mathcal{H}(\kappa_{\widetilde{\textnormal{M}}}^{\vee})
\end{equation*}
such that the functor $\Hom_{\mathcal{H}(\kappa_{\textnormal{M}}^{\vee})}(\mathcal{H}(\kappa_{\widetilde{\textnormal{M}}}^{\vee}),-)$ matches with $\textnormal{n-Ind}_{\textnormal{P}\cap \widetilde{\textnormal{M}}}^{\widetilde{\textnormal{M}}}$ under the equivalence of categories given by Corollary \ref{corollary_equiv_with_cat_of_Heckemodules} for $\textnormal{M}$ and $\widetilde{\textnormal{M}}$, respectively. Moreover, by \cite{BK98} Theorem 12.1, there is a unique support preserving algebra isomorphism
\begin{equation*}
    t_{\textnormal{P}\widetilde{\textnormal{M}}}\colon\mathcal{H}(\kappa_{\widetilde{\textnormal{M}}}^{\vee})\xrightarrow{\sim}\mathcal{H}(\kappa^{\vee})
\end{equation*}
such that the functor $\Hom_{\mathcal{H}(\kappa_{\widetilde{\textnormal{M}}}^{\vee})}(\mathcal{H}(\kappa^{\vee}),-)$ matches with $\textnormal{n-Ind}_{\widetilde{\textnormal{M}}}^{\textnormal{G}}$
under the equivalence of categories provided by Corollary \ref{corollary_equiv_with_cat_of_Heckemodules} for $\widetilde{\textnormal{M}}$ and $\textnormal{G}$, respectively.

In \cite[7.6]{BK93} it is shown that $\mathcal{H}(\kappa^{\vee})$ is isomorphic to a tensor product of affine Hecke algebras of type $A$ in a way compatible with the algebra map $t_{\textnormal{P}}:=t_{\textnormal{P}\widetilde{\textnormal{M}}} \circ t_{\textnormal{P}\cap \widetilde{\textnormal{M}}}$. To discuss this, let us first assume that $\widetilde{\textnormal{M}}=\textnormal{G}$. In this case, we have $\textnormal{M}=\textnormal{GL}_d(\textnormal{L})\times...\times \textnormal{GL}_d(L)$ and $\pi=\pi_0\otimes...\otimes \pi_0$ for a supercuspidal $\overline{\mathbf{Q}}_p$-representation $\pi_0$ of $\textnormal{GL}_d(L)$. They consider a certain subfield $E\hookrightarrow \textnormal{End}_L(L^n)=M_n(L)$ of degree $d$ over $L$ such that the intersection of the image of the induced embedding 
\begin{equation}
    \iota\colon\textnormal{H}:=\textnormal{GL}_{n/d}(    E)=\textnormal{Aut}_E(L^n)\hookrightarrow \textnormal{Aut}_L(L^n)=\textnormal{GL}_n(L)
\end{equation}
with $\textnormal{M}$ is a maximal torus $\textnormal{T}_{\textnormal{H}}\subset \textnormal{H}$. The intersection $I_{\textnormal{H}}:=\textnormal{H}\cap J\subset \textnormal{H}$ is the Iwahori subgroup associated with the Borel $\textnormal{B}_{\textnormal{H}}:=\textnormal{P}\cap \textnormal{H}\subset \textnormal{H}$ of upper triangular matrices. We set $\textnormal{T}^{\circ}_{\textnormal{H}}:=\textnormal{T}_{\textnormal{H}}\cap I_{\textnormal{H}}$. Then the map
\begin{equation*}
    t_{\textnormal{B}_{\textnormal{H}}}\colon\mathcal{H}(\textnormal{T}_{\textnormal{H}},\textnormal{T}^{\circ}_{\textnormal{H}})\to \mathcal{H}(\textnormal{H},I_{\textnormal{H}}),
\end{equation*}
upon identifying the source with $\overline{\mathbf{Q}}_p[X_{\ast}(\textnormal{T}_{\textnormal{H}})]$, sends a  minuscule cocharacter $\mu$ to $\delta_{\textnormal{B}_{\textnormal{H}}}(\mu(\varpi_E))^{1/2}[I_{\textnormal{H}}\mu(\varpi_E)I_{\textnormal{H}}]$ and therefore coincides with Bernstein's map $\Theta\colon\overline{\mathbf{Q}}_p[X_{\ast}(\textnormal{T}_{\textnormal{H}})]\hookrightarrow \mathcal{H}(\textnormal{H},I_{\textnormal{H}})$.

In addition, by \cite[7.6.19]{BK93} one has an algebra isomorphism
\begin{align*}
    \mathcal{H}(\textnormal{T}_{\textnormal{H}},\textnormal{T}_{\textnormal{H}}^{\circ}) &\xrightarrow{\sim}\mathcal{H}(\kappa_{\textnormal{M}}^{\vee}) \\
    \mu &\mapsto T_{\mu}
\end{align*}
with $T_{\mu}$ being supported on  $J_{\textnormal{M}}\iota(\mu(\varpi_E))J_{\textnormal{M}}$ that only depends on the choice of $\pi_0$ in the inertial equivalence class.
By \cite[Theorem 7.6.20]{BK93}, there is an algebra isomorphism $\mathcal{H}(\textnormal{H},I_{\textnormal{H}})\xrightarrow{\sim}\mathcal{H}(\kappa^{\vee})$ that fits into a commutative diagram
\begin{equation*}
    \begin{tikzcd}
	{\mathcal{H}(\textnormal{T}_{\textnormal{H}},\textnormal{T}_{\textnormal{H}}^{\circ})} & {\mathcal{H}(\kappa_{\textnormal{M}}^{\vee})} \\
	{\mathcal{H}(\textnormal{H},I_{\textnormal{H}})} & {\mathcal{H}(\kappa^{\vee}).}
	\arrow["\sim", from=1-1, to=1-2]
	\arrow["{t_{\textnormal{B}_{\textnormal{H}}}}"', from=1-1, to=2-1]
	\arrow["{t_{\textnormal{P}}}", from=1-2, to=2-2]
	\arrow["\sim", from=2-1, to=2-2]
\end{tikzcd}
\end{equation*}

Now assume that $\widetilde{\textnormal{M}}\subset \textnormal{G}$ is arbitrary. Then $\textnormal{M}=\textnormal{M}_1\times...\times \textnormal{M}_k\leq \widetilde{\textnormal{M}}=\textnormal{G}_1\times...\times \textnormal{G}_k\leq \textnormal{G}$ is a group of block diagonal matrices with $\textnormal{G}_i=\textnormal{GL}_{n_i}(L)$ for $i=1,...,k$. Write $\textnormal{M}=\textnormal{M}_1\times...\times \textnormal{M}_k\leq \widetilde{\textnormal{M}}=\textnormal{G}_1\times...\times \textnormal{G}_k$ for the induced product decomposition. By the previous discussion, we obtain field extensions $E_i/L$ of degree $d_i$, embeddings $\iota_i\colon\textnormal{H}_i=\textnormal{GL}_{n_i/d_i}(E_i)\hookrightarrow \textnormal{G}_i$ with $\textnormal{H}_i\cap \textnormal{M}_i=\textnormal{T}_{\textnormal{H}_i}$ and a commutative diagram
\begin{equation}\label{Diagram_CompatibilityWithIwahori}
    \begin{tikzcd}
	{\otimes_{i=1,...,k}\mathcal{H}(\textnormal{T}_{\textnormal{H}_i},\textnormal{T}_{\textnormal{H}_i}^{\circ})} & {\mathcal{H}(\kappa_{\textnormal{M}}^{\vee})} \\
	{\otimes_{i=1,...,k}\mathcal{H}(\textnormal{H}_i,I_{\textnormal{H}_i})} & {\mathcal{H}(\kappa_{\widetilde{\textnormal{M}}}^{\vee}} )\\
	& {\mathcal{H}(\kappa^{\vee})}
	\arrow["\sim", from=1-1, to=1-2]
	\arrow["{\otimes_{i}t_{\textnormal{B}_{\textnormal{H}_i}}}"', hook', from=1-1, to=2-1]
	\arrow["{t_{\textnormal{P}\cap \widetilde{\textnormal{M}}}}", hook', from=1-2, to=2-2]
	\arrow["\sim", from=2-1, to=2-2]
	\arrow["{t_{\textnormal{P}\widetilde{\textnormal{M}}}}", from=2-2, to=3-2]
\end{tikzcd}
\end{equation}
where $t_{\textnormal{P}\widetilde{\textnormal{M}}}$ is an algebra isomorphism and, in particular, $t_{\textnormal{P}}:=t_{\textnormal{P}\widetilde{\textnormal{M}}}\circ t_{\textnormal{P}\cap \widetilde{\textnormal{M}}}$ is an injective algebra homomorphism.

Let us now make this slightly more explicit. For $1\leq i\leq k$, and $1\leq j\leq n_i/d_i$, write $\mu_{i,j}\in X_{\ast}(\textnormal{T}_{\textnormal{H}_i})$ for the minuscule cocharacter $(1,...,1,0,...,0)$ with $1$ appearing exactly $j$ times and $T_{\mu_{i,j}}\in \mathcal{H}(\kappa_{\textnormal{M}}^{\vee})$ for the image of $1\otimes...\otimes1\otimes \mu_{i,j}\otimes1\otimes...\otimes 1$. This is a function that is supported on
\begin{equation*}
    J_{\textnormal{M}}(\mathbf{1}_{n_1},...,\mathbf{1}_{n_{i-1}},\iota_i(\mu_{i,j}(\varpi_{E_i})),\mathbf{1}_{n_{i+1}},...,\mathbf{1}_{n_k})J_{\textnormal{M}}.
\end{equation*}
We then have an isomorphism
\begin{equation*}
    \mathcal{H}(\kappa_{\textnormal{M}}^{\vee})\cong \overline{\mathbf{Q}}_p[T_{\mu_{i,j}}^{\pm 1}\mid 1\leq i\leq k, 1\leq j\leq n_i/d_i].
\end{equation*}

Recall the smooth $\overline{\mathbf{Q}}_p$-representation $\sigma_{\Omega}=\textnormal{c-Ind}_J^K\kappa$. We have an algebra isomorphism
\begin{equation*}
    A\colon\mathcal{H}(\kappa^{\vee})\xrightarrow{\sim}\mathcal{H}(\sigma^{\vee}_{\Omega})
\end{equation*}
induced by $B\colon\textnormal{c-Ind}_J^{\textnormal{G}}\kappa^{\vee}\xrightarrow{\sim} \textnormal{c-Ind}_K^{\textnormal{G}}\sigma^{\vee}_{\Omega}$ (with inverse sending $f\in \textnormal{c-Ind}_K^{\textnormal{G}}\sigma^{\vee}_{\Omega}$ to $f(-)(1)\in \textnormal{c-Ind}_J^{\textnormal{G}}\kappa^{\vee}$), sending an element $F\in \mathcal{H}(\kappa^{\vee})=\textnormal{End}_{\textnormal{G}}(\textnormal{c-Ind}_J^{\textnormal{G}}\kappa^{\vee})$ to $B\circ F\circ B^{-1}\in \mathcal{H}(\sigma^{\vee}_{\Omega})=\textnormal{End}_{\textnormal{G}}(\textnormal{c-Ind}_K^{\textnormal{G}}\sigma^{\vee}_{\Omega})$.
\begin{lemma}\label{Lemma_explicitdescriptionofgenerators}
    Suppose that $F\in \mathcal{H}(\kappa^{\vee})$, viewed as a function $\textnormal{G}\to \textnormal{End}_{\overline{\mathbf{Q}}_p}(\kappa^{\vee})$ is supported on the double coset $JgJ$ for an element $g\in \textnormal{G}$. Then $A(F)\colon\textnormal{G}\to \textnormal{End}_{\overline{\mathbf{Q}}_p}(\sigma^{\vee}_{\Omega})$ is supported on the double coset $KgK$.

    In particular, for $i=1,...,k$, $j=1,...,n_i/d_i$, the element $Y_{i,j}:=A(t_{\textnormal{P}}(T_{\mu_{i,j}}))\in \mathcal{H}(\sigma^{\vee}_{\Omega})$ is of the form $[Kt_{i,j}K,\psi_{i,j}]$ where $t_{i,j}=\nu_{i,j}(\varpi_L)$ for some \textit{minuscule cocharacter} $\nu_{i,j}\in X_{\ast}(\textnormal{T})$ and $\psi_{i,j}\in\Hom_{t_{i,j}^{-1}Kt_{i,j}\cap K}(\sigma_{\Omega}^{\vee},t_{i,j}^{\ast}\sigma_{\Omega}^{\vee})$.
\end{lemma}
\begin{proof}
    Pick a $\overline{\mathbf{Q}}_p$-basis $v_1,...,v_h$ for the representation space of $\kappa^{\vee}$. Then the functions $[Jk,v_i]:K\to \kappa^{\vee}$ supported on $Jk$ and sending $k$ to $v_i$, $k\in J\backslash K$, $1\leq i\leq h$ form a basis of the representation space of $\sigma^{\vee}_{\Omega}=\textnormal{c-Ind}_J^K\kappa^{\vee}$.

    Let $[K,[Jk,v_i]]\colon\textnormal{G}\to \sigma^{\vee}_{\Omega}$ be the function supported on $K$, sending $1$ to $[Jk,v_i]\in \sigma^{\vee}_{\Omega}$. We claim that $A(F)$ is supported on $KgK$ if and only if $A(F)\ast [K,[Jk,v_i]]$ is supported on $KgK$ for every $k\in J\backslash K$, $1\leq i\leq h$. Indeed, this follows from the fact that for $x\in \textnormal{G}$, $A(F)\ast [K,[Jk,v_i]](x)=A(F)(x)\bigl([Jk,v_i]\bigr)$ and the elements $[Jk,v_i]$ form a basis for the representation space of $\sigma^{\vee}_{\Omega}$.

    We now compute
    \begin{align*}
        A(F)\ast [K,[Jk,v_i]] & =B\Bigl(F\ast\bigl(B^{-1}([K,[Jk,v_i]])\bigr)\Bigr)
    \\
        &= B\bigl(F\ast[Jk,v_i]\bigr) \\
        &= k^{-1}B(F\ast [J,v_i])
    \end{align*}
    where the second equality used $B([Jk,v_i])=[K,[Jk,v_i]]$ and the last equality that $k^{-1}[J,v_i]=[Jk,v_i]$ and both $F\ast (-)$, and $B$ are $\textnormal{G}$-equivariant. The function $F\ast[J,v_i]\colon\textnormal{G}\to \kappa^{\vee}$ is supported on $JgJ$, so $B(F\ast[J,v_i])$ is supported on $KgJ$. In particular, $k^{-1}B(F\ast [J,v_i])$ is supported on $KgK$ showing that so is $A(F)$.

    To prove the second statement, recall that $t_{\textnormal{P}}(T_{\mu_{i,j}})$ is supported on
    \begin{equation*}
        J\textnormal{diag}(\mathbf{1}_{n_1},...,\mathbf{1}_{n_{i-1}},\iota_i(\mu_{i,j}(\varpi_{E_i})),\mathbf{1}_{n_{i+1}},...,\mathbf{1}_{n_k})J.
    \end{equation*}
    We claim that
    \begin{equation}\label{equation_minusculeclaim}
        \iota_i(\mu_{i,j}(\varpi_{E_i}))\in \textnormal{GL}_{n_i}(\mathcal{O}_L)\widetilde{\mu}_{i,j}(\varpi_L)\textnormal{GL}_{n_i}(\mathcal{O}_L)\subset \textnormal{GL}_{n_i}(L)
    \end{equation}
    for some minuscule cocharacter $\widetilde{\mu}_{i,j}\colon\mathbf{G}_m(L)\to T_{n_i}(L)$. Once the claim is proved, the second part of the Lemma follows as $A(t_{\textnormal{P}}(T_{\mu_{i,j}}))$ is then supported on
    \begin{equation*}
        K\textnormal{diag}(\mathbf{1}_{n_1},...,\mathbf{1}_{n_{i-1}},\widetilde{\mu}_{i,j}(\varpi_{L}),\mathbf{1}_{n_{i+1}},...,\mathbf{1}_{n_k})K=K\nu_{i,j}(\varpi_L)K
    \end{equation*}
    for some minuscule cocharacter $\nu_{i,j}\colon\mathbf{G}_m(L)\to \textnormal{T}$.

    Let $K_i/L$ be the maximal unramified subextension of $E_i / L$, let $f_i:=[K_i:L]$ and let $r_i:=[E_i:K_i]$. We have a commutative diagram
    \begin{equation*}
    \begin{tikzcd}[scale cd=0.7, column sep = small]
        {\textnormal{GL}_{n_i/d_i}(E_i)=\textnormal{Aut}_{E_i}(L^{n_i})} & {\textnormal{GL}_{n_i/f_i}(K_i)=\textnormal{Aut}_{K_i}(L^{n_i})} & {\textnormal{GL}_{n_i}(L)=\textnormal{Aut}_{L}(L^{n_i})=G_i} \\
	{T_{n_i/d_i}(E_i)=\textnormal{M}_i\cap \textnormal{GL}_{n_i/d_i}(E_i)} & {\textnormal{GL}_{r_i}(K_i)^{\times n_i/d_i}=\textnormal{M}_i\cap \textnormal{GL}_{n_i/f_i}(K_i)} & {\textnormal{GL}_{d_i}(L)^{\times n_i/d_i}=\textnormal{M}_i,}
	\arrow["\rho", hook, from=1-1, to=1-2]
	\arrow["\upsilon", hook, from=1-2, to=1-3]
	\arrow["{\Delta}"', hook, from=2-1, to=1-1]
	\arrow["{\widetilde{\rho}^{\times n_i/d_i}}"', hook, from=2-1, to=2-2]
	\arrow["{\Delta}"', hook, from=2-2, to=1-2]
	\arrow["{\widetilde{\upsilon}^{\times n_i/d_i}}"', hook, from=2-2, to=2-3]
	\arrow["{\Delta}"', hook, from=2-3, to=1-3]
\end{tikzcd}
    \end{equation*}
    where the top maps factorise $\iota_i$ and the bottom maps are induced by intersecting the members of the top row with $\textnormal{M}_i$.
    
Concretely, $\widetilde{\rho}$ can be identified with the natural inclusion $\textnormal{Aut}_{E_i}(E_i)=\textnormal{GL}_1(E_i)\hookrightarrow \textnormal{Aut}_{K_i}(E_i)=\textnormal{GL}_{r_i}(K_i)$. In particular, the $\varpi_{K_{i}}$-adic valuation of $\det(\widetilde{\rho}(\varpi_{E_i}))$ is $1$. Indeed, this is because the characteristic polynomial of multiplication by $\varpi_{E_i}$ on the $K_i$-vector space $E_i$ is Eisenstein. Therefore, $\widetilde{\rho}(\varpi_{E_i})\in \textnormal{GL}_{r_i}(\mathcal{O}_{K_i})\textnormal{diag}(\varpi_{K_i},1,...,1)\textnormal{GL}_{r_i}(\mathcal{O}_{K_i})$ by the Cartan decomposition. 
    
    Moreover, $\widetilde{\upsilon}$ can be identified with the natural inclusion $\textnormal{Aut}_{K_i}(K_i^{r_i})=\textnormal{GL}_{r_i}(K_i)\hookrightarrow \textnormal{Aut}_{L}(K_i^{r_i})=\textnormal{GL}_{d_i}(L)$. In particular, $\widetilde{\upsilon}(\textnormal{diag}(\varpi_{K_i},1,...,1))=\textnormal{diag}(\varpi_L,...,\varpi_L,1,...,1)\in \textnormal{GL}_{d_i}(L)$ with the uniformiser occurring in the diagonal exactly $f_i$ times. The claim \eqref{equation_minusculeclaim} now follows from the identity 
    \[ \iota_i\Big(\Delta\big(\mu_{i,j}(\varpi_{E_i})\big)\Big)=\Delta \biggl(\widetilde{\upsilon}^{\times n_i/d_i}\Bigl(\widetilde{\rho}^{\times n_i/d_i}\big(\mu_{i,j}(\varpi_{E_i})\big)\Bigr)\biggr). \]
\end{proof}

We can finally introduce the commutative $\overline{\mathbf{Q}}_p$-subalgebras 
\begin{align}\label{equation_BernsteinSubalgPlus}
    \mathcal{H}^{B,+}(\sigma_{\Omega}) &:= \overline{\mathbf{Q}}_p[Y_{i,j}\mid 1\leq i\leq k, 1\leq j\leq n_i/d_i]\subset \mathcal{H}(\sigma^{\vee}_{\Omega}), \\
\label{equation_BernsteinSubalg}
    \mathcal{H}^B(\sigma_{\Omega})&:=\mathcal{H}^{B,+}(\sigma_{\Omega})[Y_{i,j}^{-1}\mid i,j] \\
    & =\overline{\mathbf{Q}}_p[Y_{i,j}^{\pm1}\mid 1\leq i\leq k,1\leq j\leq n_i/d_i]\subset \mathcal{H}(\sigma_{\Omega}^{\vee}). \nonumber
\end{align}
\begin{lemma}\label{Lemma_centercontainedinBernsteinsalgebra}
    The centre of $\mathcal{H}(\sigma_{\Omega}^{\vee})$ is contained in $\mathcal{H}^B(\sigma_{\Omega})$.
\end{lemma}
\begin{proof}
    By \ref{Diagram_CompatibilityWithIwahori}, it suffices to prove that the image of $t_{\textnormal{B}_{\textnormal{H}_i}}$ contains the centre of $\mathcal{H}(\textnormal{H}_i,I_{\textnormal{H}_i})$ for $i=1,...,k$. This follows from Bernstein's presentation of the centre of the Iwahori Hecke algebra noting that $t_{\textnormal{B}_{\textnormal{H}_i}}$ coincides with Bernstein's map $\Theta\colon\overline{\mathbf{Q}}_p[X_{\ast}(\textnormal{T}_{\textnormal{H}_i})]\hookrightarrow \mathcal{H}(\textnormal{H}_i,I_{\textnormal{H}_i})$.
\end{proof}

\begin{example}
Let $\textnormal{B}=\textnormal{T}\textnormal{U}\subset \textnormal{G}$ be the Borel subgroup of upper triangular matrices with its Levi decomposition.
For the convenience of the reader, we spell out the case when $\Omega$ is the principal block $[\textnormal{T},\mathbf{1}]_{\textnormal{G}}$. In this case, $J$ is the Iwahori subgroup $I:=\textnormal{red}^{-1}(B(\mathcal{O}_L/\varpi_L))\leq \textnormal{G}$ and $\kappa$ is the trivial representation. We have $\widetilde{\textnormal{M}}=\textnormal{G}$ and $\textnormal{M}=\textnormal{T}$. Moreover, $\mathcal{H}(\kappa^{\vee})$ is the usual Iwahori Hecke algebra $\mathcal{H}(G,I)=\overline{\mathbf{Q}}_p[I\backslash \textnormal{G}/I]$. 
    
    For $1\leq i\leq n$, let $\mu_i=(1,...,1,0,...0)$ be the minuscule cocharacter with $1$ occurring exactly $i$ times. 
    The map sending $\mu_i$ to $\delta_{\textnormal{B}}^{1/2}(\mu_i(\varpi_L))[I\mu_i(\varpi_L)I]=q_L^{\frac{-i(n-i)}{2}}[I\mu_i(\varpi_L)I]$ can be promoted to an algebra homomorphism \[
    \overline{\mathbf{Q}}_p[X_{\ast}(\textnormal{T})^+]\to  \mathcal{H}(\textnormal{G},I),
    \]
    sending a dominant $\mu$ to $\delta_{\textnormal{B}}^{1/2}(\mu(\varpi_L))[I\mu(\varpi_L)I]$. 
    In other words, for dominant cocharacters $\mu_1,\mu_2$, $[I\mu_1(\varpi_L)I]\ast [I\mu_2(\varpi_L)I]=[I\mu_1\mu_2(\varpi_L)I]$. Finally, a well known fact is that for dominant $\mu$, $[I\mu(\varpi_L)I]$ is invertible in the Iwahori Hecke algebra. We then obtain the algebra homomorphism $t_{\textnormal{B}}\colon\overline{\mathbf{Q}}_p[X_{\ast}(\textnormal{T})]=\mathcal{H}(\textnormal{T},\textnormal{T}^0)\to \mathcal{H}(\textnormal{G},I)$ by sending $\mu^{-1}$ for $\mu$ dominant to $t_{\textnormal{B}}(\mu)^{-1}$. We observe that this coincides with the algebra homomorphism $\Theta$ considered by Bernstein.

    Finally, in the definition of $\mathcal{H}^{B,+}(\sigma)$, $k=1$, and $Y_{1,j}\colon\textnormal{G}\to \textnormal{End}(\sigma)$ is the function associated with $\delta_{\textnormal{B}}^{1/2}(\mu_j(\varpi_L))[I\mu_j(\varpi_L)I]$ that, as observed, is supported on $K\mu_j(\varpi_L)K$.

\end{example}

\subsection{Locally algebraic representations with characteristic \texorpdfstring{$0$}{0} coefficients}\label{Subsection_locallyalgreps} Let $\textnormal{T}_n\textnormal{U}_n=\textnormal{B}_n\subset\textnormal{GL}_{n,L}$ be the Borel subgroup of upper triangular matrices with its standard Levi decomposition. Let $E/\mathbf{Q}_p$ be a finite field extension so that $\#\Hom(L,E)=[L:\mathbf{Q}_p]$. Let $\mathcal{O}\subset E$ denote its ring of integers and $\varpi\in \mathcal{O}$ a choice of uniformiser. 
We set $W_{n,L}:=W(\textnormal{Res}_{L/\mathbf{Q}_p}\textnormal{GL}_n,\textnormal{Res}_{L/\mathbf{Q}_p}\textnormal{B}_n)\cong W_n^{[L:\mathbf{Q}_p]}$. 

For a standard parabolic subgroup $P=MN\subset \textnormal{GL}_{n,L}$, we set $W_{M,L}:=W(\textnormal{Res}_{L/\mathbf{Q}_p}M,\textnormal{Res}_{L/\mathbf{Q}_p}(M\cap \textnormal{B}_n))\cong W_M^{[L:\mathbf{Q}_p]}$ and $W^M_L$ to be the set of minimal length representatives of $W_{L,n}/W_{M,L}$.

For a $\textnormal{Res}_{L/\mathbf{Q}_p}\textnormal{B}_n$-dominant weight
\begin{equation*}
    \lambda=(\lambda_{\iota})_{\iota:L\hookrightarrow E}\in X^{\ast}(\textnormal{Res}_{L/\mathbf{Q}_p}\textnormal{T}_n)_+=(\mathbf{Z}_+^n)^{\Hom(L,E)},
\end{equation*} we set $V_{\lambda}=\otimes_{\iota:L\hookrightarrow E}V_{\lambda_{\iota}}$ to be the absolutely irreducible algebraic $E$-representation of $\textnormal{G}$ of highest weight $\lambda$. 

For this subsection, fix $\sigma\in \textnormal{Rep}_{\textnormal{sm}}(K,E)$, finite dimensional over $E$. 
Set $\sigma_{\lambda}:=\sigma\otimes_EV_{\lambda}$, a locally algebraic $E$-representation of $K$. 
Moreover, we set $\mathcal{H}(\sigma_{\lambda})$ to be the convolution algebra of compactly supported functions $f\colon\textnormal{G}\to \textnormal{End}_{E}(\sigma_{\lambda})$ satisfying $f(k_1gk_2)=\sigma_{\lambda}(k_1)\circ f(g)\circ \sigma_{\lambda}(k_2)$ for every $k_1,k_2\in K$, $g\in \textnormal{G}$. 
Using the Cartan decomposition for $\textnormal{G}$, an $E$-basis of $\mathcal{H}(\sigma_{\lambda})$ is given by functions $[Kt_{\mu}K,\psi]$ supported on $Kt_{\mu}K$, sending $t_{\mu}$ to $\psi$, where
$t_{\mu}=\mu(\varpi_L)$ for some $\mu\in X_{\ast}(\textnormal{T}_n)^+$ and $\psi$ runs over an $E$-basis of $\Hom_{K\cap t_{\mu}^{-1}Kt_{\mu}}(\sigma_{\lambda},t_{\mu}^{\ast}\sigma_{\lambda})$.

We note that acting on $\textnormal{c-Ind}_K^{\textnormal{G}}\sigma_{\lambda}$ 
via convolution induces isomorphisms $\mathcal{H} (\sigma_{\lambda})\cong \textnormal{End}_{\textnormal{G}}(\textnormal{c-Ind}_K^{\textnormal{G}}\sigma_{\lambda})$, see for instance \cite[Lemma 1.2]{ST07}. We also note that there is an algebra isomorphism
\begin{align*}
    \mathcal{H}(\sigma) &\xrightarrow{\sim}\mathcal{H}(\sigma_{\lambda}) \\
    [Kt_{\mu}K,\phi] &\mapsto [Kt_\mu K,\phi\otimes t_{\mu}],
\end{align*}
see \cite[Lemma 1.4]{ST07}. 

We finally consider a standard parabolic subgroup $P=MN\subset \textnormal{GL}_{n,L}$ and write $\textnormal{P}=P(L)$, $\textnormal{M}=M(L)$, $\textnormal{N}=N(L)$. Let $\mathfrak{n}$ denote the Lie algebra of $\textnormal{N}$. We recall the following theorem of Lazard (cf. \cite[Ch. V, (2.4.10), Th\'eor\`eme]{Laz65}).

\begin{theorem}[Lazard]\label{Theorem_Lazard}
     Given a finite dimensional continuous $E$-representation $V$ of $K_{\textnormal{P}}$, we have a functorial isomorphism
    \begin{equation*}
\varinjlim_{K_{\textnormal{N}}'\trianglelefteq K_{\textnormal{N}}}H^{\ast}_{\textnormal{cont}}(K_{\textnormal{N}}',V) \xrightarrow{\sim}H^{\ast}(\mathfrak{n},V)
    \end{equation*}
     of $K_{\textnormal{P}}$-representations where the colimit runs over open normal subgroups and the transition maps are the ones given by restriction on group cohomology.

     Moreover, pre-composition with the canonical map 
     \begin{equation*}
         H^{\ast}_{\textnormal{cont}}(K_{\textnormal{N}}',V)\to\varinjlim_{K_{\textnormal{N}}'\trianglelefteq K_{\textnormal{N}}}H^{\ast}_{\textnormal{cont}}(K_{\textnormal{N}}',V)
     \end{equation*} induces an isomorphism
     \begin{equation*}
         \alpha_{K_{\textnormal{N}}'}:H^{\ast}_{\textnormal{cont}}(K_{\textnormal{N}}',V)\xrightarrow{\sim}\Gamma(K_{\textnormal{N}}',H^{\ast}(\mathfrak{n},V)).
     \end{equation*}
\end{theorem}
Combining Lazard's theorem with Kostant's formula (cf. \cite{Kos61}) yields the following result.
\begin{theorem}\label{Theorem_Kostant}
    We have a collection of canonical isomorphisms
    \begin{equation*}
        H^{\ast}(K_{\textnormal{N}}',V_{\lambda})\xrightarrow{\sim} \oplus_{w\in W^{\textnormal{M}}}V_{w\cdot \lambda}
    \end{equation*}
     for compact open normal subgroups $K_{\textnormal{N}}'\trianglelefteq K_{\textnormal{N}}$ that are compatible under restriction in group cohomology. In particular, each of these restriction maps are isomorphisms.
\end{theorem}

We will need the following simple corollary of Lazard's comparison isomorphisms and Kostant's formula.
\begin{lemma}\label{Lemma_cupproductiso}
     For every compact open normal subgroup $K'_{\textnormal{N}}\trianglelefteq K_{\textnormal{N}}$, the cup product on group cohomology induces an isomorphism
     \begin{equation}\label{equation_cupproductiso}
         \cup\colon H^0(K'_{\textnormal{N}},\sigma)\otimes_E H^{\ast}(K'_{\textnormal{N}},V_{\lambda})\xrightarrow{\sim} H^{\ast}(K'_{\textnormal{N}},\sigma_{\lambda}).
     \end{equation}
     Moreover, these isomorphisms fit into commutative diagrams
     \begin{equation*}
         \begin{tikzcd}
	{H^{0}(K_{\textnormal{N}}',\sigma)\otimes_EH^{\ast}(K_{\textnormal{N}}',V_{\lambda})} & {H^{\ast}(K_{\textnormal{N}}',\sigma_{\lambda})} \\
	{H^{0}(K_{\textnormal{N}},\sigma)\otimes_EH^{\ast}(K_{\textnormal{N}},V_{\lambda})} & {H^{\ast}(K_{\textnormal{N}},\sigma_{\lambda}).}
	\arrow["\cup", from=1-1, to=1-2]
	\arrow["{\textnormal{tr}_{K_{\textnormal{N}}/K_{\textnormal{N}}'}\otimes \textnormal{res}^{-1}}", from=1-1, to=2-1]
	\arrow["{\textnormal{cores}}", from=1-2, to=2-2]
	\arrow["\cup", from=2-1, to=2-2]
\end{tikzcd}
     \end{equation*}
\end{lemma}
\begin{proof}
To prove the first claim, note that by Theorem~\ref{Theorem_Lazard}, we have $H^{\ast}(K_{\textnormal{N}}',V)=H^0(K_{\textnormal{N}}'/K''_{\textnormal{N}},H^{\ast}(K_{\textnormal{N}}'',V))$ for any compact open normal subgroup $K''_{\textnormal{N}}\trianglelefteq K'_{\textnormal{N}}$ and $V=\sigma, V_{\lambda},\sigma_{\lambda}$. In particular, as restriction commutes with cup product, we can assume that $K_{\textnormal{N}}'$ acts trivially on $\sigma$. In this case, the claim follows from the fact that group cohomology commutes with direct sums and the canonical comparing isomorphism can be identified with the cup product.

To see the second claim, note that for $v\otimes w\in H^{0}(K_{\textnormal{N}}',\sigma)\otimes_EH^{\ast}(K_{\textnormal{N}}',V_{\lambda})$ we have
\begin{align*}
    \textnormal{cores}(v\cup w) &=\textnormal{cores}(v\cup \textnormal{res}(\textnormal{res}^{-1}w)) \\
    &= \textnormal{cores}(v)\cup\textnormal{res}^{-1}(w) \\
    &=\textnormal{tr}_{K_{\textnormal{N}}/K_{\textnormal{N}}'}(v)\cup \textnormal{res}^{-1}(w)
\end{align*}
where we used Theorem~\ref{Theorem_Kostant} to see that $\textnormal{res}:H^{\ast}(K_{\textnormal{N}},V_{\lambda})\to H^{\ast}(K_{\textnormal{N}'},V_{\lambda})$ is invertible.
\end{proof}
Define $\sigma_{\lambda, \textnormal{P}}:=H^{\ast}(K_{\textnormal{N}},\sigma_{\lambda})\in \textnormal{Rep}_{\textnormal{sm}}(K_{\textnormal{M}},E)$. By the previous discussion, there is a canonical isomorphism
\begin{equation*}
\sigma_{ \lambda, \textnormal{P}}\cong\bigoplus_{i\in\mathbf{Z}}\left(\bigoplus_{w\in W^{\textnormal{M}}, \textnormal{ }l(w)=i}\sigma_{\textnormal{P},w\cdot\lambda}\right)
\end{equation*}
where, following the previous notation, $\sigma_{\textnormal{P},w\cdot \lambda}=\sigma^{K_{\textnormal{N}}}\otimes_E V_{w\cdot \lambda}$ is a locally algebraic $E$-representation of $K_{\textnormal{M}}$. We have an induced isomorphism
\begin{equation*}
    \mathcal{H}(\sigma_{\lambda,\textnormal{P}})\cong \prod_{w\in W^{\textnormal{M}}}\mathcal{H}(\sigma_{\textnormal{P},w\cdot \lambda})
\end{equation*}
(noting that the weights $(w \cdot \lambda)_{w \in W}$ are distinct, as the stabilizer for the dotted action must fix, under the usual action, $\lambda + \rho$ which lies inside an open Weyl chamber).

We can now define a ``locally algebraic Satake transform"
\begin{equation*}
    \mathcal{S}_{\textnormal{P},\lambda} \colon \mathcal{H}(\sigma_{\lambda})\cong \mathcal{H}(\sigma)\xrightarrow{\mathcal{S}_{\textnormal{P}}}\mathcal{H}(\sigma_{\textnormal{P}})\xrightarrow{\Delta}\prod_{w\in W^{\textnormal{M}}}\mathcal{H}(\sigma_{\textnormal{P}})\cong\prod_{w\in W^{\textnormal{M}}}\mathcal{H}(\sigma_{\textnormal{P},w\cdot \lambda})\cong \mathcal{H}(\sigma_{\lambda,\textnormal{P}}),
\end{equation*}
where $\Delta$ denotes the diagonal embedding.

\begin{remark}\label{Remark_MinusculeSatakeTransform}
    Let us spell out the effect of $\mathcal{S}_{\textnormal{P},\lambda}$ on basis elements $[Kt_{\mu}K,\phi\otimes t_{\mu}]$ for \textit{minuscule} cocharacters $\mu\in X_{\ast}(\textnormal{T}_n)^+$ and arbitrary intertwining maps $\phi\in \Hom_{t_{\mu}^{-1}Kt_{\mu}\cap K}(\sigma,t_{\mu}^{\ast}\sigma)$.

    To do this, we just need to compute $\mathcal{S}_{\textnormal{P}}([Kt_{\mu}K,\phi])$. Recall that $\mathcal{S}_{\textnormal{P}}$ is the composition of the algebra homomorphisms $r_{\textnormal{P}}\colon\mathcal{H}(\sigma)\to \mathcal{H}(\sigma\mid_{K_{\textnormal{P}}})$ and $r_{\textnormal{M}}\colon\mathcal{H}(\sigma\mid_{K_{\textnormal{P}}})\to \mathcal{H}(\sigma_{\textnormal{P}})$, restricting a function $f\colon\textnormal{G}\to \textnormal{End}(\sigma)$ to $\textnormal{P}$, and integrating along the unipotent radical $\textnormal{N}$, respectively.
We first unravel $r_{\textnormal{P}}([Kt_{\mu}K,\phi])$. We compute double cosets
\begin{align*}
    Kt_{\mu}K &= K/(t_{\mu}Kt_{\mu}^{-1}\cap K)t_{\mu}K \\
    &= (K/\mathcal{P}_{\mu})t_{\mu}K \\
    & =\coprod_{w\in ^{\textnormal{M}}W^{\textnormal{M}_{\mu}}}K_{\textnormal{P}}wt_{\mu}K \\
    &=\coprod_{w\in ^{\textnormal{M}}W^{\textnormal{M}_{\mu}}}K_{\textnormal{P}}t_{\mu}^{w}K,
\end{align*}
where $\textnormal{P}_{\mu}=\textnormal{M}_{\mu}\textnormal{N}_{\mu}\subset \textnormal{GL}_n$ is the \textit{opposite} parabolic subgroup associated with $\mu$, $\mathcal{P}_{\mu}:=\textnormal{red}^{-1}(\textnormal{P}_{\mu}(\mathcal{O}_L/\varpi_L))$ the corresponding parahoric subgroup, $^{\textnormal{M}}W^{\textnormal{M}_{\mu}}\subset W:=W(\textnormal{GL}_{n,L},\textnormal{B}_n)$ is the set of minimal length representatives for $W_{\textnormal{M}}\backslash W/W_{\textnormal{M}_{\mu}}$, and $t_{\mu}^w:=wt_{\mu}w^{-1}$. In particular, we obtain
\begin{equation*}
    Kt_{\mu}K\cap \textnormal{P}=\coprod_{w\in {}^{\textnormal{M}}W^{\textnormal{M}_{\mu}}}K_{\textnormal{P}}t_{\mu}^wK_{\textnormal{P}}.
\end{equation*}
Consequently, we obtain that
\begin{equation*}
    r_{\textnormal{P}}([Kt_{\mu}K,\phi])=\sum_{w\in {}^{\textnormal{M}}W^{\textnormal{M}_{\mu}}}[K_{\textnormal{P}}t_{\mu}^{w}K_{\textnormal{P}},\phi^{w}]
\end{equation*}
where $\phi^{w}:=\sigma(w)\circ \phi\circ \sigma(w^{-1})$.

To compute $\mathcal{S}_{\textnormal{P}}([Kt_{\mu}K,\phi])$, we are left with computing $r_{\textnormal{M}}([K_{\textnormal{P}}t_{\mu}^{w}K_{\textnormal{P}},\phi^w])$ for $w\in {}^{\textnormal{M}}W^{\textnormal{M}_{\mu}}$. To give the formula, we define
\begin{align*}
    \phi_{\textnormal{P}}^w \colon \sigma_{\textnormal{P}} &\to \sigma_{\textnormal{P}} \\
    v &\mapsto \sum_{n\in K_{\textnormal{N}}/t_{\mu}^wK_{\textnormal{N}}t_{\mu}^{w,-1}\cap K_{\textnormal{N}}}(\sigma(n)\circ \phi^w)(v).
\end{align*}
Then the formula is as follows
\begin{equation*}
    r_{\textnormal{M}}([K_{\textnormal{P}}t_{\mu}^wK_{\textnormal{P}},\phi^w])=[K_{\textnormal{M}}t_{\mu}^wK_{\textnormal{M}},\phi^w_{\textnormal{P}}].
\end{equation*}

Finally we obtain the formula
\begin{equation*}
    \mathcal{S}_{\textnormal{P},\lambda}([Kt_{\mu}K,\phi\otimes t_{\mu}])=\left(\sum_{w\in ^{\textnormal{M}}W^{\textnormal{M}_{\mu}}}[K_{\textnormal{M}}t^{w}_{\mu}K_{\textnormal{M}},\phi^w_{\textnormal{P}}\otimes t_{\mu}^w]\right)_{\widetilde{w}\in W^{\textnormal{M}}}\in \prod_{\widetilde{w}\in W^{\textnormal{M}}}\mathcal{H}(\sigma_{\textnormal{P},\widetilde{w}\cdot\lambda}).
\end{equation*}
\end{remark}

\subsection{Integral structures}\label{Subsection_integralstructures}
We now combine the setup of the previous subsections. Let $E/\mathbf{Q}_p$ be a finite field extension with ring of integers $\mathcal{O}$ satisfying $\#\Hom(L,E)=[L:\mathbf{Q}_p]$. Let $\lambda\in (\mathbf{Z}_+^n)^{\Hom(L,E)}$ be a highest weight vector and set $\mathcal{V}_{\lambda}\subset V_{\lambda}$ be the usual $K$-stable $\mathcal{O}$-lattice given by the dual Weyl module.

Fix a Bernstein block $\Omega=[\widetilde{\textnormal{M}},\widetilde{\pi}]_{\textnormal{G}}$ with its standard presentation and the associated Bushnell--Kutzko type $(J,\kappa)$.
We introduced the associated $\overline{\mathbf{Q}}_p$-representations $\sigma_{\Omega}$ of $K$, and, for standard parabolic subgroups $\textnormal{P}=\textnormal{M}\textnormal{N}\leq \textnormal{G}$, $\sigma_{\textnormal{P},\Omega}$, $\sigma_{\textnormal{M},\Omega}$ of $K_{\textnormal{M}}$. 

Write $\widetilde{\textnormal{M}}=\textnormal{GL}_{n_1}(L)\times...\times\textnormal{GL}_{n_h}(L)$ and $\widetilde{\pi}=\pi_1\otimes...\otimes \pi_h$. For a subset $S\subset \{1,...,h\}$ equipped with an \textit{ordering}, set $n_S:=\sum_{i\in S}n_i$, $\widetilde{\textnormal{M}}_S:=\prod_{i\in S}\textnormal{GL}_{n_i}(L)\leq G_{\textnormal{S}}:=\textnormal{GL}_{n_S}(L)$, $\widetilde{\pi}_S:=\otimes_{i\in S}\pi_i$, supercuspidal $\overline{\mathbf{Q}}_p$-representations of $\widetilde{\textnormal{M}}_S$. Set $\Omega_{S}:=[\widetilde{\textnormal{M}}_S,\widetilde{\pi}_S]_{\textnormal{G}_S}$ and note that $\Omega_S$ does not depend on the ordering we choose on $S$.

For our analysis of the boundary cohomology of $\textnormal{GL}_n$-locally symmetric spaces in \S\ref{sec_boundarycoh}, we will have to make several auxiliary choices on which our results will depend. Namely, given $\lambda$ and $\Omega$, we will need to fix a choice of $\mathcal{O}$-lattice in $\sigma_{\Omega}\otimes V_{\lambda}$ and similarly for locally algebraic representations induced by $\lambda$ and $\Omega$ for Levi subgroups of $\textnormal{G}$. Moreover, for our inductive argument, we need to fix comparisons between $H^{\ast}(K_{\textnormal{N}},\sigma_{\Omega}\otimes V_{\lambda})$ and $\oplus_{w\in W^{\textnormal{M}},\Omega_{\textnormal{M}}}\sigma_{\Omega_{\textnormal{M}}}\otimes V_{w\cdot \lambda}$ for standard parabolic subgroups $\textnormal{P}=\textnormal{M}\textnormal{N}\leq \textnormal{G}$.

For the algebraic part, there are canonical choices we can make. Namely, the dual Weyl module for choice of lattices and the combination of Lazard's and Kostant's isomorphisms for the comparison maps. However, for the smooth part, we have no canonical choices at our disposal. The purpose of the following definition is to spell out the choices that we need make.

\begin{definition}\label{Definition_integralsystem}
An $\mathcal{O}$-integral $\Omega$-system is a tuple
\begin{equation*}
    \mathcal{L}:=\left((\sigma_{\Omega_{S}}^{\circ})_{\Omega_S},(\varphi_{\textnormal{P}_S,\textnormal{M}_S})_{(\Omega_S,\textnormal{M}_S)}\right)
\end{equation*}
containing the following data.
\begin{itemize}
    \item For every Bernstein block of the form $\Omega_S$ for some subset $S\subset \{1,...,h\}$, a choice of a $K_{\textnormal{S}}:=\textnormal{GL}_{n_S}(\mathcal{O}_L)$-stable $\mathcal{O}$-lattice $\sigma_{\Omega_S}^{\circ}\subset \sigma_{\Omega_S}^{\vee}$.
    \item For every pair $(\Omega_S,\textnormal{M}_S)$ of a Bernstein block of the form $\Omega_S=[\widetilde{\textnormal{M}}_{S},\widetilde{\pi}_S]_{\textnormal{G}_S}$ for some subset $S\subset \{1,...,h\}$ equipped with an ordering and a standard parabolic subgroup $\widetilde{\textnormal{P}}_S\leq  \textnormal{P}_{S}=\textnormal{M}_S\textnormal{N}_S\leq \textnormal{G}_S$, a choice of isomorphism
    \begin{align*}
        \varphi_{\textnormal{P}_S,\textnormal{M}_S}\colon\textnormal{c-Ind}_{K_{\textnormal{M}_S}}^{\textnormal{M}_S}\Gamma(K_{\textnormal{N}_{S}},\sigma_{\Omega_{S}}^{\circ}\otimes_{\mathcal{O}}E)^{\vee}\xrightarrow{\sim} \\
        \textnormal{c-Ind}_{K_{\textnormal{M}_S}}^{\textnormal{M}_S}\left(\bigoplus_{\Omega_{\textnormal{M}_S}\in \mathfrak{B}(\textnormal{M}_S,\Omega_S)}(\sigma_{\Omega_{\textnormal{M}_S}}^{\circ}\otimes_{\mathcal{O}}E)^{\vee} \right)
    \end{align*}
    where the sum runs over the set $\mathfrak{B}(\textnormal{M}_S,\Omega_S)$ from Proposition \ref{Proposition_BushnellKutzkoTypesIso}.
\end{itemize}
\end{definition}
To ease notation, for a Bernstein block of the form $\Omega_S=[\widetilde{\textnormal{M}}_{S},\widetilde{\pi}_S]_{\textnormal{G}_S}$ for some subset $S\subset \{1,...,h\}$ equipped with an ordering and a standard parabolic subgroup $\widetilde{\textnormal{P}}_S\leq  \textnormal{P}_{S}=\textnormal{M}_S\textnormal{N}_S\leq \textnormal{G}_S$, we write
\begin{align*}
\sigma_{\textnormal{P}_S,\Omega_S}^{\circ} &:=\Gamma(K_{\textnormal{N}_S},\sigma_{\Omega_S}^{\circ}), \\
\sigma_{\textnormal{M}_S,\Omega_{S}}^{\circ} &:=\bigoplus_{\Omega_{\textnormal{M}_S}\in \mathfrak{B}(\textnormal{M}_S,\Omega_S)}\sigma_{\Omega_{\textnormal{M}_S}}^{\circ}.
\end{align*}
\begin{lemma}\label{Lemma_changeofweightIII}
    Given an $\mathcal{O}$-integral $\Omega$-system, there is an integer $N\geq 0$ such that, for every $S$ and $\textnormal{M}_S$, the maps $p^N\varphi_{\textnormal{P}_S,\textnormal{M}_S}$, and $p^N\varphi_{\textnormal{P}_S,\textnormal{M}_S}^{-1}$ restrict to morphisms
    \begin{equation*}
        \textnormal{c-Ind}_{K_{\textnormal{M}_S}}^{\textnormal{M}_S}(\sigma_{\textnormal{P}_S,\Omega_S}^{\circ})^{\vee}\to \textnormal{c-Ind}_{K_{\textnormal{M}_S}}^{\textnormal{M}_S}(\sigma_{\textnormal{M}_S,\Omega_S}^{\circ})^{\vee},
    \end{equation*}
    and
    \begin{equation*}
        \textnormal{c-Ind}_{K_{\textnormal{M}_S}}^{\textnormal{M}_S}(\sigma_{\textnormal{M}_S,\Omega_S}^{\circ})^{\vee} \to \textnormal{c-Ind}_{K_{\textnormal{M}_S}}^{\textnormal{M}_S}(\sigma_{\textnormal{P}_S,\Omega_S}^{\circ})^{\vee} ,
    \end{equation*}
    respectively.
\end{lemma}
\begin{proof}
    Note that there are only finitely many maps involved. In particular, the lemma follows from the fact that for a smooth $\mathcal{O}$-representation $\sigma^{\circ}$ of a compact open subgroup $U$ of a locally profinite group $\textnormal{H}$ with finite free underlying $\mathcal{O}$-module, $(\textnormal{c-Ind}_{U}^{\textnormal{H}}\sigma^{\circ})[1/p]=\textnormal{c-Ind}_{U}^{\textnormal{H}}(\sigma^{\circ}[1/p])$.
\end{proof}
\begin{remark}
    Given an $\mathcal{O}$-integral $\Omega$-system $\mathcal{L}$ and a subset $S\subset \{1,...,h\}$, we obtain an $\mathcal{O}$-integral $\Omega_{S}$-system $\mathcal{L}_S$ by remembering only the lattices and isomorphisms for $\Omega_{S'}$ with $S'\subset S$.
\end{remark}
\begin{prop}
    Given a Bernstein block $\Omega$ of $\textnormal{Rep}_{\textnormal{sm}}(\textnormal{G},\overline{\mathbf{Q}}_p)$, after possibly enlarging $E$, there exists an $\mathcal{O}$-integral $\Omega$-system.
\end{prop}
\begin{proof}
    The existence of $K_S$-stable $\mathcal{O}$-lattices is clear. Proposition~\ref{Proposition_BushnellKutzkoTypesIso} implies the existence of the required isomorphisms after scalar extension to $\overline{\mathbf{Q}}_p$. As these are isomorphisms over $\overline{\mathbf{Q}}_p$ between finitely generated $E$-representations, they are already defined over some finite extension of $\mathbf{Q}_p$ (see for instance \cite[Lemma 5.1]{Pas13}).
\end{proof}
Given a highest weight vector $\lambda$, a Bernstein block $\Omega$ and an $\mathcal{O}$-integral $\Omega$-system $\mathcal{L}$, we set
\begin{align*}
\sigma_{\Omega,\lambda}^{\circ} &:=\sigma_{\Omega}^{\circ}\otimes_{\mathcal{O}}\mathcal{V}_{\lambda}\subset \sigma_{\Omega}^{\vee}\otimes_EV_{\lambda}, \\
\sigma_{\Omega^{\vee},\lambda^{\vee}}^{\circ} &:=\sigma_{\Omega}^{\circ,\vee}\otimes_{\mathcal{O}}\mathcal{V}_{-w_0^{\textnormal{G}}\lambda}\subset \sigma_{\Omega}\otimes_EV_{\lambda}^{\vee},\\
\sigma_{\textnormal{P},\Omega,\lambda}^{\circ} &:=\Gamma(K_{\textnormal{N}},\sigma_{\Omega}^{\circ})\otimes_{\mathcal{O}}\left(\bigoplus_{w\in W^{\textnormal{M}}}\mathcal{V}_{w\cdot \lambda}\right) \\
&\subset \sigma_{\textnormal{P},\Omega}^{\vee}\otimes_E(\bigoplus_{w\in W^{\textnormal{M}}}V_{w\cdot \lambda}) \cong
     H^{\ast}(K_{\textnormal{N}},\sigma_{\Omega}^{\vee}\otimes_EV_{\lambda}),\\
\sigma_{\textnormal{M},\Omega,\lambda}^{\circ} &:=\left(\bigoplus_{\Omega_{\textnormal{M}}\in \mathfrak{B}(\textnormal{M},\Omega)}\sigma_{\Omega_{\textnormal{M}}}^{\circ}\right)\otimes_{\mathcal{O}}\left(\bigoplus_{w\in W^{\textnormal{M}}}\mathcal{V}_{w\cdot \lambda}\right) \\
&\subset \sigma_{\textnormal{M},\Omega}^{\vee}\otimes_E(\bigoplus_{w\in W^{\textnormal{M}}}V_{w\cdot\lambda})\cong
    \sigma_{\textnormal{M},\Omega}^{\vee}\otimes_EH^{\ast}(K_{\textnormal{N}},V_{\lambda}),
\end{align*}
where in the last two definitions we run over standard parabolic subgroups $\textnormal{P}=\textnormal{M}\textnormal{N}\leq \textnormal{G}$ and the indicated isomorphisms are induced by Theorem~\ref{Theorem_Kostant} and Lemma~\ref{Lemma_cupproductiso}. 

Set $\sigma_{\Omega,\lambda}:=\sigma_{\Omega,\lambda}^{\circ}[1/p]$, $\sigma_{\textnormal{P},\Omega,\lambda}:=\sigma_{\textnormal{P},\Omega,\lambda}^{\circ}[1/p]$, $\sigma_{\textnormal{M},\Omega,\lambda}:=\sigma_{\textnormal{M},\Omega,\lambda}^{\circ}[1/p]$. Note that there is an isomorphism $\phi_{\Omega,\lambda}\colon\sigma_{\Omega,\lambda}^{\vee}\xrightarrow{\sim}\sigma_{\Omega^{\vee},\lambda^{\vee}}^{\circ}[1/p]$ unique up to scalar. 

We have the locally algebraic Satake transforms
\begin{equation*}
    \mathcal{S}_{\textnormal{P},\Omega,\lambda}\colon\mathcal{H}(\sigma_{\Omega,\lambda})\to \mathcal{H}(\sigma_{\textnormal{P},\Omega,\lambda})
\end{equation*}
and the isomorphisms
\begin{equation*}
    r_{\mathcal{L},\textnormal{P},\textnormal{M}}\colon\mathcal{H}(\sigma_{\textnormal{P},\Omega,\lambda})\xrightarrow{\sim}\mathcal{H}(\sigma_{\textnormal{M},\Omega,\lambda})
\end{equation*}
induced by the isomorphisms $\varphi_{\textnormal{P},\textnormal{M}}$ of $\mathcal{L}$ and the anti-isomorphisms $\mathcal{H}(\sigma)\xrightarrow{\sim}\mathcal{H}(\sigma^{\vee})$. Moreover, induced by $\phi_{\Omega,\lambda}$, we have an anti-isomorphism
\begin{align*}
    \iota\colon\mathcal{H}(\sigma_{\Omega,\lambda}^{\circ}[1/p]) &\xrightarrow{\sim}\mathcal{H}(\sigma_{\Omega^{\vee},\lambda^{\vee}}^{\circ}[1/p]) \\
    [KgK,\psi] &\mapsto [Kg^{-1}K,\psi^t].
\end{align*}
This isomorphism is independent of the choice of $\phi_{\Omega,\lambda}$, making it canonical.

Recall the $\overline{\mathbf{Q}}_p$-subalgebras $\mathcal{H}^{B,+}(\sigma_{\Omega})\leq\mathcal{H}^B(\sigma_{\Omega})\leq \mathcal{H}(\sigma_{\Omega}^{\vee})$ from \eqref{equation_BernsteinSubalgPlus}, \eqref{equation_BernsteinSubalg}. After possibly enlarging $E$, each of the elements $Y_{i,j}^{\pm 1}$ are defined over $E$ and so, under the canonical isomorphism
\begin{equation*}
    \mathcal{H}((\sigma_{\Omega}^{\circ}\otimes_{\mathcal{O}}E)^{\vee})\xrightarrow{\sim}\mathcal{H}(\sigma_{\Omega,\lambda}),
\end{equation*}
we obtain $E$-subalgebras
\begin{equation*}
    \mathcal{H}^{B,+}(\sigma_{\Omega,\lambda})\cong E[\{Y_{i,j}\}_{i,j}]\leq \mathcal{H}^{B}(\sigma_{\Omega,\lambda})\cong E[\{Y_{i,j}^{\pm1}\}_{i,j}]\leq\mathcal{H}(\sigma_{\Omega,\lambda})
\end{equation*}
where in formulas we have $Y_{i,j}=[K t_{i,j}K,\psi_{i,j}\otimes t_{i,j}]$ for $t_{i,j}=\nu_{i,j}(\varpi_{L})$ for a minuscule cocharacter $\nu_{i,j}\in X_{\ast}(\textnormal{T})^+$.

The choice of $\mathcal{L}$ induces $\mathcal{O}$-models of the algebras $\mathcal{H}^{B,+}(\sigma_{\Omega,\lambda}),$ $\mathcal{H}^{B}(\sigma_{\Omega,\lambda})$. 
Namely, for $1\leq i\leq k, 1\leq j\leq n_i/d_i$, let $N_{i,j}$ be the smallest integer such that the following are satisfied.
\begin{itemize}
    \item We have $p^{N_{i,j}}Y_{i,j}\in \mathcal{H}(\sigma_{\Omega,\lambda}^{\circ})\cap \iota^{-1}(\mathcal{H}(\sigma_{\Omega^{\vee},\lambda^{\vee}}^{\circ}))$.
    \item For every standard parabolic subgroup $\textnormal{P}=\textnormal{M}\textnormal{N}\leq \textnormal{G}$ and element $w\in {}^{\textnormal{M}}W^{\textnormal{M}_{\nu_{i,j}}}$, we have
    \begin{equation*}
        p^{N_{i,j}}(\psi_{i,j})_{\textnormal{P}}^w\otimes t_{i,j}^w\in \textnormal{End}_{\mathcal{O}}(\sigma_{\textnormal{P},\Omega,\lambda}^{\circ}).
    \end{equation*}
    In particular, $\mathcal{S}_{\textnormal{P},\Omega,\lambda}(p^{N_{i,j}}Y_{i,j})\in \mathcal{H}(\sigma_{\textnormal{P},\Omega,\lambda}^{\circ})$ for every standard parabolic subgroup $\textnormal{P}=\textnormal{M}\textnormal{N}\leq \textnormal{G}$ (cf. Remark~\ref{Remark_MinusculeSatakeTransform}).
    \item We have $r_{\mathcal{L},\textnormal{P},\textnormal{M}}\circ\mathcal{S}_{\textnormal{P},\Omega,\lambda}(p^{N_{i,j}}Y_{i,j})\in \mathcal{H}(\sigma_{\textnormal{M},\Omega,\lambda}^{\circ})$ for every standard parabolic subgroup $\textnormal{P}=\textnormal{M}\textnormal{N}\leq \textnormal{G}$.
\end{itemize}

We set 
\begin{equation*}
    \mathcal{H}^{B,+}(\sigma_{\Omega,\lambda})^{\circ}\leq \mathcal{H}(\sigma_{\Omega,\lambda})
\end{equation*}
to be the $\mathcal{O}$-subalgebra generated by the set $\{V_{i,j}:=p^{N_{i,j}}Y_{i,j}\mid 1\leq i\leq k,1\leq j\leq n_i/d_i\}$. Note that we have a canonical isomorphism $\mathcal{H}^{B,+}(\sigma_{\Omega,\lambda})^{\circ}\otimes_{\mathcal{O}}\overline{\mathbf{Q}}_p\cong\mathcal{H}^{B,+}(\sigma_{\Omega})$.

Given a standard parabolic subgroup $\textnormal{P}=\textnormal{M}\textnormal{N}\leq \textnormal{G}$, a Weyl group element $w\in W^{\textnormal{M}}$ and Bernstein block $\Omega_{\textnormal{M}}$ of $\textnormal{Rep}_{\textnormal{sm}}(\textnormal{M},\overline{\mathbf{Q}}_p)$ inducing $\Omega$ under the parabolic induction, the integral $\Omega$-system $\mathcal{L}$ allows us to define $\mathcal{H}^{B,+}(\sigma_{\Omega_{\textnormal{M}},w\cdot \lambda})^{\circ}$ as well.
\begin{definition}\label{Definition_lambdaelladmissible}
    We say that an integer $N$ is $(\lambda,\mathcal{L})$-admissible if the following are satisfied.
    \begin{itemize}
    \item For every proper standard Levi subgroup $\textnormal{M}=\textnormal{G}_{S_1}\times...\times \textnormal{G}_{S_h}\leq \textnormal{G}$, weight vector $w\cdot \lambda=(\lambda_1,...,\lambda_h)$ for $w\in W^{\textnormal{M}}$ and Bernstein block $\Omega_{\textnormal{M}}=(\Omega_{S_1},...,\Omega_{S_h})$ of $\textnormal{Rep}_{\textnormal{sm}}(\textnormal{M},\overline{\mathbf{Q}}_p)$ inducing $\Omega$ under parabolic induction, $N$ is $(\lambda_i,\mathcal{L}_{S_i})$-admissible.
    \item We have $p^{N}V_{i,j}^{-1}\in \mathcal{H}(\sigma_{\Omega,\lambda}^{\circ})\cap\iota^{-1}(\mathcal{H}(\sigma_{\Omega^{\vee},\lambda^{\vee}}^{\circ}))$.
    \item We have $r_{\mathcal{L},\textnormal{P},\textnormal{M}}\circ\mathcal{S}_{\textnormal{P},\Omega,\lambda}(p^N V_{i,j}^{-1})\in \mathcal{H}(\sigma_{\textnormal{M},\Omega,\lambda}^{\circ})$ for every standard parabolic subgroup $\textnormal{P}=\textnormal{M}\textnormal{N}\leq \textnormal{G}$.
\end{itemize}
\end{definition}

For a $(\lambda,\mathcal{L})$-admissible integer $N$, we set
\begin{equation}\label{equation_integralmodel}
    \mathcal{H}^B(\sigma_{\Omega,\lambda})_N^{\circ}\leq \mathcal{H}^B(\sigma_{\Omega,\lambda})
\end{equation}
to be the $\mathcal{O}$-subalgebra generated by the set $\{V_{i,j},p^NV_{i,j}^{-1}\mid 1\leq i\leq k,1\leq j\leq n_i/d_i\}$. Note that we have a canonical isomorphism $\mathcal{H}^{B}(\sigma_{\Omega,\lambda})^{\circ}\otimes_{\mathcal{O}}\overline{\mathbf{Q}}_p\cong\mathcal{H}^{B}(\sigma_{\Omega})$ and, in particular, the former contains the Bernstein centre $\mathfrak{Z}_{\textnormal{G},\Omega}$. Similarly, we can define $\mathcal{H}^{B}(\sigma_{\textnormal{M},\Omega,\lambda})^{\circ}_N$ for every standard parabolic subgroup $\textnormal{P}=\textnormal{M}\textnormal{N}\leq \textnormal{G}$.

We note that we have a commutative diagram
\begin{equation}\label{explicitpresentation}
    \begin{tikzcd}
	{\mathcal{O}[V_{i,j}\mid i,j]} & {\mathcal{H}^{B,+}(\sigma_{\Omega,\lambda})^{\circ}} \\
	{\mathcal{O}[V_{i,j},W_{i,j}\mid i,j]/(V_{i,j}W_{i,j}-p^{N}\mid i,j)} & {\mathcal{H}^{B}(\sigma_{\Omega,\lambda})^{\circ}.}
	\arrow["\sim", from=1-1, to=1-2]
	\arrow[hook, from=1-1, to=2-1]
	\arrow[hook, from=1-2, to=2-2]
	\arrow["\sim", from=2-1, to=2-2]
\end{tikzcd}
\end{equation}

Finally, we note that, by definition, the $E$-algebra homomorphisms $\mathcal{S}_{\textnormal{P},\Omega,\lambda}$, and $r_{\mathcal{L},\textnormal{P},\textnormal{M}}\circ\mathcal{S}_{\textnormal{P},\Omega,\lambda}$ restrict to $\mathcal{O}$-algebra homomorphisms
\begin{align*}
\mathcal{S}_{\textnormal{P},\Omega,\lambda}^{\circ(,+)}\colon\mathcal{H}^{B(,+)}(\sigma_{\Omega,\lambda})^{\circ} &\to \mathcal{H}(\sigma_{\textnormal{P},\Omega,\lambda}^{\circ}), \\
 \mathcal{S}_{\textnormal{M},\mathcal{L},\lambda}^{(+)} \colon\mathcal{H}^{B(,+)}(\sigma_{\Omega,\lambda})^{\circ} &\to \mathcal{H}(\sigma_{\textnormal{M},\Omega,\lambda}^{\circ}).
\end{align*}

\subsubsection{Interaction of Satake transform with interpolation of semisimple local Langlands}
We have a commutative diagram
\begin{equation}\label{equation_Satakediagram}
    \begin{tikzcd}
	{\mathfrak{Z}_{\textnormal{G},\Omega}} & {\mathcal{H}^B(\sigma_{\Omega,\lambda})^{\circ}\otimes_{\mathcal{O}}\overline{\mathbf{Q}}_p} \\
	{\mathfrak{Z}_{\textnormal{M},\Omega}} & {\mathcal{H}(\sigma_{\textnormal{M},\Omega,\lambda}^{\circ})\otimes_{\mathcal{O}}\overline{\mathbf{Q}}_p}
	\arrow[hook, from=1-1, to=1-2]
	\arrow["{\textnormal{BD}}"', from=1-1, to=2-1]
	\arrow["{\mathcal{S}_{\textnormal{M},\mathcal{L},\lambda}}", from=1-2, to=2-2]
	\arrow[hook, from=2-1, to=2-2]
\end{tikzcd}
\end{equation}
matching the unnormalised Satake transform with the morphism on centres intertwining unnormalised parabolic induction.
We will also need an analysis of the interaction of the map $\textnormal{BD}$ with interpolation of the semisimple local Langlands. In particular, write $\textnormal{M}=\textnormal{GL}_{n_1}(L)\times...\times\textnormal{GL}_{n_r}(L)$ for a partition $n=n_1+...+n_r$. Let $k=\mathcal{O}/\varpi$, and, for $1\leq i\leq r$, $\overline{D}_i:G_L\to k$ be a continuous $n_i$-dimensional determinant and set $\overline{D}:=\overline{D}_1(n-n_1)\oplus...\oplus \overline{D}_{r-1}(n-(n_1+...+n_{r-1}))\oplus \overline{D}_r$. Let $w\in W^{\textnormal{M}}$ and $\Omega_{\textnormal{M}}$ be a Bernstein block of $\textnormal{Rep}_{\textnormal{sm}}(\textnormal{M},\overline{\mathbf{Q}}_p)$ that induces $\Omega$ under parabolic induction. Write $\lambda_i=((w\cdot\lambda)_{\tau,n_1+...+n_{i-1}+1},...,(w\cdot\lambda)_{\tau,n_1+...+n_{i}})\in (\mathbf{Z}_+^{n_i})^{\Hom(L,E)}$ and set $\Omega_i$ to be the Bernstein block of $\textnormal{Rep}_{\textnormal{sm}}(\textnormal{GL}_{n_i}(L),\overline{\mathbf{Q}}_p)$ given by the $i$th factor of $\Omega_{\textnormal{M}}$.

Let $R^{\lambda,\Omega}_{\overline{D}}$, and $ R^{\lambda_i,\Omega_i}_{\overline{D}_i}$ be the corresponding potentially semistable pseudodeformation rings and $\eta_{\textnormal{G}}\colon\mathfrak{Z}_{\textnormal{G},\Omega}\to R^{\lambda,\Omega}_{\overline{D}}\otimes_{\mathcal{O}}\overline{\mathbf{Q}}_p$, and $\eta_i\colon\mathfrak{Z}_{\textnormal{GL}_{n_i}(L),\Omega_i}\to R^{\lambda_i,\Omega_i}_{\overline{D}_i}\otimes_{\mathcal{O}}\overline{\mathbf{Q}}_p$ be the maps interpolating the Tate normalised semisimple local Langlands correspondences.
We set $R^{w\cdot \lambda,\Omega_{\textnormal{M}}}_{(\overline{D}_i)_i}:=\widehat{\otimes}_{i}R^{\lambda_i,\Omega_i}_{\overline{D}_i}$ and $\eta_{\textnormal{M}}:=\otimes_i\eta_i$ with its source canonically identified with $\mathfrak{Z}_{\textnormal{M},\Omega_{\textnormal{M}}}$.

\begin{lemma}\label{Lemma_InterpolationOfSSvsEisenstein}
    We have a commutative diagram
    \begin{equation}\label{equation_interpolationdiagram}
        \begin{tikzcd}
	{\mathfrak{Z}_{\textnormal{G},\Omega}} & {R^{\lambda,\Omega}_{\overline{D}}\otimes_{\mathcal{O}}\overline{\mathbf{Q}}_p} \\
	{\mathfrak{Z}_{\textnormal{M},\Omega_{\textnormal{M}}}} & {R^{w\cdot \lambda,\Omega_{\textnormal{M}}}_{(\overline{D}_i)_i}\otimes_{\mathcal{O}}\overline{\mathbf{Q}}_p.}
	\arrow["{\eta_{\textnormal{G}}}", from=1-1, to=1-2]
	\arrow["{\textnormal{BD}}"', from=1-1, to=2-1]
	\arrow["{(D_i)_i\mapsto\oplus_iD_i(n-(n_1+...+n_i))}", from=1-2, to=2-2]
	\arrow["{\eta_{\textnormal{M}}}"', from=2-1, to=2-2]
\end{tikzcd}
    \end{equation}
\end{lemma}
\begin{proof}
    The right bottom corner of the diagram is a reduced Jacobson ring. In particular, it suffices to check commutativity after post-composition along an arbitrary $\overline{\mathbf{Q}}_p$-point $R^{w\cdot \lambda,\Omega_{\textnormal{M}}}_{(\overline{D}_i)_i}\to \overline{\mathbf{Q}}_p$. In this case, the claim follows from the definitions.
\end{proof}

We finally introduce $\mathcal{O}$-models of Bernstein centres that are compatible with both \eqref{equation_Satakediagram} and \eqref{equation_interpolationdiagram}. We do this by induction on parabolic subgroups ordered by inclusion.
For a standard parabolic subgroup $\widetilde{\textnormal{M}}\leq \textnormal{G}$ such that $\Omega=[\widetilde{\textnormal{M}},\widetilde{\pi}]_{\textnormal{G}}$, set $\mathfrak{Z}_{\widetilde{\textnormal{M}},\mathcal{L},\lambda}^{\circ}:=\mathfrak{Z}_{\widetilde{\textnormal{M}},\Omega}\cap \mathcal{H}^B(\sigma_{\widetilde{\textnormal{M}},\Omega,\lambda})^{\circ}$.

Let $\textnormal{M}\leq \textnormal{G}$ be a standard Levi subgroup such that the direct product factor $\textnormal{Rep}_{\textnormal{sm}}(\textnormal{M},\overline{\mathbf{Q}}_p)[\Omega]\subset\textnormal{Rep}_{\textnormal{sm}}(\textnormal{M},\overline{\mathbf{Q}}_p)$ of blocks inducing $\Omega$ is non-trivial. Assume that, for every proper standard Levi subgroup $\textnormal{M}'\lneq\textnormal{M}$ with non-trivial $\textnormal{Rep}_{\textnormal{sm}}(\textnormal{M}',\overline{\mathbf{Q}}_p)[\Omega]$, $\mathfrak{Z}_{\textnormal{M}',\mathcal{L},\lambda}^{\circ}$ is defined. We define the $\mathcal{O}$-model
\begin{equation*}
\mathfrak{Z}_{\textnormal{M},\mathcal{L},\lambda}^{\circ}:=\big(\bigcap_{\textnormal{M}'\lneq\textnormal{M}}\textnormal{BD}^{-1}(\mathfrak{Z}_{\textnormal{M}',\mathcal{L},\lambda}^{\circ})\big)\cap \mathcal{H}^{B}(\sigma_{\textnormal{M},\Omega,\lambda})^{\circ}\leq \mathfrak{Z}_{\textnormal{M},\Omega},
\end{equation*}
where the intersection runs over $\textnormal{M}'\lneq \textnormal{M}$ as above and $\textnormal{BD}:\mathfrak{Z}_{\textnormal{M},\Omega}\to \mathfrak{Z}_{\textnormal{M}',\Omega}$ denotes the map of centres intertwining $\textnormal{Ind}_{\textnormal{M}'}^{\textnormal{M}}$.

Finally, in the situation of \eqref{equation_interpolationdiagram} define the $\mathcal{O}$-models
\begin{align*}
    \mathfrak{Z}_{\textnormal{G},\mathcal{L},\lambda,\overline{D}}^{\circ} &:=\mathfrak{Z}_{\textnormal{G},\mathcal{L},\lambda}^{\circ}\cap \eta_{\textnormal{G}}^{-1}(R_{\overline{D}}^{\lambda,\Omega})\leq \mathfrak{Z}_{\textnormal{G},\Omega}, \\
    \mathfrak{Z}_{\textnormal{M},\mathcal{L},\lambda,(\overline{D}_i)_i}^{\circ} &:=\mathfrak{Z}_{\textnormal{M},\mathcal{L},\lambda}^{\circ}\cap \bigcap_{w,\Omega_{\textnormal{M}}}\eta_{\textnormal{M}}^{-1}(R^{w\cdot \lambda,\Omega_{\textnormal{M}}}_{(\overline{D}_i)_i}) \leq \mathfrak{Z}_{\textnormal{M},\Omega}.
\end{align*}
The following commutative diagram summarises the situation. 
\begin{equation*}
    \begin{tikzcd}
	& {\mathcal{H}^{B}(\sigma_{\Omega,\lambda})^{\circ}} & {\mathfrak{Z}_{\textnormal{G},\mathcal{L},\lambda,\overline{D}}^{\circ}} & {R^{\lambda,\Omega}_{\overline{D}}} \\
	{\mathcal{H}(\sigma_{\textnormal{M},\Omega,\lambda}^{\circ})} & {\mathcal{H}^B(\sigma_{\textnormal{M},\Omega,\lambda})^{\circ}} & {\mathfrak{Z}_{\textnormal{M},\mathcal{L},\lambda,(\overline{D}_i)_i}^{\circ}} & {R^{w\cdot \lambda,\Omega_{\textnormal{M}}}_{(\overline{D}_i)_i},}
	\arrow["{\mathcal{S}_{\textnormal{M},\mathcal{L},\lambda}}"', from=1-2, to=2-1]
	\arrow[hook', from=1-3, to=1-2]
	\arrow["{\eta_{\textnormal{G}}^{\circ}}", from=1-3, to=1-4]
	\arrow["{\textnormal{BD}^{\circ}}", from=1-3, to=2-3]
	\arrow["{(D_i)_i\mapsto\oplus_iD_i(m_i)}", from=1-4, to=2-4]
	\arrow[hook', from=2-2, to=2-1]
	\arrow[hook', from=2-3, to=2-2]
	\arrow["{\eta_{\textnormal{M}}^{\circ}}"', from=2-3, to=2-4]
\end{tikzcd}
\end{equation*}
where $m_i = n - (n_1 + \dots + n_i)$.

\subsection{Smooth representations with \texorpdfstring{$p$}{p}-power torsion coefficients}\label{subsection_smooth_modp}

Consider $\sigma\in\textnormal{Rep}(K,E)$, finite dimensional over $E$. Fix a choice of $K$-stable $\mathcal{O}$-lattice $\sigma^{\circ}\subset\sigma$.
Set $\Lambda:=\mathcal{O}/\varpi^m$ for an integer $m\geq 1$. In this subsection, we consider $\overline{\sigma}:=\sigma^{\circ}\otimes_{\mathcal{O}}\Lambda\in \textnormal{Rep}_{\textnormal{sm}}(K,\Lambda)$, finite free over $\Lambda$.

Given $\pi\in D^+_{\textnormal{sm}}(\textnormal{G},\Lambda)$, $R\Gamma(K,\pi\otimes_{\Lambda}\overline{\sigma})$ is naturally isomorphic in $D^+(\Lambda)$ to an object of $D^+(\mathcal{H}(\sigma)\otimes_{\mathcal{O}}\Lambda).$ In particular, there is an algebra homomorphism
\begin{equation*}
    \mathcal{H}(\sigma)\to \textnormal{End}_{D^+(\Lambda)}(R\Gamma(K,\pi\otimes_{\Lambda} \overline{\sigma})).
\end{equation*}
To see this, note that both $\Res_K^{\textnormal{G}}(-)$ and $-\otimes_{\Lambda}\overline{\sigma}$ preserve injectives. The first functor by \cite[Proposition 2.1.2]{Eme10} and the second by admitting $\Hom_{\Lambda}(\overline{\sigma},-)\cong -\otimes_{\Lambda}\overline{\sigma}^{\vee}$ as an exact left adjoint. Consequently, $R\Gamma(K,-)\circ (-\otimes_{\Lambda}\overline{\sigma})\circ \textnormal{Res}_K^{\textnormal{G}}\cong R\left(\Gamma(K,-)\circ (-\otimes_{\Lambda}\overline{\sigma})\circ \textnormal{Res}_K^{\textnormal{G}}\right)$. Concretely, for $\pi\xrightarrow{\sim}\mathcal{I}^{\bullet}$ an injective resolution in $D^+_{\textnormal{sm}}(\textnormal{G},\Lambda)$, applying the forgetful functor to $\Gamma(K,\mathcal{I}^{\bullet}\otimes \overline{\sigma})\in D^+(\mathcal{H}(\sigma)\otimes_{\mathcal{O}}\Lambda)$ computes $R\Gamma(K,\pi\otimes_{\Lambda}\overline{\sigma})$.

Let $r_{\textnormal{P}}\colon\mathcal{H}(\sigma)\to \mathcal{H}(\sigma\mid_{K_P})$ be the algebra homomorphism restricting a function $f\colon\textnormal{G}\to \textnormal{End}(\sigma)$ to $\textnormal{P}$.

\begin{lemma}\label{RestrictionFunc}
    For $\pi\in D^+_{\textnormal{sm}}(\textnormal{P},\Lambda)$, there is an isomorphism
    \begin{equation*}
        R\Gamma(K,(\textnormal{Ind}_{K_{\textnormal{P}}}^K\pi)\otimes\overline{\sigma})\cong r_{\textnormal{P}}^{\ast}R\Gamma(K_{\textnormal{P}},\pi\otimes \overline{\sigma})
    \end{equation*}
    in $D^+(\mathcal{H}(\sigma)\otimes_{\mathcal{O}}\Lambda)$.
\end{lemma}
\begin{proof}
    As $\Ind_{K_{\textnormal{P}}}^{K}$ admits $\Res_{K_{\textnormal{P}}}^K$ as a left adjoint, an exact functor, it suffices to check the claim on underived functors. In other words, for $\pi\in \textnormal{Rep}_{\textnormal{sm}}(\textnormal{P},\Lambda)$, we seek an isomorphism
    \begin{equation*}
        \Gamma(K,\textnormal{Ind}_{K_{\textnormal{P}}}^K\pi\otimes_{\mathcal{O}}\sigma^{\circ})\cong\Gamma(K_{\textnormal{P}},\pi\otimes_{\mathcal{O}}\sigma^{\circ})
    \end{equation*}
    such that, for $p\in \textnormal{P}$, and $\phi\in \Hom_{p^{-1}Kp\cap K}(\sigma^{\circ},p^{\ast}\sigma^{\circ})$, the action of
    $[KpK,\phi]$ on the LHS matches the action of $r_{\textnormal{P}}([KpK,\phi])$
    on the RHS. This is clear.
\end{proof}

Let $m\in \textnormal{M}$ and $\phi\in \Hom_{m^{-1}K_{\textnormal{P}}m\cap K_{\textnormal{P}}}(\sigma^{\circ},m^{\ast}\sigma^{\circ})$.
For $\pi\in D^+_{\textnormal{sm}}(\textnormal{P},\Lambda)$ one obtains a map
\begin{equation*}
    \overline{\phi}_{\textnormal{P},\pi}\colon\Res_{m^{-1}K_{\textnormal{M}}m\cap K_{\textnormal{M}}}(R\Gamma(K_N,\pi\otimes_{\Lambda}\overline{\sigma}))\to m^{\ast}\Res_{K_{\textnormal{M}}\cap m K_{\textnormal{M}}m^{-1}}(R\Gamma(K_{\textnormal{N}},\pi\otimes_{\Lambda}\overline{\sigma}))
\end{equation*}
in $D^+_{\textnormal{sm}}(m^{-1}K_{\textnormal{M}}m\cap K_{\textnormal{M}},\Lambda)$ defined by the usual Hecke action. 

We can define it as follows. Pick an injective resolution $\pi\cong \mathcal{I}^{\bullet}$ in $D^+_{\textnormal{sm}}(\textnormal{P},\Lambda)$. Then $\mathcal{I}^{\bullet}\otimes_{\Lambda}\overline{\sigma}$ is an injective resolution of $\pi\otimes_{\Lambda}\overline{\sigma}$ in $D^{+}_{\textnormal{sm}}(K_{\textnormal{P}},\Lambda)$ and we can compute the source and target of $\overline{\phi}_{\textnormal{P}}$ by applying the underived variants of the appearing functors to it.
We define
\begin{align*}
\overline{\phi}_{\textnormal{P},\pi}^i\colon\Gamma(K_{\textnormal{N}},\mathcal{I}^i\otimes_{\mathcal{O}}\sigma^{\circ}) &\to \Gamma(K_{\textnormal{N}},\mathcal{I}^{i}\otimes_{\mathcal{O}}\sigma^{\circ}) \\
    v &\mapsto \sum_{n\in K_{\textnormal{N}}/(mK_{\textnormal{N}}m^{-1}\cap K_{\textnormal{N}})}\bigl(\sigma^{\circ}(n)\circ (m\otimes \phi)\bigr)(v),
\end{align*}
that can be easily checked to be $m^{-1}K_{\textnormal{M}}m\cap K_{\textnormal{M}}$-equivariant when we consider the twisted action on the target. We then set $\overline{\phi}_{\textnormal{P},\pi}:=\overline{\phi}_{\textnormal{P},\pi}^{\bullet}$.

This construction will be of special interest in the case of the trivial representation $\pi=\mathbf{1}[0]$.
\begin{lemma}\label{TransitivityOfHeckeAction}
    For $\pi\in D^+_{\textnormal{sm}}(\textnormal{P},\Lambda)$, there is an isomorphism
    \begin{equation*}
R\Gamma(K_{\textnormal{P}},\pi\otimes_{\Lambda}\overline{\sigma})\cong R\Gamma(K_{\textnormal{M}},R\Gamma(K_{\textnormal{N}},\pi\otimes_{\Lambda}\overline{\sigma}))
    \end{equation*}
    in $D^{+}(\Lambda)$, for every $m\in \textnormal{M}$, matching the endomorphism induced by $[K_{\textnormal{P}}mK_{\textnormal{P}},m\otimes \phi]$ on the left with the endomorphism induced by $[K_{\textnormal{M}}mK_{\textnormal{M}},\overline{\phi}_{\textnormal{P},\pi}]$ on the right.
\end{lemma}
\begin{proof}
    We can use the injective resolution $\pi\otimes_{\Lambda}\overline{\sigma}\cong\mathcal{I}^{\bullet}\otimes_{\Lambda}\overline{\sigma}$
    to reduce the question to the underived statement.
    In particular, we can assume that $\pi$ is injective, pick $v\in \Gamma(K_{\textnormal{M}},\Gamma(K_{\textnormal{N}},\pi\otimes_{\Lambda}\overline{\sigma}))=\Gamma(K_{\textnormal{P}},\pi\otimes_{\Lambda}\overline{\sigma})$ and compute
    \begin{align*}
    [K_{\textnormal{P}}mK_{\textnormal{P}},\phi]\ast v &=\sum_{k\in K_{\textnormal{P}}/m K_{\textnormal{P}}m^{-1}\cap K_{\textnormal{P}}}(\sigma^{\circ}(k)\circ (m\otimes \phi))(v) \\
    &= \sum_{(k_1,k_2)\in K_{\textnormal{M}}/(mK_{\textnormal{M}}m^{-1}\cap K_{\textnormal{M}})\times K_{\textnormal{N}}/(mK_{\textnormal{N}}m^{-1}\cap K_{\textnormal{N}})}\sigma^{\circ}(k_1)(\sigma^{\circ}(k_2)(m\otimes\phi(v))) \\
    &= \sum_{k_1\in K_{\textnormal{M}}/mK_{\textnormal{M}}m^{-1}\cap K_{\textnormal{M}}}\sigma^{\circ}(k_1)(\overline{\phi}_{\textnormal{P},\pi}\ast v)= [K_{\textnormal{M}}mK_{\textnormal{M}},\overline{\phi}_{\textnormal{P},\pi}]\ast v.
    \end{align*}
    To check the validity of the second equality, we note that, for every $k_1,\widetilde{k}_1\in K_{\textnormal{M}}$, $k_2,\widetilde{k}_2\in K_{\textnormal{N}}$, we have $k_1k_2=m\widetilde{k}_1\widetilde{k}_2m^{-1}$ if and only if $k_1=m\widetilde{k}_1m^{-1}$ and $k_2=m\widetilde{k}_2m^{-1}$.
\end{proof}

Note that $\textnormal{Rep}_{\textnormal{sm}}(\textnormal{M},\Lambda)$ has enough $\Lambda$-flat objects (given by arbitrary direct sums of representations of the form $\textnormal{c-Ind}_{J}^{\textnormal{M}}\Lambda$ for open subgroups $J\subset \textnormal{K}$) to compute $-\otimes_{\Lambda}^{\mathbf{L}}-$ in $D^{-}_{\textnormal{sm}}(\textnormal{M},\Lambda)$. We say that a complex $\pi\in D^{-}_{\textnormal{sm}}(\textnormal{M},\Lambda)$ has \textit{bounded Tor-dimension} if there are integers $a\leq b$ such that $H^i(\pi\otimes^{\mathbf{L}}_{\Lambda}M)=0$ for every $\Lambda$-module $M$ and integer $i\notin [a,b]$.

\begin{lemma}\label{TorDimLemma}
    If $\pi\in D^{b}_{\textnormal{sm}}(\textnormal{M},\Lambda)$ has bounded Tor-dimension, then it admits a \textit{bounded} $\Lambda$-flat resolution.
\end{lemma}
\begin{proof}
    Pick a bounded above $\Lambda$-flat resolution $F^{\bullet}\to \pi$ in $D^{-}_{\textnormal{sm}}(\textnormal{M},\Lambda)$. We claim that $\tau_{\geq N}F^{\bullet}$ is a $\Lambda$-flat resolution of $\pi$ for any $N\leq a$ where $\pi$ has Tor-amplitude $[a,b]$.

    To see this, we need to prove that $\textnormal{coker}(d^{N}:F^{N-1}\to F^{N})$ is $\Lambda$-flat. This is \cite[\href{https://stacks.math.columbia.edu/tag/0653}{Tag 0653}]{stacks-project}.
\end{proof}

Given $\pi \in D^b_{\textnormal{sm}}(\textnormal{M},\Lambda)$ with bounded Tor-dimension, using the Lemma above, we can now define
\begin{equation*}
\pi\otimes^{\mathbf{L}}_{\Lambda}-:D^+_{\textnormal{sm}}(K_{\textnormal{M}},\Lambda)\to D^+_{\textnormal{sm}}(K_{\textnormal{M}},\Lambda)
\end{equation*}
as $\textnormal{Tot}(F^{\bullet}\otimes_{\Lambda}-)$ for any bounded $\Lambda$-flat resolution $F^{\bullet}\to \pi$ in $D_{\textnormal{sm}}^+(\textnormal{M},\Lambda)$. One checks that this functor sends quasi-isomorphisms to quasi-isomorphisms (cf. \cite[\href{https://stacks.math.columbia.edu/tag/06Y0}{Tag 06Y0}]{stacks-project}) and that it is independent of the choice of the resolution up to canonical quasi-isomorphism (cf. \cite[\href{https://stacks.math.columbia.edu/tag/064L}{Tag 064L}]{stacks-project}).

\begin{lemma}\label{ProjectionFormula}
    Let $\pi\in D^b_{\textnormal{sm}}(\textnormal{M},\Lambda)$ having bounded Tor-dimension. Then there is an isomorphism
    \begin{equation*}
        R\Gamma(K_{\textnormal{P}},(\textnormal{Inf}_{\textnormal{M}}^{\textnormal{P}}\pi)\otimes_{\Lambda}\overline{\sigma})\cong R\Gamma(K_{\textnormal{M}},\pi\otimes^{\mathbf{L}}_{\Lambda}R\Gamma(K_{\textnormal{N}},\overline{\sigma}))
    \end{equation*}
    matching the endomorphism induced by $[K_{\textnormal{P}}mK_{\textnormal{P}},m\otimes \phi]$ on the left with the endomorphism induced by $[K_{\textnormal{M}}mK_{\textnormal{M}},m\otimes\overline{\phi}_{\textnormal{P},\mathbf{1}}]$ on the right.
\end{lemma}
\begin{proof}
    Using Lemma~\ref{TransitivityOfHeckeAction}, one reduces to the assertion that there is an isomorphism
    \begin{equation*}
       R\Gamma(K_{\textnormal{N}},(\textnormal{Inf}_{\textnormal{M}}^{\textnormal{P}}\pi)\otimes_{\Lambda}\overline{\sigma})\cong \pi\otimes^{\mathbf{L}}_{\Lambda}R\Gamma(K_{\textnormal{N}},\overline{\sigma})
    \end{equation*}
    in $D^+_{\textnormal{sm}}(K_{\textnormal{M}},\Lambda)$
    that matches the intertwining map $\overline{\phi}_{\textnormal{P},\textnormal{Inf}_{\textnormal{M}}^{\textnormal{P}}\pi}$ on the left with $m\otimes \overline{\phi}_{\textnormal{P},\mathbf{1}}$ on the right.

    To prove this, we pick an injective resolution $\mathbf{1}\xrightarrow{\sim}\mathcal{I}^{\bullet}$ of the trivial representation in $D^+_{\textnormal{sm}}(\textnormal{P},\Lambda)$ and a $\Lambda$-flat bounded resolution $\pi\xrightarrow{\sim} F^{\bullet}$ in $D^+_{\textnormal{sm}}(\textnormal{M},\Lambda)$. We claim that the resolution $(\textnormal{Inf}_{\textnormal{M}}^{\textnormal{P}}\pi)\otimes_{\Lambda}\overline{\sigma}\xrightarrow{\sim}\textnormal{Tot}((\textnormal{Inf}_{M}^{\textnormal{P}}F^{\bullet})\otimes_{\Lambda}\mathcal{I}^{\bullet}\otimes_{\Lambda}\overline{\sigma})$
    is $\Gamma(K_{N},-)$-acyclic. Indeed, by a theorem of Lazard \cite[\href{https://stacks.math.columbia.edu/tag/058G}{Tag 058G}]{stacks-project} the members of the resolution (as objects of $\textnormal{Rep}_{\textnormal{sm}}(K_{\textnormal{N}},\Lambda)$) are filtered colimits of finite direct sums of the injective objects $\mathcal{I}^i\otimes_{\Lambda}\overline{\sigma}$. In particular, they are injective in $D^+_{\textnormal{sm}}(K_{\textnormal{N}},\Lambda)$ by \cite[Proposition 2.1.3.]{Eme10}

    Therefore,  we have
    \begin{align*}
         R\Gamma(K_{\textnormal{N}},(\textnormal{Inf}_{\textnormal{M}}^{\textnormal{P}}\pi)\otimes_{\Lambda}\overline{\sigma}) &\cong \Gamma(K_{\textnormal{N}},\textnormal{Tot}((\textnormal{Inf}_{\textnormal{M}}^{\textnormal{P}}F^{\bullet})\otimes_{\Lambda}\mathcal{I}^{\bullet}\otimes_{\Lambda}\overline{\sigma})) \\
        &= \textnormal{Tot}(F^{\bullet}\otimes_{\Lambda}\Gamma(K_{\textnormal{N}},\mathcal{I}^{\bullet}\otimes_{\Lambda}\overline{\sigma}))\\
        &\cong\pi\otimes^{\mathbf{L}}_{\Lambda}R\Gamma(K_{\textnormal{N}},\overline{\sigma}).
    \end{align*}
    The Hecke actions can be easily compared between the second and third members of the sequence of isomorphisms above.
\end{proof}

We finish the subsection with two key lemmas on ``change of coefficient systems up to $p^N$-torsion".
\begin{lemma}\label{Lemma_changeofweightI}
    Assume that we are given finite dimensional $E$-representations $\sigma_1,\sigma_2$ of $K$ with fixed choices of $K$-stable $\mathcal{O}$-lattices $\sigma_1^{\circ}\subset \sigma_1$, $\sigma_{2}^{\circ}\subset \sigma_2$, an isomorphism
    \begin{equation*}
        \varphi\colon\textnormal{c-Ind}_{K}^{\textnormal{G}}\sigma_1\xrightarrow{\sim}\textnormal{c-Ind}_K^{\textnormal{G}}\sigma_2,
    \end{equation*}
    an integer $N\geq 0$ such that $p^N\varphi$ and $p^N\varphi^{-1}$ restricts to a morphism
    \begin{equation*}
        \textnormal{c-Ind}_K^{\textnormal{G}}\sigma_1^{\circ}\to\textnormal{c-Ind}_K^{\textnormal{G}}\sigma_2^{\circ}
    \end{equation*}
    and
    \begin{equation*}
        \textnormal{c-Ind}_K^{\textnormal{G}}\sigma_2^{\circ}\to\textnormal{c-Ind}_K^{\textnormal{G}}\sigma_1^{\circ},
    \end{equation*}
    respectively.

Let $r\colon\mathcal{H}(\sigma_1)\to \mathcal{H}(\sigma_2)$
be the homomorphism $h\mapsto \varphi\circ h\circ \varphi^{-1}$ and set $\mathcal{H}^{\circ}\subset \mathcal{H}(\sigma_1^{\vee})$ be the $\mathcal{O}$-subalgebra corresponding to $\mathcal{H}(\sigma_1^{\circ})\cap r^{-1}(\mathcal{H}(\sigma_2^{\circ}))$ under the usual anti-isomorphism $\mathcal{H}(\sigma_1^{\vee})\cong \mathcal{H}(\sigma_1)$. For $\pi\in D^+_{\textnormal{sm}}(\textnormal{G},\Lambda)$ we have algebra homomorphisms
\begin{align*}
    \mathcal{H}^{\circ} &\to \textnormal{End}_{D^+(\Lambda)}(R\Gamma(K,\pi\otimes_{\mathcal{O}}\sigma_1^{\circ,\vee})), \\
     \mathcal{H}^{\circ} &\to \textnormal{End}_{D^+(\Lambda)}(R\Gamma(K,\pi\otimes_{\mathcal{O}}\sigma_2^{\circ,\vee})),
\end{align*}
with the latter being induced by $r$.

Moreover, there are $\mathcal{H}^{\circ}$-equivariant homomorphisms
\begin{align*}
    \alpha \colon R\Gamma(K,\pi\otimes_{\mathcal{O}}\sigma_1^{\circ,\vee}) &\to R\Gamma(K,\pi\otimes_{\mathcal{O}}\sigma_2^{\circ,\vee}), \\
    \beta \colon R\Gamma(K,\pi\otimes_{\mathcal{O}}\sigma_2^{\circ,\vee}) &\to R\Gamma(K,\pi\otimes_{\mathcal{O}}\sigma_1^{\circ,\vee})
\end{align*}
in $D^+(\Lambda)$ satisfying $\alpha\circ \beta=p^{2N}$, $\beta\circ \alpha=p^{2N}$.
\end{lemma}

\begin{proof}
    The first claim follows from the definition of $\mathcal{H}^{\circ}$ and the fact that the object $R\Gamma(K,\pi\otimes_{\mathcal{O}}\sigma_i^{\circ,\vee})$ naturally lies in $D^+_{\textnormal{sm}}(\mathcal{H}(\sigma_i^{\circ,\vee})\otimes_{\mathcal{O}}\Lambda)$.

    To see the second part, we note that we have a $\mathcal{H}(\sigma_i^{\circ,\vee})$-equivariant isomorphism
    \begin{equation*}
        R\Gamma(K,\pi\otimes_{\mathcal{O}}\sigma_i^{\circ,\vee})\cong R\Hom_{\textnormal{G}}(\textnormal{c-Ind}_K^{\textnormal{G}}\sigma_i^{\circ},\pi)
    \end{equation*}
    where on the RHS we act through the anti-isomorphism $\mathcal{H}(\sigma_i^{\circ,\vee})\xrightarrow{\sim}\mathcal{H}(\sigma_i^{\circ})=\textnormal{End}_{\textnormal{G}}(\textnormal{c-Ind}_K^{\textnormal{G}}\sigma_i^{\circ})$ with the right action of the latter.

    It suffices to see that pre-composition by $p^N\varphi^{-1}$ and $p^N\varphi$ induce morphisms
    \begin{equation*}
        R\Hom_{\textnormal{G}}(\textnormal{c-Ind}_K^{\textnormal{G}}\sigma_1^{\circ},-)\to R\Hom_{\textnormal{G}}(\textnormal{c-Ind}_K^{\textnormal{G}}\sigma_2^{\circ},-),
    \end{equation*}
    and
    \begin{equation*}
        R\Hom_{\textnormal{G}}(\textnormal{c-Ind}_K^{\textnormal{G}}\sigma_2^{\circ},-)\to R\Hom_{\textnormal{G}}(\textnormal{c-Ind}_K^{\textnormal{G}}\sigma_1^{\circ},-),
    \end{equation*}
    respectively,
    of derived functors $D^+_{\textnormal{sm}}(\textnormal{G},\Lambda)\to D^+(\mathcal{H}^{\circ}\otimes_{\mathcal{O}}\Lambda)$.

    To verify the first claim it suffices to see that for every $h\in \mathcal{H}^\circ$ we have $h\circ(p^N\varphi^{-1})=(p^N\varphi^{-1})\circ r(h)$. When we view $h$ and $r(h)$ as elements of $\mathcal{H}(\sigma_1^{\circ})$ and $\mathcal{H}(\sigma_2^{\circ})$ (under the usual anti-isomorphisms), we have the formula $r(h)=\varphi\circ h\circ\varphi^{-1}$. We then compute
    \begin{equation*}
        (p^N\varphi^{-1})\circ r(h)=(p^N\varphi^{-1})\circ \varphi\circ h\circ\varphi^{-1}=h\circ (p^N\varphi^{-1}).
    \end{equation*}
    The second claim follows by symmetry.
\end{proof}

Finally, we will need a lemma that is more specific to the setup of locally algebraic types. Let $\lambda\in (\mathbf{Z}_+^n)^{\Hom(L,E)}$ be a highest weight vector and set $\mathcal{V}_{\lambda}\subset V_{\lambda}$ be our fixed choice of $K$-stable $\mathcal{O}$-lattice given by the dual Weyl module. Let $K'\trianglelefteq K$ be a fixed compact open normal subgroup and $\sigma$ be a finite dimensional $E$-representation of $K$ on which $K'$ acts trivially. Consider a choice of $K$-stable $\mathcal{O}$-lattice $\sigma^{\circ}\subset \sigma$.

For a standard parabolic subgroup $\textnormal{P}=\textnormal{M}\textnormal{N}$, call a pair $(t,\phi)$, consisting of an element $t\in \textnormal{M}$ and an intertwining map $\phi\in \Hom_{t^{-1}K_{\textnormal{P}}t\cap K_{\textnormal{P}}}(\sigma,t^{\ast}\sigma)$, $\lambda$\textit{-integral} if $\phi\otimes t\in \Hom_{t^{-1}K_{\textnormal{P}}t\cap K_{\textnormal{P}}}(\sigma_{\lambda}^{\circ},t^{\ast}\sigma_{\lambda}^{\circ})\subset \Hom_{t^{-1}K_{\textnormal{P}}t\cap K_{\textnormal{P}}}(\sigma\otimes V_{\lambda},t^{\ast}(\sigma\otimes V_{\lambda}))$.
\begin{lemma}\label{Lemma_changeofweightII}
    There is an integer $N\geq 0$ depending only on $\lambda$ and $[K:K']$ satisfying the following property. For all integers $m\in \mathbf{Z}_{\geq 1}$, $q\in \mathbf{Z}$ and standard parabolic subgroup $\textnormal{P}=\textnormal{M}\textnormal{N}\leq \textnormal{G}$, there are maps
    \begin{align*}
        \alpha_m &\colon H^q(K_{\textnormal{N}},\sigma_{\lambda}^{\circ}/\varpi^m)\to \Gamma(K_{\textnormal{N}},\sigma^{\circ})\otimes_{\mathcal{O}}\left(\bigoplus_{w\in W^{\textnormal{M}},l(w)=q}\mathcal{V}_{w\cdot \lambda}\right)/\varpi^m, \\
        \beta_m &\colon \Gamma(K_{\textnormal{N}},\sigma^{\circ})\otimes_{\mathcal{O}}\left(\bigoplus_{w\in W^{\textnormal{M}},l(w)=q}\mathcal{V}_{w\cdot \lambda}\right)/\varpi^m\to H^q(K_{\textnormal{N}},\sigma_{\lambda}^{\circ}/\varpi^m)
    \end{align*}
    such that the following are satisfied.
    \begin{itemize}
        \item We have equalities of maps $\alpha_m\circ \beta_m=p^N$, $\beta_m\circ \alpha_m=p^N$.
        \item For every $\lambda$-integral pair $(t,\phi)$, the endomorphism of the source of $\alpha_m$ induced by $(\overline{\phi\otimes t})_{\textnormal{P},\mathbf{1}}$ and the endomorphism of the target of $\alpha_m$ given by
        \begin{equation}\label{equation_locallyalgHeckeaction}
            x\otimes y\mapsto \sum_{k\in K_{\textnormal{N}}/tK_{\textnormal{N}}t^{-1}\cap K_{\textnormal{N}}}k\big((\phi\otimes t)(x\otimes y)\big)
        \end{equation}
        are intertwined by both $\alpha_m$ and $\beta_m$.
    \end{itemize}
\end{lemma}
\begin{proof}
To ease notation, write
\begin{equation*}
    \sigma_{\textnormal{P},\lambda,q}^{\circ}:=\Gamma(K_{\textnormal{N}},\sigma^{\circ})\otimes_{\mathcal{O}}\left(\bigoplus_{w\in W^{\textnormal{M}},l(w)=q}\mathcal{V}_{w\cdot \lambda}\right).
\end{equation*}
    We first claim that it suffices to find $N$ for which we can construct maps
    \begin{align*}
        \alpha &\colon H^{q}_{\textnormal{cont}}(K_{\textnormal{N}},\sigma_{\lambda}^{\circ})\to \sigma_{\textnormal{P},\lambda,q}^{\circ} \\
        \beta &\colon \sigma_{\textnormal{P},\lambda,q}^{\circ}\to H^{q}_{\textnormal{cont}}(K_{\textnormal{N}},\sigma_{\lambda}^{\circ})
    \end{align*}
    satisfying the following.
    \begin{itemize}
        \item We have equalities $\alpha\circ \beta=p^N$ and $\beta\circ\alpha=p^N$.
        \item For every $\lambda$-integral pair $(t,\phi)$, the endomorphism
        \begin{equation}\label{equation_derivedsmoothHeckeaction}
            \begin{tikzcd}
            H^q(K_{\textnormal{N}},\sigma_{\lambda}^{\circ}) \arrow{r}{\textnormal{res}} & H^q(t^{-1}K_{\textnormal{N}}t\cap K_{\textnormal{N}},\sigma_{\lambda}^{\circ}) \arrow{ld}[swap]{\phi\cdot} \\
            H^q(K_{\textnormal{N}}\cap tK_{\textnormal{N}}t^{-1},\sigma_{\lambda}^{\circ}) \arrow{r}{\textnormal{cores}} & H^{q}(K_{\textnormal{N}},\sigma_{\lambda}^{\circ})
            \end{tikzcd}
        \end{equation}
        and \eqref{equation_locallyalgHeckeaction} are intertwined by both $\alpha$ and $\beta$.
    \end{itemize}
To see that this indeed suffices, note that there is a short exact sequence
\begin{equation*}
    0\to H^q_{\textnormal{cont}}(K_{\textnormal{N}},\sigma_{\lambda}^{\circ})/\varpi^m\to H^q(K_{\textnormal{N}},\sigma_{\lambda}^{\circ}/\varpi^m)\to H^{q+1}_{\textnormal{cont}}(K_{\textnormal{N}},\sigma_{\lambda}^{\circ})[\varpi^m]\to 0
\end{equation*}
such that endomorphism of the LHS induced by \eqref{equation_derivedsmoothHeckeaction} matches with the endomorphism of the middle induced by $\overline{(\phi\otimes t)}_{\textnormal{P},\mathbf{1}}$. Moreover, since continuous group cohomology of $K_{\textnormal{N}}$ on finitely generated $\mathcal{O}$-modules is finitely generated, the RHS of the short exact sequence is contained in $H^{q+1}_{\textnormal{cont}}(K_{\textnormal{N}},\sigma_{\lambda}^{\circ})[p^{N'}]$ for an integer $N'$ independent of $m\geq 1$.
In fact a simple argument with the Hochschild--Serre spectral sequence shows that $N'$ depends solely on $[K_{\textnormal{N}}:K_{\textnormal{N}}']$ and $\lambda$. In particular, we verified that it suffices to prove the auxiliary claim.

Using once more that the exponent of the $p$-torsion in $H^{\ast}(K_{\textnormal{N}},\sigma_{\lambda}^{\circ})$ is bounded by an integer depending solely on $\lambda$ and $[K:K']$, the existence of the maps follows from Theorem~\ref{Theorem_Kostant} and the first part of Lemma~\ref{Lemma_cupproductiso}. The fact that it also matches the Hecke actions follows from Theorem~\ref{Theorem_Lazard} and the second part of Lemma~\ref{Lemma_cupproductiso}.
\end{proof}

\section{Local-global compatibility for Betti cohomology}\label{sec_LGC_for_Betti}
The goal of this final section is to prove Theorem \ref{Theorem_MainLGCThm}, our main result on local–global compatibility for automorphic Galois representations. We deduce the characteristic $0$ statement from Proposition \ref{Proposition_MainLGCProp}, which establishes a more general local–global compatibility result for the Betti cohomology of $\mathrm{Res}_{F/\mathbf{Q}}\mathrm{GL}_n$-locally symmetric spaces. The proof proceeds by induction on $n \geq 1$.

The inductive hypothesis is used to establish local–global compatibility for the contribution of the Borel–Serre boundary. This is carried out in \S\ref{sec_boundarycoh} and is the only point where the results of \S\ref{sec_local_representation_theory} are needed. Our analysis yields a clean description of the boundary cohomology only after discarding $p^N$-torsion, where the exponent $N$ depends only on $n$, $F$, and the chosen local system. Consequently, Proposition \ref{Proposition_MainLGCProp} is also formulated modulo such $p^N$-torsion.

It remains to treat the interior cohomology, where it suffices to prove local–global compatibility for cohomology with coefficients in certain flat $\mathbf{Z}_p$-local systems.\footnote{Namely, those associated with weights that will be called “cohomologically cuspidal”.} This is achieved by following the strategy of \cite{10author}. More precisely, we employ a slightly modified version of the $P$-ordinary degree-shifting argument developed in \cite{CN25}. In the course of the proof, to produce the required congruences with conjugate self-dual automorphic representations of $\mathrm{GL}_{2n}(\mathbf{A}_F)$, we replace the vanishing result of \cite{caraiani_scholze} with Theorem \ref{VanishingUpToBoundedTorsion}. As a consequence, we must also ignore $(M\Delta)^N$-torsion, for some integers $M, N \geq 1$ depending only on $n$ and $F$.

\subsection{Generalities on locally symmetric spaces}\label{subsection_locallysymmetricspaces}
We introduce some key notation, following \cite{New16} and \cite{10author}. If $F$ is a number field, and $G$ is a connected linear algebraic group over $F$, we will write $X^G$ for the symmetric space associated by Borel--Serre \cite{Bor73} to the group $\Res_{F / \mathbf{Q}} G$, $\overline{X}^G$ for its partial compactification, and $\partial X^G = \overline{X}^G - X^G$. 

In the presence of an integral model for $G$ over $\cO_F$, an open compact subgroup $K_G \leq G(\mathbf{A}_F^\infty)$ will be said to be \textit{good} if it is neat, and decomposes as a product $K_G = \prod_v K_{G, v}$ over the set of finite places of $F$, where each $K_{G, v}$ is contained in $G(\cO_{F_v})$. For a finite set of finite places $S$ of $F$, a good subgroup $K_G$ will further be called $S$\textit{-good} if $K_{G,v}=G(\mathcal{O}_{F_v})$ for every finite place $v\notin S$. 

If $K_G$ is a good subgroup, then we introduce the quotient
\[ X^G_{K_G} = G(F) \backslash (X^G \times G(\mathbf{A}^\infty_F) / K_G), \]
and define $\overline{X}^G_{K_G}$, $\partial X^G_{K_G}$ similarly. Then $X^G_{K_G}$ is a smooth manifold with finitely many connected components; $\overline{X}^G_{K_G}$ is a compact manifold with corners; and $\partial X^G_{K_G}$ is a compact manifold with corners.

Following \cite{CN25}, we set $\overline{\mathfrak{X}}_G$ to be the projective limit
\begin{equation*}
    \overline{\mathfrak{X}}_G=\varprojlim_{K_G}\overline{X}_{K_G}^G
\end{equation*}
in the category of compact Hausdorff spaces where the limit runs over good subgroups $K_G\leq G(\mathbf{A}_F^{\infty})$. We similarly define $\partial \mathfrak{X}_G=\varprojlim_{K_G}\partial X_{K_G}^G$. These spaces come equipped with the continuous right regular action of $G(\mathbf{A}_F^{\infty})$. The restriction of this action to any good subgroup $K_G$ makes both $\overline{\mathfrak{X}}_G$ and $\partial\mathfrak{X}_G$ a free $K_G$-space in the sense of \cite[Definition 2.23.]{New16}

To define compactly supported cohomology via the space $\overline{\mathfrak{X}}_G$, we also consider the projective limit
\begin{equation*}
    \mathfrak{X}_G:=\varprojlim_{K_G}X_{K_G}^G
\end{equation*}
and write $j_G\colon\mathfrak{X}_G\hookrightarrow\overline{\mathfrak{X}}_G$ for the open embedding.

Let $S$ be a finite set of finite places of $F$, $R$ be a commutative ring and $\mathcal{V}$ be a smooth $R[K_{\textnormal{G,S}}]$-module. For a $G^S\times K_{G,S}$-space $X$, set $\textnormal{Sh}_{G^S\times K_{G,S}}(X,R)$ to be the abelian category of $G^S\times K_{G,S}$-equivariant sheaves of $R$-modules in the sense of \cite[Definition 2.22]{New16}. By \textit{loc. cit.}, Lemma 2.26, $\mathcal{F}\mapsto \mathcal{F}(\ast)$ yields an equivalence of categories $\textnormal{Sh}_{G^S\times K_{G,S}}(\ast,R)\xrightarrow{\sim}\textnormal{Mod}_{\textnormal{sm}}(G^S\times K_{G,S},R)$. In particular, given $\mathcal{V}\in \textnormal{Mod}_{\textnormal{sm}}(K_{G,S},R)$, inflation to $G^S\times K_{G,S}$ and pullback of equivariant of sheaves along $\overline{\mathfrak{X}}_G\to \ast$ yields an object of $\textnormal{Sh}_{G^S\times K_{G,S}}(\overline{\mathfrak{X}}_G,R)$ that, by abusing the notation, will be denoted by $\mathcal{V}$. By taking derived global sections, we obtain
\begin{equation*}
    R\Gamma(\overline{\mathfrak{X}}_G,\mathcal{V})\in D^+_{\textnormal{sm}}(G^S\times K_{G,S},R).
\end{equation*}
Analogous observations apply to the space $\partial\mathfrak{X}_G$.

On the other hand, by descent (cf. \cite[Lemma 2.24]{New16}), we have an equivalence $\textnormal{Sh}_{K_G}(\overline{\mathfrak{X}}_G,R)\cong\textnormal{Sh}(\overline{X}_{K_G}^G,R)$ between the category of $K_G$-equivariant sheaves on the free $K_G$-space $\overline{\mathfrak{X}}_G$ and sheaves of $R$-modules on $\overline{X}_{K_G}^G$. In particular, $\mathcal{V}$ descends to a sheaf on $\overline{X}_{K_G}^G$ that we keep denoting by $\mathcal{V}$, just as its restriction along $j_{K_G}:X_{K_G}\hookrightarrow \overline{X}_{K_G}^G$. There are natural isomorphisms
\begin{align*}
    R\Gamma(X_{K_G}^G,\mathcal{V}) &\cong R\Gamma(\overline{X}_{K_G}^G,\mathcal{V}) \cong R\Gamma(K_G, R\Gamma(\overline{\mathfrak{X}}_G,\mathcal{V})), \\
    R\Gamma_c(X_{K_G}^G,\mathcal{V}) &:= R\Gamma(\overline{X}_{K_G}^G,j_{K_G,!}\mathcal{V})\cong  R\Gamma(K_G,R\Gamma(\overline{\mathfrak{X}}_G,j_{G,!}\mathcal{V}))
\end{align*}
in $D^+(R)$ where the first isomorphism is induced by restriction along the homotopy equivalence $j_{K_G}$. In particular, we obtain algebra homomorphisms
\begin{equation}\label{equation_primetoSHeckeaction}
    \mathcal{H}(G^S,K_{G}^S)\otimes_{\mathbf{Z}}R\to \textnormal{End}_{D^+(R)}(R\Gamma_{(c)}(X_{K_G}^G,\mathcal{V})).
\end{equation}
The analogous conclusion holds for boundary cohomology as well.

We now specialise to $p$-torsion coefficients. Let $E/\mathbf{Q}_p$ be a finite field extension, $\mathcal{O}$ be its ring of integers and $\varpi\in \mathcal{O}$ be a uniformiser. Set $\Lambda=\mathcal{O}/\varpi^m$. The cohomology of the complexes
\begin{equation*}
    R\Gamma(\overline{\mathfrak{X}}_G,\Lambda),R\Gamma(\overline{\mathfrak{X}}_G,j_{G,!}\Lambda)\in D^+_{\textnormal{sm}}(G(\mathbf{A}_F^{\infty}),\Lambda),
\end{equation*}

recover regular and compactly supported completed cohomology 
\begin{equation*}
    \varinjlim_{K_G}H^{\ast}(X_{K_G},\Lambda), \varinjlim_{K_G}H^{\ast}_c(X_{K_G},\Lambda)\in \textnormal{Mod}_{\textnormal{sm}}(G(\mathbf{A}_F^{\infty}),\Lambda),
\end{equation*} 
respectively cf. \cite[Lemma 2.1.7]{CN25}. 

If $\sigma$ is a finite free $\mathcal{O}$-module with an action of $K_{G,S}$, $\overline{\sigma}:=\sigma\otimes_{\mathcal{O}}\Lambda$ becomes a smooth $\Lambda[K_{G,S}]$-module, fitting into the formalism described above. The projection formula for $R\Gamma(\overline{\mathfrak{X}}_G,-)$ (cf. \cite[Lemma 2.1.8]{CN25}) yields natural isomorphisms
\begin{align*}
    R\Gamma(K_G,R\Gamma(\overline{\mathfrak{X}}_G,\Lambda)\otimes_{\mathcal{O}}\sigma)\cong R\Gamma(X_{K_G}^G,\overline{\sigma}), \\
    R\Gamma(K_G,R\Gamma(\overline{\mathfrak{X}}_G,j_{G,!}\Lambda)\otimes_{\mathcal{O}}\sigma)\cong R\Gamma_c(X_{K_G}^G,\overline{\sigma})
\end{align*}
in $D^+(\Lambda)$. For any subset $T\subset S$, the source of these isomorphisms are naturally isomorphic to objects of $D^+(\mathcal{H}(G^T,\sigma)\otimes_{\mathcal{O}}\Lambda)$ for the convolution algebra $\mathcal{H}(G^T,\sigma)$ of compactly supported $\sigma$-bi-invariant functions $f:G(\mathbf{A}_F^{\infty\cup T})\to \textnormal{End}_{\mathcal{O}}(\sigma)$. In particular, we obtain algebra homomorphisms
\begin{equation*}
    \mathcal{H}(G^T,\sigma)\to \textnormal{End}_{D^+(\Lambda)}(R\Gamma_{(c)}(X_{K_G}^G,\overline{\sigma})).
\end{equation*}
Their restriction along $\mathcal{H}(G^S,K_G^S)\xrightarrow{}\mathcal{H}(G^T,\sigma)$ recovers the prime-to-$S$ Hecke action \eqref{equation_primetoSHeckeaction}. The analogous conclusion holds for boundary cohomology just as well and the natural maps between them respect the Hecke action.

\subsubsection{The Borel--Serre boundary}
Suppose now that $G$ is a connected reductive group. 
We have a $G(\mathbf{A}_{F}^{\infty})$-equivariant decomposition
\begin{equation*}
    \partial \mathfrak{X}_G=\coprod_{Q}\mathfrak{X}_G^{Q}
\end{equation*}
into locally closed subspaces
\begin{equation*}
    \mathfrak{X}_G^Q:=(\mathfrak{X}_Q\times G(\mathbf{A}_{F}^{\infty}))/Q(\mathbf{A}_{F}^{\infty})
\end{equation*}
labelled by representatives of the $G(F)$-conjugacy classes of $F$-rational parabolic subgroups $Q\leq G$, and $G(\mathbf{A}_{F}^{\infty})$ is equipped with its profinite topology. This yields a stratification
\begin{equation}\label{equation_BorelSerreStratification}
    \partial X_{K_G}=\coprod_{Q}X_{K_G}^Q
\end{equation}
for any good subgroup $K_G\subset G(\mathbf{A}_{F}^{\infty})$.

We similarly introduce the compactified variants
\begin{equation*}
    \overline{\mathfrak{X}}_{G}^{Q}:=(\overline{\mathfrak{X}}_Q\times G(\mathbf{A}_{F}^{\infty}))/Q(\mathbf{A}_{F}^{\infty})
\end{equation*}
and write $j_G^Q$ for the $G(\mathbf{A}_F^{\infty})$-equivariant open embedding $\mathfrak{X}_G^Q\hookrightarrow\overline{\mathfrak{X}}_G^Q$.

Write now $Q=M\rtimes N$ and consider a finite set of finite places $S$ of $F$ satisfying the following.
\begin{itemize}
    \item For $v\notin S$ the base change $G\times_FF_v$ admits a reductive model over $\mathcal{O}_{F_v}$.
\end{itemize}

Given a compact open subgroup $K_{G,S}\leq G(\widehat{\mathcal{O}}_{F,S})$, to state the forthcoming lemma, we introduce a $G(\mathbf{A}_{F}^{\{\infty\}\cup S})\times K_{G,S}$-equivariant subspace $\overline{\mathfrak{X}}_{G,K_{G,S}}^Q\subset \overline{\mathfrak{X}}_{G}^Q$. Write $G_S=\coprod_{i=1}^rQ_{S}g_iK_{G,S}$ and let $g_1=1$. Consider the decomposition
    \begin{equation}\label{decomposition}
        \overline{\mathfrak{X}}_{G}^Q\cong \coprod_{i=1}^r\left(\overline{\mathfrak{X}}_Q\times(G^S\times K_{G,S})\right)/\left(Q^S\times ((g_iK_{G,S}g_{i}^{-1})\cap Q_{S})\right),
    \end{equation}
    where the quotient on the $i$-th factor is taken with respect to the action $(x,g)\cdot q=(xq,g_i^{-1}q^{-1}g_ig)$. This is a $G^S\times K_{G,S}$-equivariant decomposition and we define $\overline{\mathfrak{X}}_{G,K_{G,S}}^{Q}$ to be the factor corresponding to $g_1$.

    We similarly define the space $\mathfrak{X}_{G,K_{G,S}}^Q\subset \overline{\mathfrak{X}}_{G,K_{G,S}}^{Q}$ and denote by $j_{G,K_{G,S}}^{Q}$ the open embedding.
    
    Using the Iwasawa decomposition and \ref{decomposition}, one sees that both $\overline{\mathfrak{X}}_G^Q$ and $\overline{\mathfrak{X}}_{G,K_{G,S}}^Q$ are compact Hausdorff spaces admitting a free $K_G$-action.
    
    \begin{lemma}\label{Lemma_BorelSerreLemma}
        Suppose we are given an $F$-rational parabolic subgroup $Q=M\rtimes N\leq G$, a finite set of finite places $S$ of $F$ such that for $v\notin S$ the base change $G\times_FF_v$ admits a reductive model over $\mathcal{O}_{F_v}$. For every compact open subgroup $K_{G,S}\leq G_S$, we have a commutative diagram
        \begin{equation*}
\begin{tikzcd}
	{R\Gamma(\overline{\mathfrak{X}}_{G,K_{G,S}}^Q,j_{G,K_{G,S},!}^Q\mathcal{O}/\varpi^m)} && {\textnormal{Ind}^{G^S\times K_{G,S}}_{Q^S\times K_{Q,S}}R\Gamma(\overline{\mathfrak{X}}_M,j_{M,!}\mathcal{O}/\varpi^m)} \\
	{R\Gamma(\overline{\mathfrak{X}}_{G,K_{G,S}}^Q,\mathcal{O}/\varpi^m)} && {\textnormal{Ind}^{G^S\times K_{G,S}}_{Q^S\times K_{Q,S}}R\Gamma(\overline{\mathfrak{X}}_M,\mathcal{O}/\varpi^m)}
	\arrow["\sim", from=1-1, to=1-3]
	\arrow["{i_{c,G,K_{G,S}}^Q}", from=1-1, to=2-1]
	\arrow["{\textnormal{Ind}_Q^G(i_{c,M})}", from=1-3, to=2-3]
	\arrow["\sim", from=2-1, to=2-3]
\end{tikzcd}
    \end{equation*}
    in $D^+_{\textnormal{sm}}(\mathcal{O}/\varpi^m[G^S\times K_{G,S}])$, where the horizontal maps are isomorphisms and the vertical maps are the ones induced by forgetting the support.
    \end{lemma}
    \begin{proof}
        This is Lemma 3.2 of \cite{Hev25}.
    \end{proof}

\subsection{Formulating local-global compatibility}
\subsubsection{The cohomology of $\textnormal{Res}_{F/\mathbf{Q}}\textnormal{GL}_n$-locally symmetric spaces}\label{subsubsection_GL_nlocallysymmetricspaces}
Let $F$ be a number field and $n\geq 1$ be an integer.
We write $T_n\leq B_n \leq \textnormal{GL}_{n,\mathcal{O}_F}$ for the subgroups of diagonal and upper-triangular matrices, respectively. For a prime $\ell$, let $S_{\ell}$ denote the set of $\ell$-adic places of $F$.

Let $p$ be a prime, and let $E\subset \overline{\mathbf{Q}}_p$ be a subfield, finite over $\mathbf{Q}_p$, and large enough to contain the images of all embeddings $F\hookrightarrow\overline{\mathbf{Q}}_p$. We write $\mathcal{O}$ for the ring of integers of $E$, $\varpi\in \mathcal{O}$ for a choice of uniformiser and $k=\mathcal{O}/\varpi$ for the residue field. The character group $X^\ast((\Res_{F / \mathbf{Q}} T_n)_E)$ may be identified with $(\mathbf{Z}^n)^{\Hom(F, E)}$ in the usual way, and the subset of weights that are dominant with respect to the Borel subgroup $(\Res_{F  / \mathbf{Q}} B_n)_E$ with the subset $(\mathbf{Z}^n_+)^{\Hom(F, E)}\subset (\mathbf{Z}^n)^{\Hom(F,E)}$ of tuples $\lambda = (\lambda_\tau)_{\tau \in \Hom(F, E)}$ satisfying 
\[ \lambda_{\tau, 1} \geq \lambda_{\tau, 2} \geq \dots \geq \lambda_{\tau, n} \]
for each $\tau \in \Hom(F, E)$. Associated to each such $\lambda$ is an $E[\GL_n(F \otimes_\mathbf{Q} \mathbf{Q}_p)]$-module $V_\lambda$. It may be described explicitly as follows. Let $V_{\lambda_\tau} = \operatorname{Ind}^{\GL_{n, E}}_{\overline{B}_{n, E}} \lambda_\tau$ denote the algebraic induction as in \cite[I, Ch. 3]{jantzen2003}; it is an irreducible representation of highest weight $\lambda_\tau$.
If $v$ is the place of $F$ induced by $\tau$, we make $\GL_n(F_v)$ act on $V_{\lambda_\tau}$ via the map $\tau \colon \GL_n(F_v) \to \GL_n(E)$. Finally, we set $V_\lambda = \otimes_{\tau \in \Hom(F, E)} V_{\lambda_\tau}$. 

We use the same recipe to define an $\cO[\GL_n(\cO_F \otimes_{\mathbf{Z}} \mathbf{Z}_p)]$-module associated to any $\lambda \in (\mathbf{Z}^n_+)^{\Hom(F, E)}$: we write $\mathcal{V}_{\lambda_\tau} = \Ind^{\GL_{n, \cO}}_{\overline{B}_{n, \cO}} \lambda_\tau$ for the algebraic induction, an $\cO$-lattice in $V_{\lambda_\tau}$, and define $\mathcal{V}_\lambda = \otimes_\tau \mathcal{V}_{\lambda_\tau}$.

Given a dominant weight $\lambda=(\lambda_{\tau})_{\tau:F\hookrightarrow E} \in (\mathbf{Z}^n_+)^{\Hom(F,E)}$ and a $p$-adic place $v|p$ of $F$, we will write $\lambda_v=(\lambda_{\tau})_{\tau\colon F_v\hookrightarrow E}\in (\mathbf{Z}^n_+)^{\Hom(F_v,E)}$ where we identify $\Hom(F_v,E)$ with the subset of $\Hom(F,E)$ of embeddings inducing $v$. We further define the $E[\textnormal{GL}_n(F_v)]$-module $V_{\lambda_v}:=\otimes_{\tau\in \Hom(F_v,E)}V_{\lambda_{\tau}}$ and the $\textnormal{GL}_n(\mathcal{O}_{F_v})$-stable $\mathcal{O}$-lattice $\mathcal{V}_{\lambda_v}=\otimes_{\tau\in \Hom(F_v,E)}\mathcal{V}_{\lambda_{\tau}}$.

For a finite place place $v$ of $F$, let $\mathfrak{B}(F_v,n)$ denote the set of Bernstein blocks of $\textnormal{Rep}_{\textnormal{sm}}(\textnormal{GL}_n(F_v),\overline{\mathbf{Q}}_p)$. Given $\Omega=(\Omega_v)_{v\in S_p}\in \prod_{v\in S_p}\mathfrak{B}(F_v,n)$, we set $\sigma_{\Omega}=\otimes_{v\in S_p}\sigma_{\Omega_v}$, a smooth $\overline{\mathbf{Q}}_p[\textnormal{GL}_n(\mathcal{O}_F\otimes_{\mathbf{Z}}\mathbf{Z}_p)]$-module, where $\sigma_{\Omega_v}$ is the smooth $\overline{\mathbf{Q}}_p[\textnormal{GL}_n(\mathcal{O}_{F_v})]$-module introduced in \S\ref{Subsubsection_smooth_char0}.

We will refer to a tuple of $\mathcal{O}$-integral $\Omega_v$-systems\footnote{See Definition~\ref{Definition_integralsystem}.} $\mathcal{L}=(\mathcal{L}_v)_{v\in S_p}$ as an $\mathcal{O}$-integral $\Omega$-system and we set $\sigma_{\Omega}^{\circ}=\otimes_{v\in S_p}\sigma_{\Omega_v}^{\circ}$, a $\textnormal{GL}_n(\mathcal{O}_F\otimes_{\mathbf{Z}}\mathbf{Z}_p)$-stable $\mathcal{O}$-lattice in $\sigma_{\Omega}^{\vee}$.

Combining the discussion of the previous paragraphs, given a dominant weight $\lambda\in (\mathbf{Z}^n_+)^{\Hom(F,E)}$, a collection of Bernstein blocks $\Omega\in \prod_{v\in S_p}\mathfrak{B}(F_v,n)$ and an $\mathcal{O}$-integral $\Omega$-system $\mathcal{L}$, we set
\begin{equation*}
    \sigma_{\Omega,\lambda}^{\circ}:=\sigma_{\Omega}^{\circ}\otimes_{\mathcal{O}}\mathcal{V}_{\lambda}\in \textnormal{Mod}(\mathcal{O}[\textnormal{GL}_n(\mathcal{O}_F\otimes_{\mathbf{Z}}\mathbf{Z}_p)])
\end{equation*}
and note that it is finite free over $\mathcal{O}$. In particular, for an integer $m\geq 1$, $\sigma_{\Omega,\lambda}^{\circ}/\varpi^m$ is a smooth $\mathcal{O}/\varpi^m[\textnormal{GL}_n(\mathcal{O}_{F}\otimes_{\mathbf{Z}}\mathbf{Z}_p)]$-module, fitting into the formalism of \ref{subsection_locallysymmetricspaces}. 

Given any finite set of finite places $T$ of $F$ and a $T$-good subgroup $K\leq \textnormal{GL}_n(\mathbf{A}_F^{\infty})$, we obtain algebra homomorphisms
\begin{equation*}
    \mathcal{H}((\textnormal{GL}_{n})^{T},\sigma_{\Omega,\lambda}^{\circ})\to \textnormal{End}_{D^+(\mathcal{O}/\varpi^m)}(R\Gamma(X_{K}^{\textnormal{GL}_n},\sigma_{\Omega,\lambda}^{\circ}/\varpi^m)).
\end{equation*}
Moreover, we define the homotopy limit
\begin{equation*}
    R\Gamma(X_K^{\textnormal{GL}_n},\sigma_{\Omega,\lambda}^{\circ}):=\varprojlim _{m\geq 1}R\Gamma(X_K^{\textnormal{GL}_n},\sigma_{\Omega,\lambda}^{\circ}/\varpi^m)
\end{equation*}
in $D^+(\mathcal{O})$. In particular, we obtain an algebra homomorphism
\begin{equation*}
    \mathcal{H}((\textnormal{GL}_{n,F})^{ T},\sigma_{\Omega,\lambda}^{\circ})\to \textnormal{End}_{D^+(\mathcal{O})}(R\Gamma(X_{K}^{\textnormal{GL}_n},\sigma_{\Omega,\lambda}^{\circ})).
\end{equation*}

Finally, we discuss the `unnormalised Satake transform' associated to a standard parabolic subgroup of $\textnormal{GL}_n$. Let $Q$ be a standard parabolic subgroup of $\GL_{n, \cO_F}$, with standard (i.e.\ block diagonal) Levi subgroup $M$, and unipotent radical $N$. Writing,
\begin{equation*}
    \mathbf{T}^S=\mathbf{T}^S_n:=\mathcal{H}(\textnormal{GL}_n(\mathbf{A}_F^{\infty,S}),\textnormal{GL}_n(\widehat{\mathcal{O}}_F^S))\otimes_{\mathbf{Z}}\mathcal{O},
\end{equation*}
we see that if $M = \GL_{n_1 ,\cO_F} \times \dots \times \GL_{n_r, \cO_F}$, then we can naturally identify
\[ \mathcal{H}(M(\mathbf{A}_F^{S, \infty}), M(\widehat{\cO}_F^S)) \otimes_\mathbf{Z} \mathcal{\cO} = \mathbf{T}^S_{n_1} \otimes_\cO \dots \otimes_{\mathcal{O}} \mathbf{T}^S_{n_r}. \]
We write $\mathbf{T}^S_M$ for this algebra. The map `restrict to $Q$ and integrate along $N$' gives an $\cO$-algebra homomorphism $\mathcal{S}_M \colon \mathbf{T}^S \to \mathbf{T}^S_M$, that we call the unnormalised Satake transform (see e.g.\ \cite[\S 2.2.6]{New16}; an explicit description of this homomorphism is given in \cite[Lemma 4.6]{New16}).

\subsubsection{Axiomatising local-global compatibility}
As in the previous subsection, let $E\subset \overline{\mathbf{Q}}_p$ be a subfield, finite over $\mathbf{Q}_p$ and large enough to contain all field embeddings $F\hookrightarrow \overline{\mathbf{Q}}_p$. Let $\mathcal{O}\subset E$ denote its ring of integers and $\varpi\in \mathcal{O}$ a fixed choice of uniformiser.

Let $S$ be a finite set of finite places of $F$ containing the set $S_p$ of $p$-adic places of $F$. We set $\mathbf{T}^S:=\mathcal{H}(\textnormal{GL}_n(\mathbf{A}_F^{\infty,S}),\textnormal{GL}_n(\widehat{\mathcal{O}}_F^S))\otimes_{\mathbf{Z}}\mathcal{O}$. For a finite place $v\notin S$ of $F$, let $P_v(X)\in \mathbf{T}^S[X]$ denote the Hecke polynomial
\begin{equation*}
    \sum_{i=0}^{n}(-1)^iq_v^{\frac{i(i-1)}{2}}T_{v,i}X^{n-i}\in \mathcal{H}(\textnormal{GL}_n(F_v),\textnormal{GL}_n(\mathcal{O}_{F_v}))\otimes_{\mathbf{Z}}\mathcal{O}
\end{equation*}
where for $i=0,\dots,n$,
\begin{equation*}
    T_{v,i}:=[\textnormal{GL}_n(\mathcal{O}_{F_v})\textnormal{diag}(\underbrace{\varpi_v,...,\varpi_v}_{i\textnormal{ times}},1,...,1)\textnormal{GL}_n(\mathcal{O}_{F_v})]\in \mathcal{H}(\textnormal{GL}_n(F_v),\textnormal{GL}_n(\mathcal{O}_{F_v})).
\end{equation*}
For an unramified $\overline{\mathbf{Q}}_p$-representation $\pi_v$ of $\textnormal{GL}_n(F_v)$, specialisation of $P_v(X)$ along the corresponding point $\mathbf{T}^S\to \overline{\mathbf{Q}}_p$ yields the characteristic polynomial of $\textnormal{Frob}_v$ acting on $\textnormal{rec}_{F_v}^T(\pi_v)$.

\begin{definition}
    We say that a local Artinian $\mathcal{O}$-algebra $A$ admitting a surjective algebra homomorphism $x\colon\mathbf{T}^S\to A$ is \textit{of Galois type} if there exists a continuous determinant $D_x:G_{F,S}\to A$ satisfying the following condition.
    \begin{itemize}
        \item For all finite places $v\notin S$ of $F$, we have $D_x(X-\textnormal{Frob}_v)=P_v(X).$
    \end{itemize}

    More generally, we say that an Artinian $\mathcal{O}$-algebra $A$ admitting an algebra homomorphism $x\colon\mathbf{T}^S\to A$ is of Galois type if $A^S:=x(\mathbf{T}^{S})$ is of Galois type.\footnote{In particular, it means that $A^S$ must be local.}
\end{definition}

In addition to $S$, consider a $p$-adic place $a$ of $F$, a dominant weight $\lambda\in (\mathbf{Z}_+^n)^{\Hom(F,E)}$, a collection of Bernstein blocks $\Omega=(\Omega_v)_{v\in S_p}\in \prod_{v\in S_p}\mathfrak{B}(F_v,n)$ and an $\mathcal{O}$-integral $\Omega$-system $\mathcal{L}=(\mathcal{L}_v)_{v\in S_p}$ (see Definition~\ref{Definition_integralsystem}). The tuple $(S,a,\lambda_a,\mathcal{L}_a)$ and a choice of an $(\lambda_a,\mathcal{L}_a)$-admissible integer $N\geq 0$ (cf. Definition \ref{Definition_lambdaelladmissible}) allows us to introduce the commutative Hecke algebras
\begin{equation*}
    \mathbf{T}_a:=\mathbf{T}^S\otimes_{\mathcal{O}}\mathfrak{Z}_{\textnormal{GL}_n(F_a),\mathcal{L}_a,\lambda_a}^{\circ}\leq
\mathbf{T}_{a}^{B}:=\mathbf{T}^S\otimes_{\mathcal{O}}\mathcal{H}^B(\sigma_{\Omega_a,\lambda_a})^{\circ},
\end{equation*}
where $\mathfrak{Z}_{\textnormal{GL}_n(F_a),\mathcal{L}_a,\lambda_a}^{\circ}\leq\mathcal{H}^B(\sigma_{\Omega_a,\lambda_a})^{\circ}:=\mathcal{H}^B(\sigma_{\Omega_a,\lambda_a})^{\circ}_N$ are introduced in \eqref{equation_integralmodel}. The appearing Hecke algebras at $a$ are $\mathcal{O}$-models of the $\overline{\mathbf{Q}}_p$-algebras $\mathfrak{Z}_{\textnormal{GL}_n(F_a),\Omega_a}\leq\mathcal{H}(\sigma_{\Omega_a}^{\vee})$.

For a continuous $n$-dimensional determinant $\overline{D}:G_{F,S}\to \mathcal{O}/\varpi$, a finite extension $L_a/F_a$ and integers $h_1\leq h_2$, in Theorem \ref{Thm_WWE_global} we introduced the quotient $R_{\overline{D}}^{\{a\},L_a,[h_1,h_2]}$ of $R_{\overline{D}}$ satisfying the following property.
\begin{itemize}
    \item An $A$ point $x:R_{\overline{D}}\to A$ for an Artinian $\mathcal{O}$-algebra, with corresponding lift $D_x:G_{F,S}\to A$, factors through $R_{\overline{D}}^{\{a\},L_a,[h_1,h_2]}$ if and only if there is a Cayley--Hamilton representation $(\mathcal{E},\rho:G_{F,S}\to \mathcal{E}^{\times},D':\mathcal{E}\to A)$ such that $D=D'\circ \rho$ and $\mathcal{E}$ is torsion semistable with Hodge--Tate weights within $[h_1,h_2]$ as a $A[G_{L_a}]$-module.
\end{itemize}

On the other hand, for a continuous $n$-dimensional determinant $\overline{D}_a:G_{F_a}\to \mathcal{O}/\varpi$, in Definition \ref{def_defnring_notation} we set $R^{\lambda_a,\Omega_a}_{\overline{D}_a}$ to be the $\mathcal{O}$-flat quotient of the unrestricted pseudodeformation ring $R_{\overline{D}_a}$ satisfying the following property.
\begin{itemize}
    \item An $\mathcal{O}$-algebra homomorphism $x:R_{\overline{D}_a}\to \overline{\mathbf{Q}}_p$ with corresponding semisimple Galois representation $\rho_x:G_{F_a}\to \textnormal{GL}_n(\overline{\mathbf{Q}}_p)$ factors through $R^{\lambda_a,\Omega_a}_{\overline{D}_a}$ if and only if $\rho_x$ is potentially semistable with $\tau$-labelled Hodge--Tate weights
    \begin{equation*}
        \textnormal{HT}(\rho_x)=\{\lambda_{\tau, 1}+n-1>...>\lambda_{\tau,n}\}_{\tau:F_a\hookrightarrow E}
    \end{equation*}
    and with $(\textnormal{rec}^T_{F_a})^{-1}(\textnormal{WD}(\rho_x)^{F-ss})$ lying in $\Omega_a$.
\end{itemize}
We further recall the map from Corollary \ref{Corollary_InperpolationOfSSLL} interpolating the semisimple (Tate-normalised) local Langlands correspondence
\begin{equation*}
    \eta\colon \mathfrak{Z}_{\textnormal{GL}_n(F_a),\Omega_a}\to R_{\overline{D}_a}^{\lambda_a,\Omega_a}\otimes_{\mathcal{O}}\overline{\mathbf{Q}}_p
\end{equation*}
and the integral $\mathcal{O}$-model
\begin{equation*}
    \mathfrak{Z}_{\overline{D}_a}^{\circ}:=\mathfrak{Z}_{\textnormal{GL}_n(F_a),\mathcal{L}_a,\lambda_a,\overline{D}_a}^{\circ}\leq \mathfrak{Z}_{\textnormal{GL}_n(F_a),\mathcal{L}_a,\lambda_a}^{\circ}\leq  \mathfrak{Z}_{\textnormal{GL}_n(F_a),\Omega_a}
\end{equation*}
lying in $\mathcal{H}^{B}(\sigma_{\Omega,\lambda})^{\circ}$ and descending $\eta$ to an $\mathcal{O}$-algebra homomorphism
\begin{equation*}
    \eta^{\circ}\colon\mathfrak{Z}_{\overline{D}_a}^{\circ}\to R_{\overline{D}_a}^{\lambda_a,\Omega_a}.
\end{equation*}

For a $\mathbf{T}_a$-algebra $A$, we will write $\textnormal{nat}_A\colon\mathfrak{Z}_{\overline{D}_a}^{\circ}\to A$ for the map induced by the canonical map $\mathfrak{Z}_{\overline{D}_a}^{\circ}\to\mathbf{T}_a$. We can finally axiomatise local-global compatibility.
\begin{definition}\label{def_LGC}
    Suppose that $A$ is an Artinian $\mathcal{O}$-algebra admitting an $\mathcal{O}$-algebra homomorphism $x\colon\mathbf{T}_a\to A$ that is of Galois type as a $\mathbf{T}^S$-algebra with corresponding determinant $D_x:G_{F,S}\to A$ and $D_{x,a}:=D_{x}|_{G_{F_a}}$. We say that $\textnormal{LGC}(S,a,\lambda,\Omega,\mathcal{L})$ holds for $D_x$ if the following properties are satisfied.
    \begin{enumerate}
        \item\textit{(Global condition)} 
        There is a factorisation
        \begin{equation}
            \begin{tikzcd}
	{R_{\overline{D}_x}} & A \\
	{R_{\overline{D}_x}^{\{a\}, L_a,[h_1,h_2]},} & 
	\arrow["{D_{x}}", from=1-1, to=1-2]
	\arrow[two heads, from=1-1, to=2-1]
	\arrow[dotted, from=2-1, to=1-2]
\end{tikzcd}
        \end{equation}
        where $L_a/F_a$ and $h_{1}, h_{2}\in \mathbf{Z}$ only depend on $\Omega$, $\lambda$ and $a$.
        \item\textit{(Local condition)} There exists a dotted arrow making the following diagram commutative
        \begin{equation}\label{equation_LGCdiagram}
            \begin{tikzcd}
	{R_{\overline{D}_{x,a}}} & A \\
	{R_{\overline{D}_{x,a}}^{\lambda_a,\Omega_a}} & {\mathfrak{Z}^{\circ}_{\overline{D}_{x,a}}.}
	\arrow["{D_{x,a}}", from=1-1, to=1-2]
	\arrow[two heads, from=1-1, to=2-1]
	\arrow[dotted, from=2-1, to=1-2]
	\arrow["{\textnormal{nat}_A}"', from=2-2, to=1-2]
	\arrow["{\eta^{\circ}}", from=2-2, to=2-1]
\end{tikzcd}
        \end{equation}
    \end{enumerate}
\end{definition}

\subsubsection{Local-global compatibility for Betti cohomology of $\textnormal{Res}_{F/\mathbf{Q}}\textnormal{GL}_n$-locally symmetric spaces}\label{subsubsection_LGCstatement}
We keep the notation from \ref{subsubsection_GL_nlocallysymmetricspaces} and for the rest of the subsection we further impose the following assumption on $F$.
\begin{assumption}\label{assumption_CM}
The field $F$ is a CM number field that contains an imaginary quadratic subfield $F_0\leq F$ in which $p$ splits. In particular, each $p$-adic place of its maximal totally real subfield $F^+$ splits in $F$.
\end{assumption}
We write $\overline{S}_p$ for the set of $p$-adic places of $F^+$ and fix a choice of decomposition $S_p = \widetilde{S}_p \sqcup \widetilde{S}_p^c$.

We are now ready to state our main technical result on local-global compatibility. We need to induct on $n$, so we will proceed on an axiomatic basis with a clearly stated hypothesis. We suppose $F$ and $S$ to be fixed, and fix too a choice of place $a \in \widetilde{S}_p$ lying above a place $\overline{a} \in \overline{S}_p$. We formulate hypothesis $P(n, a, S, T)$, where $n \geq 1$ and $T \in \mathbf{T}^S$:
\begin{itemize}
    \item Let $\lambda \in (\mathbf{Z}^n_+)^{\Hom(F, E)}$, $\Omega\in \prod_{v\in S_p}\mathfrak{B}(F_v,n)$, $\mathcal{L}$ be an $\mathcal{O}$-integral $\Omega$-system. Then we can find integers $N_1, N_2 \geq 1$ with the following properties:
    \begin{itemize}
        \item $N_1$ depends only on $n$ and $F$.
        \item $N_2$ depends only on $n$, $F$, $a$, $\lambda$ and $\mathcal{L}$.
        \item For any $q \in \mathbf{Z}$, $S\setminus\{a\}$-good subgroup $K \leq \GL_n(\mathbf{A}_F^\infty)$, and for any integer $m \geq 1$, there exist ideals $J \leq \mathbf{T}_a(H^q(X_K^{\textnormal{GL}_n}, \sigma_{\Omega,\lambda}^{\circ}/ \varpi^m))$ and $J_c \leq \mathbf{T}_a(H^q_c(X_K^{\textnormal{GL}_n}, \sigma_{\Omega,\lambda}^{\circ}/ \varpi^m))$ satisfying $J^{N_1}_{(c)} = 0$ and such that for any maximal ideal $\mathfrak{m}\trianglelefteq \mathbf{T}^S$ with residue field $k$ supported on $H^q_{(c)}(X_K^{\textnormal{GL}_n},\sigma_{\Omega,\lambda}^{\circ}/\varpi^m)$, the localisation 
        \[ \mathbf{T}:=\mathbf{T}_a(H^q_{(c)}(X_K^{\textnormal{GL}_n}, \sigma_{\Omega,\lambda}^{\circ}/ \varpi^m))_{\mathfrak{m}}/(J_{(c)})_{\mathfrak{m}} \]
        satisfies the following properties:
       \begin{enumerate}
           \item As a $\mathbf{T}^S$-algebra, it is of Galois type with associated determinant $D_{\mathfrak{m}}:G_{F,S}\to \mathbf{T}$.
           \item The determinant $D_{\mathfrak{m}}:G_{F,S}\xrightarrow{}\mathbf{T}\to \mathbf{T}/\mathbf{T}[p^{N_2}T^2]$ satisfies $\textnormal{LGC}(S,a,\lambda,\Omega,\mathcal{L}).$
       \end{enumerate}
    \end{itemize}
\end{itemize}
In light of the main result of \cite{Sch15}, the hypothesis admits a slightly different formulation that will be more convenient for us to work with.
\begin{lemma}\label{lemma_equivalent_LGC_hypothesis}
    The hypothesis $P(n,a,S,T)$ is equivalent to the following hypothesis.
    \begin{itemize}
    \item Let $\lambda \in (\mathbf{Z}^n_+)^{\Hom(F, E)}$, $\Omega\in \prod_{v\in S_p}\mathfrak{B}(F_v,n)$, $\mathcal{L}$ be an $\mathcal{O}$-integral $\Omega$-system. Then we can find integers $N_1, N_2 \geq 1$ with the following properties:
    \begin{itemize}
        \item $N_1$ depends only on $n$ and $F$.
        \item $N_2$ depends only on $n$, $F$, $a$, $\lambda$ and $\mathcal{L}$.
        \item For any $q \in \mathbf{Z}$, $S\setminus\{a\}$-good subgroup $K \leq \GL_n(\mathbf{A}_F^\infty)$, integer $m \geq 1$, and for any maximal ideal $\mathfrak{m}\trianglelefteq \mathbf{T}^S$ with residue field $k$ supported on $H^q_{(c)}(X_K^{\textnormal{GL}_n},\sigma_{\Omega,\lambda}^{\circ}/\varpi^m)$, there exist ideals $J^{\mathfrak{m}} \leq \mathbf{T}_a(H^q(X_K^{\textnormal{GL}_n}, \sigma_{\Omega,\lambda}^{\circ}/ \varpi^m))_{\mathfrak{m}}$ and $J_{c}^{\mathfrak{m}} \leq \mathbf{T}_a(H^q_c(X_K^{\textnormal{GL}_n}, \sigma_{\Omega,\lambda}^{\circ}/ \varpi^m))_{\mathfrak{m}}$ satisfying $(T_n^2p^{N_2}J_{(c)}^{\mathfrak{m}})^{N_1} = 0$ and such that the $\mathbf{T}_a$-algebra 
        \[ \mathbf{T}:=\mathbf{T}_a(H^q_{(c)}(X_K^{\textnormal{GL}_n}, \sigma_{\Omega,\lambda}^{\circ}/ \varpi^m))_{\mathfrak{m}}/J_{(c)}^{\mathfrak{m}} \]
        satisfies the following properties:
       \begin{enumerate}
           \item As a $\mathbf{T}^S$-algebra, it is of Galois type with associated determinant $D_{\mathfrak{m}}:G_{F,S}\to \mathbf{T}$.
           \item The determinant $D_{\mathfrak{m}}:G_{F,S}\xrightarrow{}\mathbf{T}$ satisfies $\textnormal{LGC}(S,a,\lambda,\Omega,\mathcal{L}).$
       \end{enumerate}
    \end{itemize}
\end{itemize}
\end{lemma}
\begin{proof}
    Fix $\lambda,\Omega,\mathcal{L},q,K$ and $m$ as in the statement and let $J_{(c)}$ be the ideal provided by $P(n,a,S,T)$. Set $A_{(c)}:= \mathbf{T}_a(H^q_{(c)}(X_K^{\textnormal{GL}_n},\sigma_{\Omega,\lambda}^{\circ}/\varpi^m))$.
    
    For a choice of maximal ideal as in the statement, we can set $J_{(c)}^{\mathfrak{m}}$ to be the kernel of the surjection
    \begin{equation*}
        A_{(c,)\mathfrak{m}}=\mathbf{T}_a(H^q_{(c)}(X_K^{\textnormal{GL}_n},\sigma_{\Omega,\lambda}^{\circ}/\varpi^m))_{\mathfrak{m}}\to 
        \big(A_{(c,)\mathfrak{m}}/J_{(c,)\mathfrak{m}}\big)\big/\Big(\big( A_{(c,)\mathfrak{m}}/J_{(c,)\mathfrak{m}}\big)[p^{N_2}T^2]\Big).
    \end{equation*}

    Conversely, assume that we are given ideals $J_{(c)}^{\mathfrak{m}}$ as in the statement for every maximal ideal $\mathfrak{m}$ of residue field $k$ and supported on $H^q_{(c)}(X_K^{\textnormal{GL}_n},\sigma_{\Omega,\lambda}^{\circ}/\varpi^m)$. By \cite{Sch15}, there is an integer $M_1\geq 0$ depending only on $n$ and $[F:\mathbf{Q}]$ and ideals $I_{(c)}\trianglelefteq A_{(c)}= \mathbf{T}_a(H^q_{(c)}(X_K^{\textnormal{GL}_n},\sigma_{\Omega,\lambda}^{\circ}/\varpi^m))$ with $I_{(c)}^{M_1}=0$ such that $A_{(c,)\mathfrak{m}}/I_{(c,)\mathfrak{m}}$ is of Galois type for every $\mathfrak{m}$. Then $J_{(c)}:=I_{(c)}+T^2p^{N_2}\sum_{\mathfrak{m}}\Big(J_{(c)}^{\mathfrak{m}}A_{(c)}\Big)\trianglelefteq A_{(c)}=\mathbf{T}_a\big(H^q_{(c)}(X_{K},\sigma_{\Omega,\lambda}/\varpi^m)\big)$ is so that $J^{M_1+N_1}_{(c)}=0$ and the corresponding localised quotients $A_{(c,)\mathfrak{m}}/J_{(c,)\mathfrak{m}}$ are of Galois type satisfying local-global compatibility modulo $(I_{(c,)\mathfrak{m}}+J_{(c)}^{\mathfrak{m}})/J_{(c,)\mathfrak{m}}\leq \Big(A_{(c,)\mathfrak{m}}/J_{(c,)\mathfrak{m}}\Big)[T^2p^{N_2}]$. 
\end{proof}

Our main result on local-global compatibility is then the following.

\begin{prop}\label{Proposition_MainLGCProp}
     Given an integer $n\geq 1$, a finite set $S\subset S(F)$ that is invariant under complex conjugation and contains the $p$-adic places, and $a\in S_p$. Suppose that the following are satisfied.
     \begin{enumerate}
         \item For $v\notin S$ with residual characteristic $l$, either $l$ is unramified in $F$ and $S_l(F)\cap S\neq \emptyset$ or $l$ splits in an imaginary quadratic subfield of $F$.
         \item There is a place $\overline{b} \in \overline{S}_p - \{ \overline{a} \}$ such that
     \begin{equation}\label{eqn_enough_shifting_new} \sum_{\overline{v} \in \overline{S}_p - \{ \overline{a}, \overline{b} \} } [F^+_{\overline{v}} : \mathbf{Q}_p] > \frac{1}{2}[F^+ :  \mathbf{Q}]. 
     \end{equation}
     \end{enumerate}
        
    Then there exists an \textit{explicit} element $T_n\in \mathbf{T}^S$ satisfying the following.
    \begin{itemize}
        \item For any $\textnormal{GL}_n(\widehat{\mathcal{O}}_{F}^S)$-spherical cuspidal automorphic representation $\pi$ of $\textnormal{GL}_n(\mathbf{A}_F)$, $T_n$ acts on $\pi\otimes(\chi\circ\det)$ by a non-zero scalar for some finite order Hecke character $\chi$.
        \item Hypothesis $P(n,a,S,T_n)$ holds.
    \end{itemize}
\end{prop}
We will make explicit the definition of the Hecke operator $T_n$ in \S\ref{subsubsection_theelementT_n} once we introduced the quasi-split unitary group $\widetilde{G}$ admitting $\textnormal{Res}_{F/F^+}\textnormal{GL}_n$ as a Levi subgroup. The Hecke operator will essentially be given by the image under the unnormalised Satake transform from $\widetilde{G}$ to $\textnormal{Res}_{F/F^+}\textnormal{GL}_n$ of the Hecke operator $\Delta_{\mathfrak{p}}$ at an auxiliary rational prime provided by \S\ref{VanishingUpToBoundedTorsion}. 

As a corollary, we obtain semisimple local-global compatibility at $p$ for cohomological cuspidal automorphic representations of $\textnormal{GL}_n(\mathbf{A}_F)$.
\begin{theorem}\label{Theorem_MainLGCThm}
    Let $F$ be a CM number field and $\pi$ be a cohomological cuspidal automorphic representation of $\textnormal{GL}_n(\mathbf{A}_F)$ of weight $\lambda\in (\mathbf{Z}_+^n)^{\Hom(F,\mathbf{C})}$. Let $p$ be a prime and let $\iota\colon\overline{\mathbf{Q}}_p\xrightarrow{\sim}\mathbf{C} $ be an isomorphism of fields. Then, for any $p$-adic place $v|p$ of $F$, the following hold true.
    \begin{enumerate}
        \item The local Galois representation $r_{\pi,\iota}|_{G_{F_v}}$ is de Rham, with $\tau$-labelled Hodge--Tate weights 
        \begin{equation*}
            \textnormal{HT}(r_{\pi,\iota}|_{G_{F_v}})=\{\lambda_{\iota\tau ,1}+n-1> \lambda_{\iota\tau ,2}+n-2>...>\lambda_{\iota\tau,n}\}_{\tau\colon F_{v}\hookrightarrow\overline{\mathbf{Q}}_p}.
        \end{equation*}
        \item There is an isomorphism
    \begin{equation*}
        \textnormal{WD}(r_{\pi,\iota}|_{G_{F_v}})^{\textnormal{ss}}\cong \iota^{-1}\textnormal{rec}_{F_v}^T(\pi_v)^{\textnormal{ss}}.
    \end{equation*}
    \end{enumerate}
\end{theorem}
\begin{proof}
    We can proceed `one place at a time', so fix a place $a|p$ of $F$. After making a cyclic base change (using \cite{Art89}), we can assume that the following conditions are satisfied:
    \begin{itemize}
        \item $F$ contains an imaginary quadratic field in which $p$ splits.
        \item There is a place $\overline{b}\in \overline{S}_p\setminus \{\overline{a}\}$ such that \eqref{eqn_enough_shifting_new} holds.
    \end{itemize}

    Choose a good subgroup $K\leq \textnormal{GL}_n(\mathbf{A}_F^{\infty})$ with $K_p=\prod_{v\mid p}\textnormal{GL}_n(\mathcal{O}_{F_v})$ and $\pi^{K^p}\neq 0$. Let $S$ be a finite set of places of $F$ containing the $p$-adic places, satisfying $(1)$ from Proposition \ref{Proposition_MainLGCProp} and so that $K$ is $S$-good. Let $\Omega=(\Omega_v)_{v\mid p}$ be the collection of Bernstein blocks associated with $\pi_p=\otimes_{v\mid p}\pi_v$. Let $\mathcal{L}$ be an $\mathcal{O}$-integral $\Omega$-system for the ring of integers of a suitable finite field extension $E/\mathbf{Q}_p.$

    After possibly enlarging $E$, the action of $\mathbf{T}^S$ on $\pi$ is via an $\mathcal{O}$-valued point $x_{\pi}$ and we write $\mathfrak{m}\trianglelefteq \mathbf{T}^S$ for the corresponding maximal ideal with residue field $k=\mathcal{O}/\varpi$.

    By \cite{Fra98}, \cite{Fra98a}, we can assume that the Hecke algebra $\mathbf{T}_a:=\mathbf{T}^S\otimes_{\mathcal{O}}\mathfrak{Z}_{\textnormal{GL}_n(F_a),\mathcal{L}_a,\lambda_a,\overline{D}_{\mathfrak{m},a}}^{\circ}$ acts on $\Hom_K(\sigma_{\Omega},\pi)$ through the faithful quotient $A_{\mathfrak{m}}=\mathbf{T}_a(H^{\ast}(X_K,\sigma_{\Omega,\lambda}^{\circ})_{\mathfrak{m}})$. By Proposition \ref{Proposition_MainLGCProp}, there are integers $N_1\geq 1$, $N_2\geq 0$ and, for every $m\geq 1$, an ideal $J_m\trianglelefteq A_{\mathfrak{m}}^m:=\mathbf{T}_a(H^{\ast}(X_K,\sigma_{\Omega,\lambda}^{\circ}/\varpi^m)_{\mathfrak{m}})$  with $(T_n^2p^{N_2}J_m)^{N_1}=0$ such that $A_{\mathfrak{m}}^m/J_m$ satisfies $\textnormal{LGC}(S,a,\iota^{-1}\lambda,\Omega,\mathcal{L})$. Set
    \begin{equation*}
        J:=\ker(A_{\mathfrak{m}}\hookrightarrow\prod_mA_{\mathfrak{m}}^m\twoheadrightarrow \prod_m(A_{\mathfrak{m}}^m/J_m))
    \end{equation*}
    and note that $(T_n^2p^{N_2}J)^{N_1}=0$ and $\big(A_{\mathfrak{m}}/J\big)/\mathfrak{m}^k$ satisfies $\textnormal{LGC}(S,a,\iota^{-1}\lambda,\Omega,\mathcal{L})$ for every $k\geq 1$.
    By \cite{Sha74}, after possibly twisting by a Hecke character, $T_n$ is sent under $x_{\pi}$ to a non-zero element. In particular, the extension to $\mathbf{T}_a$ of the point $x_{\pi}$ factors through $A_{\mathfrak{m}}/J$ and so it satisfies $\textnormal{LGC}(S,a,\iota^{-1}\lambda,\Omega,\mathcal{L})$ mod $\varpi^k$ for every $k\geq 1$. The theorem now follows from Corollary \ref{cor:local_global_deformation_ring}.
\end{proof}

Our proof of Proposition \ref{Proposition_MainLGCProp} will be broken into two separate parts. First, we describe the cohomology of the Borel--Serre boundary of $X^{\textnormal{GL}_n}$ in terms of ``torsion Eisenstein series" with desired local behaviour at our fixed $p$-adic place allowing us to show that the boundary cohomology satisfies local-global compatibility by induction on $n$. As we have limited control over the local systems appearing in this analysis, this forces us to ignore $p^{N_2}$-torsion of exponent depending on our initial local system. This is a new phenomenon compared to previous works \cite{10author}, \cite{AC23}, \cite{CN25} and \cite{Hev24} originating in relaxing the residual irreducibility condition in their theorems. Indeed, in any of the listed articles, results are only proved after localisation at a non-Eisenstein maximal ideal, which has the effect of immediately killing all boundary cohomology for $X^{\textnormal{GL}_n}$. 

Finally, to fill the gap between the boundary cohomology and cohomology (without compact supports) with $\mathcal{O}/\varpi^m$-coefficients, it turns out to be sufficient to prove local-global compatibility at $a\in S_p$ for the interior cohomology of $X^{\textnormal{GL}_n}$ with $\mathcal{O}$-coefficients with values in a local system \textit{suitably chosen at places away from $a$}. This puts us into the situation that is handled by the ``degree shifting argument" introduced in \cite{10author}. More precisely, we closely follow the $P$-ordinary version of the argument as presented in \cite{CN25} and \cite{Hev24}, building on \cite{10author}, and \cite{AC23}, with the key difference coming from replacing the `decomposed generic' assumption on our maximal ideal to instead proving results up to $T_n$-torsion, using our vanishing result established in \S \ref{sec_torsion_in_the_cohomology}. Roughly speaking, saying that a maximal ideal is decomposed generic corresponds to saying that the image of $T_n$ in the localised Hecke algebra is a unit.

\begin{remark}
    We emphasise that to deduce Theorem \ref{Theorem_MainLGCThm} it appears to be necessary to analyse the boundary cohomology as well. Indeed, the degree shifting argument only proves local-global compatibility for \textit{interior} cohomology with coefficients in specific $\mathcal{O}$-local systems $\mathcal{W}$ (the `cohomologically cuspidal' ones) which consist of only torsion classes. Moreover, even when the Hecke eigensystem $\psi$ of $\pi$ appears in $H^{\ast}_!(X_K,\mathcal{V})$ for some $\mathcal{O}$-local system $\mathcal{V}$ with $\mathcal{V}/\varpi^m\cong \mathcal{W}/\varpi^m$, we cannot ensure that $\psi\mod{\varpi^m}$ appears in the image of $H^{\ast}_!(X_K,\mathcal{W})\to H^{\ast}_!(X_K,\mathcal{W}/\varpi^m)$.
\end{remark}

\subsection{Torsion Eisenstein series in boundary cohomology}\label{sec_boundarycoh}
We continue to work with the setup of the previous subsection. In particular, we have an integer $n\geq 1$, a number field $F/\mathbf{Q}$ fixed and $\mathcal{O}$ will denote the ring of integers of a sufficiently large coefficient field $E/\mathbf{Q}_p$ with a fixed choice of uniformiser $\varpi\in \mathcal{O}$.
\subsubsection{The automorphic side}

Fix a dominant weight $\lambda\in (\mathbf{Z}_+^n)^{\Hom(F,E)}$, a collection of Bernstein blocks $\Omega=(\Omega_v)_{v\mid p}\in \prod_{v\mid p}\mathfrak{B}(F_v,n)$, an $\mathcal{O}$-integral $\Omega$-system $\mathcal{L}=(\mathcal{L}_v)_{v\mid p}$ and an integer $N\geq 1$ that is $(\lambda_v,\mathcal{L}_v)$-admissible for every $v\in S_p$.

Consider a standard parabolic subgroup $Q=MN\leq \textnormal{GL}_{n,\mathcal{O}_F}$. For a finite place $v$ of $F$ we will write $\textnormal{Q}_v:=Q(F_v),$ $\textnormal{M}_v:=M(F_v)$ and $\textnormal{N}_v:=N(F_v)$. Recall that in Proposition~\ref{Proposition_BushnellKutzkoTypesIso}, we introduced the multiset $\mathfrak{B}(\textnormal{M}_v,\Omega_v)$ of Bernstein blocks $\Omega_{\textnormal{M}_v}$ of $\textnormal{Rep}_{\textnormal{sm}}(\textnormal{M}_v,\overline{\mathbf{Q}}_p)$ for every finite place $v\in S_p$.

Fix some finite set of finite places $S$ of $F$ containing all $p$-adic places such that $K^S=G(\widehat{\mathcal{O}}_F^{S})$ and an integer $N\geq 0$ that is $(\lambda_v,\mathcal{L}_v)$-admissible for every $v\in S_p$. In \S\ref{Subsection_integralstructures} we introduced the Satake transforms $\mathcal{S}_{\textnormal{Q}_v,\Omega_v,\lambda_v}^{\circ
}$, $\mathcal{S}_{\textnormal{M}_v,\mathcal{L}_v,\lambda_v}$ with source given by the commutative $\mathcal{O}$-algebra $\mathcal{H}^B(\sigma_{\Omega_v,\lambda_v})^{\circ}:=\mathcal{H}^B(\sigma_{\Omega_v,\lambda_v})^{\circ}_N$. We introduce the Satake transforms
\begin{equation*}
    \mathcal{S}_{Q,v}^{\circ}:=\mathcal{S}_M\otimes \mathcal{S}_{\textnormal{Q}_v,\Omega_v,\lambda_v}^{\circ}\textnormal{, }\mathcal{S}_{M,v}^{\circ}:=\mathcal{S}_M\otimes \mathcal{S}_{\textnormal{M}_v,\mathcal{L}_v,\lambda_v}
\end{equation*}
with source given by the Hecke algebra 
\begin{equation*}
    \mathbf{T}_v^B=\mathbf{T}^S\otimes_{\mathcal{O}}\mathcal{H}^B(\sigma_{\Omega_v,\lambda_v})^{\circ}
\end{equation*}
 introduced in the previous subsection. We will also need the $\mathcal{O}$-subalgebra
 \begin{equation*}
     \mathbf{T}^{B,+}_v:=\mathbf{T}^S\otimes_{\mathcal{O}}\mathcal{H}^{B,+}(\sigma_{\Omega_v,\lambda_v})^{\circ}\leq \mathbf{T}^B_v.
 \end{equation*}

The next Proposition is our main result on ``torsion Eisenstein functoriality" in the cohomology of $\textnormal{Res}_{F/\mathbf{Q}}\textnormal{GL}_n$-locally symmetric spaces. Although the statement appears to be rather technical, it roughly states the following. The Hecke eigensystems (with Hecke operators at a fixed $p$-adic place $a$ of $F$ included) appearing in 
\begin{equation*}
    H^{\ast}_{(c)}(X_K^{Q},\sigma_{\Omega,\lambda}^{\circ}/\varpi^m),
\end{equation*} up to nilpotency of degree $N_1$ and $p$-power torsion of exponent $N_2$, already appear (under the Satake transform) in
\begin{equation*}
    \bigoplus_{w\in W^{\textnormal{M}_a}}\bigoplus_{\Omega_{\textnormal{M}}\in \mathfrak{B}(\textnormal{M}_a,\Omega)}H^{\ast}_{(c)}(X_{K_{M}},\sigma_{\Omega_{\textnormal{M}_a},w\cdot\lambda_a}^{\circ}/\varpi^m)
\end{equation*}
where the integers $N_1,N_2$ are independent of $m\geq 1$.

\begin{prop}\label{Proposition_EisFunc}
There are integers $N_1\geq 1$, $N_2\geq 0$ satisfying the following.
\begin{itemize}
    \item The integer $N_1$ only depends on $n$ and $F$.
    \item The integer $N_2$ only depends on $n$, $F$, $\lambda$ and $\mathcal{L}$.
    \item For every integer $m\geq 1$, place $a\in S_p$, standard parabolic subgroup $Q=MN\leq \textnormal{GL}_{n,\mathcal{O}_F}$ and $S\setminus\{a\}$-good level subgroup $K\leq \textnormal{GL}_n(\mathbf{A}_F^{\infty})$, there is an $S\setminus \{a\}$-good level subgroup $K_{M}\leq M(\mathbf{A}_F^{\infty})$ and an inclusion
    \begin{align}\label{equation_goaleisfunc}
    \begin{split}
\left(p^{N_2}\textnormal{Ann}_{\mathbf{T}_a^B}\left(\bigoplus_{w,\Omega_{\textnormal{M}_a}}H^{\ast}_{(c)}(X_{K_{M}},\sigma_{\Omega_{\textnormal{M}_a},w\cdot\lambda_a}^{\circ}/\varpi^m)\right)\right)^{N_1} \\
        \leq \textnormal{Ann}_{\mathbf{T}_a^B}\left(R\Gamma_{(c)}(X_K^Q,\sigma_{\Omega,\lambda}^{\circ}/\varpi^m)\right),
    \end{split}
    \end{align}
    where the direct sum runs over $w\in W^{\textnormal{M}_a}$, $\Omega_{\textnormal{M}_a}\in \mathfrak{B}(\textnormal{M}_a,\Omega)$.
\end{itemize}
\end{prop}

\begin{remark}
    We note that even though our local-global compatibility results are concerned with the action of $\mathfrak{Z}_{\textnormal{GL}_n(F_v),\mathcal{L}_v,\lambda_v}^{\circ}$, for our analysis of the boundary cohomology we work with the larger commutative Hecke algebra $\mathcal{H}^B(\sigma_{\Omega_v,\lambda_v})^{\circ}$ that admits an explicit presentation by `sufficiently simple' Hecke operators.
Namely, we have explicit presentations
\begin{equation*}
    \begin{tikzcd}
	{\mathbf{T}^S[V_{i,j}\mid i,j]} & {\mathbf{T}^{B,+}_a} \\
	{\mathbf{T}^S[V_{i,j},W_{i,j}\mid i,j]/(V_{i,j}W_{i,j}-p^{N})} & {\mathbf{T}_a^{B}}
	\arrow["\sim", from=1-1, to=1-2]
	\arrow[hook, from=1-1, to=2-1]
	\arrow[hook, from=1-2, to=2-2]
	\arrow["\sim", from=2-1, to=2-2]
\end{tikzcd}
\end{equation*}
that we will use in the course of the proof of Proposition~\ref{Proposition_EisFunc}.
\end{remark}

For a $\mathbf{Z}_p$-module $M$ and an integer $k\geq 0$, we introduce the following simple notation
\begin{equation*}
    \overline{M}^k:=\im(M\xrightarrow{p^k\cdot}M).
\end{equation*}
If $M$ is in fact a module over some $\mathbf{Z}_p$-algebra $R$, then $\overline{M}^k$ is naturally an $R$-submodule $\overline{M}^k\leq M$. Moreover, given two $R$-modules $M,M'$ and an $R$-module homomorphism between them $\alpha:M\to M'$, we obtain an induced map $\overline{\alpha}^k\colon\overline{M}^k\to \overline{M'}^k$. If $\alpha$ is surjective or injective, respectively, so is $\overline{\alpha}^k$. Moreover, if we have an exact sequence of $R$-modules
\begin{equation*}
    0\to M'\to M\to M''\to 0,
\end{equation*}
then the natural surjection
\begin{equation*}
    \overline{M}^k/\overline{M'}^k\to \overline{M''}^k\to 0
\end{equation*}
has kernel killed by $p^k$. Indeed, if  $x\in \overline{M}^k$ is sent to $0$ in $\overline{M''}^k$, we have $x\in M'$ and so $p^kx\in \overline{M'}^k$.

\begin{lemma}\label{Lemma_EisFuncMorphism}
    Let $M'$ and $M$ be $\mathbf{T}^B_a$-modules and $\alpha:M\to M'$ be a $\mathbf{T}^{B,+}_a$-algebra homomorphism. Then the induced map $\overline{\alpha}^N\colon\overline{M}^N\to \overline{M'}^N$ is $\mathbf{T}_a^B$-equivariant.
\end{lemma}

\begin{proof}
    By the formulas $W_{i,j}V_{i,j}=p^N$, and $\alpha\circ V_{i,j}=V_{i,j}\circ \alpha$, we obtain $W_{i,j}\circ \alpha|_{V_{i,j}(M)}=\alpha\circ W_{i,j}|_{V_{i,j}(M)}$. Finally, using $V_{i,j}W_{i,j}=p^N$, we see that $\overline{M}^N=V_{i,j}(W_{i,j}(M))\subset V_{i,j}(M)$.
\end{proof}

\begin{lemma}\label{Lemma_EisFuncInjection}
    Let $M$ be a $\mathbf{T}_a^B$-module and $M'\leq M$ be a sub-$\mathbf{T}_a^{B,+}$-module. Then $\overline{M'}^N\leq \overline{M}^N$ is $\mathbf{T}_a^B$-invariant.
\end{lemma}
\begin{proof}
    For every $i$ and $j$, $W_{i,j}$ sends $V_{i,j}(M')$ into $\overline{M'}^N$ as $W_{i,j}V_{i,j}=p^N$. The claim follows from $\overline{M'}^N=V_{i,j}W_{i,j}(M')\leq V_{i,j}(M)$.
\end{proof}
\begin{corollary}\label{Corollary_EisFuncSubquotient}
    Let $M,M'$ be $\mathbf{T}_a^B$-modules such that $M'$ is a $\mathbf{T}_a^{B,+}$-equivariant subquotient of $M$. Then $\overline{M'}^{2N}$ is a $\mathbf{T}_a^B$-equivariant subquotient of $\overline{M}^{2N}$.
\end{corollary}
\begin{proof}
    Consider a $\mathbf{T}_a^{B,+}$-equivariant diagram $M\hookleftarrow M''\twoheadrightarrow M'$ for a $\mathbf{T}_a^{B,+}$-module $M''$. Then $\overline{M''}^N\leq \overline{M}^N$ is preserved by $\mathbf{T}_a^B$ by Lemma~\ref{Lemma_EisFuncInjection}. Moreover, $\overline{M''}^N\to \overline{M'}^N$ being a $\mathbf{T}_a^{B,+}$-equivariant map, $\overline{M''}^{2N}\to \overline{M'}^{2N}$ is $\mathbf{T}_a^B$-equivariant by Lemma~\ref{Lemma_EisFuncMorphism}.
\end{proof}

\begin{lemma}\label{Lemma_EisFuncSurjection}
    Let $M, M'$ be $\mathbf{T}_a^B$-modules, $k\geq    0$ an integer and $M\twoheadrightarrow M'$ be a surjection of $\mathbf{T}_a^B$-modules with kernel killed by $p^k$. Then $(p^k\textnormal{Ann}_{\mathbf{T}_a^B}(M'))^2\leq \textnormal{Ann}_{\mathbf{T}_a^B}(M)$. 
\end{lemma}
\begin{proof}
    Denote by $K\leq M$ the kernel. We then have $(\textnormal{Ann}_{\mathbf{T}_a^B}(M')\cap \textnormal{Ann}_{\mathbf{T}_a^B}(K))^2\leq \textnormal{Ann}_{\mathbf{T}_a^B}(M)$ and $p^k\mathbf{T}_a^B\leq \textnormal{Ann}_{\mathbf{T}_a^B}(K)$. Therefore, we have $p^k\textnormal{Ann}_{\mathbf{T}_a^B}(M')\leq \textnormal{Ann}_{\mathbf{T}_a^B}(M')\cap \textnormal{Ann}_{\mathbf{T}_a^B}(K)$.
\end{proof}

\begin{lemma}\label{Lemma_EisFuncTStructure}
    Let $M$ be a $\mathbf{T}^B_a$-module, $M'$ be a $\mathbf{T}_a^{B,+}$-module and $\alpha:M\twoheadrightarrow M'$ be a $\mathbf{T}_a^{B,+}$-equivariant surjection. Then $\overline{M'}^N$ can be equipped with a $\mathbf{T}_a^B$-module structure, making the surjection $\beta:M\to M'\to\overline{M'}^N$ $\mathbf{T}^B_a$-equivariant.
\end{lemma}
\begin{proof}
    We need to show that $\ker(\beta)=\{x\in M\mid \alpha(x)\in M'[p^N]\}$ is preserved by $\mathbf{T}_a^B$. In other words, we need to check that for every 
    $i, j$ and $x\in \textnormal{ker}(\beta)$, $\alpha(W_{i,j}(x))$ is killed by $p^N$.

    To see this, we compute
    \begin{equation*}
        V_{i,j}(\alpha(W_{i,j}(x)))=\alpha(V_{i,j}W_{i,j}(x))=\alpha(p^Nx)=p^N\alpha(x)=0.
    \end{equation*}
    In other words, $\alpha(W_{i,j}(x))\in \ker(M'\xrightarrow{V_{i,j}}M')$. However, $\ker(M'\xrightarrow{V_{i,j}}M')\leq M'[p^N]$ as $W_{i,j}V_{i,j}=p^N$, finishing the proof.
\end{proof}
\begin{lemma}\label{BoundedTorDim}
    The complex $R\Gamma(\overline{\mathfrak{X}}_M,\mathcal{O}/\varpi^m)\in D^+_{\textnormal{sm}}(M(\mathbf{A}_F^{\infty}),\mathcal{O}/\varpi^m)$ has bounded Tor-dimension for every integer $m\geq 1$.
\end{lemma}

\begin{proof}
    We know that the compact Hausdorff space $\overline{\mathfrak{X}}_M$ has finite cohomological dimension cf. \cite[Lemma 2.1.4]{CN25}. For an $\mathcal{O}/\varpi^m$-module $N$, write $\mathcal{N}$ for the $M(\mathbf{A}_F^{\infty})$-equivariant sheaf associated with the trivial representation with underlying module $N$. Then, by \cite[Lemma 2.1.9]{CN25}, there is an isomorphism
    \begin{equation*}
        R\Gamma(\overline{\mathfrak{X}}_M,\mathcal{N})\cong R\Gamma(\overline{\mathfrak{X}}_M,\mathcal{O}/\varpi^m)\otimes^{\mathbf{L}}_{\mathcal{O}/\varpi^m}N.
    \end{equation*}
    The claim follows.
    \end{proof}

    \begin{proof}[Proof of Proposition~\ref{Proposition_EisFunc}]
    We only prove the case of cohomology without compact support. The compactly supported case follows the same argument. 
    Note that we have an inclusion
\begin{equation*}
    \textnormal{Ann}_{\mathbf{T}_a^B}(H^{\ast}(X_K^Q,\sigma_{\Omega,\lambda}^{\circ}/\varpi^m))^{n^2[F:\mathbf{Q}]/2}\leq \textnormal{Ann}_{\mathbf{T}_a^B}(R\Gamma(X_K^Q,\sigma_{\Omega,\lambda}^{\circ}/\varpi^m)).
\end{equation*}
In particular, it suffices to find integers $N_1$, $N_2$ such that \eqref{equation_goaleisfunc} holds with $\textnormal{Ann}_{\mathbf{T}_a^B}(H^{\ast}(X_K^Q,\sigma_{\Omega,\lambda}^{\circ}/\varpi^m))$ on the RHS.

By Lemma \ref{Lemma_BorelSerreLemma} and the projection formula (cf. \cite[Lemma 2.1.8]{CN25}), we have a series of isomorphisms
        \begin{align*}
            R\Gamma(X_K^Q,\sigma_{\Omega,\lambda}^{\circ}/\varpi^m) &\cong R\Gamma(K,(\textnormal{Ind}_{Q(\mathbf{A}_F^{\infty})}^{G(\mathbf{A}_F^{\infty})}R\Gamma(\overline{\mathfrak{X}}_Q,\mathcal{O}/\varpi^m))\otimes_{\mathcal{O}/\varpi^m}^{\mathbf{L}}\sigma_{\Omega,\lambda}^{\circ}/\varpi^m) \\
            &\cong \bigoplus_gR\Gamma(K_{Q,g},R\Gamma(\overline{\mathfrak{X}}_Q,\mathcal{O}/\varpi^m)\otimes^{\mathbf{L}}_{\mathcal{O}/\varpi^m}\sigma_{\Omega,\lambda}^{\circ}/\varpi^m)
        \end{align*}
        where the sum runs over representatives $g\in G^{S\setminus \{a\}}$ of $Q(\mathbf{A}_{F}^{\infty})\backslash G(\mathbf{A}_F^{\infty})/K$ and $K_{Q,g}:=gKg^{-1}\cap Q(\mathbf{A}_F^{\infty})$. 
        By Lemma~\ref{RestrictionFunc}, this isomorphism is $\mathbf{T}_a^{B}$-equivariant when we act on the RHS via restricting functions to $Q(\mathbf{A}_F^{\infty\cup (S\setminus \{a\})})$. By combining this with the existence of the Hochschild--Serre spectral sequence, we see that it suffices to find integers $N_1, N_2$ such that
        \begin{align*}    \bigg(p^{N_2}\textnormal{Ann}_{\mathbf{T}_a^B}\big(\bigoplus_{w,\Omega_{\textnormal{M}_a}}H^{\ast}(X_{K_{M}},\sigma_{\Omega_{\textnormal{M}_a},w\cdot\lambda_a}^{\circ}/\varpi^m)\big)\bigg)^{N_1} \\
        \leq \textnormal{Ann}_{\mathbf{T}_a^B}\bigg(H^{\ast}\big(R\Gamma(K_{Q},R\Gamma(\overline{\mathfrak{X}}_Q,\mathcal{O}/\varpi^m)\otimes_{\mathcal{O}/\varpi^m}^{\mathbf{L}}\sigma_{\Omega,\lambda}^{\circ}/\varpi^m)\big)\bigg)
        \end{align*}
        where $K_M=K_Q\cap M(\mathbf{A}_F^{\infty})$ for a subgroup $K_{Q}\leq Q(\mathbf{A}_F^{\infty})$ satisfying the following.
        \begin{itemize}
            \item For $v\notin S\setminus\{a\}$, $K_{Q,v}=K_v\cap Q(F_v)$.
            \item For $v\in S\setminus \{a\}$, $K_{Q,v}= \textnormal{ker}\big(Q(\mathcal{O}_{F_v})\to Q(\mathcal{O}_{F_v}/\varpi_v^h)\big)$ for $h\geq m$ such that $K_{Q,v}$ acts trivially on $\sigma_{\Omega_{v}}^{\circ}$.
        \end{itemize}
        In particular, we have
        \begin{align*}
        &\textnormal{Ann}_{\mathbf{T}_a^B}\bigg(H^{\ast}\big(R\Gamma(K_{Q},R\Gamma(\overline{\mathfrak{X}}_Q,\mathcal{O}/\varpi^m)\otimes_{\mathcal{O}/\varpi^m}^{\mathbf{L}}\sigma_{\Omega,\lambda}^{\circ}/\varpi^m)\big)\bigg) \\
        =& \textnormal{Ann}_{\mathbf{T}_a^B}\bigg(H^{\ast}\big(R\Gamma(K_{Q},R\Gamma(\overline{\mathfrak{X}}_Q,\mathcal{O}/\varpi^m)\otimes_{\mathcal{O}/\varpi^m}^{\mathbf{L}}\sigma_{\Omega_a,\lambda}^{\circ}/\varpi^m)\big)\bigg)
        \end{align*}
        for $\sigma_{\Omega_a,\lambda}^{\circ}:=\sigma_{\Omega_a}^{\circ}\otimes_{\mathcal{O}} \mathcal{V}_{\lambda}$.
    
        Combining \cite[Lemma 4.1.6]{CN25} with Lemmas \ref{TransitivityOfHeckeAction}, \ref{BoundedTorDim} and \ref{ProjectionFormula}, we further have a $\mathbf{T}^S$-equivariant isomorphism
        \begin{align*}
            & R\Gamma(K_{Q},R\Gamma(\overline{\mathfrak{X}}_Q,\mathcal{O}/\varpi^m)\otimes_{\mathcal{O}/\varpi^m}^{\mathbf{L}}\sigma_{\Omega_a,\lambda}^{\circ}/\varpi^m) \\
            \cong & 
            R\Gamma(K_{Q},\textnormal{Inf}_{M(\mathbf{A}_{F}^{\infty})}^{Q(\mathbf{A}_F^{\infty})}R\Gamma(\overline{\mathfrak{X}}_M,\mathcal{O}/\varpi^m)\otimes_{\mathcal{O}/\varpi^m}^{\mathbf{L}}\sigma_{\Omega_a,\lambda}^{\circ}/\varpi^m) \\
            \cong & R\Gamma\left(K_{M},R\Gamma(K_{N,p},\textnormal{Inf}_{M_p}^{Q_p}R\Gamma(\overline{\mathfrak{X}}_M,\mathcal{O}/\varpi^m)\otimes_{\mathcal{O}/\varpi^m}^{\mathbf{L}}\sigma_{\Omega_a,\lambda}^{\circ}/\varpi^m)\right) \\
            \cong & R\Gamma(K_{M},R\Gamma(\overline{\mathfrak{X}}_M,\mathcal{O}/\varpi^m)\otimes_{\mathcal{O}/\varpi^m}^{\mathbf{L}}R\Gamma(K_{N,p},\sigma_{\Omega_a,\lambda}^{\circ}/\varpi^m))\\ 
            \cong &
            R\Gamma(X_{K_{M}},R\Gamma(K_{N,p},\sigma_{\Omega_a,\lambda}^{\circ}/\varpi^m))
        \end{align*}
that matches the action of $V_{i,j}=[K_at_{i,j}K_a,p^{N_{i,j}}\psi_{i,j}\otimes t_{i,j}]$ on the LHS with the action of
\begin{equation*}
    \mathcal{S}_a(V_{i,j}):=\sum_{w\in {}^{\textnormal{M}_a}W^{M_{\nu_{i,j},a}}}[K_{\textnormal{M},a}t_{i,j}^wK_{\textnormal{M},a},\overline{(p^{N_{i,j}}\psi_{i,j}^{w}\otimes  t_{i,j}^w)}_{\textnormal{Q}_a,\mathbf{1}}]
\end{equation*}
on the RHS.

Consider the $\mathbf{T}^{S}$-equivariant hypercohomology spectral sequence
\begin{align*}
E_2^{r,s}=H^r(X_{K_{M}},H^s(K_{N,p},\sigma_{\Omega_a,\lambda}^{\circ}/\varpi^m))\Rightarrow  \\
    H^{r+s}:=H^{r+s}(X_{K_{M}},R\Gamma(K_{N,p},\sigma_{\Omega_a,\lambda}^{\circ}/\varpi^m))
\end{align*}
and write $\textnormal{Fil}^{\bullet}$ for the corresponding filtration on $H^{r+s}$ with graded pieces $E_{\infty}^{r,s}$.
To ease notation, for $w\in W^M$ and $r\in \mathbf{Z}$, set
\begin{equation*}
    M^{r}_w:=\bigoplus_{\Omega_{\textnormal{M}_a}\in \mathfrak{B}(\textnormal{M}_a,\Omega_a)}H^{r}(X_{K_{M}},\sigma_{\Omega_{\textnormal{M}_a},w_a\cdot \lambda_a}^{\circ}/\varpi^m) .  
\end{equation*}
It suffices to show that for every integer $q\geq 0$, there are integers $N_1^q$, $N_2^q\geq 0$ as in the statement (in particular independent of $m$) such that there is an inclusion
\begin{equation*}
    \left(p^{N_1^q}\textnormal{Ann}_{\mathbf{T}_a^B}(\bigoplus_{l(w)\leq q}M_w^{q-l(w)})\right)^{N_2^q}\leq\textnormal{Ann}_{\mathbf{T}_a^B}(H^q).
\end{equation*}
We therefore fix an integer $q\geq 0$. 

Note that the spectral sequence is equivariant with respect to the action of the Hecke operators $\mathcal{S}_a(V_{i,j})$ and the filtration $\textnormal{Fil}^{\bullet}$ on $H^{q}$ is preserved by the $\mathbf{T}_a^{B,+}$-action. By Lemma~\ref{Lemma_EisFuncInjection}, the induced filtration $\overline{\textnormal{Fil}^{\bullet}}^{N}$ on $\overline{H^{q}}^N$ is preserved by the action of $\mathbf{T}^B_a$.  We further see that, for any integer $k\geq 0$, there is a $\mathbf{T}_a^{B,+}$-equivariant surjection
\begin{equation}\label{Equation_FiltrationSurjection}
   \overline{\textnormal{gr}^i}^k:=\overline{\textnormal{Fil}^{r+1}}^k/\overline{\textnormal{Fil}^{r}}^k\twoheadrightarrow\overline{E_{\infty}^{r,q-r}}^{k}
\end{equation}
with kernel killed by $p^k$. Consequently, by Lemma~\ref{Lemma_EisFuncTStructure}, $\overline{E_{\infty}^{r,q-r}}^{2N}$ can be promoted to a $\mathbf{T}^{B}_a$-module.

We claim that, there are $\mathbf{T}^{S}$-equivariant maps
\begin{align*}
    \alpha_{r} \colon & E_{2}^{r,q-r}\xrightarrow{}M^{r}:=\bigoplus_{l(w)=q-r}M_{w}^r= \\
    & H^r(X_{K_{M}},\bigoplus_{w\in W^{M}, l(w)=q-r}\bigoplus_{\Omega_{\textnormal{M}_a}\in \mathfrak{B}(\textnormal{M}_a,\Omega_a)}\sigma_{\Omega_{\textnormal{M}_a},w_a\cdot \lambda_a}^{\circ}/\varpi^m), \\
    \beta_{r} \colon & M^{r}\to E_2^{r,q-r}
\end{align*}
satisfying $\alpha_{r}\circ\beta_{r}=\beta_{r}\circ\alpha_{r}=p^{N_3}$ for an integer $N_3\geq N$ that only depends on $n, F, \lambda, \mathcal{L}$.
Moreover, these maps match the action of $\mathcal{S}_a(V_{i,j})$ on $E_{2}^{r,q-r}$ with the action of $\mathcal{S}_{\textnormal{Q}_a,\Omega_a,\lambda_a}^{\circ}(V_{i,j})$ on $M^{r}$. The claim follows from combining Lemma~\ref{Lemma_changeofweightI}, Lemma~\ref{Lemma_changeofweightII}, Lemma~\ref{Lemma_changeofweightIII} and the fact that, for every $v\in S_p\setminus \{a\}$ and $w_v\in W^{\textnormal{M}_v}$, $\mathcal{V}_{w_v\cdot \lambda_v}/\varpi^m$ is trivial as a $K_{M,v}$-representation.

We set $E^r:=\overline{\beta_r(M^r)}^{N_3}\leq E_2^{r,q-r}$, and $\overline{\beta}_r:=p^{N_3}\beta_r:M^r\to E^r$ a surjection with kernel killed by $p^{2N_3}$, equipping $E^r$ with a $\mathbf{T}_a^{B}$-module structure by Lemma~\ref{Lemma_EisFuncTStructure} (as $N_3\geq N$).

Moreover, $\overline{E_2^{r,q-r}}^{2N_3}\leq E^r$ is a $\mathbf{T}_a^{B,+}$-invariant submodule as $\beta_r\circ\alpha_r=p^{N_3}$. Since $\overline{E_2^{r,q-r}}^{2N_3}$ admits $\overline{E_{\infty}^{r,q-r}}^{2N_3}$ as a $\mathbf{T}_a^{B,+}$-equivariant subquotient, we see that $\overline{E^r}^{2N}$, and consequently $M^r$, admits $\overline{E_{\infty}^{r,q-r}}^{2N+2N_3}$ as a $\mathbf{T}_a^B$-equivariant subquotient by Corollary~\ref{Corollary_EisFuncSubquotient}.

We obtain that
\begin{equation*}
    \textnormal{Ann}_{\mathbf{T}_a^B}(M^r)\leq \textnormal{Ann}_{\mathbf{T}_a^B}(\overline{E^r}^{2N})\leq \textnormal{Ann}_{\mathbf{T}_a^B}(\overline{E_{\infty}^{r,q-r}}^{2N+2N_3}).
\end{equation*}

On the other hand, recall the surjection \eqref{Equation_FiltrationSurjection} for $k=2N+2N_3$ with kernel killed by $p^{2N+2N_3}$. In particular, by Lemma~\ref{Lemma_EisFuncSurjection}, we have an inclusion
\begin{equation*}
    (p^{2N+2N_3}\textnormal{Ann}_{\mathbf{T}_a^B}(\overline{E_{\infty}^{r,q-r}}^{2N+2N_3}))^2\leq \textnormal{Ann}_{\mathbf{T}_a^B}(\overline{\textnormal{gr}^r}^{2N+2N_3}).
\end{equation*}
We also have inclusions
\begin{equation*}
    (\cap_{r=0,...,q}\textnormal{Ann}_{\mathbf{T}_a^B}(\overline{\textnormal{gr}^r}^{2N+2N_3}))^q\leq \textnormal{Ann}_{\mathbf{T}_a^B}(\overline{H^q}^{2N+2N_3}), 
\end{equation*}
and (by Lemma~\ref{Lemma_EisFuncSurjection})
\begin{equation*}
    (p^{2N+2N_3}\textnormal{Ann}_{\mathbf{T}_a^B}(\overline{H^q}^{2N+2N_3}))^2\leq \textnormal{Ann}_{\mathbf{T}_a^B}(H^q).
\end{equation*}
This finishes the proof.
\end{proof}

\subsubsection{Consequences on the Galois side}
To discuss the corollary of our description of boundary cohomology via ``torsion Eisenstein series" to local-global compatibility for the cohomology of the Borel--Serre boundary, we specialise to the setup of \S\ref{subsubsection_LGCstatement}. In particular, $F$ will now be a CM field as in Assumption \ref{assumption_CM}. We fix a finite set of finite places $S$ of $F$ containing $S_p$ and a place $a\in S_p$. Moreover, we fix a dominant weight $\lambda\in (\mathbf{Z}_+^n)^{\Hom(F,E)}$, a collection of Bernstein blocks $\Omega\in \prod_{v\in S_p}\mathfrak{B}(F_v,n)$ and an $\mathcal{O}$-integral $\Omega$-system $\mathcal{L}$.

The following simple lemma will be used repeatedly in the proof of Lemma \ref{Lemma_BoundaryLGC}.
\begin{lemma}\label{lemma_shortexactsequencereduction}
    Suppose we are given an integer $h\geq 1$, an element $T\in \mathbf{T}^S$, a short exact sequence
    \begin{equation*}
        0\to M'\to M\to M''\to 0
    \end{equation*}
    of $\mathbf{T}_a$-modules of finite cardinality and ideals $J'\trianglelefteq \mathbf{T}_a(M')$, $J''\trianglelefteq \mathbf{T}_a(M'')$ such that, for every maximal ideal $\mathfrak{m}\trianglelefteq \mathbf{T}^S$ with residue field $k$, the following are satisfied. 
    \begin{enumerate}
        \item We have $(TJ')^{h}=0$ and $(TJ'')^{h}=0$.
        \item The $\mathbf{T}^S$-algebras $(\mathbf{T}_a(M')/J')_{\mathfrak{m}}$, $(\mathbf{T}_a(M'')/J'')_{\mathfrak{m}}$ are of Galois type with associated determinants $D'$, $D''$.
        \item The determinants $D'$, $D''$ satisfy $\textnormal{LGC}(S,a,\lambda,\Omega,\mathcal{L})$.
    \end{enumerate}
    Then there is an ideal $J\trianglelefteq \mathbf{T}_a(M)$ such that $(TJ)^{2h}=0$, and for every maximal ideal $\mathfrak{m}$ as above, $(\mathbf{T}_a(M)/J)_{\mathfrak{m}}$ satisfies $(2)$ and $(3)$.
\end{lemma}
\begin{proof}
    There is a square-zero ideal $I\trianglelefteq \mathbf{T}_a(M)$ and a surjection
    \begin{equation*}
        f\colon\mathbf{T}_a(M'\oplus M'')\twoheadrightarrow\mathbf{T}_a(M)/I
    \end{equation*}
    of $\mathbf{T}_a$-algebras. Let $I'\trianglelefteq \mathbf{T}_a(M'\oplus M'')$ be the preimage of $(J',J'')$ under the inclusion
    \begin{equation*}
        \mathbf{T}_a(M'\oplus M'')\hookrightarrow\mathbf{T}_a(M')\times \mathbf{T}_a(M'') 
    \end{equation*}
    and note that it satisfies $(TI')^{h}=0$. Let $J$ be the preimage of $f(I')$ under $\mathbf{T}_a(M)\twoheadrightarrow\mathbf{T}_a(M)/I$ and note that we have $(TJ)^{2h}=0$. Moreover, one checks easily that $(\mathbf{T}_a(M)/J)_{\mathfrak{m}}$ satisfies conditions $(2)$ and $(3)$.
\end{proof}

\begin{lemma}\label{Lemma_BoundaryLGC}

Suppose that, for $1\leq n'\leq n$, we are given $T_{n'}\in\mathbf{T}_{n'}^S$ such that the following hold.
\begin{enumerate}
    \item For each $1\leq n'<n$, $P(n',a,S,T_{n'})$ holds true.
    \item For each proper standard Levi subgroup $M=\textnormal{GL}_{n_1}\times ...\times \textnormal{GL}_{n_r}$ of $\textnormal{GL}_{n,\mathcal{O}_F}$, $T_{n_1}\otimes...\otimes T_{n_r}$ divides $\mathcal{S}_M(T_n)$.
\end{enumerate}
    Then we can find integers $N_1,N_2\geq 1$ with the following properties:
    \begin{itemize}
        \item The integer $N_1$ depends only on $n, F$ and $T_{n}$.
        \item The integer $N_2$ depends only on $n, F,\lambda,\mathcal{L}$ and $T_{n}$.
        \item For every $S\setminus\{a\}$-good subgroup $K\leq \textnormal{GL}_n(\mathbf{A}_F^{\infty})$ and integer $m\geq 1$, there exists an ideal $J \leq \mathbf{T}_a(R\Gamma(\partial X_K, \sigma_{\Omega,\lambda}^{\circ}/ \varpi^m))$ satisfying $(T_n^2p^{N_2}J)^{N_1} = 0$ and such that for any maximal ideal $\mathfrak{m}\trianglelefteq \mathbf{T}^S$ of residue field $k$ and such that we have $J_{\mathfrak{m}}\neq \mathbf{T}_a(R\Gamma(\partial X_K, \sigma_{\Omega,\lambda}^{\circ}/ \varpi^m))_{\mathfrak{m}}$, the localisation 
        \[ \mathbf{T}_{\partial}:=\mathbf{T}_a(R\Gamma(\partial X_K, \sigma_{\Omega,\lambda}^{\circ}/ \varpi^m))_{\mathfrak{m}}/J_{\mathfrak{m}} \]
        satisfies the following properties:
       \begin{enumerate}
           \item As a $\mathbf{T}^S$-algebra, it is of Galois type with associated determinant $D_{\mathfrak{m}}:G_{F,S}\to \mathbf{T}_{\partial}$.
           \item The determinant $D_{\mathfrak{m}}:G_{F,S}\xrightarrow{}\mathbf{T}_{\partial}$ satisfies $\textnormal{LGC}(S,a,\lambda,\Omega,\mathcal{L}).$
        \end{enumerate}
    \end{itemize}
\end{lemma}
\begin{proof}
    Fix $m\geq 1$, $K$ as in the statement. By considering the $\mathbf{T}_a$-equivariant filtration of $H^{\ast}(\partial X_K, \sigma_{\Omega,\lambda}^{\circ}/ \varpi^m)$ provided by excision long exact sequences induced by the stratification \eqref{equation_BorelSerreStratification}, it suffices to prove the Lemma for $H^{\ast}_c(X_K^Q, \sigma_{\Omega,\lambda}^{\circ}/ \varpi^m)$ in place of $H^{\ast}(\partial X_K, \sigma_{\Omega,\lambda}^{\circ}/ \varpi^m)$ for a fixed standard parabolic subgroup $Q=MN\leq \textnormal{GL}_n$.
    
    By Proposition \ref{Proposition_EisFunc}, it suffices to show the existence of integers $N_3,N_4\geq 1$ such that, for every $w\in W^{\textnormal{M}_a}$ and $\Omega_{\textnormal{M}_a}\in \mathfrak{B}(\textnormal{M}_a,\Omega)$, there is an ideal
    \begin{equation*}
        J\trianglelefteq \mathbf{T}_a\left(H_{\Omega_{\textnormal{M}_a},w} \right)
    \end{equation*}
    satisfying $(T_n^2p^{N_4}J)^{N_3}=0$ such that, for every maximal ideal $\mathfrak{m}$, the assertions $(1)$ and $(2)$ hold for the $\mathbf{T}_a$-algebra $(\mathbf{T}_a(H_{\Omega_{\textnormal{M}_a},w})/J)_{\mathfrak{m}}$ , where
    \begin{equation*}
        H_{\Omega_{\textnormal{M}_a},w}:=H^{\ast}_{c}(X_{K_{M}},\sigma_{\textnormal{M}_a,w\cdot\lambda}^{\circ}/\varpi^m)
    \end{equation*}
    for some $S\setminus\{a\}$-good level subgroup $K_M\leq M(\mathbf{A}_{F}^{\infty})$.

    Note that the $\mathbf{T}_a$ action on $H_{\Omega_{\textnormal{M}_a},w}$ is via
    \begin{equation*}
        \mathcal{S}_M\otimes\textnormal{BD}^{\circ}\colon\mathbf{T}_a\to \mathbf{T}_{M,a}:=\mathbf{T}_M^S\otimes\mathfrak{Z}_{M(F_a),\mathcal{L}_a,\lambda_a}^{\circ}.
    \end{equation*}
    To proceed further, we introduce some notation. Write $M=\textnormal{GL}_{n_1}\times...\times\textnormal{GL}_{n_r}$ and consider the corresponding isomorphism $\mathbf{T}_{M}^S\cong \mathbf{T}_{n_1}^{S}\otimes...\otimes\mathbf{T}^S_{n_r}$. Given a maximal ideal $\mathfrak{n}\trianglelefteq\mathbf{T}^{S}_M$ of residue field $k$, for $i=1,...,r$, write $\mathfrak{n}_i\trianglelefteq \mathbf{T}_{n_i}^S$ for the corresponding maximal ideal of residue field $k$. Assume that $\mathfrak{n}_i$ are of Galois type, write $\overline{D}_{\mathfrak{n}_i}:G_{F,S}\to k$ for the associated determinant and set $\overline{D}_{\mathfrak{n}_i,a}:=\overline{D}_{\mathfrak{n}_i}\mid_{G_{F_a}}$.
    Moreover, for $w\in W^{\textnormal{M}_a}$ and $\Omega_{\textnormal{M}_{a}}$, write $w\cdot \lambda_a=(\lambda_{1},\dots,\lambda_{r})\in (\mathbf{Z}_{+}^{n_1}\times\dots\times\mathbf{Z}_+^{n_r})^{\Hom(F_a,E)}$ and $\Omega_{\textnormal{M}_a}=(\Omega_1,\dots,\Omega_r)\in \mathfrak{B}(F_a,n_1)\times\dots\times \mathfrak{B}(F_a,n_r)$, and write $\mathcal{L}_i$ for the $\mathcal{O}$-integral $\Omega_i$-system induced by $\mathcal{L}$. Finally, set
    \begin{equation*}
        R^{w\cdot \lambda,\Omega_{\textnormal{M}_a}}_{\overline{D}_{\mathfrak{n},a}}:=R^{\lambda_1,\Omega_1}_{\overline{D}_{\mathfrak{n}_1,a}}\widehat{\otimes}_{\mathcal{O}}\dots\widehat{\otimes}_{\mathcal{O}}R^{\lambda_r,\Omega_r}_{\overline{D}_{\mathfrak{n}_r,a}}.
    \end{equation*}

    By the Künneth formula, Lemma \ref{lemma_shortexactsequencereduction} and the assumptions of the Lemma, there are integers $N_3,N_4\geq 1$ such that for every $w$ and $\Omega_{\textnormal{M}_a}$, there is an ideal
    \begin{equation*}
        J_M\trianglelefteq \mathbf{T}_{M,a}\left(H_{\Omega_{\textnormal{M}_a},w} \right)
    \end{equation*}
    satisfying $(T_n^2p^{N_4}J_{M})^{N_3}=0$ such that, for every maximal ideal $\mathfrak{n}=(\mathfrak{n}_1,...,\mathfrak{n}_r)\trianglelefteq \mathbf{T}_M^S$, the localised $\mathbf{T}_{M,a}$-algebra
    \begin{equation*}
        (A_{M})_{\mathfrak{n}}:=\big(\mathbf{T}_{M,a}\left(H_{\Omega_{\textnormal{M}_a},w} \right)/J_{M}\big)_{\mathfrak{n}}
    \end{equation*}
   is of Galois type as a $\mathbf{T}^S_{n_i}$-algebra for every $i=1,...,r$ with associated determinant $D_{\mathfrak{n}_i}:G_{F,S}\to (A_{M})_{\mathfrak{n}}$ satisfying $\textnormal{LGC}(S,a,\lambda_i,\Omega_i,\mathcal{L}_i)$.

   We set
   \begin{equation*}
       D_{\mathfrak{n}}:= D_{\mathfrak{n}_1}(n-n_1)\oplus...\oplus D_{\mathfrak{n}_r}(n-(n_1+...+n_r))\colon G_{F,S}\to (A_{M})_{\mathfrak{n}}.
   \end{equation*}

   Now let $\mathfrak{m}\trianglelefteq \mathbf{T}^S$ be a maximal ideal of residue field $k$ supported on $H_{\Omega_{\textnormal{M}_a},w}$ and let $\mathfrak{S}$ be the set of maximal ideals of $\mathbf{T}_M^S$ supported on $H_{\Omega_{\textnormal{M}_a},w}$ such that
   \begin{equation*}
      \mathbf{T}_{M,a}(H_{\Omega_{\textnormal{M}_a},w})_{\mathfrak{m}}= \prod_{\mathfrak{n}\in \mathfrak{S}}\mathbf{T}_{M,a}(H_{\Omega_{\textnormal{M}_a},w})_{\mathfrak{n}}.
   \end{equation*}
   After possibly enlarging the coefficient field $E$, we can assume that each $\mathfrak{n}\in \mathfrak{S}$ has residue field $k$. This can be done uniformly for each $\mathfrak{m}$ in the support of $\oplus_{w,\Omega_{\textnormal{M}_a}}H_{\Omega_{\textnormal{M}_a},w}$ as there are finitely many such maximal ideals. Moreover, let $J:=J_{M}\cap \mathbf{T}_a(H_{\Omega_{\textnormal{M}_a},w})$ and note that, for $v\notin S$, $D_{\mathfrak{n}}(X-\textnormal{Frob}_v)\in (A_{M})_{\mathfrak{n}}[X]$ is given by the image of $P_v(X)\in \mathbf{T}^S[X]$ and, in particular, its coefficients lie in the subalgebra
   \begin{equation*}
(A)_{\mathfrak{m}}:=\big(\mathbf{T}_a(H_{\Omega_{\textnormal{M}_a},w})/J\big)_{\mathfrak{m}}\leq (A_{M})_{\mathfrak{m}}=\prod_{\mathfrak{n}}(A_{M})_{\mathfrak{n}}.
   \end{equation*}
Therefore, by \cite[Corollary 1.14]{Che14}  and Chebotarev's density theorem, we see that $D_{\mathfrak{m}}:=\prod_{\mathfrak{n}}D_{\mathfrak{n}}$
takes values in $(A)_{\mathfrak{m}}$, showing that $(A)_{\mathfrak{m}}$ is of Galois type as a $\mathbf{T}^S$-algebra.

It is clear that $D_{\mathfrak{m}}$ satisfies (1) of $\textnormal{LGC}(S,a,\lambda,\Omega,\mathcal{L})$.

Finally, to check the existence of the commutative diagram \eqref{equation_LGCdiagram} for $(A)_{\mathfrak{m}}$, it suffices to show that the diagram exists after extending scalars to $(A_M)_{\mathfrak{m}}$ and projecting onto any of the factors $(A_{M})_{\mathfrak{n}}$. We have a diagram
\begin{equation*}
    \begin{tikzcd}
	{\mathfrak{Z}_{\overline{D}_{\mathfrak{m},a}}^{\circ}} && {R^{\lambda,\Omega_a}_{\overline{D}_{\mathfrak{m},a}}} & {R_{\overline{D}_{\mathfrak{m},a}}} \\
	& {(A_{M})_{\mathfrak{n}}} \\
	{\mathfrak{Z}_{M(F_a),\mathcal{L},w\cdot\lambda,(\overline{D}_{\mathfrak{n}_i})_i}^{\circ}} && {R^{w\cdot \lambda,\Omega_{\textnormal{M}_a}}_{(\overline{D}_{\mathfrak{n}_i,a})_i}} & {R_{(\overline{D}_{\mathfrak{n}_i,a})_i}}
	\arrow["{\eta^{\circ}}", from=1-1, to=1-3]
	\arrow["{\textnormal{nat}}", from=1-1, to=2-2]
	\arrow["{\textnormal{BD}^{\circ}}", from=1-1, to=3-1]
	\arrow["{D_{\mathfrak{n},a}}"', dotted, from=1-3, to=2-2]
	\arrow[, from=1-3, to=3-3]
	\arrow[two heads, from=1-4, to=1-3]
	\arrow["{(D_i)_i\mapsto \oplus_iD_i(n_{i+1}+..+n_r)}"{description}, from=1-4, to=3-4]
	\arrow["{\textnormal{nat}}", from=3-1, to=2-2]
	\arrow["{\eta_{M(F_a)}^{\circ}}"', from=3-1, to=3-3]
	\arrow["{(D_{\mathfrak{n}_i,a})_i}", dashed, from=3-3, to=2-2]
	\arrow[two heads, from=3-4, to=3-3]
\end{tikzcd}
\end{equation*}
where the existence of the dotted factorisation of $D_{\mathfrak{n},a}$ follows from the existence of the dashed factorisation of $(D_{\mathfrak{n}_i,a})_i$ (implied by the inductive hypothesis) and the commutativity of the right square. The commutativity of the top triangle follows from the commutativity of the other three triangles (from left to right, by definition, by induction and by construction, respectively) and of the left square (thanks to Lemma \ref{Lemma_InterpolationOfSSvsEisenstein}).
\end{proof}
\subsection{The degree shifting argument}
We have now arrived to the analysis of interior cohomology. As before, $n\geq 1$ is a fixed integer, $F$ is fixed a CM number field satisfying Assumption \ref{assumption_CM} and we have a fixed decomposition $S_p=\widetilde{S}_p\coprod \widetilde{S}_p^c$. Moreover, $E/\mathbf{Q}_p$ is a finite field extension such that $\# \Hom(F,E)=[F:\mathbf{Q}]$, with ring of integers $\mathcal{O}$, a fixed choice of uniformiser $\varpi\in \mathcal{O}$ and residue field $k=\mathcal{O}/\varpi$.

\subsubsection{The quasi-split unitary group, and its Siegel parabolic}
We now define the key linear algebraic groups needed for the degree shifting argument. Let $\Psi_n$ denote the $n \times n$ matrix with $1$s on the antidiagonal and $0$s elsewhere, and let $J_n$ denote the $2n \times 2n$ matrix $J_n = \left( \begin{smallmatrix} 0 & \Psi_n \\ -\Psi_n & 0 \end{smallmatrix}\right)$. We introduce the group scheme $\widetilde{G}:=\widetilde{G}_n$ over $\cO_{F^+}$ whose functor of points is given by
\[ \widetilde{G}_n(R) = \{ g \in \GL_{2n}(R \otimes_{\cO_{F^+}} \cO_F) \mid {}^t g J_n g^c = J_n \}. \]
We write $T \leq B \leq \widetilde{G}$ for the subgroups of diagonal and upper-triangular matrices, respectively. After extension to $F^+$, $\widetilde{G}$ becomes a reductive group scheme and $T, B$ a maximal torus and Borel subgroup. 

We write $G \leq P \leq \widetilde{G}$ for the subgroups of block diagonal matrices and block upper-triangular matrices, respectively, where the blocks have size $n, n$. After extension to $F^+$, $G$ becomes a Levi subgroup of the maximal parabolic subgroup $P$, that we call that Siegel parabolic. We identify $G$ with $\Res_{\cO_F / \cO_{F^+}} \GL_n$ using the formula 
\[ g = \left( \begin{array}{cc}  A & 0 \\ 0 & D \end{array}\right) \in G(R) \mapsto D \in \GL_n(R \otimes_{\cO_{F^+}} \cO_F).  \]
Under this identification, $T$ is identified with $T_n$ and $B\cap G$ with $B_n$. In particular, given $\lambda\in (\mathbf{Z}_+^n)^{\Hom(F,E)}$, the discussion of \S\ref{subsubsection_GL_nlocallysymmetricspaces} yields an $E[G(F^{+}\otimes_{\mathbf{Q}}\mathbf{Q}_p)]$-module $V_{\lambda} $ with a $G(\mathcal{O}_{F^+\otimes_{\mathbf{Z}}}\mathbf{Z}_p)$-stable $\mathcal{O}$-lattice $\mathcal{V}_{\lambda}$.

To define coefficient systems for the group $\widetilde{G}$ coming from algebraic representations, we now use our Assumption \ref{assumption_CM}. Namely, the choice of $\widetilde{S}_p$ induces an identification
\[ \prod_{v \in \widetilde{S}_p} \iota_v \colon \widetilde{G}(F^+ \otimes_\mathbf{Q} \mathbf{Q}_p) \cong \prod_{v \in \widetilde{S}_p} \GL_{2n}(F_v). \]
Writing $\widetilde{I}_p \subset \Hom(F, E)$ for the set of embeddings inducing a place of $\widetilde{S}_p$, we may identify the character group of the torus $(\Res_{F^+ / \mathbf{Q}} T)_E$ with $(\mathbf{Z}^{2n})^{\widetilde{I}_p}$, and the subset of characters which are $(\Res_{F^+ / \mathbf{Q}} B)_E$-dominant with the subset of weights $\widetilde{\lambda} \in (\mathbf{Z}^{2n}_+)^{\widetilde{I}_p}$. The implicit identification $(\mathbf{Z}^n)^{\Hom(F, E)} = (\mathbf{Z}^{2n})^{\widetilde{I}_p}$ is given by the formula
\[ \widetilde{\lambda}_\tau = (-\lambda_{\tau c, n}, \dots, -\lambda_{\tau c, 1}, \lambda_{\tau, 1}, \dots, \lambda_{\tau, n}), \]
if $\tau \in \widetilde{I}_p$. If $\lambda \in (\mathbf{Z}^n)^{\Hom(F, E)}$, $\widetilde{\lambda} \in (\mathbf{Z}^{2n})^{\widetilde{I}_p}$ correspond under this identification, then we will say that they are \textit{matching weights}. If $\widetilde{\lambda} \in (\mathbf{Z}^{2n}_+)^{\widetilde{I}_p}$, then we define  associated representation $\widetilde{V}_{\widetilde{\lambda}} = \otimes_{\tau \in \widetilde{I}_p} \widetilde{V}_{\widetilde{\lambda}_\tau}$ and $\widetilde{\mathcal{V}}_{\widetilde{\lambda}} = \otimes_\tau \widetilde{\mathcal{V}}_{\widetilde{\lambda}_\tau}$ as before.

We adopt the notations for Weyl groups $W_\tau$, $W_{\overline{v}}$, $W$, and $W_{P, \tau}$, $W_{P, \overline{v}}$, and $W_P$, and $W^P_\tau$, $W^P_{\overline{v}}$, and $W^P$, as well as weights $\widetilde{\lambda}_\tau$, $\widetilde{\lambda}_{\overline{v}}$, $\lambda_\tau$, and $\lambda_{\overline{v}}$ introduced in \cite[\S 4.2]{10author}. Thus especially $W$ is identified with the Weyl group $W((\Res_{F^+ / \mathbf{Q}} \widetilde{G})_E, (\Res_{F^+ / \mathbf{Q}} T)_E)$, and if $\widetilde{\lambda} \in (\mathbf{Z}^{2n}_+)^{\widetilde{I}_p}$ and $w \in W^P$, and $\lambda \in (\mathbf{Z}^n)^{\Hom(F, E)}$ matches $w \cdot \widetilde{\lambda} = w(\widetilde{\lambda} + \rho) - \rho$, then $\lambda$ in fact lies in $(\mathbf{Z}^n_+)^{\Hom(F, E)}$.

We next introduce unramified Hecke algebras. Let $S$ be a finite set of finite places of $F$, containing the $p$-adic ones, and let $K$ be an $S$-good subgroup of $\GL_n(\mathbf{A}_F^{\infty})$. We have set $\mathbf{T}^S = \mathcal{H}(\GL_n(\mathbf{A}_F^{\infty, S}), K^S) \otimes_\mathbf{Z}\cO$.

Similarly, let $S$ be a finite set of finite places of $F$, containing the $p$-adic ones, and stable under complex conjugation. Then we write $\overline{S}$ for the set of places of $F^+$ lying below a place of $S$. If $\widetilde{K}$ is an $S$-good subgroup of $\widetilde{G}(\mathbf{A}_{F^+}^\infty)$, then we set $\widetilde{\mathbf{T}}^S = \mathcal{H}(\widetilde{G}(\mathbf{A}_{F^+}^{\infty, \overline{S}}), \widetilde{K}^{\overline{S}}) \otimes_\mathbf{Z} \cO$. 

Finally, we discuss the `unnormalised Satake transform' associated to a standard parabolic subgroup of $\widetilde{G}$. Choose a partition $n = n_1 + n_2 + \dots + n_{r-1} + n_r$ with $n_i \geq 1$ for each $i = 1, \dots, r$, and let $Q \leq \widetilde{G}$ denote the intersection of $\widetilde{G}$ with the subgroup of block upper-triangular matrices, with blocks of sizes \[ n_1, n_2, \dots, n_{r-1}, 2n_r, n_{r-1}, \dots, n_2, n_1. \]
Let $M \leq Q$ denote the block diagonal matrices, and $N \leq Q$ the kernel of the projection $Q \to M$. If $\overline{S}$ contains the ramified places of $F / F^+$, then the base extension of $\widetilde{G}$ to $\cO_{F^+, \overline{S}}$ is a reductive group scheme and $Q$ is a parabolic subgroup with Levi subgroup $M$. Further, we can identify (with the obvious notation)
\[ \mathcal{H}(M(\mathbf{A}_{F^+}^{\overline{S}, \infty}), M(\widehat{\cO}_{F^+}^{\overline{S}})) \otimes_\mathbf{Z} \mathcal{\cO} = \mathbf{T}^S_{n_1} \otimes_\cO \dots \otimes \mathbf{T}^S_{n_{r-1}} \otimes \widetilde{\mathbf{T}}^S_{n_r}, \]
and the unnormalised Satake transform $\mathcal{S}^{\widetilde{G}}_M \colon \widetilde{\mathbf{T}}^S \to \mathbf{T}^S_M$ is described explicitly by \cite[Proposition-Definition 5.3]{New16}. 

The most important special case arises when $Q = P$ is the Siegel parabolic. In this case, we write simply $\mathcal{S} = \mathcal{S}_n$ for the induced homomorphism $\mathcal{S} \colon \widetilde{\mathbf{T}}^S \to \mathbf{T}^S$. 

\subsubsection{The element $T_n$}\label{subsubsection_theelementT_n}

We finally provide a construction of $T_n$. Let $\mathfrak{p}$ be a \textit{rational} prime that splits completely in $F$ and such that $S_{\mathfrak{p}}(F)\cap S=\emptyset$. For $n\geq 1$, let $\Delta_{\mathfrak{p},n}:=\prod_{v\mid \mathfrak{p}}\Delta_{v,n}\in \widetilde{\mathbf{T}}^S$ where $\Delta_{v,n}$ is the spherical Hecke operator defined in \S\ref{subsection_Vanishing statement}. 
\begin{lemma}\label{lemma_existence_of_T_n}
    There are integers $M_1,M_2\geq 1$ only depending on $n$, $F$ and $\mathfrak{p}$ such that, if we set $\widetilde{T}_{n'}:=M_2^{n'}\Delta_{\mathfrak{p},n'}^{M_1}$ and $T_{n'}:=\mathcal{S}_{n'}(\widetilde{T}_{n'})$, then the following are satisfied.
    \begin{itemize}
        \item For each proper Levi subgroup $M=\textnormal{GL}_{n_1}\times...\times \textnormal{GL}_{n_r}$ of $\textnormal{GL}_{n,\mathcal{O}_F}$, $\mathcal{S}_{n_1}(\widetilde{T}_{n_1})\otimes...\otimes \mathcal{S}_{n_r}(\widetilde{T}_{n_r})$ divides $\mathcal{S}^{\widetilde{G}}_M(\widetilde{T}_n)$.
        \item  For each $1\leq n'\leq n$, each good subgroup $\widetilde{K}\leq \widetilde{G}_{n'}(\mathbf{A}_{F^+}^{\infty})$ and any local system $\mathcal{W}$ on $X^{\widetilde{G}_{n'}}_{\widetilde{K}}$ associated with a $\widetilde{K}_p$-stable $\mathcal{O}$-lattice in a finite dimensional locally algebraic $E$-representation of $\widetilde{K}_p$, we have
        \[ \widetilde{T}_{n'} H^q(X^{\widetilde{G}_{n'}}_{\widetilde{K}}, \mathcal{W}/ \varpi^m) = 0 \text{ for all }m \geq 1, q < \frac{1}{2} \dim_{\mathbf{R}} X^{\widetilde{G}_{n'}}; \]
        and 
        \[ \widetilde{T}_{n'} H^q_c(X^{\widetilde{G}_{n'}}_{\widetilde{K}}, \mathcal{W}/ \varpi^m) = 0 \text{ for all }m \geq 1, q > \frac{1}{2} \dim_{\mathbf{R}} X^{\widetilde{G}_{n'}}. \]
    \end{itemize}
\end{lemma}
\begin{proof}
Combining Theorem \ref{VanishingUpToBoundedTorsion} with \cite[Lemma 2.1.1]{caraiani_scholze}  and the observation that for each proper Levi subgroup $M = \GL_{n_1} \times \dots \times \GL_{n_r}$ of $G$, $\mathcal{S}_{n_1}(\Delta_{\mathfrak{p}, n_1}) \otimes \dots \otimes \mathcal{S}_{n_r}(\Delta_{\mathfrak{p}, n_r})$ divides $\mathcal{S}^{\widetilde{G}}_M(\Delta_{\mathfrak{p}, n})$ in $\mathbf{T}^S_M$, we obtain the desired vanishing when $\mathcal{W}$ is the trivial coefficient system.

A standard argument with the Hochschild--Serre spectral sequence for cohomology without compact support, combined with Poincar\'e duality shows that the desired vanishing for general $\mathcal{W}$ holds after changing $M_1$ to $n^2[F:\mathbf{Q}]M_1$ and $M_2$ to $M_2^{n^2[F:\mathbf{Q}]}$.
\end{proof}
From now on, we fix a rational prime $\mathfrak{p}$ as above and a choice of Hecke operators $\widetilde{T}_{n'}$ as in the statement of the Lemma. Moreover, we set
\[
    T_{n'}=\mathcal{S}_{n'}(\widetilde{T}_{n'}) \text{ for }1\leq n'\leq n.
\]
\begin{remark}
    Note that for a $\textnormal{GL}_n(\mathcal{O}_F\otimes_{\mathbf{Z}}\mathbf{Z}_{\mathfrak{p}})$-spherical cuspidal automorphic representation $\pi$ of $\textnormal{GL}_n(\mathbf{A}_F)$, after possibly twisting by a finite order Hecke character, $T_n$ acts on $\pi^{\textnormal{GL}_n(\mathcal{O}_F\otimes_{\mathbf{Z}}\mathbf{Z}_{\mathfrak{p}})}$ via a non-zero scalar as $\pi$ is generic by \cite{Sha74}.
\end{remark}

\subsubsection{The Local System}
We now introduce the local system on the locally symmetric spaces $X^{\widetilde{G}}_{\widetilde{K}}$ that we will use in our degree shifting argument.

Let $\Lambda=\mathcal{O}$ or $\mathcal{O}/\varpi^m$. Given a smooth $\Lambda$-representation $\sigma$ of $K_{\overline{a}}(0):=\textnormal{GL}_n(\mathcal{O}_{F_a})\times \textnormal{GL}_n(\mathcal{O}_{F_{a^c}})$ on a finite free $\Lambda$-module, we first explain how to lift it to a representation of some deep enough $P$-parahoric subgroup of $\widetilde{G}(\mathcal{O}_{F^+_{\overline{a}}})$. For an integer $t\geq 0$, set 
\begin{equation*}
    K_a(t):=\ker (\textnormal{GL}_n(\mathcal{O}_{F_a})\to\textnormal{GL}_n(\mathcal{O}_{F_a}/\varpi^t_a))
\end{equation*} 
and similarly define $K_{a^c}(t)$ and $K_{\overline{a}}(t)$. For integers $t\geq 0$ and $u\geq \max(t,1)$, we set $\mathcal{P}_{\overline{a}}(t,u)$ to be the preimage under $\iota_{a}\colon\widetilde{G}(F^+_{\overline{a}})\xrightarrow{\sim}\textnormal{GL}_{2n}(F_a)$ of the compact open subgroup
\begin{equation*}
    \left\{\begin{pmatrix}
A & B \\
C & D 
\end{pmatrix}\in \textnormal{GL}_{2n}(\mathcal{O}_{F_a})\mid C\in \varpi_a^uM_{n\times n}(\mathcal{O}_{F_a}),A,D\in K_{a}(t) \right\}.
\end{equation*}
Let now $u\geq 1$ be an integer such that $K_{\overline{a}}(u)$ acts trivially on $\sigma$. Then we set
\begin{equation*}
    \textnormal{Inf}^{\mathcal{P}_{\overline{a}}(0,u)}(\sigma)\in \textnormal{Mod}_{\textnormal{sm}}(\mathcal{P}_{\overline{a}}(0,u),\Lambda)
\end{equation*}
to be the representation defined as follows. The underlying $\Lambda$-module is given by that of $\sigma$ and, for $g=\iota_a^{-1}\left(\begin{pmatrix}
A & B \\
C & D 
\end{pmatrix}\right)\in \mathcal{P}_{\overline{a}}(0,u)$, we set
\begin{equation*}
    \textnormal{Inf}^{\mathcal{P}_{\overline{a}}(0,u)}(\sigma)(g)=\sigma(A,\theta_n^{-1}(D))
\end{equation*}
where\footnote{We note that the embedding $G(F^+_{\overline{a}})\hookrightarrow \textnormal{GL}_{2n}(F_a)$ induced by $\iota_a$ sends $(A,B)$ to $(\theta_n(B),A)$.} $\theta_n\colon\textnormal{GL}_n(F_{a^c})\xrightarrow{\sim}\textnormal{GL}_n(F_{a})$, $A\mapsto (\Psi_n{}^tA^{-1}\Psi_n)^c$. Given another integer $u'\geq u\geq 1$, we have $\textnormal{Res}_{\mathcal{P}(0,u')}(\textnormal{Inf}^{\mathcal{P}_{\overline{a}}(0,u)}(\sigma))=\textnormal{Inf}^{\mathcal{P}_{\overline{a}}(0,u')}(\sigma)$.

Suppose that we are given $\overline{a}\in \overline{S}_p$, a dominant weight $\widetilde{\lambda}\in (\mathbf{Z}_+^{2n})^{\widetilde{I}_p}$, a Bernstein block $\Omega_{\overline{a}}=(\Omega_a,\Omega_{a^c})$ of $\textnormal{Rep}_{\textnormal{sm}}(G(F^+_{\overline{a}}),\overline{\mathbf{Q}}_p)$, and a pair of an $\mathcal{O}$-integral $\Omega_{a}$- and $\Omega_{a^c}$-system $\mathcal{L}_{\overline{a}}=(\mathcal{L}_a,\mathcal{L}_{a^c})$. For the smallest integer $u=u_{\Omega_{\overline{a}}}\geq 1$ such that $K_{\overline{a}}(u)$ acts trivially on $\sigma_{\Omega_{\overline{a}}}$, we define
\begin{align*}
    \mathcal{W}_{\widetilde{\lambda}_{\overline{a}},\mathcal{L}_{\overline{a}}} & :=\textnormal{Inf}^{\mathcal{P}_{\overline{a}}(0,u)}(\sigma_{\Omega_{\overline{a}}}^{\circ})\otimes_{\mathcal{O}}\mathcal{V}_{\widetilde{\lambda}_{\overline{a}}}\in \textnormal{Mod}_{\textnormal{sm}}\left( \mathcal{P}_{\overline{a}}(0,u),\mathcal{O}\right), \\
    \mathcal{W}_{\widetilde{\lambda},\mathcal{L}_{\overline{a}}} &:=\textnormal{Inf}^{\mathcal{P}_{\overline{a}}(0,u)}(\sigma_{\Omega_{\overline{a}}}^{\circ})\otimes_{\mathcal{O}}\mathcal{V}_{\widetilde{\lambda}}\in \textnormal{Mod}_{\textnormal{sm}}\left(\mathcal{P}_{\overline{a}}(0,u)\times\prod_{\overline{v}\in\overline{S}_p\setminus\{\overline{a}\}}\widetilde{G}(\mathcal{O}_{F^+_{\overline{v}}}),\mathcal{O}\right) .
\end{align*}
For any integer $u\geq u_{\Omega_{\overline{a}}}$ and $u\geq t\geq 0$, we further set
\begin{equation*}
    \mathcal{W}_{\widetilde{\lambda}_{\overline{a}},\mathcal{L}_{\overline{a}}}^{t,u}:= \textnormal{Res}_{\mathcal{P}_{\overline{a}}(t,u)}(\mathcal{W}_{\widetilde{\lambda}_{\overline{a},\mathcal{L}_{\overline{a}}}}).
\end{equation*}

\subsubsection{Some explicit Hecke operators}
We first quickly recall the definition of the Hecke polynomials in $\widetilde{\mathbf{T}}^S[X]$ and the effect of the unnormalised Satake transform on them.
For a finite place $\overline{v}\notin \overline{S}$ that is unramified in $F$ with a place $v\notin S$ above it and an integer $1\leq i\leq 2n$, let $\widetilde{T}_{v,i}\in \widetilde{\mathbf{T}}^S$ be the Hecke operator $T_{G,v,i}$ from \cite{New16}, Proposition-Definition 5.2 and set
\begin{equation*}
    \widetilde{P}_v(X)=X^{2n}+\sum_{i=1}^{2n}(-1)^iq_v^{\frac{i(i-1)}{2}}\widetilde{T}_{v,i}X^{2n-i}\in \widetilde{\mathbf{T}}^S[X].
\end{equation*}
We then have
\begin{equation}\label{equation_effectofSatake}
    \mathcal{S}(\widetilde{P}_v(X))=P_v(X)q_v^{n(2n-1)}P_{v^c}^{\vee}(q_v^{1-2n}X)\in \mathbf{T}^S[X]
\end{equation}
where, for a polynomial $f(X)$ of degree $d$ with invertible constant term $a_0$, $f^{\vee}(X):=a_0^{-1}X^df(X^{-1})$.

Finally, at $P$-parahoric level, we introduce the $U_p$-operators we need to define $P$-ordinary parts. We set 
\begin{equation*}
    \widetilde{u}_{a,n}:=\textnormal{diag}(\underbrace{\varpi_a,...,\varpi_a}_{n},1,...,1),\widetilde{u}_{a,2n}:=\textnormal{diag}(\varpi_a,...,\varpi_a)\in \textnormal{GL}_{2n}(F_a).
\end{equation*}
Given a dominant weight $\widetilde{\lambda}\in (\mathbf{Z}_+^{2n})^{\widetilde{I}_p}$ viewed as an $E$-rational character, we further set
\begin{equation*}
    \alpha_{\widetilde{\lambda}_a,n}:=\prod_{\tau:F_a\hookrightarrow E}\widetilde{\lambda}_{\tau}(\tau(\widetilde{u}_{a,n})), \textnormal{ }\alpha_{\widetilde{\lambda}_a,2n}:=\prod_{\tau:F_a\hookrightarrow E}\widetilde{\lambda}_{\tau}(\tau(\widetilde{u}_{a,2n}))\in E.
\end{equation*}
By lowest weight considerations (cf. \cite[Lemma 2.1.17]{CN25}), the elements $\alpha_{\widetilde{\lambda}_{\overline{a}},n}^{-1}\iota_a^{-1}(\widetilde{u}_{a,n})$ and $\alpha_{\widetilde{\lambda}_{\overline{a}},2n}^{-1}\iota_a^{-1}(\widetilde{u}_{a,2n})$ preserve $\mathcal{V}_{\widetilde{\lambda}_{\overline{a}}}\subset V_{\widetilde{\lambda}_{\overline{a}}}$.
In particular, for any Bernstein block $\Omega_{\overline{a}}=(\Omega_a,\Omega_{a^c})$ of $\textnormal{Rep}_{\textnormal{sm}}(G(F^+_{\overline{a}}),\overline{\mathbf{Q}}_p)$, pair of an $\mathcal{O}$-integral $\Omega_{a}$- and $\Omega_{a^c}$-system $\mathcal{L}_{\overline{a}}=(\mathcal{L}_a,\mathcal{L}_{a^c})$,  and integers $u\geq u_{\Omega_{\overline{a}}}$, $u\geq t\geq 0$, we can define the Hecke operators
\begin{align*}
    U_{a,n} &:=\bigg[\mathcal{P}_{\overline{a}}(t,u)\iota^{-1}_a(\widetilde{u}_{a,n})\mathcal{P}_{\overline{a}}(t,u),\id\otimes\big(\alpha_{\widetilde{\lambda}_{\overline{a}},n}^{-1}\iota_a^{-1}(\widetilde{u}_{a,n})\big)\bigg]\in \mathcal{H}(\mathcal{W}_{\widetilde{\lambda}_{\overline{a}},\mathcal{L}_{\overline{a}}}^{t,u}), \\
    U_{a,2n} &:=\bigg[\mathcal{P}_{\overline{a}}(t,u)\iota^{-1}_a(\widetilde{u}_{a,2n})\mathcal{P}_{\overline{a}}(t,u),\id\otimes \big(\alpha_{\widetilde{\lambda}_{\overline{a}},2n}^{-1}\iota_a^{-1}(\widetilde{u}_{a,2n})\big)\bigg]\in \mathcal{H}(\mathcal{W}_{\widetilde{\lambda}_{\overline{a}},\mathcal{L}_{\overline{a}}}^{t,u})^{\times},
\end{align*}
where, as before, $\mathcal{H}(\mathcal{W}_{\widetilde{\lambda}_{\overline{a}},\mathcal{L}_{\overline{a}}}^{t,u})$ denotes the convolution algebra of $\mathcal{W}_{\widetilde{\lambda}_{\overline{a}},\mathcal{L}_{\overline{a}}}^{t,u}$-bi-invariant functions $\widetilde{G}(F^+_{\overline{a}})\to \textnormal{End}_{\mathcal{O}}(\mathcal{W}_{\widetilde{\lambda}_{\overline{a}},\mathcal{L}_{\overline{a}}}^{t,u})$.
In particular, we obtain algebra homomorphisms
\begin{align*}
    \widetilde{\mathbf{T}}^S[U_{a,n},U_{a,2n}^{\pm 1}] &\to \textnormal{End}_{D^+(\mathcal{P}_{\overline{a}}(0,u)/\mathcal{P}_{\overline{a}}(t,u),\mathcal{O}/\varpi^m)}(R\Gamma_{(c)}(\widetilde{X}_{\widetilde{K}^{\overline{a}}\mathcal{P}_{\overline{a}}(t,u)},\mathcal{W}_{\widetilde{\lambda},\mathcal{L}_{\overline{a}}})), \\
    \widetilde{\mathbf{T}}^S[U_{a,n},U_{a,2n}^{\pm 1}] &\to \textnormal{End}_{D^+(\mathcal{P}_{\overline{a}}(0,u)/\mathcal{P}_{\overline{a}}(t,u),\mathcal{O}/\varpi^m)}(R\Gamma(\partial\widetilde{X}_{\widetilde{K}^{\overline{a}}\mathcal{P}_{\overline{a}}(t,u)},\mathcal{W}_{\widetilde{\lambda},\mathcal{L}_{\overline{a}}}))
\end{align*}
and these Hecke actions are compatible with respect to the natural maps between cohomology complexes with and without compact supports and the maps induced by restriction along $\mathcal{P}_{\overline{a}}(t',u')\leq \mathcal{P}_{\overline{a}}(t,u)$.

\subsubsection{Reminder on $P$-ordinary Hida theory}
We give a minimalistic reminder on the notion of $P$-ordinary parts at $\overline{a}\in \overline{S}_p$. For details and the whole picture, we refer to \cite[\S2.2]{CN25} and \cite[\S3]{Hev24}. Let $\widetilde{K}^{\overline{a}}\leq \widetilde{G}(\mathbf{A}_{F^+}^{\infty\cup\{\overline{a}\}})$ be a fixed $\overline{S}$-good subgroup. For integers $u\geq 1$, $u\geq t\geq 0$, set
\begin{equation*}
    \widetilde{K}(t,u):=\widetilde{K}^{\overline{a}}\mathcal{P}_{\overline{a}}(t,u)\leq \widetilde{G}(\mathbf{A}_{F^+}^{\infty}).
\end{equation*}
Given an integer $m\geq 1$ and $\widetilde{\lambda}$, $\Omega_{\overline{a}}$ and $\mathcal{L}_{\overline{a}}$ as before, following \cite{Kha17}, we define
\begin{equation*}
    R\Gamma_{(c)}(\widetilde{X}_{\widetilde{K}(t,u)},\mathcal{W}_{\widetilde{\lambda},\mathcal{L}_{\overline{a}}}/\varpi^m)^{\overline{a}\textnormal{-ord}}\in D^+(\mathcal{P}_{\overline{a}}(0,u)/\mathcal{P}_{\overline{a}}(t,u),\mathcal{O}/\varpi^m)
\end{equation*}
to be the maximal direct summand on which $U_{a,n}$ acts as an automorphism.
Let $\lambda_{\overline{a}}\in (\mathbf{Z}_+^n)^{\Hom(F_a,E)\coprod \Hom(F_{a^c},E)}$ be the dominant weight matching $\widetilde{\lambda}_{\overline{a}}$. The content of the following Proposition is the observation that, as a consequence of $P$-ordinary Hida theory, we have a canonical algebra homomorphism
\begin{equation*}
    \mathcal{H}(\sigma_{\Omega_{\overline{a}},\lambda_{\overline{a}}}^{\circ})\to \textnormal{End}_{D^+(\mathcal{O}/\varpi^m)}(R\Gamma_{(c)}(\widetilde{X}_{\widetilde{K}(0,u)},\mathcal{W}_{\widetilde{\lambda},\mathcal{L}_{\overline{a}}}/\varpi^m)^{\overline{a}\textnormal{-ord}}).
\end{equation*}
To state the Proposition, note that we have a natural $P(\mathcal{O}_{F^+_{\overline{a}}})\cong^{\iota_a} P_{(n,n)}(\mathcal{O}_{F_a})$-equivariant surjection\footnote{See for instance the discussion around \cite[Lemma 2.1.12]{CN25}.}
\begin{equation*}
    \mathcal{V}_{\widetilde{\lambda}_{\overline{a}}}\to \mathcal{V}_{\lambda_a}\otimes_{\mathcal{O}}\theta_n^{-1}\mathcal{V}_{\lambda_{a^c}}.
\end{equation*}
In particular, for every integer $m\geq 1$, 
we obtain a $\mathcal{P}_{\overline{a}}(0,m)$-equivariant surjection
\begin{equation}\label{equation_ModpmSurjection}
    \textnormal{pr}_m\colon\mathcal{W}_{\widetilde{\lambda}_{\overline{a}},\mathcal{L}_{\overline{a}}}^{0,m}/\varpi^m\to \textnormal{Inf}^{\mathcal{P}_{\overline{a}}(0,m)}(\sigma^{\circ}_{\Omega_{\overline{a}},\lambda_{\overline{a}}}/\varpi^m).
\end{equation}
\begin{prop}\label{proposition_OrdinaryPartsHeckeAction}
    Suppose that we are given $\overline{a}\in \overline{S}_p$, $\widetilde{\lambda}\in (\mathbf{Z}_+^{2n})^{\widetilde{I}_p}$, $\Omega_{\overline{a}}=(\Omega_a,\Omega_{a^c})$ of $\textnormal{Rep}_{\textnormal{sm}}(G(F^+_{\overline{a}}),\overline{\mathbf{Q}}_p)$, and $\mathcal{L}_{\overline{a}}=(\mathcal{L}_a,\mathcal{L}_{a^c})$ as before. Moreover, assume we are given integers $m\geq 1$, $u\geq u_{\Omega_{\overline{a}}}$ and set $h:=\max(m,u)$.

    Then the natural maps
    \begin{align*}
        R\Gamma_{(c)}(\widetilde{X}_{\widetilde{K}(0,u)},\mathcal{W}_{\widetilde{\lambda},\mathcal{L}_{\overline{a}}}/\varpi^m)^{\overline{a}\textnormal{-ord}}\to R\Gamma_{(c)}(\widetilde{X}_{\widetilde{K}(0,h)},\mathcal{W}_{\widetilde{\lambda},\mathcal{L}_{\overline{a}}}/\varpi^m)^{\overline{a}\textnormal{-ord}}\to \\
        R\Gamma_{(c)}\bigg(\widetilde{X}_{\widetilde{K}(0,h)},\textnormal{Inf}^{\mathcal{P}_{\overline{a}}(0,h)}\big(\sigma_{\Omega_{\overline{a}},\lambda_{\overline{a}}}^{\circ}/\varpi^m\big)\otimes_{\mathcal{O}}\mathcal{V}_{\widetilde{\lambda}^{\overline{a}}}\bigg)^{\overline{a}\textnormal{-ord}}
    \end{align*}
    induced by restriction and the $\mathcal{P}_{\overline{a}}(0,h)$-equivariant surjection \eqref{equation_ModpmSurjection}, respectively, are isomorphisms. In particular, there are canonical algebra homomorphisms
    \begin{equation*}
        \mathcal{H}(\sigma_{\Omega_{\overline{a}},\lambda_{\overline{a}}}^{\circ})\to \textnormal{End}_{D^+(\mathcal{O}/\varpi^m)}(R\Gamma_{(c)}(\widetilde{X}_{\widetilde{K}(0,u)},\mathcal{W}_{\widetilde{\lambda},\mathcal{L}_{\overline{a}}}/\varpi^m)^{\overline{a}\textnormal{-ord}})
    \end{equation*}
    compatible when changing $m\geq 1$ and $u\geq u_{\Omega_{\overline{a}}}$ and
    commuting with the $\widetilde{\mathbf{T}}^{S}$-action.
\end{prop}
\begin{proof}
    The first map is an isomorphism by the ``independence of level" property in $P$-ordinary Hida theory (cf. \cite[Corollary 2.2.9]{CN25}) and the second map is an isomorphism by the ``independence of weight" property (cf. combination of \cite[Proposition 2.2.12 and Proposition 2.2.17]{CN25}).

    To see the last claim, note that, by \cite[Lemma 3.8]{Hev24}, there is an algebra homomorphism
    \begin{align*}
        \mathcal{H}(\sigma_{\Omega_{\overline{a}},\lambda_{\overline{a}}}^{\circ})^+ &\to \mathcal{H}(\textnormal{Inf}^{\mathcal{P}_{\overline{a}}(0,h)}(\sigma_{\Omega_{\overline{a},\lambda_{\overline{a}}}}^{\circ}/\varpi^m)) \\
        [K_{\overline{a}}gK_{\overline{a}},\psi] &\mapsto [\mathcal{P}_{\overline{a}}(0,h)g\mathcal{P}_{\overline{a}}(0,h),\textnormal{Inf}^{\mathcal{P}_{\overline{a}}(0,h)}\psi],
    \end{align*}
    where $ \mathcal{H}(\sigma^{\circ}_{\Omega_{\overline{a}},\lambda_{\overline{a}}})^+\leq  \mathcal{H}(\sigma_{\Omega_{\overline{a}},\lambda_{\overline{a}}}^{\circ})$ is the subalgebra of functions supported on
    \begin{equation*}
        G(F^+_{\overline{a}})^+:=\{g\in G(F^+_{\overline{a}})\mid gU(\mathcal{O}_{F^+_{\overline{a}}})g^{-1}\leq U(\mathcal{O}_{F^+_{\overline{a}}}), g^{-1}\overline{U}(\mathcal{O}_{F^+_{\overline{a}}})g\leq \overline{U}(\mathcal{O}_{F^+_{\overline{a}}})\}.
    \end{equation*}
    Moreover, one has $\mathcal{H}(\sigma_{\Omega_{\overline{a}},\lambda_{\overline{a}}}^{\circ})=\mathcal{H}(\sigma_{\Omega_{\overline{a}},\lambda_{\overline{a}}}^{\circ})^+[[K_{\overline{a}}\widetilde{u}_{a,n}K_{\overline{a}},\textnormal{id}]^{\pm 1}]$ and $[K_{\overline{a}}\widetilde{u}_{a,n}K_{\overline{a}},\textnormal{id}]$ acts invertibly on ordinary parts by definition.
\end{proof}
We finally note that the discussion holds equally well for the cohomology complex of the boundary $\partial\widetilde{X}_{\widetilde{K}(0,u)}$, and the Hecke actions introduced at $\overline{a}$ are compatible with the natural maps between these complexes.

\subsubsection{Rational $P$-ordinary middle degree cohomology}
Let $d$ denote the middle degree $\frac{1}{2}\dim_{\mathbf{R}}X^{\widetilde{G}}$ for $\widetilde{G}$.
Suppose that we are given $\overline{a}\in \overline{S}_p$, a dominant weight $\widetilde{\lambda}\in (\mathbf{Z}_+^{2n})^{\widetilde{I}_p}$, a Bernstein block $\Omega_{\overline{a}}=(\Omega_a,\Omega_{a^c})$, and a pair of an $\mathcal{O}$-integral $\Omega_{a}$- and $\Omega_{a^c}$-system $\mathcal{L}_{\overline{a}}=(\mathcal{L}_a,\mathcal{L}_{a^c})$, as before. For a fixed finite set of finite places $\overline{S}$ of $F^+$ containing $\overline{S}_p$, and $\lambda_{\overline{a}}=(\lambda_a,\lambda_{a^c})\in (\mathbf{Z}_+^n)^{\Hom(F_a,E)\times \Hom(F_{a^c},E)}$ matching $\widetilde{\lambda}_{\overline{a}}$, we define
\begin{equation*}
    \widetilde{\mathbf{T}}_a:=\widetilde{\mathbf{T}}^S\otimes_{\mathcal{O}}\mathfrak{Z}_{\textnormal{GL}_n(F_a),\mathcal{L}_a,\lambda_a}^{\circ}.
\end{equation*}
Given an $S$-good subgroup $\widetilde{K}^{\overline{a}}\leq \widetilde{G}(\mathbf{A}_{F^+}^{\infty\cup\{\overline{a}\}})$ and integer $u\geq u_{\Omega_{\overline{a}}}$, we would like to describe the  $\widetilde{\mathbf{T}}_a[1/p]$-module
\begin{equation*}
    H^d(\widetilde{X}_{\widetilde{K}(0,u)},\mathcal{W}_{\widetilde{\lambda},\mathcal{L}_{\overline{a}}})^{\overline{a}\textnormal{-ord}}[1/p]
\end{equation*}
in terms of cohomological cusp forms for $\widetilde{G}$.
The following definition plays a similar role here to \cite[Definition 4.3.5]{10author}. 
\begin{definition}
    Let $\widetilde{\lambda} \in (\mathbf{Z}^{2n}_+)^{\widetilde{I}_p}$. We say that $\widetilde{\lambda}$ is cohomologically cuspidal if it satisfies the following condition: for all proper standard parabolic subgroups $Q \leq  \widetilde{G}$, for all $k \in \{1, \dots, n\}$, and for all $w \in \widetilde{W}^{M_Q}$, there exist $\tau, \tau' \in \widetilde{I}_p$ such that
    \[ (w \cdot \widetilde{\lambda})_{\tau, 2n} - (w \cdot \widetilde{\lambda})_{\tau, k} \neq (w \cdot \widetilde{\lambda})_{\tau', 2n} - (w \cdot \widetilde{\lambda})_{\tau', k}. \]
\end{definition}
Here $\widetilde{W} = W( (\Res_{F^+ / \mathbf{Q}} \widetilde{G})_E, (\Res_{F^+ / \mathbf{Q}} \widetilde{T})_E)$ denotes the absolute Weyl group of the reductive group $\Res_{F^+ / \mathbf{Q}} \widetilde{G}$ over $\mathbf{Q}$. If $Q < \widetilde{G}$ is a proper standard  parabolic subgroup, then there is an isomorphism 
\[
M_Q \cong \Res_{F/F^+} \GL_{n_1} \times \Res_{F/F^+} \GL_{n_2} \times \dots \times \Res_{F/F^+} \GL_{n_{r-1}} \times \widetilde{G}_{n_r},
\]
where $\widetilde{G}_{n_r}$ is a quasi-split unitary group in $2n_r$ variables.
The point of the definition is that a cuspidal, regular algebraic automorphic representation $\sigma = \sigma_1 \otimes \dots \otimes \sigma_r$ of $M_Q(\mathbf{A}_{F^+})$ 
cannot have weight $w \cdot \widetilde{\lambda}$, since $\sigma_1$ would have weight $\lambda \in (\mathbf{Z}^{n_1}_+)^{\Hom(F,E)}$ given by
\[
    \lambda_{\tau, i} = (w \cdot \widetilde{\lambda})_{\tau, 2n - n_1 + i}, \qquad
    \lambda_{\tau c, i} = -(w \cdot \widetilde{\lambda})_{\tau, n_1 + 1 - i},
\]
for $i = 1, \dots, n_1$ and $\tau \in \widetilde{I}_p$. This contradicts Clozel's purity lemma, which implies that $\lambda_{\tau, n_1} + \lambda_{\tau c, 1} = (w \cdot \widetilde{\lambda})_{\tau, 2n} - (w \cdot \widetilde{\lambda})_{\tau, n_1}$ is an integer independent of $\tau$.
Therefore, any regular algebraic automorphic representation of $\widetilde{G}(\mathbf{A}_{F^+})$ of cohomologically cuspidal weight $\widetilde{\lambda}$ is cuspidal.

\begin{lemma}\label{lem_existence_of_cohomologically_cuspidal_weights_new}
    Let $\widetilde{\lambda} \in (\mathbf{Z}^{2n}_+)^{\widetilde{I}_p}$, let $\tau_0 \in \widetilde{I}_p$, and suppose that $[F^+ : \mathbf{Q}] > 1$. Then there exists $\widetilde{\lambda}' \in (\mathbf{Z}^{2n}_+)^{\widetilde{I}_p}$ satisfying the following conditions:
    \begin{enumerate}
        \item $\widetilde{\lambda}'$ is cohomologically cuspidal.
        \item $\widetilde{\lambda}'_\tau = \widetilde{\lambda}_\tau$ for all $\tau \neq \tau_0$.
    \end{enumerate}
\end{lemma}
\begin{proof}
    Define $\widetilde{\lambda}'$ by $\widetilde{\lambda}'_\tau = \widetilde{\lambda}_\tau$ if $\tau \neq \tau_0$, and 
    \[ \widetilde{\lambda}'_{\tau_0} = \widetilde{\lambda}_{\tau_0} + ((2n-1)a, (2n-3)a, \dots, (1-2n) a) \]
    for some integer $a > 0$. If $Q$ is a proper standard parabolic subgroup of $\widetilde{G}$, $k \in \{1, \dots, n\}$, and $w \in \widetilde{W}^{M_Q}$, then we have
    \[ (w \cdot \widetilde{\lambda}')_{\tau_0, 2n} - (w \cdot \widetilde{\lambda}')_{\tau_0, k} = (w \cdot \widetilde{\lambda})_{\tau_0, 2n} - (w \cdot \widetilde{\lambda})_{\tau_0, k} + r(k, w) a, \]
    where $r(k, w)$ is a non-zero integer that depends on $k$ and $w$.
    In particular, if we choose $a$ to be larger than all the finitely many integers
    \[ |((w \cdot \widetilde{\lambda})_{\tau, 2n} - (w \cdot \widetilde{\lambda})_{\tau, k}) - ((w \cdot \widetilde{\lambda})_{\tau_0, 2n} - (w \cdot \widetilde{\lambda})_{\tau_0, k})|, \]
    as $w$ ranges over the sets $\widetilde{W}^{M_Q}$, $k$ ranges over $\{1, \dots, n\}$, and $\tau$ ranges over the non-empty set $\widetilde{I}_p - \{ \tau_0 \}$, then $\widetilde{\lambda}'$ will be cohomologically cuspidal and have the desired form. 
\end{proof}

To proceed further, we also need to discuss $P$-ordinary parts in characteristic $0$. To ease notation, set $\widetilde{\textnormal{G}}:=\textnormal{GL}_{2n}(F_a)$, $\textnormal{P}:=P_{(n,n)}(F_a)$, $\textnormal{K}:=\textnormal{GL}_n(\mathcal{O}_{F_a})\times \textnormal{GL}_n(\mathcal{O}_{F_a})\leq\textnormal{G}:=\textnormal{GL}_n(F_a)\times\textnormal{GL}_n(F_a)$ and $\textnormal{U}^0:=N_{(n,n)}(\mathcal{O}_{F_a})\leq\textnormal{U}:=N_{(n,n)}(F_a)$. Moreover, for integers $u\geq t\geq 0$, $u\geq 1$, write $\mathcal{P}(t,u)\leq \textnormal{GL}_{2n}(\mathcal{O}_{F_a})$ to be the subgroup of matrices that are block upper-triangular modulo $\varpi^u$ and strictly block upper-triangular modulo $\varpi^t$ with blocks of size $n$. For $\widetilde{\mu}\in (\mathbf{Z}_+^{2n})^{\Hom(F_a,E)}$, introduce the double coset operators
\begin{align*}
    U^{t,u}_{n,\widetilde{\mu}}&:=\alpha_{\widetilde{\mu},n}^{-1}[\mathcal{P}(t,u)\widetilde{u}_{a,n}\mathcal{P}(t,u)] \\
    U^{t,u}_{2n,\widetilde{\mu}}&:=\alpha_{\widetilde{\mu},2n}^{-1}[\mathcal{P}(t,u)\widetilde{u}_{a,2n}\mathcal{P}(t,u)].
\end{align*}
Let $\widetilde{\Delta}^{u}\leq \widetilde{\textnormal{G}}$ be a submonoid containing the monoid $\widetilde{\Delta}_0^u:=\langle\mathcal{P}(0,u),\widetilde{u}_{a,n},\widetilde{u}_{a,2n}^{\pm1}\rangle$.
For a smooth admissible $E$-representation $\pi$ of $\widetilde{\Delta}^u$, we write
\begin{equation*}
    \pi^{\mathcal{P}(t,u),\widetilde{\mu}\textnormal{-ord}}\hookrightarrow \pi^{\mathcal{P}(t,u)}
\end{equation*}
for the maximal direct summand consisting of generalised eigenspaces of $\pi^{\mathcal{P}(t,u)}$ with respect to the action of $U^{t,u}_{n,\widetilde{\mu}}$ and $U^{t,u}_{2n,\widetilde{\mu}}$ with eigenvalues in $\overline{\mathbf{Z}}_p^{\times}$. One sees easily that, for $u'\geq t'$, $t'\geq t$, $u'\geq u$, 
\begin{equation}\label{equation_rationalindependenceoflevel}
    (\pi^{\mathcal{P}(t',u'),\widetilde{\mu}\textnormal{-ord}})^{\mathcal{P}(t,u)}=\pi^{\mathcal{P}(t,u),\widetilde{\mu}\textnormal{-ord}}.
\end{equation} 
Set \begin{equation*}
    \Delta^+_0=\widetilde{\Delta}_0^u\cap\textnormal{G}\leq \Delta^+:=\{g\in \widetilde{\Delta}^u\cap\textnormal{G}\mid g\textnormal{U}^0g^{-1}\leq \textnormal{U}^0,g^{-1}\overline{\textnormal{U}}^0g\leq \overline{\textnormal{U}}^0\}
\end{equation*}
and $\Delta_0:=\langle\Delta^+_0,\widetilde{u}_{a,n}^{-1}\rangle\leq \Delta:=\langle\Delta^+,\widetilde{u}_{a,n}^{-1}\rangle$.
We set
\begin{equation*}
    \pi^{\textnormal{P-}\widetilde{\mu}\textnormal{-ord}}:=\varinjlim_{t\to \infty}\pi^{\mathcal{P}(t,t),\widetilde{\mu}\textnormal{-ord}}\in \textnormal{Rep}_{\textnormal{sm}}(\Delta,E).
\end{equation*}
We finally note that, for $0\leq t\leq u$ and $\textnormal{K}(t):=\ker\Big(\textnormal{K}\to \textnormal{GL}_n(\mathcal{O}_{F_a}/\varpi^t_a)\times\textnormal{GL}_n(\mathcal{O}_{F_a}/\varpi^t_a)\Big)$, we have
\begin{equation}\label{equation_rationalordinaryinfiniteleveltofinitelevel}
    (\pi^{\textnormal{P-}\widetilde{\mu}\textnormal{-ord}})^{\textnormal{K}(t)}=\pi^{\mathcal{P}(t,u),\widetilde{\mu}\textnormal{-ord}}.
\end{equation}

Now let $\pi$ be a smooth admissible $E$-representation of $\widetilde{\textnormal{G}}$, and $\sigma$ be a smooth irreducible $E$-representation of $\textnormal{K}$ inflated from $\textnormal{K}/\textnormal{K}(u)$ and treat it as a $\Delta_0^+$-module via inflation. Write $\textnormal{Inf}^u\sigma$ for the induced $\widetilde{\Delta}^u_0$-module obtained by inflation along $\widetilde{\Delta}^u_0\to\widetilde{\Delta}^u_0/\mathcal{P}(u,u)=\Delta^+_0/\textnormal{K}(u)$. We then have a canonical identification
\begin{equation}\label{equation_infinitelvl_to_finitelvl}
    (\pi\otimes_E\textnormal{Inf}^u\sigma)^{\mathcal{P}(0,u),\widetilde{\mu}\textnormal{-ord}}=((\pi\otimes_E\textnormal{Inf}^u\sigma)^{\textnormal{P-}\widetilde{\mu}\textnormal{-ord}})^{\textnormal{K}}=(\pi^{\textnormal{P-}\widetilde{\mu}\textnormal{-ord}}\otimes_E\sigma)^{\textnormal{K}} 
\end{equation}
of sub-$E$-vector spaces of $\varinjlim_{t\to \infty}(\pi\otimes_E\textnormal{Inf}^u\sigma) ^{\mathcal{P}(t,t)}=(\varinjlim_{t\to \infty}\pi^{\mathcal{P}(t,t)})\otimes_E\sigma$.
Moreover, for $g\in \textnormal{G}^+$ and $\psi\in \Hom_{\textnormal{K}\cap g^{-1}\textnormal{K}g}(\sigma,g^{\ast}\sigma)$, \eqref{equation_infinitelvl_to_finitelvl} intertwines the action of $[\mathcal{P}(0,u)g\mathcal{P}(0,u),\textnormal{Inf}^u\psi]$ on the source with the action of $[\textnormal{K}g\textnormal{K},\psi]$ on the target. Let $\mathcal{H}(\sigma)^+\leq \mathcal{H}(\sigma)$ denote the subalgebra of functions supported on $\textnormal{G}^+$ and note that, as $\textnormal{G}=\langle \textnormal{G}^+,\widetilde{u}_{a,n}^{-1}\rangle$, we have $\mathcal{H}(\sigma)=\mathcal{H}(\sigma)^+[[\textnormal{K}\widetilde{u}_{a,n}\textnormal{K},\textnormal{id}]^{\pm1}]$. Therefore, $(\pi\otimes_E\textnormal{Inf}^u\sigma)^{\mathcal{P}(0,u),\widetilde{\mu}\textnormal{-ord}}$ admits a canonical action of $\mathcal{H}(\sigma)$ compatible with \eqref{equation_infinitelvl_to_finitelvl}.

We now apply this discussion to the admissible $E$-representation
\begin{equation*}
    H:=\iota_a^{-1}\bigg(\varinjlim_{\widetilde{K}_{\overline{a}}}H^{\ast}(\widetilde{X}_{\widetilde{K}^{\overline{a}}\widetilde{K}_{\overline{a}}},V_{\widetilde{\lambda}})\bigg)
\end{equation*}
of $\widetilde{\textnormal{G}}$.
\begin{lemma}
    There is a $\widetilde{\mathbf{T}}_a[1/p]$-equivariant isomorphism
    \begin{align}\label{equation_compatibilityofordinaryheckeactions}
    \begin{split}
        H^{\ast}(\widetilde{X}_{\widetilde{K}(0,u)}, \mathcal{W}_{\widetilde{\lambda},\mathcal{L}_{\overline{a}}})^{\overline{a}\textnormal{-ord}}[1/p]\xrightarrow{\sim} \\
        \Hom_{\textnormal{K}}\bigg((\sigma_{\Omega_a}^{\circ}[1/p])^{\vee}\otimes_E (\theta_n^{-1}\sigma_{\Omega_{a^c}}^{\circ}[1/p])^{\vee},H^{\textnormal{P-}\widetilde{\lambda}_a\textnormal{-ord}}\bigg)
    \end{split}
    \end{align}
    where, on the LHS we act at $a$ via the action induced by Proposition \ref{proposition_OrdinaryPartsHeckeAction} and on the RHS via the natural action of the centre $\mathfrak{Z}_{\textnormal{GL}_n(F_a),\mathcal{L}_a,\lambda_a}^{\circ}$.
\end{lemma}

\begin{proof}
    Set $\sigma:= \sigma_{\Omega_a}^{\circ}[1/p]\otimes_E\theta_n^{-1}\sigma_{\Omega_{a^c}}^{\circ}[1/p]$. By the discussion prior to the Lemma, we have
    \begin{equation*}
        H^{\ast}(\widetilde{X}_{\widetilde{K}(0,u)},\mathcal{W}_{\widetilde{\lambda},\mathcal{L}_{\overline{a}}}[1/p])^{\widetilde{\lambda}_{\overline{a}}\textnormal{-ord}}=(H\otimes_E\textnormal{Inf}^u\sigma)^{\mathcal{P}(0,u),\widetilde{\lambda}_{\overline{a}}\textnormal{-ord}}=\Hom_{\textnormal{K}}(\sigma^{\vee},H^{\textnormal{P-}\widetilde{\lambda}_{\overline{a}}\textnormal{-ord}})
    \end{equation*}
    under which the action of $[\textnormal{K}g\textnormal{K},\psi]\in\mathcal{H}(\sigma)^+$ on the RHS matches the action of $[\mathcal{P}(0,u)g\mathcal{P}(0,u),\textnormal{Inf}^u\psi]$ on the LHS. Moreover, unravelling the definitions, we see that
    \begin{equation*}
        H^{\ast}(\widetilde{X}_{\widetilde{K}(0,u)},\mathcal{W}_{\widetilde{\lambda},\mathcal{L}_{\overline{a}}})^{\overline{a}\textnormal{-ord}}[1/p]=H^{\ast}(\widetilde{X}_{\widetilde{K}(0,u)},\mathcal{W}_{\widetilde{\lambda},\mathcal{L}_{\overline{a}}}[1/p])^{\widetilde{\lambda}_{\overline{a}}\textnormal{-ord}}
    \end{equation*}
    as $\mathcal{H}(\textnormal{Inf}^u\sigma)$-submodules of $H^{\ast}(\widetilde{X}_{\widetilde{K}(0,u)},\mathcal{W}_{\widetilde{\lambda},\mathcal{L}_{\overline{a}}}[1/p])$. In particular, we obtain \eqref{equation_compatibilityofordinaryheckeactions} and see that it is clearly $\widetilde{\mathbf{T}}^S$-equivariant. 
    
    To match the Hecke actions at $a$, it suffices to see that, for every $[\textnormal{K}g\textnormal{K},\psi]\in \mathcal{H}(\sigma)^+\cap \mathcal{H}(\sigma_{\Omega_{\overline{a}},\lambda_{\overline{a}}}^{\circ})$ and some sufficiently large integer $M$, the action of $p^M[\textnormal{K}g\textnormal{K},\psi]=[\textnormal{K}g\textnormal{K},p^M\psi]$ on $H^{\ast}(\widetilde{X}_{\widetilde{K}(0,u)},\mathcal{W}_{\widetilde{\lambda},\mathcal{L}_{\overline{a}}})^{\overline{a}\textnormal{
    -ord}}[1/p]$ defined via Proposition \ref{proposition_OrdinaryPartsHeckeAction} matches with the action of the Hecke operator $[\mathcal{P}(0,u)g\mathcal{P}(0,u),p^M\textnormal{Inf}^u\psi]$. We choose $M$ to be large enough so that $p^M(\textnormal{Inf}^u\psi)\otimes g$ preserves $\mathcal{W}_{\widetilde{\lambda}_{\overline{a}},\mathcal{L}_{\overline{a}}}\subset \Big(\textnormal{Inf}^{\mathcal{P}(0,u)}\sigma_{\Omega_{\overline{a}}}^{\circ}[1/p]\Big)\otimes_E V_{\widetilde{\lambda}_{\overline{a}}}$. In particular, $[\mathcal{P}(0,u)g\mathcal{P}(0,u),p^M\textnormal{Inf}^u\psi\otimes g]$ acts on $H^{\ast}(\widetilde{X}_{\widetilde{K}(0,u)},\mathcal{W}_{\widetilde{\lambda},\mathcal{L}_{\overline{a}}})$, recovering the action of $[\mathcal{P}(0,u)g\mathcal{P}(0,u),p^M\textnormal{Inf}^u\psi]$ on $H^{\ast}(\widetilde{X}_{\widetilde{K}(0,u)},\mathcal{W}_{\widetilde{\lambda},\mathcal{L}_{\overline{a}}})[1/p]$ after inverting $p$. 
    
    On the other hand, for an integer $m\geq u$, the action of $[\textnormal{K}g\textnormal{K},p^M\psi]$ on $H^{\ast}(\widetilde{X}_{\widetilde{K}(0,u)},\mathcal{W}_{\widetilde{\lambda},\mathcal{L}_{\overline{a}}}/\varpi^m)^{\overline{a}\textnormal{-ord}}$ is defined by acting with
    \begin{equation*}
        [\mathcal{P}(0,m)g\mathcal{P}(0,m),p^M\textnormal{Inf}^{\mathcal{P}_{\overline{a}}(0,m)}(\psi\otimes g)]
    \end{equation*} on the target of the isomorphism induced by taking $P$-ordinary parts of the following composition
    \begin{align}\label{equation_prmcircres}
    \begin{split}
        H^{\ast}(\widetilde{X}_{\widetilde{K}(0,u)},\mathcal{W}_{\widetilde{\lambda},\mathcal{L}_{\overline{a}}}/\varpi^m) & \xrightarrow{\textnormal{res}}H^{\ast}(\widetilde{X}_{\widetilde{K}(0,m)},\mathcal{W}_{\widetilde{\lambda},\mathcal{L}_{\overline{a}}}/\varpi^m) \\
        & \xrightarrow{\textnormal{pr}_m} H^{\ast}(\widetilde{X}_{\widetilde{K}(0,m)},\textnormal{Inf}^{\mathcal{P}_{\overline{a}}(0,m)}\big(\sigma_{\Omega_{\overline{a}},\lambda_{\overline{a}}}^{\circ}/\varpi^m\big)\otimes_{\mathcal{O}}\mathcal{V}_{\widetilde{\lambda}^{\overline{a}}})
    \end{split}
    \end{align}
    where $\textnormal{pr}_m\colon\mathcal{W}_{\widetilde{\lambda}_{\overline{a}},\mathcal{L}_{\overline{a}}}/\varpi^m\to \textnormal{Inf}^{\mathcal{P}_{\overline{a}}(0,m)}\big(\sigma_{\Omega_{\overline{a}},\lambda_{\overline{a}}}^{\circ}/\varpi^m\big)$ is the $\mathcal{P}(0,m)$-equivariant map introduced in \eqref{equation_ModpmSurjection}. Using that $g\in \textnormal{G}^+$, one checks easily that \eqref{equation_prmcircres} intertwines the action of
    $[\mathcal{P}(0,u)g\mathcal{P}(0,u),p^M\textnormal{Inf}^u\psi\otimes g]$ on the source with the action of
    $[\mathcal{P}(0,m)g\mathcal{P}(0,m),p^M\textnormal{Inf}^{\mathcal{P}_{\overline{a}}(0,m)}(\psi\otimes g)]$ on the target. We conclude by passing to the limit over $m\geq 1$.
\end{proof}

As a corollary, we obtain a generalisation of \cite[Theorem 2.4.11]{10author} and \cite[Proposition 6.5]{Hev24}. After combining it with the base change result \cite{Shin_2014}, reciprocity for conjugate self-dual regular algebraic cuspidal automorphic representations of $\textnormal{GL}_{2n,F}$ (see for instance \cite[Theorem 2.3.3]{10author}) and the main characteristic $0$ result of \cite{Hev24} on $P$-ordinary local-global compatibility cf. \cite[Theorem 4.27]{Hev24} we obtain the following theorem.

\begin{theorem}\label{Theorem_selfdualLGC}
   Suppose that we are given a finite set $\overline{S}$ of $F^+$ containing $\overline{S}_p$, a place $\overline{a}\in \overline{S}_p$, a dominant weight $\widetilde{\lambda}\in (\mathbf{Z}^{2n}_+)^{\widetilde{I}_p}$ with $\lambda_{\overline{a}}=(\lambda_{a},\lambda_{a^c})\in (\mathbf{Z}_+^n)^{\Hom(F_a,E)\times \Hom(F_{a^c},E)}$ matching $\widetilde{\lambda}_{\overline{a}}$, a Bernstein block $\Omega_{\overline{a}}=(\Omega_{a},\Omega_{a^c})$ of $\textnormal{Rep}_{\textnormal{sm}}(G(F^+_{\overline{a}}),\overline{\mathbf{Q}}_p)$, a pair of $\mathcal{O}$-integral $\Omega_a$- and $\Omega_{a^c}$-system $\mathcal{L}_{\overline{a}}=(\mathcal{L}_a,\mathcal{L}_{a^c})$, and an $\overline{S}$-good subgroup $\widetilde{K}\leq \widetilde{G}(\mathbf{A}_{F^+}^{\infty})$. 
   Assume that the following are satisfied.
   \begin{itemize}
       \item For $v\notin S$ with residual characteristic $l$, either $l$ is unramified in $F$ and $S_l(F)\cap S\neq \emptyset$ or $l$ splits in an imaginary quadratic subfield of $F$.
       \item The weight $\widetilde{\lambda}$ is cohomologically cuspidal.
   \end{itemize}
   Set $d=\frac{1}{2}\dim_{\mathbf{R}}X^{\widetilde{G}}=n^2[F^+:\mathbf{Q}]$, $\widetilde{\mathbf{T}}_a:=\widetilde{\mathbf{T}}^S\otimes_{\mathcal{O}}\mathfrak{Z}_{\textnormal{GL}_n(F_a),\mathcal{L}_a,\lambda_a}^{\circ}$ and, for an integer $u\geq u_{\Omega_a}$, write $\widetilde{K}(0,u):=\widetilde{K}^{\overline{a}}\mathcal{P}_{\overline{a}}(0,u)$. Then the following statements hold true.
   \begin{itemize}
       \item The middle degree cohomology group
       \begin{equation*}
           H^d(\widetilde{X}_{\widetilde{K}(0,u)},\mathcal{W}_{\widetilde{\lambda},\mathcal{L}_{\overline{a}}})^{\overline{a}\textnormal{-ord}}[1/p]
       \end{equation*}
       is a semisimple $\widetilde{\mathbf{T}}_a[1/p]$-module.
       \item To every $\overline{\mathbf{Q}}_p$-point $f\colon\widetilde{\mathbf{T}}_a\Big(H^d(\widetilde{X}_{\widetilde{K}(0,u)},\mathcal{W}_{\widetilde{\lambda},\mathcal{L}_{\overline{a}}})^{\overline{a}\textnormal{-ord}}\Big)\to \overline{\mathbf{Q}}_p$, we can associate a $2n$-dimensional continuous Galois representation
       \begin{equation*}
           \widetilde{\rho}_f:G_{F,S}\to \textnormal{GL}_{2n}(\overline{\mathbf{Q}}_p)
       \end{equation*}
       with associated determinant $\widetilde{D}_f$
       satisfying the following properties.
       \begin{enumerate}
           \item For $v\notin S$, $\widetilde{D}_f(X-\textnormal{Frob}_v)=f\big(\widetilde{P}_v(X)\big)$ in $\overline{\mathbf{Q}}_p[X]$.
           \item The representation $\widetilde{\rho}_f$ is potentially semistable with $\tau$-labelled Hodge--Tate weights
           \begin{align*}
               \textnormal{HT}(\widetilde{\rho}_{f}\mid_{G_{F_a}})=\{-\lambda_{\tau c, n}+2n-1>...>-\lambda_{\tau c,1}+n> \\
               \lambda_{\tau,1}+n-1>...>\lambda_{\tau,n}\}_{\tau:F_a\hookrightarrow E}.
           \end{align*}
           \item There is an isomorphism
           \begin{equation*}
               \widetilde{\rho}|_{G_{F_a}}\sim \begin{pmatrix}\rho_{1} & \ast \\0 & \rho_{2} \end{pmatrix}
           \end{equation*}
           where $\rho_1$ has $\tau$-labelled Hodge--Tate weights
           \begin{equation*}
               \textnormal{HT}(\rho_1)=\{\lambda_{\tau,1}+n-1>...>\lambda_{\tau,n}\}_{\tau:F_a\hookrightarrow E}.
           \end{equation*}
           \item The supercuspidal support of $\textnormal{rec}_{F_a}^{-1}(\textnormal{WD}(\rho_1))|\det|_{F_a}^{\frac{n-1}{2}}$ is given by
           \begin{equation*}
               \mathfrak{Z}_{\textnormal{GL}_n(F_a),\mathcal{L}_a,\lambda_a}^{\circ}\xrightarrow{\textnormal{nat}}\widetilde{\mathbf{T}}_a\Big(H^d(\widetilde{X}_{\widetilde{K}(0,u)},\mathcal{W}_{\widetilde{\lambda},\mathcal{L}_{\overline{a}}})^{\overline{a}\textnormal{-ord}}\Big)\xrightarrow{f}\overline{\mathbf{Q}}_p.
           \end{equation*}
       \end{enumerate}
   \end{itemize}
\end{theorem}

\subsubsection{Local-global compatibility for interior cohomology}
For a subset $\overline{T}\subset \overline{S}_p$ and an integer $m\geq 1$, set 
\begin{equation*}
\mathcal{V}_U(\overline{T},m):=R\Gamma(U(\mathcal{O}_{F^+,\overline{T}}),\mathcal{O}/\varpi^m)\in D^b_{\textnormal{sm}}(G(\mathcal{O}_{F^+,\overline{T}}),\mathcal{O}/\varpi^m)
\end{equation*}
and, for a good subgroup $K\leq G(\mathbf{A}_F^{\infty})$ with $K_{\overline{T}}\leq G(\mathcal{O}_{F^+,\overline{T}})$, use the same notation to denote the corresponding object in the bounded derived category $D^b(\textnormal{Sh}(\overline{X}_K,\mathcal{O}/\varpi^m))$. Set $\mathcal{V}^j_U(\overline{T},m)$ to be its $j^{th}$ cohomology group, a locally constant sheaf on $\overline{X}_K$ that is non-zero if and only if $j\in [0,n^2(\sum_{\overline{v}\in \overline{T}}[F^+_{\overline{v}}:\mathbf{Q}_p])]$.
\begin{lemma}\label{lemma_degreeshiftingpreliminarylemma1}
       Suppose that we are given a finite set $\overline{S}$ of places of $F^+$ containing $\overline{S}_p$, distinct places $\overline{a},\overline{b}\in \overline{S}_p$, a dominant weight $\widetilde{\lambda}\in (\mathbf{Z}^{2n}_+)^{\widetilde{I}_p}$ with $\lambda_{\overline{a}}=(\lambda_{a},\lambda_{a^c})\in (\mathbf{Z}_+^n)^{\Hom(F_a,E)\times \Hom(F_{a^c},E)}$ and $\lambda_{\overline{b}}=(\lambda_{b},\lambda_{b^c})\in (\mathbf{Z}_+^n)^{\Hom(F_b,E)\coprod \Hom(F_{b^c},E)}$ matching $\widetilde{\lambda}_{\overline{a}}$ and $\widetilde{\lambda}_{\overline{b}}$, respectively, a Bernstein block $\Omega_{\overline{a}}=(\Omega_{a},\Omega_{a^c})$ of $\textnormal{Rep}_{\textnormal{sm}}(G(F^+_{\overline{a}}),\overline{\mathbf{Q}}_p)$, a pair of $\mathcal{O}$-integral $\Omega_a$- and $\Omega_{a^c}$-system $\mathcal{L}_{\overline{a}}=(\mathcal{L}_a,\mathcal{L}_{a^c})$, and an $\overline{S}$-good subgroup $\widetilde{K}\leq \widetilde{G}(\mathbf{A}_{F^+}^{\infty})$. 
   Assume that the following are satisfied.
   \begin{itemize}
       \item The subgroup $\widetilde{K}$ is decomposed with respect to $P$ and, for every $\overline{v}\in \overline{S}_p$, $\widetilde{K}\cap U(F^+_{\overline{v}})=U(\mathcal{O}_{F^+_{\overline{v}}})$.
       \item For any embedding $\tau:F^+\hookrightarrow E$ inducing a place different from $\overline{a}$ and $\overline{b}$, $\widetilde{\lambda}_{\tau}=(0,...,0)$.
   \end{itemize}
   For an integer $u\geq u_{\Omega_{\overline{a}}}$, write $\widetilde{K}(0,u):=\widetilde{K}^{\overline{a}}\mathcal{P}_{\overline{a}}(0,u)$ and $K:=\widetilde{K}(0,u)\cap G(\mathbf{A}_{F^+}^{\infty})$. Moreover, set $\mathbf{T}_a:=\mathbf{T}^S\otimes_{\mathcal{O}}\mathfrak{Z}_{\textnormal{GL}_n(F_a),\mathcal{L}_a,\lambda_a}^{\circ}$, $\widetilde{\mathbf{T}}_a:=\widetilde{\mathbf{T}}^S\otimes_{\mathcal{O}}\mathfrak{Z}_{\textnormal{GL}_n(F_a),\mathcal{L}_a,\lambda_a}^{\circ}$ and $\mathcal{S}_a:=\mathcal{S}\otimes \id\colon\widetilde{\mathbf{T}}_a\to \mathbf{T}_a$.

   Then, for any integer $q\geq 0$, 
   \begin{equation*}
       \varprojlim_{m\geq 1}H^q_!\big(X_K,\sigma_{\Omega_{\overline{a}},\lambda_{\overline{a}}}^{\circ}\otimes_{\mathcal{O}}\mathcal{V}_{\lambda_{\overline{b}}}\otimes_{\mathcal{O}}\mathcal{V}_U(\overline{S}_p\setminus\{\overline{a},\overline{b}\},m)\big)
   \end{equation*}
   is a $\widetilde{\mathbf{T}}_a$-module subquotient of $H^q(\partial\widetilde{X}_{\widetilde{K}(0,u)},\mathcal{W}_{\widetilde{\lambda},\mathcal{L}_{\overline{a}}})^{\overline{a}\textnormal{-ord}}$ where $\widetilde{\mathbf{T}}_a$ acts on the former via $\mathcal{S}_a$. 
\end{lemma}
\begin{proof}
    Let $m\geq 1$ be an integer. By Lemma \ref{Lemma_BorelSerreLemma}, the Mackey formula, Lemma \ref{RestrictionFunc}, Lemma \ref{ProjectionFormula} and Lemma \ref{BoundedTorDim}, the complex
    \begin{equation*}
        R\Gamma\big(\widetilde{K}^{\{\overline{a},\overline{b}\}},R\Gamma(\overline{\mathfrak{X}}_{\widetilde{G}}^{P},(j_!)\mathcal{O}/\varpi^m)\big)
    \end{equation*}
    admits
    \begin{equation*}
        \textnormal{Ind}_{P(F^+_{\{\overline{a},\overline{b}\}})}^{\widetilde{G}(F^+_{\{\overline{a},\overline{b}\}})}R\Gamma\Big(K^{\{\overline{a},\overline{b}\}},R\Gamma\big(\overline{\mathfrak{X}}_G,(j_!)\mathcal{V}_U(\overline{S}_p\setminus\{\overline{a},\overline{b}\},m)\big)\Big)
    \end{equation*}
    as a $\widetilde{\mathbf{T}}^S$-equivariant direct summand in $D^+_{\textnormal{sm}}(\widetilde{G}(F^+_{\{\overline{a},\overline{b}\}}),\mathcal{O}/\varpi^m)$ in a manner compatible with the respective maps induced by forgetting the support.

    On the other hand, as shown in the proof of \cite[Theorem 2.4.4]{10author}, the complex $R\Gamma(U(\mathcal{O}_{F^+_{\overline{b}}}),\mathcal{V}_{\widetilde{\lambda}_{\overline{b}}}/\varpi^m)$ admits  $\mathcal{V}_{\lambda_{\overline{b}}}/\varpi^m$ as a direct summand in $D^+_{\textnormal{sm}}(K_{\overline{b}},\mathcal{O}/\varpi^m)$. Consequently, the complex
    \begin{equation*}
        R\Gamma\big(\widetilde{K}^{\overline{a}},R\Gamma(\overline{\mathfrak{X}}_{\widetilde{G}}^{P},(j_!)\mathcal{V}_{\widetilde{\lambda}_{\overline{b}}}/\varpi^m)\big)
    \end{equation*}
    admits 
    \begin{equation*}
        \textnormal{Ind}_{P(F^+_{\overline{a}})}^{\widetilde{G}(F^+_{\overline{a}})}R\Gamma\bigg(K^{\overline{a}},R\Gamma\Big(\overline{\mathfrak{X}}_G,(j_!)\big(\mathcal{V}_{\lambda_{\overline{b}}}/\varpi^m\otimes_{\cO/\varpi^m}\mathcal{V}_U(\overline{S}_p\setminus\{\overline{a},\overline{b}\},m)\big)\Big)\bigg)
    \end{equation*}
    as a $\widetilde{\mathbf{T}}^S$-equivariant direct summand in $D^+_{\textnormal{sm}}(\widetilde{G}(F^+_{\overline{a}}),\cO/\varpi^m)$ in a manner that is compatible with the respective maps forgetting the support.

    By \cite[Corollary 3.30]{Hev24} (see \cite[Proposition 2.3.11]{CN25} for the statement in the case of the Siegel parabolic and without asserting Hecke equivariance), there are $\widetilde{\mathbf{T}}_a$-equivariant injections
    \begin{align*}
        &H^q_{(c)}\Big(X_K,\sigma_{\Omega_{\overline{a}},\lambda_{\overline{a}}}^{\circ}\otimes_{\cO}\mathcal{V}_{\lambda_{\overline{b}}}\otimes_{\cO}\mathcal{V}_U(\overline{S}_p\setminus\{\overline{a},\overline{b}\},m)\Big) \\
\cong & H^q\bigg(K,R\Gamma\Big(\overline{\mathfrak{X}}_G,(j_!)\big(\sigma_{\Omega_{\overline{a}},\lambda_{\overline{a}}}^{\circ}\otimes_{\cO}\mathcal{V}_{\lambda_{\overline{b}}}\otimes_{\cO}\mathcal{V}_U(\overline{S}_p\setminus\{\overline{a},\overline{b}\},m)\big)\Big)\bigg) \\
        \hookrightarrow & H^q\Big(\widetilde{K},R\Gamma(\overline{\mathfrak{X}}_{\widetilde{G}}^P,(j_!)\mathcal{W}_{\widetilde{\lambda},\mathcal{L}_{\overline{a}}}/\varpi^m)\Big)^{\overline{a}\textnormal{-ord}} \\
        \cong & H^q_{(c)}(\widetilde{X}_{\widetilde{K}}^P,\mathcal{W}_{\widetilde{\lambda},\mathcal{L}_{\overline{a}}}/\varpi^m)^{\overline{a}\textnormal{-ord}}
    \end{align*}
    compatible with the respective maps forgetting the support.\footnote{Here it is implicit that we are applying the result to the Weyl group element $w_0^P$ corresponding to the open Bruhat stratum and in this case the argument of \textit{loc. cit.} yields an injection, not a subquotient.}
    Finally, the natural map
    $H^q_{c}(\widetilde{X}_{\widetilde{K}}^P,\mathcal{W}_{\widetilde{\lambda},\mathcal{L}_{\overline{a}}}/\varpi^m)^{\overline{a}\textnormal{-ord}}\to H^q_{}(\widetilde{X}_{\widetilde{K}}^P,\mathcal{W}_{\widetilde{\lambda},\mathcal{L}_{\overline{a}}}/\varpi^m)^{\overline{a}\textnormal{-ord}}$ forgetting the support factors as
    \begin{align*}
        H^q_{c}(\widetilde{X}_{\widetilde{K}}^P,\mathcal{W}_{\widetilde{\lambda},\mathcal{L}_{\overline{a}}}/\varpi^m)^{\overline{a}\textnormal{-ord}} &\to H^q_{}(\partial \widetilde{X}_{\widetilde{K}},\mathcal{W}_{\widetilde{\lambda},\mathcal{L}_{\overline{a}}}/\varpi^m)^{\overline{a}\textnormal{-ord}} \\
        & \to
        H^q_{}(\widetilde{X}_{\widetilde{K}}^P,\mathcal{W}_{\widetilde{\lambda},\mathcal{L}_{\overline{a}}}/\varpi^m)^{\overline{a}\textnormal{-ord}}
    \end{align*}
    where both are the natural maps on cohomology with and without compact supports, respectively, induced by the open embedding $\widetilde{X}_{\widetilde{K}}^P\hookrightarrow \partial \widetilde{X}_{\widetilde{K}}$.
    In particular, we obtain a $\widetilde{\mathbf{T}}_a$-equivariant diagram
    \begin{equation*}
        \begin{tikzcd}[scale cd=0.6, column sep = tiny]
	& {H^q_{}(\partial \widetilde{X}_{\widetilde{K}},\mathcal{W}_{\widetilde{\lambda},\mathcal{L}_{\overline{a}}}/\varpi^m)^{\overline{a}\textnormal{-ord}}} \\
	{H^q_{c}(\widetilde{X}_{\widetilde{K}}^P,\mathcal{W}_{\widetilde{\lambda},\mathcal{L}_{\overline{a}}}/\varpi^m)^{\overline{a}\textnormal{-ord}}} && {H^q(\widetilde{X}_{\widetilde{K}}^P,\mathcal{W}_{\widetilde{\lambda},\mathcal{L}_{\overline{a}}}/\varpi^m)^{\overline{a}\textnormal{-ord}}} \\
	{H^q_{c}\Big(X_K,\sigma_{\Omega_{\overline{a}},\lambda_{\overline{a}}}^{\circ}\otimes_{\mathcal{O}}\mathcal{V}_{\lambda_{\overline{b}}}\otimes_{\mathcal{O}}\mathcal{V}_U(\overline{S}_p\setminus\{\overline{a},\overline{b}\},m)\Big)} && {H^q\Big(X_K,\sigma_{\Omega_{\overline{a}},\lambda_{\overline{a}}}^{\circ}\otimes_{\mathcal{O}}\mathcal{V}_{\lambda_{\overline{b}}}\otimes_{\mathcal{O}}\mathcal{V}_U(\overline{S}_p\setminus\{\overline{a},\overline{b}\},m)\Big),}
	\arrow[from=1-2, to=2-3]
	\arrow[from=2-1, to=1-2]
	\arrow[from=2-1, to=2-3]
	\arrow[hook, from=3-1, to=2-1]
	\arrow[from=3-1, to=3-3]
	\arrow[hook, from=3-3, to=2-3]
\end{tikzcd}
    \end{equation*}
    where the bottom map is the one induced by forgetting the support. Furthermore, these diagrams are compatible with respect to changing $m\geq 1$. Passing to the limit along $m\geq 1$ finishes the proof.
\end{proof}

\begin{lemma}\label{lem_p_torsion_spectral_sequence}
    Let $E^{p, q}_r$ be a cohomological spectral sequence of $\mathbf{Z}_p$-modules concentrated in the quadrant $[0, N] \times [0, N]$. Suppose there exists $M \geq 1$ such that $p^M$ annihilates all differentials $d^{p, q}_r$, $r \geq 1$. Then there exists $N' \geq 1$ depending only on $N$ and $M$, and natural maps $\alpha_{p, q} \colon E^{p, q}_2 \to E^{p, q}_\infty$, with kernel and cokernel annihilated by $p^{N'}$. 
\end{lemma}
\begin{proof}
    The spectral sequence degenerates after a number of steps that can be bounded in terms of $N$, so it suffices to show that for each $r \geq 2$, there is a natural map $\beta_{p, q}^r \colon E^{p, q}_r \to E^{p, q}_{r+1}$ with kernel and cokernel annihilated by $p^{2M}$. Let $B \leq Z \leq E_{p, q}^r$ denote the coboundaries and cocycles. Then $p^M E^{p, q}_r \leq Z$, so we can define $\beta_{p, q}^r(x) = p^M x \text{ mod }B$. It's then easy to check that the kernel and cokernel of $\beta_{p, q}^r$ are annihilated by $p^{2M}$. 
\end{proof}

\begin{lemma}[Degree shifting]\label{lemma_degreeshifting}Suppose that we are in the setup of Lemma \ref{lemma_degreeshiftingpreliminarylemma1}. Fix integers $r,s\geq 0$ and $w\in W^{P}_{\overline{S}_p\setminus\{\overline{a},\overline{b}\}}$ such that $l(w)=s$. Let $\lambda\in (\mathbf{Z}_+^n)^{\Hom(F,E)}$ match $w\cdot \widetilde{\lambda}$. Then there exists an integer $N_1\geq 0$ depending only on $n$ and $[F:\mathbf{Q}]$ such that
\begin{equation*}
    p^{N_1}H^r_!(X_K,\sigma_{\Omega_{\overline{a}},\lambda_{\overline{a}}}^{\circ}\otimes_{\mathcal{O}}\mathcal{V}_{\lambda^{\overline{a}}})
\end{equation*}
is a subquotient $\widetilde{\mathbf{T}}_a$-module of $H^{r+s}(\partial \widetilde{X}_{\widetilde{K}(0,u)},\mathcal{W}_{\widetilde{\lambda},\mathcal{L}_{\overline{a}}})^{\overline{a}\textnormal{-ord}}$.

In particular, if we further assume that our fixed rational prime $\mathfrak{p}$ is so that $S_{\mathfrak{p}}(F^+)\cap\overline{S}=\emptyset$ and, in particular, $T_n\in \widetilde{\mathbf{T}}^S$ from \S\ref{subsubsection_theelementT_n}, then the following holds true. For $d:=\frac{1}{2}\dim_{\mathbf{R}}X^{\widetilde{G}}$,
\begin{equation*}
    T_n^2p^{N_1}H^{d-s}_!(X_K,\sigma_{\Omega_{\overline{a}},\lambda_{\overline{a}}}^{\circ}\otimes_{\mathcal{O}}\mathcal{V}_{\lambda^{\overline{a}}})
\end{equation*}
is a $\widetilde{\mathbf{T}}_a$-equivariant subquotient of the $\widetilde{\mathbf{T}}_a[1/p]$-module 
\begin{equation*}
    H^{d}(\widetilde{X}_{\widetilde{K}(0,u)},\mathcal{W}_{\widetilde{\lambda},\mathcal{L}_{\overline{a}}})^{\overline{a}\textnormal{-ord}}[1/p].
\end{equation*}
    
\end{lemma}
\begin{proof}
For $?=\varnothing, c,!$, and integer $q\geq 0$, set
\begin{equation*}
    H^{q}_?:= \varprojlim_{m\geq 1}H^{q}_?\big(X_K,\sigma_{\Omega_{\overline{a}},\lambda_{\overline{a}}}^{\circ}\otimes_{\mathcal{O}}\mathcal{V}_{\lambda_{\overline{b}}}\otimes_{\mathcal{O}}\mathcal{V}_U(\overline{S}_p\setminus\{\overline{a},\overline{b}\},m)\big)
\end{equation*}
and note that $H_!^q=\im (H_c^q\to H^q)$ as the inverse system computing $\ker (H_c^q\to H^q)$ satisfies Mittag--Leffler.

There are spectral sequences
$E_{c,q}^{r,s},E_q^{r,s}$ of $\mathbf{T}_a$-modules and a morphism between them $i_{c,q}^{r,s}\colon E_{c,q}^{r,s}\to E_q^{r,s}$ as follows.
\begin{equation*}
    \begin{tikzcd}[scale cd=0.9]
	{E_{2,c}^{r,s}=H^r_c\Big(X_K,\sigma_{\Omega_{\overline{a}},\lambda_{\overline{a}}}^{\circ}\otimes_{\mathcal{O}}\mathcal{V}_{\lambda_{\overline{b}}}\otimes_{\mathcal{O}}H^s_{\textnormal{cont}}(U(\mathcal{O}_{F^+_{\overline{S}_p\setminus\{\overline{a},\overline{b}\}}}),\mathcal{O})\Big)} & \Rightarrow & {H_c^{r+s}} \\
	{E_{2}^{r,s}=H^r\Big(X_K,\sigma_{\Omega_{\overline{a}},\lambda_{\overline{a}}}^{\circ}\otimes_{\mathcal{O}}\mathcal{V}_{\lambda_{\overline{b}}}\otimes_{\mathcal{O}}H^s_{\textnormal{cont}}(U(\mathcal{O}_{F^+_{\overline{S}_p\setminus\{\overline{a},\overline{b}\}}}),\mathcal{O})\Big)} & \Rightarrow & {H^{r+s}}
	\arrow["{i_{c,2}^{r,s}}", from=1-1, to=2-1]
	\arrow["{i_c^{r+s}}", from=1-3, to=2-3]
\end{tikzcd}
\end{equation*}
where $i_{c,2}^{r,s}$ and $i_c^{r+s}$ are the maps induced by forgetting the support.

Arguing as in the proof of \cite[Lemma 4.2.3]{10author}, we can find an element $z \in \mathcal{O}\Big[Z\big(G(\mathcal{O}_{F^+_{\overline{S}_p\setminus\{\overline{a},\overline{b}\}}})\big)\Big]$ that acts as multiplication by a distinct element of $\cO$ on each nonzero cohomology group $H^s_{\textnormal{cont}}(U(\mathcal{O}_{F^+_{\overline{S}_p\setminus\{\overline{a},\overline{b}\}}}), \cO)$. It follows that there is an integer $M \geq 1$, depending only on $n$ and $[F : \mathbf{Q}]$, such that the differentials in these spectral sequences $E_{c,2}^{r,s}, E_2^{r,s}$ are annihilated by $p^M$.

By Lemma \ref{lem_p_torsion_spectral_sequence}, there are maps $\alpha$, $\alpha_c$ of $\mathbf{T}_a$-modules fitting into a diagram
\begin{equation*}
    \begin{tikzcd}
	{H^r_c\Big(X_K,\sigma_{\Omega_{\overline{a}},\lambda_{\overline{a}}}^{\circ}\otimes_{\mathcal{O}}\mathcal{V}_{\lambda_{\overline{b}}}\otimes_{\mathcal{O}}H^s_{\textnormal{cont}}(U(\mathcal{O}_{F^+_{\overline{S}_p\setminus\{\overline{a},\overline{b}\}}}),\mathcal{O})\Big)} && {E_{c,\infty}^{r,s}} \\
	{H^r\Big(X_K,\sigma_{\Omega_{\overline{a}},\lambda_{\overline{a}}}^{\circ}\otimes_{\mathcal{O}}\mathcal{V}_{\lambda_{\overline{b}}}\otimes_{\mathcal{O}}H^s_{\textnormal{cont}}(U(\mathcal{O}_{F^+_{\overline{S}_p\setminus\{\overline{a},\overline{b}\}}}),\mathcal{O})\Big)} && {E_{\infty}^{r,s}}
	\arrow["{\alpha_c}", from=1-1, to=1-3]
	\arrow["{i_{c,2}^{r,s}}", from=1-1, to=2-1]
	\arrow["{i_{c,\infty}^{r,s}}", from=1-3, to=2-3]
	\arrow["\alpha", from=2-1, to=2-3]
\end{tikzcd}
\end{equation*}
with kernel and cokernel killed by $p^{N_2}$ for an integer $N_2\geq 0$ depending solely on $n$ and $F$. An easy diagram chase combined with Lemma \ref{lemma_degreeshiftingpreliminarylemma1} then shows that 
\begin{equation*}
    p^{N_2}  H^r_!(X_{K},\sigma_{\Omega_{\overline{a}},\lambda_{\overline{a}}}^{\circ}\otimes_{\mathcal{O}}  \mathcal{V}_{\lambda_{\overline{b}}} \otimes_\cO\Hom_{\mathbf{Z}_p}(\bigwedge^s U(\cO_{F^+_{\overline{S}_p\setminus\{\overline{a},\overline{b}\}}}), \cO))
\end{equation*} is a subquotient $\widetilde{\mathbf{T}}_a$-module of $H^{r+s}(\partial\widetilde{X}_{\widetilde{K}(0,u)},\mathcal{W}_{\widetilde{\lambda},\mathcal{L}_{\overline{a}}})^{\overline{a}\textnormal{-ord}}$. 

        Kostant's theorem shows that $\Hom_{\mathbf{Z}_p}(\bigwedge^s U(\cO_{F^+_{\overline{S}_p\setminus\{\overline{a},\overline{b}\}}}), E)$ admits $V_{\lambda_{\overline{S}_{p}\setminus\{\overline{a},\overline{b}\}}}$ as a direct summand $E[\GL_n(\cO_{F_{S_p\setminus\{a,a^c,b,b^c\}}})]$-module, and therefore that we can find homomorphisms of $\cO[\GL_n(\cO_{F_{S_p\setminus\{a,a^c,b,b^c\}}})]$-modules
        \[ \mathcal{V}_{\lambda^{\{\overline{a},\overline{b}\}}} \to  \Hom_{\mathbf{Z}_p}(\bigwedge^s U(\cO_{F^+_{\overline{S}_{p}\setminus\{\overline{a},\overline{b}\}}}), \cO) \]
        and
        \[ \Hom_{\mathbf{Z}_p}(\bigwedge^s U(\cO_{F^+_{\overline{S}_{p}\setminus\{\overline{a},\overline{b}\}}}), \cO) \to \mathcal{V}_{\lambda^{\{\overline{a},\overline{b}\}}} \]
        whose composite is multiplication by a power of $p$ bounded solely in terms of $n$ and $[F : \mathbf{Q}]$ (and in particular, not depending on $\widetilde{\lambda}$). Increasing $N_2$ to exceed this power, we find that $ H^r_!(X_{K},\sigma_{\Omega_{\overline{a}},\lambda_{\overline{a}}}^{\circ}\otimes_{\mathcal{O}}  \mathcal{V}_{\lambda_{\overline{b}}} \otimes_\cO\Hom_{\mathbf{Z}_p}(\bigwedge^s U(\cO_{F^+_{\overline{S}_p\setminus\{\overline{a},\overline{b}\}}}), \cO))$ admits
        \begin{equation*}
            p^{N_2}  H^r_!(X_{K},\sigma_{\Omega_{\overline{a}},\lambda_{\overline{a}}}^{\circ}\otimes_{\mathcal{O}}  \mathcal{V}_{\lambda^{\overline{a}}})
        \end{equation*}
         as a $\mathbf{T}_a$-equivariant subquotient, and therefore that $p^{2 N_2}H^r_!(X_{K},\sigma_{\Omega_{\overline{a}},\lambda_{\overline{a}}}^{\circ}\otimes_{\mathcal{O}}  \mathcal{V}_{\lambda^{\overline{a}}})$ is a subquotient $\widetilde{\mathbf{T}}_a$-module of $H^{r+s}(\partial\widetilde{X}_{\widetilde{K}(0,u)},\mathcal{W}_{\widetilde{\lambda},\mathcal{L}_{\overline{a}}})^{\overline{a}\textnormal{-ord}}$. This completes the proof of the first part of the Lemma by setting $N_1=2N_2$. 

         For the second claim, it suffices to prove that $T_n^2H^d(\partial\widetilde{X}_{\widetilde{K}(0,u)},\mathcal{W}_{\widetilde{\lambda},\mathcal{L}_{\overline{a}}})^{\overline{a}\textnormal{-ord}}$ is a $\widetilde{\mathbf{T}}_a$-equivariant subquotient of $H^d(\widetilde{X}_{\widetilde{K}(0,u)},\mathcal{W}_{\widetilde{\lambda},\mathcal{L}_{\overline{a}}})^{\overline{a}\textnormal{-ord}}[1/p]$. By Theorem \ref{VanishingUpToBoundedTorsion}, $T_n$ annihilates the third member of the $\widetilde{\mathbf{T}}_a$-equivariant exact sequence 
         \begin{align*}
             H^d(\widetilde{X}_{\widetilde{K}(0,u)},\mathcal{W}_{\widetilde{\lambda},\mathcal{L}_{\overline{a}}})^{\overline{a}\textnormal{-ord}} &\to H^d(\partial\widetilde{X}_{\widetilde{K}(0,u)},\mathcal{W}_{\widetilde{\lambda},\mathcal{L}_{\overline{a}}})^{\overline{a}\textnormal{-ord}}\\ 
             &\to H^{d+1}_c(\widetilde{X}_{\widetilde{K}(0,u)},\mathcal{W}_{\widetilde{\lambda},\mathcal{L}_{\overline{a}}})^{\overline{a}\textnormal{-ord}}.
         \end{align*}
         In particular, $T_nH^d(\partial\widetilde{X}_{\widetilde{K}(0,u)},\mathcal{W}_{\widetilde{\lambda},\mathcal{L}_{\overline{a}}})^{\overline{a}\textnormal{-ord}}$ is a $\widetilde{\mathbf{T}}_a$-equivariant subquotient of $H^d(\widetilde{X}_{\widetilde{K}(0,u)},\mathcal{W}_{\widetilde{\lambda},\mathcal{L}_{\overline{a}}})^{\overline{a}\textnormal{-ord}}$. On the other hand, $H^d(\widetilde{X}_{\widetilde{K}(0,u)},\mathcal{W}_{\widetilde{\lambda},\mathcal{L}_{\overline{a}}})^{\overline{a}\textnormal{-ord}}$ is a finitely generated $\mathcal{O}$-module and another application of Theorem \ref{VanishingUpToBoundedTorsion} shows that its $\varpi^{\infty}$-torsion is annihilated by $T_n$, showing that $T_nH^d(\widetilde{X}_{\widetilde{K}(0,u)},\mathcal{W}_{\widetilde{\lambda},\mathcal{L}_{\overline{a}}})^{\overline{a}\textnormal{-ord}}$ is a $\widetilde{\mathbf{T}}_a$-equivariant subquotient of the module $H^d(\widetilde{X}_{\widetilde{K}(0,u)},\mathcal{W}_{\widetilde{\lambda},\mathcal{L}_{\overline{a}}})^{\overline{a}\textnormal{-ord}}[1/p]$. The proof is now complete.    
\end{proof}
\begin{lemma}\label{lemma_interiorLGC}
Suppose that we are given a finite set $\overline{S}$ of places of $F^+$ containing $\overline{S}_p$, distinct places $\overline{a},\overline{b}\in \overline{S}_p$, a dominant weight $\widetilde{\lambda}\in (\mathbf{Z}^{2n}_+)^{\widetilde{I}_p}$, 
a Bernstein block $\Omega_{\overline{a}}=(\Omega_{a},\Omega_{a^c})$ of $\textnormal{Rep}_{\textnormal{sm}}(G(F^+_{\overline{a}}),\overline{\mathbf{Q}}_p)$, a pair of $\mathcal{O}$-integral $\Omega_a$- and $\Omega_{a^c}$-system $\mathcal{L}_{\overline{a}}=(\mathcal{L}_a,\mathcal{L}_{a^c})$, and an $\overline{S}$-good subgroup $\widetilde{K}\leq \widetilde{G}(\mathbf{A}_{F^+}^{\infty})$. 
   Assume that the following are satisfied.
\begin{enumerate}
    \item For $v\notin S$ with residual characteristic $l$, either $l$ is unramified in $F$ and $S_l(F)\cap S \neq \emptyset$ or $l$ splits in an imaginary quadratic subfield of $F$.
    \item The fixed rational prime $\mathfrak{p}$ from \S\ref{subsubsection_theelementT_n} is so that $S_{\mathfrak{p}}(F^+)\cap \overline{S}=\emptyset$ and $T_n\in \mathbf{T}^S$ is the fixed element provided by Lemma \ref{lemma_existence_of_T_n}.
    \item The weight $\widetilde{\lambda}$ is cohomologically cuspidal.
    \item For any embedding $\tau:F^+\hookrightarrow E$ inducing a place different from $\overline{a}$ and $\overline{b}$, $\widetilde{\lambda}_{\tau}=(0,...,0)$.
    \item The subgroup $\widetilde{K}$ is decomposed with respect to $P$ and, for every $\overline{v}\in \overline{S}_p$, $\widetilde{K}\cap U(F^+_{\overline{v}})=U(\mathcal{O}_{F^+_{\overline{v}}})$. Write $K:=(\widetilde{K}^{\overline{a}}\cap G(\mathbf{A}_{F^+}^{\infty}))\cdot G(\mathcal{O}_{F^+_{\overline{a}}})$.
\end{enumerate}

Fix $w\in W^{P}_{\overline{S}_p\setminus\{\overline{a},\overline{b}\}}$, set $q=d-l(w)$ and let $\lambda\in (\mathbf{Z}_+^n)^{\Hom(F,E)}$ be matching $w\cdot \widetilde{\lambda}$. Then there are integers $N_1,N_2\geq 1$ only depending on $n$ and $F$, satisfying the following.
    \begin{itemize}
        \item For every maximal ideal $\mathfrak{m}\trianglelefteq \mathbf{T}^S$ with residue field $k$ supported on $ A:=\mathbf{T}_a\big(H^{q}_!(X_K,\sigma_{\Omega_{\overline{a}},\lambda_{\overline{a}}}^{\circ}\otimes_{\cO}\mathcal{V}_{\lambda^{\overline{a}}})\big)$, there is an ideal $J\trianglelefteq A_{\mathfrak{m}}$ satisfying $(T_n^2p^{N_2}J)^{N_1}=0$ such that $A_{\mathfrak{m}}/J$ is of Galois type and the associated determinant $D_{\mathfrak{m}}$ satisfies $\textnormal{LGC}(S,a,\lambda_{\overline{a}},\Omega_{\overline{a}},\mathcal{L}_{\overline{a}})$.
    \end{itemize}
\end{lemma}
\begin{proof}
    By \cite{Sch15}, there exists an integer $N_3\geq 1$ only depending on $n$ and $F$ and an ideal $I\leq A$ satisfying $I^{N_3}=0$ such that $A_{\mathfrak{m}}/I_{\mathfrak{m}}$ is of Galois type for any $\mathfrak{m}$. In particular, to verify local-global compatibility at $a$, we are allowed to pass to a larger set $\overline{S}$, still satisfying the conditions listed in the statement of the Lemma.

    For a continuous character $\overline{\chi} \colon G_{F}\to k^{\times}$ unramified at $\mathfrak{p}$ and $p$, write $\chi:G_{F}\to \mathcal{O}^{\times}$ for its Teichm\"uller lift. Let $S\subset S_{\chi}$ be a finite set of finite places of $F$ and $K_{\chi}\trianglelefteq K$ be a good normal subgroup such that the following hold.
    \begin{itemize}
        \item For each finite place $v$ of $F$, $\chi|_{G_{F_v}}\circ \textnormal{Art}_{F_v}$ is trivial on $\det(K_{\chi,v})$.
        \item The set $\overline{S}_{\chi}$ still satisfies assumption $(1)$ and $(2)$ of the Lemma.
        \item $(K_{\chi})^{S_{\chi}\setminus S}=K^{S_{\chi}\setminus S}$.
        \item $K/K_{\chi}$ is abelian of order prime to $p$.
        \item There is $\widetilde{K}_{\chi}$ with $(\widetilde{K}_{\chi})^{S_{\chi}\setminus S}=\widetilde{K}^{S_{\chi}\setminus S}$, $\widetilde{K}_\chi\cap G(\mathbf{A}_{F^+}^{\infty})=K_{\chi}$ and $\widetilde{K}_{\chi}$ still satisfies $(5)$.
    \end{itemize}
    Recall the maximal ideal (cf. discussion before \cite[Proposition 2.22]{10author}) $\mathfrak{m}(\chi)\trianglelefteq \mathbf{T}^{S_{\chi}}$ obtained by `twisting $\mathfrak{m}$ by $\chi$'. Concretely, it is defined to be $\mathfrak{m}(\chi):=f_{\chi}(\mathfrak{m})\trianglelefteq\mathbf{T}_{\chi,a}:=\mathbf{T}^{S_{\chi}}\otimes f_{\chi}(\mathfrak{Z}_{\textnormal{GL}_n(F_a),\lambda_a,\mathcal{L}_a}^{\circ})$ for the isomorphism
    
    \begin{align*} 
        f_{\chi} \colon \mathcal{H}(G^{S_{\chi}\setminus\{a\}},\sigma_{\Omega_{a},\lambda_a}^{\circ}) &\to \mathcal{H}(G^{S_{\chi}\setminus\{a\}},\sigma_{\Omega_{a},\lambda_a}^{\circ}),
     \\ 
        h \colon G^{S_{\chi}\setminus\{a\}} \to \textnormal{End}(\sigma_{\Omega_{a},\lambda_a}^{\circ}) &\mapsto \Big(f_{\chi}(h):g\mapsto \chi\big(\textnormal{Art}_F(\det(g))\big)^{-1}h(g)\Big).
     \end{align*}
    We will show the claim for some $\mathfrak{m}(\chi)$ that will then imply it for $\mathfrak{m}$.
   
    By Lemma \ref{lemma_degreeshifting} and \cite{Sch15}, for every $\chi$, $S_{\chi}$ and $K_{\chi}$ we obtain integers $N_{\chi,1}, N_{\chi,2} $ only depending on $n$ and $F$, an ideal 
    \begin{equation*}
         J_{\chi}\trianglelefteq A_{\mathfrak{m}(\chi)}:=\mathbf{T}_{\chi,a}\big(H^q_!(X_{K_{\chi}},\sigma_{\Omega_{\overline{a}},\lambda_{\overline{a}}}^{\circ}\otimes_{\cO}\mathcal{V}_{\lambda^{\overline{a}}})\big)_{\mathfrak{m}(\chi)}
    \end{equation*}
    for $\mathbf{T}_{\chi,a}:=\mathbf{T}^{S_{\chi}}\otimes f_{\chi}(\mathfrak{Z}_{\textnormal{GL}_n(F_a),\lambda_a,\mathcal{L}_a}^{\circ})$
    satisfying $(T_n^2p^{N_{\chi,2}}J_{\chi})^{N_{\chi,1}}=0$ and
    such that the following hold.
    \begin{itemize}
        \item The $\mathbf{T}^{S_{\chi}}$-algebra $A_{\mathfrak{m}(\chi)}/J_{\chi}$ is of Galois type.
        \item The Satake transform $\mathcal{S}_{\chi,a}\colon\widetilde{\mathbf{T}}_{\chi,a}=\widetilde{\mathbf{T}}^{S_{\chi}}\otimes f_{\chi}(\mathfrak{Z}_{\textnormal{GL}_n(F_a),\lambda_a,\mathcal{L}_a}^{\circ})\to \mathbf{T}_{\chi,a}$
    descends to a surjection
    \begin{equation*}
        \widetilde{A}_{\chi}:=\im \Big(\widetilde{\mathbf{T}}_a\to \textnormal{End}_{E}(H^d(\widetilde{X}_{\widetilde{K}_{\chi}(0,u)},\mathcal{W}_{\widetilde{\lambda},\mathcal{L}_{\overline{a}}})_{\widetilde{\mathfrak{m}(\chi)}}^{\overline{a}\textnormal{-ord}}[1/p])\Big)\xrightarrow{\overline{\mathcal{S}}_{\chi}} A_{\mathfrak{m}(\chi)}/J_{\chi}.
    \end{equation*}
    \end{itemize}

    By Theorem \ref{Theorem_selfdualLGC}, $\widetilde{A}_{\chi}[1/p]$ is a semisimple $\widetilde{\mathbf{T}}_{\chi,a}[1/p]$-algebra and, after possibly enlarging $E$, we can write $\widetilde{A}_{\chi}[1/p]=\prod_{i=1}^rE$. Moreover, we obtain $2n$-dimensional Galois representations $\widetilde{\rho}_{\chi,i}:G_{F,S_{\chi}}\to \textnormal{GL}_{2n}(E)$ such that
    \begin{equation*}
        \widetilde{\rho}_{\chi,i}|_{G_{F_a}}\sim \begin{pmatrix}
\rho_{\chi,i,1} & \ast \\
0 & \rho_{\chi,i,2}
\end{pmatrix}.
    \end{equation*}
    We can then set $\widetilde{\rho}_{\chi}=\prod_{i=1}^r\widetilde{\rho}_{\chi,i}$, $\rho_{\chi,j}=\prod_{i=1}^r\rho_{\chi,i,j}$ and write $\widetilde{D}_{\chi}:G_{F,S}\to \widetilde{A}_{\chi}$, $D_{\chi,j}:G_{F_a}\to \widetilde{A}_{\chi}[1/p]$ for the corresponding determinants. 
    We have a global decomposition
    \begin{equation*}
    \widetilde{D}_{\chi}\otimes_{\widetilde{A}_{\chi},\overline{\mathcal{S}}_{\chi}}A_{\mathfrak{m}(\chi)}/J_{\chi}=D_{\mathfrak{m}(\chi)}D_{\mathfrak{m}(\chi)}^{\vee,c}(1-2n)
    \end{equation*}
    and a local decomposition
    \begin{equation*}
        (\widetilde{D}_{\chi}|_{G_{F_a}})[1/p]=D_{\chi,1}D_{\chi,2}
    \end{equation*}
    that we can force to be compatible.
    Namely, by sub-lemma 1 in the proof of \cite[Proposition 4.2.13]{CN25}, after possibly enlarging $E$, we can find $\chi$ such that the following hold.
    \begin{itemize}
        \item The irreducible factors of $\overline{D}_{\mathfrak{m}(\chi)}|_{G_{F_{a}}}$ and $\overline{D}_{\mathfrak{m}(\chi)}^{\vee,c}(1-2n)|_{G_{F_a}}$ are distinct.
        \item For each $i=1,...,r$, the irreducible factors of $\overline{D}_{\rho_{\chi,i,1}}$ and of $\overline{D}_{\mathfrak{m}(\chi)}|_{G_{F_a}}$ coincide.
    \end{itemize}
    By further enlarging $E$, we may assume that $\overline{\widetilde{D}_{\chi}|_{G_{F_a}}}$ is split in the sense of \cite[Definition 2.19]{Che14} i.e. $k[G_{F_a}]/\ker (\widetilde{D}_{\chi}|_{G_{F_a}}\otimes_{\widetilde{A}_{\chi}}k)$ is a product of matrix algebras over $k$ (indexed by irreducible factors).
    Thus, we can define $\overline{e}\in k[G_{F_a}]/\ker(\widetilde{D}_{\chi}|_{G_{F_a}}\otimes_{\widetilde{A}_{\chi}}k)=\im\big(k[G_{F_a}]\to M_{2n}(k)\big)$ as the central idempotent acting as identity on the irreducible factors of $\overline{D}_{\mathfrak{m}(\chi)}|_{G_{F_a}}$ and as zero on the irreducible factors of $\overline{D}_{\mathfrak{m}(\chi)}^{\vee,c}(1-2n)\mid_{G_{F_a}}$. By \cite[Lemma 3.2.1]{CN25}, we can lift $\overline{e}$ to an idempotent $\widetilde{e}\in\widetilde{B}_a=\im\big(\widetilde{A}_{\chi}[G_{F_a}]\xrightarrow{\widetilde{\rho}_{\chi}}M_{2n}(\widetilde{A}_{\chi}[1/p])\big)$ that, when viewed as an element of $M_{2n}(\widetilde{A}_{\chi}[1/p])$, is given by idempotent projection onto the representation space of $\rho_{\chi,1}$ (see the proof of \cite[Lemma 3.2.2 (4)]{CN25}). Therefore, the natural map
    \begin{align}\label{equation_GaloisStableLattice}
    \begin{split}
        \widetilde{A}_{\chi}[G_{F_a}] &\to \widetilde{e}\widetilde{B}_a\widetilde{e} \\
        x &\mapsto \widetilde{e}x\widetilde{e},
    \end{split}
    \end{align}
    is an algebra homomorphism as can be checked after inverting $p$. Moreover, the determinant $D_{\widetilde{e}}:G_{F_a}\to \widetilde{A}_{\chi}$ we obtain by postcomposing \eqref{equation_GaloisStableLattice} with the determinant $\widetilde{e}\widetilde{B}_a\widetilde{e}\to \widetilde{A}_{\chi}$, $x\mapsto \widetilde{D}_{\chi}(x+1-\widetilde{e})$ satisfies $\overline{D}_{\widetilde{e}}=\overline{D}_{\mathfrak{m}(\chi)}|_{G_{F_a}}$ and $D_{\widetilde{e}}[1/p]=D_{\chi,1}$.
    
    By Hensel's lemma for determinants \cite[Lemma 3.2.4]{10author}, we obtain
    \begin{equation*}
        D_{\widetilde{e}}\otimes_{\widetilde{A}_{\chi},\overline{\mathcal{S}}_{\chi}}A_{\mathfrak{m}_{\chi}}/J_{\chi}=D_{\mathfrak{m}(\chi)}|_{G_{F_a}}.
    \end{equation*}
    In particular, we see that condition (2) of $\textnormal{LGC}(S_{\chi},a,\lambda_{\overline{a}},\Omega_{\overline{a}},\mathcal{L}_{\overline{a}})$ is satisfied for $A_{\mathfrak{m}(\chi)}/J_{\chi}$ (with the $\mathcal{O}$-model $f_{\chi}(\mathfrak{Z}_{\textnormal{GL}_n(F_a),\lambda_a,\mathcal{L}_a}^{\circ})\cap \eta^{-1}(R_{\overline{D}_{\chi,1}}^{\lambda_a,\Omega_a})$ appearing in the diagram). Noting that $\Gamma(K/K_{\chi},-)$ is exact, $f_{\chi^{-1}}$ descends to a surjection $A_{\mathfrak{m}(\chi)}\to A_{\mathfrak{m}}$, we can then set $J:=f_{\chi^{-1}}(J_{\chi})+I_{\mathfrak{m}}$, $D_{\mathfrak{m}}=(f_{\chi^{-1}}\circ D_{\mathfrak{m}(\chi)})\otimes \chi^{-1}\mod{J}$ to see that $A_{\mathfrak{m}}/J$ is of Galois type as a $\mathbf{T}^S$-algebra and satisfies condition $(2)$ of $\textnormal{LGC}(S,a,\lambda_{\overline{a}},\Omega_{\overline{a}},\mathcal{L}_{\overline{a}})$ (see \cite{Hev24}, proof of claim on page 92).

    We are left to prove that $D_{\mathfrak{m}}$ satisfies (1) of $\textnormal{LGC}(S,a,\lambda_{\overline{a}},\Omega_{\overline{a}},\mathcal{L}_{\overline{a}})$. It suffices to prove that the $\mathbf{Z}_p[G_{L_a}]$-module  
    \begin{equation*}
        (A_{\mathfrak{m}(\chi)}/J_{\chi})[G_{F,S_{\chi}}]/\ker(D_{\mathfrak{m}(\chi)})
    \end{equation*} is torsion semistable of weight $[h_{\lambda_{\overline{a}},1},h_{\lambda_{\overline{a}},2}]$ with $L_a/F_a$ so that $\tau_{\Omega_a}:I_{F_a}\to \textnormal{GL}_n(\overline{\mathbf{Q}}_p)$ factors through $I_{L_a/F_a}$. By flat base change, we have $\ker (\widetilde{D}_{\chi})\otimes_{\mathcal{O}}E=\ker (\widetilde{D}_{\chi}\otimes_{\mathcal{O}}E)$, yielding an embedding
\begin{align}\label{equation_embedding}
\begin{split}
    \widetilde{A}_{\chi}[G_{F,S}]/\ker(\widetilde{D}_{\chi}) &\hookrightarrow (\widetilde{A}_{\chi}\otimes_{\mathcal{O}}E)[G_{F,S_{\chi}}]/\ker(\widetilde{D}_{\chi}\otimes_{\mathcal{O}}E) \\
    &\hookrightarrow M_{2n}(\widetilde{A}_{\chi}\otimes_{\mathcal{O}}E)\cong \widetilde{\rho}_{\chi}^{2n}
\end{split}
\end{align}
of $\mathbf{Z}_p[G_{L_a}]$-modules. In particular, since the target of \eqref{equation_embedding} is semistable with Hodge--Tate weights within a range $[h_{\lambda_{\overline{a}},1},h_{\lambda_{\overline{a}},2}]$, the source is torsion semistable of weight within $[h_{\lambda_{\overline{a}},1},h_{\lambda_{\overline{a}},2}]$.

Finally we have surjections
\begin{align*}
    \widetilde{A}_{\chi}[G_{F,S_{\chi}}]/\ker(\widetilde{D}_{\chi}) &\twoheadrightarrow (A_{\mathfrak{m}(\chi)}/J_{\chi})[G_{F,S}]/\ker (\widetilde{D}_{\chi}\otimes_{\widetilde{A}_{\chi},\overline{\mathcal{S}}}A_{\mathfrak{m}(\chi)}/J_{\chi})\\ 
    & \twoheadrightarrow (A_{\mathfrak{m}(\chi)}/J_{\chi})[G_{F,S}]/\ker (D_{\mathfrak{m}(\chi)})
\end{align*}
of $\widetilde{A}_{\chi}[G_{L_a}]$-modules. Indeed, the first map uses 
\begin{equation*}
\ker(\widetilde{D}_{\chi})\otimes_{\widetilde{A}_{\chi},\overline{\mathcal{S}}_{\chi}} (A_{\mathfrak{m}(\chi)}/J_{\chi})\leq \ker (\widetilde{D}_{\chi}\otimes_{\widetilde{A}_{\chi},\overline{\mathcal{S}}_{\chi}}A_{\mathfrak{m}(\chi)})
\end{equation*}
by
\cite[Lemma 1.18 (iii)]{Che14} and the second uses that
\begin{equation*}
    \ker (\widetilde{D}_{\chi}\otimes_{\widetilde{A}_{\chi},\overline{\mathcal{S}}_{\chi}} (A_{\mathfrak{m}(\chi)}/J_{\chi}))=\ker (D_{\mathfrak{m}(\chi)})\cap\ker (D_{\mathfrak{m}(\chi)}^{\vee,c}(1-2n))\leq \ker (D_{\mathfrak{m}(\chi)})
\end{equation*}
by \cite[Lemma 3.2.4]{10author}. In particular, we see that property $(1)$ of $\textnormal{LGC}(S_{\chi},a,\lambda_{\overline{a}},\Omega_{\overline{a}},\mathcal{L}_{\overline{a}})$ is satisfied for $D_{\mathfrak{m}(\chi)}:G_{F,S_{\chi}}\to A_{\mathfrak{m}(\chi)}/J_{\chi}$. 
\end{proof}
We will need the dual statement as well.
\begin{lemma}\label{lemma_interiorLGCdual}
    Suppose that we are in the situation of the first paragraph of the statement of Lemma \ref{lemma_interiorLGC}.

    Fix $w\in W^P_{\overline{S}_p\setminus\{\overline{a},\overline{b}\}}$, set $q=d-l(w)$ and $\lambda\in (\mathbf{Z}_+^n)^{\Hom(F,E)}$ so that $-w_0^G\lambda$ matches $w\cdot \widetilde{\lambda}$. Then there are integers $N_1\geq 1$, $N_2\geq 0$ only depending on $n$, and $F$, satisfying the following.
    \begin{itemize}
         \item For every maximal ideal $\mathfrak{m}\trianglelefteq \mathbf{T}^S$ with residue field $k$ supported on\footnote{Here $\mathbf{T}_a$ acts through the restriction of the duality involution $\iota\colon\mathbf{T}_a\to \mathbf{T}^S\otimes_{\mathcal{O}}\mathcal{H}(\sigma_{\Omega_a^{\vee},\lambda_a^{\vee}}^{\circ})$.} $ A:=\mathbf{T}_a\big(H^{q}_!(X_K,\sigma_{\Omega_{\overline{a}}}^{\circ,\vee}\otimes_{\cO}\mathcal{V}_{-w_0^G\lambda})\big)$ for $\mathbf{T}_a:=\mathbf{T}^S\otimes_{\mathcal{O}}\mathfrak{Z}_{\textnormal{GL}_n(F_a),\mathcal{L}_a,\lambda_a}^{\circ}$, there is an ideal $J\trianglelefteq A_{\mathfrak{m}}$ satisfying $(T_n^2p^{N_2}J)^{N_1}=0$ such that $A_{\mathfrak{m}}/J$ is of Galois type and the associated determinant $D_{\mathfrak{m}}$ satisfies $\textnormal{LGC}(S,a,\lambda_{\overline{a}},\Omega_{\overline{a}},\mathcal{L}_{\overline{a}})$.
    \end{itemize}
\end{lemma}
\begin{proof}
    Running the proof of Lemma \ref{lemma_interiorLGC} with the local system
    \begin{equation*}
        \mathcal{W}_{\widetilde{\lambda},\mathcal{L}_{\overline{a}}^{\vee}}:=\textnormal{Inf}^{\mathcal{P}_{\overline{a}}(0,u)}(\sigma_{\Omega_{\overline{a}}}^{\circ,\vee})\otimes_{\mathcal{O}}\mathcal{V}_{\widetilde{\lambda}}
    \end{equation*}
    yields integers $N_8,N_9$ and ideal $J^{\vee}\trianglelefteq A_{\mathfrak{m}^{\vee}}^{\vee}$ satisfying $(T_n^2p^{N_9}J^\vee)^{N_8}=0$ for
    \begin{equation*}
        A^{\vee}:=\mathbf{T}_a^{\vee}(H^q_!(X_K,\sigma_{\Omega_{\overline{a}},\lambda_{\overline{a}}}^{\circ}\otimes_{\mathcal{O}}\mathcal{V}_{\lambda_{\overline{a}}})),\textnormal{ }\mathbf{T}_a^{\vee}:=\iota(\mathbf{T}_a),\textnormal{ and }\mathfrak{m}^{\vee}:=\iota(\mathfrak{m}).
    \end{equation*}
    Moreover, the quotient $A^{\vee}_{\mathfrak{m}^{\vee}}/J^{\vee}$ is of Galois type with associated determinant $D_{\mathfrak{m}^{\vee}}:G_{F,S}\to A^{\vee}_{\mathfrak{m}^{\vee}}/J^{\vee}$ satisfying the following.
    \begin{itemize}
        \item It is potentially torsion semistable of weight within a range $[h_{\lambda_{\overline{a}},1},h_{\lambda_{\overline{a}},2}]$.
        \item For the determinant $D_{\mathfrak{m}^{\vee},a}:=D_{\mathfrak{m}^{\vee}}|_{G_{F_a}}$ there is a dotted arrow making the following diagram commute.
        \begin{equation}\label{equation_dualLGC}
            \begin{tikzcd}
	{R_{\overline{D}_{\mathfrak{m}^{\vee},a}}} & A \\
	{R_{\overline{D}_{\mathfrak{m}^{\vee},a}}^{-w_0^{\textnormal{GL}_n}\lambda_a,\Omega_a^{\vee}}} & {\mathfrak{Z}^{\circ}_{\textnormal{GL}_n(F_a),\mathcal{L}_a,\lambda_a,\overline{D}_{\mathfrak{m},a}}.}
	\arrow["{D_{\mathfrak{m}^{\vee},a}}", from=1-1, to=1-2]
	\arrow[two heads, from=1-1, to=2-1]
	\arrow[dotted, from=2-1, to=1-2]
	\arrow["{\textnormal{nat}_A\circ \iota}"', from=2-2, to=1-2]
	\arrow["{\eta^{\circ}\circ\iota}", from=2-2, to=2-1]
\end{tikzcd}
        \end{equation}
    \end{itemize}
    In the diagram above we denoted by $\Omega_a^{\vee}$ the contragredient of the block $\Omega_a$ and $\overline{D}_{\mathfrak{m},a}:=(\overline{D}_{\mathfrak{m}^{\vee},a})^{\vee}$.
    To see that the map $\eta\circ\iota$ descends to a map $\eta^{\circ}\circ\iota$ as in the diagram one considers the following commutative diagram
    \begin{equation}\label{equation_dualitydiagram}
        \begin{tikzcd}
	{R_{\overline{D}_{\mathfrak{m},a}}^{\lambda_a,\Omega_a}\otimes_{\mathcal{O}}\overline{\mathbf{Q}}_p} && {R^{-w_0^{\textnormal{GL}_n}\lambda_a,\Omega_a^{\vee}}_{\overline{D}_{\mathfrak{m}^{\vee},a}}\otimes_{\mathcal{O}}\overline{\mathbf{Q}}_p} \\
	{\mathfrak{Z}_{\textnormal{GL}_n(F_a),\Omega_a}} && {\mathfrak{Z}_{\textnormal{GL}_n(F_a),\Omega_{a}^{\vee}}}
	\arrow["{D\mapsto D^{\vee}}", from=1-1, to=1-3]
	\arrow["\eta", from=2-1, to=1-1]
	\arrow["\iota"', from=2-1, to=2-3]
	\arrow["\eta"', from=2-3, to=1-3]
\end{tikzcd}
    \end{equation}
    where the horizontal maps are isomorphisms and the top horizontal map descends to a map on the $\mathcal{O}$-integral deformation rings.

    We can now set $J:=\iota^{-1}(J^{\vee})$, $A:=\iota^{-1}(A^{\vee})$ and $D_{\mathfrak{m}}:=D_{\mathfrak{m}^{\vee}}^{\vee}$. It is then clear that $D_{\mathfrak{m}}$ is potentially torsion semistable with weights within a range $[h_{\lambda_{\overline{a}},1},h_{\lambda_{\overline{a},2}}]$. Combining \eqref{equation_dualLGC}  and \eqref{equation_dualitydiagram} one easily verifies condition $2$ of $\textnormal{LGC}(S,a,\lambda_{\overline{a}},\Omega_{\overline{a}},\mathcal{L}_{\overline{a}})$ as well.
    
\end{proof}

\subsection{Finishing the proof of Proposition \ref{Proposition_MainLGCProp}}
\begin{proof}[Proof of Proposition \ref{Proposition_MainLGCProp}]
For $1\leq n'\leq n$ we set $T_{n'}$ to be the Hecke operator defined at the end of \S\ref{subsubsection_theelementT_n}. By running induction on $n$, we may assume that $P(n',a,S,T_{n'})$ holds true as long as $1\leq n'<n$.

    We recall, in light of Lemma \ref{lemma_equivalent_LGC_hypothesis}, what we are trying to do: given the hypothesis of the Proposition and $q$, $K$, $\lambda$, $\Omega$ and $\mathcal{L}$, we must find $N_1$, $N_2$ and, for every $m\geq 1$, and maximal ideal $\mathfrak{m}\trianglelefteq \mathbf{T}^S$ an ideal $J\trianglelefteq A_{\mathfrak{m}}:=\mathbf{T}_a(H^q(X_K,\sigma_{\Omega,\lambda}^{\circ}/\varpi^m)_{\mathfrak{m}})$
    satisfying $(T_n^2p^{N_2}J)^{N_1}=0$ such that
    \begin{itemize}
        \item $A_{\mathfrak{m}}/J$ is of Galois type with associated determinant $D_{\mathfrak{m}}:G_{F,S}\to A_{\mathfrak{m}}/J$ satisfying $\textnormal{LGC}(S,a,\lambda,\Omega,\mathcal{L})$.
    \end{itemize}

    We will now show that the desired property holds for each degree $q \in [0, d-1]$ (with fixed $K$, $\lambda$, $\Omega$ and $\mathcal{L}$) by descending induction on $q$. (Note that for $q$ outside the range $[0,d-1]$ there is nothing to prove as, $K$ being good, $X_K$ is a smooth real manifold of dimension $d-1$ and so its cohomology is concentrated in this range.)
     The proof proceeds by a list of steps eventually reducing the assertion to Lemma \ref{Lemma_BoundaryLGC}, Lemma \ref{lemma_interiorLGC} and Lemma \ref{lemma_interiorLGCdual}.
    \begin{enumerate}
        \item If $q < \frac{d-1}{2}$, our argument requires more involved reduction steps. First, in this case we consider the Poincar\'e duality isomorphism (respecting the action of $\mathbf{T}_a$ on the target through the restriction of the duality involution $\iota\colon\mathbf{T}_a\to \mathbf{T}^S\otimes_{\mathcal{O}}\mathcal{H}(\sigma_{\Omega_a,\lambda_a}^{\circ,\vee})$)
    \[H^q(X_K ,\sigma_{\Omega,\lambda}^{\circ}/\varpi^m ) \cong H^{d-1-q}_c(X_K, \sigma_{\Omega,\lambda}^{\circ,\vee} /\varpi^m)^\vee \]
    (where the outer $(\cdot)^\vee$ denotes Pontryagin dual).
    Further, there is an exact sequence
    \[ H^{d-2-q}(\partial X_K,\sigma_{\Omega,\lambda}^{\circ,\vee} /\varpi^m) \to H^{d-1-q}_c(X_K, \sigma_{\Omega,\lambda}^{\circ,\vee} /\varpi^m) \to  \]
    \begin{equation*}
        H^{d-1-q}(X_K, \sigma_{\Omega,\lambda}^{\circ,\vee} /\varpi^m).
    \end{equation*}
    By Poincar\'e duality, Lemma \ref{Lemma_BoundaryLGC} and induction (on $n$), the faithful Hecke algebra 
    \begin{equation*}
        \mathbf{T}_a(H^{d-2-q}(\partial X_K,\sigma_{\Omega,\lambda}^{\circ,\vee}/\varpi^m))_{\mathfrak{m}}
    \end{equation*} satisfies the assertion of the Proposition. Therefore, by Lemma \ref{lemma_shortexactsequencereduction} it suffices to consider $\mathbf{T}_a(H^{d-1-q}(X_K, \sigma_{\Omega,\lambda}^{\circ,\vee} /\varpi^m))_{\mathfrak{m}}$.
    On the other hand, we can find morphisms of $\cO[\GL_n(\cO_{F, S_p})]$-modules
    \[ \mathcal{V}_\lambda^\vee \to \mathcal{V}_{-w_0^G(\lambda)} \]
    and
    \[  \mathcal{V}_{-w_0^G(\lambda)} \to \mathcal{V}_\lambda^\vee \]
    whose composite is multiplication by a power of $p$ depending only on $\lambda$. In particular, when $q<\frac{d-1}{2}$, another application of Lemma \ref{lemma_shortexactsequencereduction} shows that it suffices to prove that 
    \begin{equation*}
        \mathbf{T}_a\big(H^{d-1-q}(X_K,\sigma_{\Omega}^{\circ,\vee}\otimes_{\cO}\mathcal{V}_{-w_0^G\lambda}/\varpi^m)\big)_{\mathfrak{m}}
    \end{equation*} 
    has the required property (where we use that $\iota(\mathbf{T}_a)\leq \mathcal{H}(\sigma_{\Omega^{\vee}_a,\lambda_a^{\vee}}^{\circ})$ by definition).
    \item Analogous to the twisting argument in the second and third paragraph of the proof of \cite[Theorem 4.5.1]{10author}, we can assume that the following are satisfied.
    \begin{itemize}
        \item If $q\geq \frac{d-1}{2}$, $\lambda$ is matched with a dominant weight in $(\mathbf{Z}_+^{2n})^{\widetilde{I}_p}$.
        \item If $q<\frac{d-1}{2}$, $-w_0^G\lambda$ is matched with a dominant weight in $(\mathbf{Z}_+^{2n})^{\widetilde{I}_p}$.
    \end{itemize}
    Indeed, it is easy to track the Hecke action and weight at $a$ throughout the argument of \cite{10author} and what it translates to on the Galois side, allowing one to verify both condition $(1)$ and $(2)$ of $\textnormal{LGC}(S,a,\lambda,\Omega,\mathcal{L})$.
    \item 
    We can replace $K$ by a good normal compact open subgroup $K'\trianglelefteq K$ with $K'^{\overline{S}\setminus\{\overline{a}\}}=K^{\overline{S}\setminus\{\overline{a}\}}$. Indeed, for $\mathcal{V}=\sigma_{\Omega,\lambda}^{\circ},\sigma_{\Omega}^{\circ,\vee}\otimes_{\cO}\mathcal{V}_{-w_0^G\lambda}$, the Hochschild--Serre spectral sequence for compactly supported cohomology, and Poincar\'e duality, together give a $\mathbf{T}_a$-equivariant spectral sequence
    \[ H^r(K/K', H^{d-1-s}(X_{K'}, \mathcal{V} / \varpi^m)^\vee) \Rightarrow H^{d-1-r-s}(X_K, \mathcal{V}/ \varpi^m)^\vee \]
    (where $(\cdot)^\vee$ denotes Pontryagin dual). In particular, $H^{q}(X_K, \mathcal{V} / \varpi^m)$ has a filtration whose graded pieces are subquotients of the groups 
    \[ H^s(K / K', H^{q + s}(X_{K'}, \mathcal{V} / \varpi^m)^\vee)^\vee, \] so Lemma \ref{lemma_shortexactsequencereduction} shows that, by induction, the desired property will hold for $H^q(X_K, \mathcal{V} / \varpi^m)$, provided that it holds for $H^q(X_{K'}, \mathcal{V} / \varpi^m)$.

    In particular, we can assume the following.
    \begin{itemize}
        \item For $\overline{v}\in \overline{S}_p\setminus\{\overline{a}\}$, $K_{\overline{v}}\leq \ker\big(G(\mathcal{O}_{F^+_{\overline{v}}})\to G(\mathcal{O}_{F^+_{\overline{v}}}/\varpi_{\overline{v}}^h)\big)$ for $h=\max (m,u_{\Omega_{\overline{v}}})$.
        \item There exists a good subgroup $\widetilde{K} \leq \widetilde{G}(\mathbf{A}_{F^+}^\infty)$ such that $K_P = K_U \rtimes K_G$ is decomposed with respect to $P$, $\widetilde{K}^S = \widetilde{G}(\widehat{\cO}_{F^+}^S)$, $K_G = K$, and $K_U = U(\widehat{\cO}_{F^+})$.
    \end{itemize}
    
    \item By Lemma \ref{lem_existence_of_cohomologically_cuspidal_weights_new}, we can find a cohomologically cuspidal weight $\widetilde{\mu} \in (\mathbf{Z}^{2n}_+)^{\widetilde{I}_p}$ with the following properties: 
    \begin{itemize}
        \item If $q\geq \frac{d-1}{2}$, $\widetilde{\mu}_{\overline{a}}$ is matched with $ \lambda_{\overline{a}}$.
        \item If $q<\frac{d-1}{2}$, $\widetilde{\mu}_{\overline{a}}$ is matched with $-w_0^{G}\lambda_{\overline{a}}$.
        \item If $\overline{v} \in \overline{S}_{p}\setminus\{\overline{a},\overline{b}\}$, then $\widetilde{\mu}_{\overline{v}} = 0$.
    \end{itemize}
    
    Choose $w \in W^P_{\overline{S}_{p}\setminus\{\overline{a},\overline{b}\}}$ with length $l(w)$ depending on $q$ as follows.
    \begin{enumerate}
        \item If $q\geq \frac{d-1}{2}$, $l(w)=d-q$.
        \item If $q< \frac{d-1}{2}$, $l(w)=d-(d-1-q)=q+1$.
    \end{enumerate}
    (The existence of $w$ follows from our assumption that \eqref{eqn_enough_shifting_new} holds, as in the proof of \cite[Proposition 4.4.1]{10author}.)
    Let $\mu \in (\mathbf{Z}^n_+)^{\Hom(F, E)}$ be the weight matched to $w \cdot \widetilde{\mu}$. Then, as our level at $\overline{S}_p\setminus\{\overline{a}\}$ is assumed to be sufficiently deep, for $q\geq \frac{d-1}{2}$ and $q<\frac{d-1}{2}$, there are identifications 
    \[ \mathbf{T}_a(H^q( X_K, \sigma_{\Omega,\lambda}^{\circ} / \varpi^m) )= \mathbf{T}_a(H^q(X_K,\sigma_{\Omega_{\overline{a}}}^{\circ}\otimes_{\cO} \mathcal{V}_\mu / \varpi^m)), \]
    and
    \begin{align*}
        &\mathbf{T}_a(H^{d-1-q}( X_K, \sigma_{\Omega}^{\circ,\vee}\otimes_{\cO}\mathcal{V}_{-w_0^G\lambda}/ \varpi^m) ) \\
        =& 
        \mathbf{T}_a(H^{d-1-q}(X_K,\sigma_{\Omega_{\overline{a}}}^{\circ,\vee}\otimes_{\cO} \mathcal{V}_{\mu} / \varpi^m)),
    \end{align*}
    respectively, where $\mathbf{T}_a$ at $a$ acts on the latter via the anti-involution
    \begin{equation*}
        \iota\colon\mathcal{H}(\sigma_{\Omega_{\overline{a}},\lambda_{\overline{a}}}^{\circ})\xrightarrow{\sim}\mathcal{H}(\sigma_{\Omega_{\overline{a}},\lambda_{\overline{a}}}^{\circ,\vee}).
    \end{equation*}
    
    \item Moreover, when $q\geq \frac{d-1}{2}$, we look at the short exact sequence of $\mathbf{T}_a$-modules
    \begin{align*}
        0\to H^q(X_K,\sigma_{\Omega_{\overline{a}}}^{\circ}\otimes_{\cO}\mathcal{V}_{\mu})/\varpi^m &\to H^q(X_K,\sigma_{\Omega_{\overline{a}}}^{\circ}\otimes_{\cO}\mathcal{V}_{\mu}/\varpi^m) \\ &\to 
        H^{q+1}(X_K,\sigma_{\Omega_{\overline{a}}}^{\circ}\otimes_{\cO}\mathcal{V}_{\mu})[\varpi^m]\to 0.
    \end{align*}
    As $H^{q+1}(X_K,\sigma_{\Omega_{\overline{a}}}^{\circ}\otimes_{\cO}\mathcal{V}_{\mu})$ is a finitely generated $\mathcal{O}$-module, there is an integer $m'\geq m$ such that the composition
    \begin{align*}
        H^{q+1}(X_K,\sigma_{\Omega_{\overline{a}}}^{\circ}\otimes_{\cO}\mathcal{V}_{\mu})[\varpi^m] &\hookrightarrow H^{q+1}(X_K,\sigma_{\Omega_{\overline{a}}}^{\circ}\otimes_{\cO}\mathcal{V}_{\mu}) \\ &\to
         H^{q+1}(X_K,\sigma_{\Omega_{\overline{a}}}^{\circ}\otimes_{\cO}\mathcal{V}_{\mu}/\varpi^{m'})
    \end{align*}
    is injective, where the second map is induced by multiplication by $\varpi^{m'}$. Thus, by the induction hypothesis and Lemma \ref{lemma_shortexactsequencereduction} it suffices to find $N_1,N_2$ and $J\trianglelefteq\mathbf{T}_a\Big(H^q(X_K,\sigma_{\Omega_{\overline{a}}}^{\circ}\otimes_{\cO}\mathcal{V}_{\mu})\Big) _{\mathfrak{m}}$ satisfying $(T_n^2p^{N_2}J)^{N_1}=0$ such that $\mathbf{T}_a\Big(H^q(X_K,\sigma_{\Omega_{\overline{a}}}^{\circ}\otimes_{\cO}\mathcal{V}_{\mu})\Big) _{\mathfrak{m}}/J$ is of Galois type and satisfies $\textnormal{LGC}(S,a,\lambda,\Omega,\mathcal{L})$.

    The analogous reduction applies just as well when $q<\frac{d-1}{2}$.
    \item Finally, when $q\geq \frac{d-1}{2}$, we can look at the exact sequence of $\mathbf{T}_a$-modules
    \begin{align*}
        0\to H^{q}_!(X_K,\sigma_{\Omega_{\overline{a}}}^{\circ}\otimes_{\cO}\mathcal{V}_{\mu}) &\to H^{q}(X_K,\sigma_{\Omega_{\overline{a}}}^{\circ}\otimes_{\cO}\mathcal{V}_{\mu}) \\ &\to
        H^{q}(\partial X_K,\sigma_{\Omega_{\overline{a}}}^{\circ}\otimes_{\cO}\mathcal{V}_{\mu}).
    \end{align*}
    In particular, using Lemma \ref{Lemma_BoundaryLGC}, and Lemma \ref{lemma_shortexactsequencereduction}, we see that it suffices to find integers $N_1,N_2$, ideal $J\trianglelefteq \mathbf{T}_a\Big(H^{q}_!(X_K,\sigma_{\Omega_{\overline{a}}}^{\circ}\otimes_{\cO}\mathcal{V}_{\mu})\Big)_{\mathfrak{m}}$ satisfying $(T_n^2p^{N_2}J)^{N_1}=0$
    such that $\mathbf{T}_a\Big(H^{q}_!(X_K,\sigma_{\Omega_{\overline{a}}}^{\circ}\otimes_{\cO}\mathcal{V}_{\mu})\Big)_{\mathfrak{m}}/J$ is of Galois type and satisfies $\textnormal{LGC}(S,a,\lambda,\Omega,\mathcal{L})$.

    Again, analogously, when $q< \frac{d-1}{2}$, it suffices to verify that 
    \begin{equation*}
        \mathbf{T}_a(H^{d-1-q}_!(X_K,\sigma_{\Omega_{\overline{a}}}^{\circ,\vee}\otimes_{\cO}\mathcal{V}_{\mu}))_{\mathfrak{m}}
    \end{equation*} has the required property.
    \end{enumerate}

The claim for the interior cohomology in $(6)$ now follows from Lemma \ref{lemma_interiorLGC} and Lemma \ref{lemma_interiorLGCdual}.    
\end{proof}

\section{Remarks on Weil--Deligne representations}\label{sec_nilp_prec}

In this section, we prove some simple results about Weil--Deligne representations, including Corollary \ref{cor_generic_is_maximal}, which shows that Theorem \ref{introthm_lgc} implies Corollary \ref{introcor_lgc_nilp}. Our approach is inspired by the proof of  \cite[Lemma 3.1.9]{acampo2024dworkmotivesmonodromypotential}.  

We work over $\mathbf{C}$, and consider finite-dimensional Weil--Deligne representations of $W_K$, where $K$ is a fixed non-archimedean local field. Let $\mathcal{I}$ denote the set of isomorphism classes of irreducible continuous representations of $W_K$ on finite-dimensional $\mathbf{C}$-vector spaces, and let $\mathcal{W}$ denote the set of equivalence classes in $\mathcal{I}$ under the relation of unramified twist. If $\omega \in \mathcal{W}$, and $(r, N)$ is a Weil--Deligne representation on $V$, then the $\omega$-isotypic part $V[\omega] \leq V$ is, by definition, the subspace generated by $W_K$-subrepresentations of $V$ that lie in $\omega$. It is preserved by $N$, so defines a sub-Weil--Deligne representation of $V$.
\begin{definition}
    Let $(r, N)$ be a Frobenius-semisimple Weil--Deligne representation on the finite-dimensional $\mathbf{C}$-vector space $V$. We define integers 
    \[ m_{1, \omega}(r, N) \geq m_{2, \omega}(r, N) \geq \dots \geq 0 \]
    by the requirement that there is an isomorphism
    \[ V[\omega] \cong \oplus_{s_i \in \omega} s_i \otimes \textnormal{Sp}(m_{i, \omega}(r, N)). \]
\end{definition}
The following is \cite[Definition 8.2]{Var24}.
\begin{definition}\label{def_prec}
    Let $(r, N)$, $(r', N')$ be Frobenius-semisimple Weil--Deligne representations. We write $(r, N) \prec (r', N')$ if for all $\omega \in \mathcal{W}$, and for all $i \geq 1$, we have
    \[ m_{1, \omega}(r, N) + \dots + m_{i, \omega}(r, N) \leq m_{1, \omega}(r', N') + \dots + m_{i, \omega}(r', N'). \]
\end{definition}
    Given a continuous semisimple representation $r : W_K \to \GL(V)$ on a finite-dimensional $\mathbf{C}$-vector space, we define (following \cite{Vog93}) $P(r) \leq \End(V)$ to be the subspace of endomorphisms $N : V \to V$ such that $(r, N)$ is a Weil--Deligne representation; equivalently, $P(r) = \Hom_{W_K}(V, V(-1))$.

    Let $G(r) = \operatorname{Cent}_{\GL(V)}(r(W_K))$. Then $G(r)$ is a (connected) reductive group, and it acts on $P(r)$. 
    \begin{lemma}\label{lem_description_of_closure_ordering}
        Let $(r, N)$, $(r, N')$ be Frobenius-semisimple Weil--Deligne representations on the finite-dimensional $\mathbf{C}$-vector space $V$. Suppose that $N$ lies in the closure of the $G(r)$-orbit of $N'$. Then $(r, N) \prec (r, N')$.
    \end{lemma}
    \begin{proof}
        If $\omega \in \mathcal{W}$, let $r[\omega]$ denote the representation of the Weil group on $V[\omega]$. Then $P(r) = \prod_\omega P(r[\omega])$ and $G(r) = \prod_\omega G(r[\omega])$ (and these product decompositions are compatible). If $N$ lies in the closure of the $G(r)$-orbit of $N'$, then for any $\omega \in \mathcal{W}$, $N[\omega] \in P(r[\omega])$ lies in the closure of the $G(r[\omega])$-orbit of $N'[\omega]$, and therefore in the closure of the $\GL(V[\omega])$-orbit of $N'[\omega]$. 
        
        It is well-known (cf. \cite[\S 7.8.1]{Bel09}) that if $M, M'$ are nilpotent endomorphisms on the finite-dimensional $\mathbf{C}$-vector space $W$, with Jordan blocks of size $t_1 \geq t_2 \geq \dots$ and $t_1' \geq t_2' \geq \dots$, then $M$ is in the closure of the $\GL(W)$-orbit of $M'$ if and only if for all $i \geq 1$ we have
        \[ t_1 + t_2 + \dots + t_i \leq t_1' + t_2' + \dots + t_i'.  \]
        The Jordan blocks of $N[\omega]$ have size $m_{1, \omega}(r, N), \dots, m_{1, \omega}(r, N)$ ($\dim s_1$ times), $m_{2, \omega}(r, N), \dots, m_{2, \omega}(r, N)$ ($\dim s_2 = \dim s_1$ times), and so on. We see that we must have $(r, N) \prec (r, N')$, as required.
    \end{proof}
\begin{definition}
    Let $(r, N)$ be a Frobenius-semisimple Weil--Deligne representation on the finite-dimensional $\mathbf{C}$-vector space $V$. We say that $(r, N)$ is generic if $\Hom_{\mathrm{WD}}((r, N), (r(1), N)) = 0$.
\end{definition}
\begin{prop}\label{prop_unique_generic}
    Let $V$ be a finite-dimensional $\mathbf{C}$-vector space, and let $r : W_K \to \GL(V)$ be a continuous semisimple representation. Then:
    \begin{enumerate}
        \item $G(r)$ has finitely many orbits in $P(r)$. In particular, there is a unique open orbit.
        \item If $N \in P(r)$, then $N$ lies in the unique open orbit if and only if $(r, N)$ is a generic Weil--Deligne representation.
        \item Suppose that $N, N' \in P(r)$ and that $(r, N')$ is a generic Weil--Deligne representation. Then $(r, N) \prec (r, N')$.
    \end{enumerate}
\end{prop}
\begin{proof}
    The first part is \cite[Proposition 4.5(4)]{Vog93}. For the second, we observe that the generic locus (as the locus where the map $\ad N : \Hom_{W_K}(r, r(1)) \to \End_{W_K}(V)$ is injective) is open in $P(r)$. It is also non-empty, and forms a single $G(r)$-orbit, by \cite[Lemma 3.1.9]{acampo2024dworkmotivesmonodromypotential}. We conclude that the unique open orbit in $P(r)$ consists precisely of those $N$ such that $(r, N)$ is generic. For the third, we observe that the first and second parts show that $N$ is in the closure of the generic orbit; we can then apply Lemma \ref{lem_description_of_closure_ordering}.
\end{proof}
\begin{corollary}\label{cor_generic_is_maximal}
    Let $n \geq 1$, let  $\pi$ be a generic irreducible admissible $\mathbf{C}[\GL_n(K)]$-module, and let $(r, N)$ be a Weil--Deligne representation on $\mathbf{C}^n$ such that $(r, N)^{\textnormal{ss}} \cong \rec_K(\pi)^{\textnormal{ss}}$. Then $(r, N)^{\textnormal{F-ss}} \prec \rec_K(\pi)$. 
\end{corollary}
\begin{proof}
    By \cite[Lemma 1.1.3]{All16}, the genericity of $\pi$ is equivalent to that of $\rec_K(\pi)$. The result therefore follows from Proposition \ref{prop_unique_generic}. 
\end{proof}

\bibliographystyle{alpha}
\bibliography{refs}
\end{document}